%% file: wstarint-arxiv.tex
\documentclass[11pt,a4paper,onecolumn,notitlepage]{article}
\let\Horig\H
\let\Oorig\O
\usepackage{cmap} 
\usepackage{refcount}
\usepackage[pdfpagemode=None,
pagebackref=true,
bookmarksopen=false,
hyperindex=true,
colorlinks=true,
linkcolor=blue,
citecolor=blue,
pdftitle={W*-algebras and noncommutative integration},
pdfauthor={Ryszard Pawel Kostecki}]{hyperref}
\usepackage[utf8x]{inputenc}
\usepackage[LGR,OT2,T1]{fontenc}
\usepackage[polutonikogreek,english]{babel}
\usepackage{CJK}
\usepackage{cyrillic}
\usepackage[all]{xy}
\usepackage{amsmath}
\usepackage{amssymb}
\usepackage{amsthm}
\usepackage{mathrsfs} 
\usepackage{upgreek} 
\usepackage{makeidx}
\usepackage{idxlayout}
\usepackage{sectsty}
\usepackage{bibspacing}
\usepackage{rpk}
\usepackage{motta}
\usepackage{tocloft}
\usepackage[pdftex,final]{graphicx}
\renewcommand*{\backref}[1]{}
\renewcommand*{\backrefalt}[4]{%
  \ifcase #1 %
    \relax
  \or
    $\uparrow$~#2.
  \else
    $\uparrow$~#2.
  \fi%
}

\newcommand{\cyrrm}[1]{\mbox{\fontencoding{OT2}\fontfamily{wncyr}\selectfont#1}} 
\makeindex
\begin{document}
\thispagestyle{empty}
\setlength\cftparskip{0pt}
\setlength\cftbeforesecskip{1pt}
\setlength\cftaftertoctitleskip{10pt}
\selectlanguage{english}
\setcounter{tocdepth}{4}
\setcounter{secnumdepth}{4}
\renewcommand*{\thefootnote}{\fnsymbol{footnote}}
\begin{center}\vspace*{-0.8cm}\hspace*{-0.66cm}\makebox{\Huge\textbf{$W^*$-algebras and noncommutative integration}}{\tiny\\\ \\\ \\}
{\LARGE Ryszard Pawe{\l} Kostecki{\tiny\\\ \\}{\small\textit{Perimeter Institute of Theoretical Physics, 31 Caroline Street North, Waterloo, Ontario N2L 2Y5, Canada}\footnote{Current affiliation.}}\\{\vskip -0.3cm}{\small\textit{Institute of Theoretical Physics, University of Warsaw, Ho\.za 69, 00-681 Warszawa, Poland}}\\{\small\vspace{3mm}\texttt{ryszard.kostecki@fuw.edu.pl}}}{\small\\\ \\}
{October 27, 2014}
\end{center}
\begin{abstract}
\noindent This text is a detailed overview of the theories of $W^*$-algebras and noncommutative integration, up to the Falcone--Takesaki theory of noncommutative $L_p$ spaces over arbitrary $W^*$-algebras, and its extension to noncommutative Orlicz spaces. The topics under consideration include the Tomita--Takesaki modular theory, the relative modular theory (featuring bimodules, spatial quotients, and canonical representation), the theory of $W^*$-dynamical systems (featuring derivations, liouvilleans, and crossed products), noncommutative Radon--Nikod\'{y}m type theorems, and operator valued weights. We pay special attention to abstract algebraic formulation of all properties (avoiding the dependence on Hilbert spaces wherever it is possible), to functoriality of canonical structures arising in the theory, and to the relationship between commutative and noncommutative integration theories. Several new results are proved.
\end{abstract}
\setcounter{footnote}{0}
\renewcommand*{\thefootnote}{\arabic{footnote}}
{\small
\tableofcontents}
\clearpage
\newif\ifvarwstarintcompile
\varwstarintcompiletrue
\input{chapter.2-algebraic.tex}

\section*{References}
\addcontentsline{toc}{section}{References}
{%
\scriptsize

Note: All links provided in references link to the free access files. Files and digital copies that are subject to any sort of restricted access were not linked. See \href{http://michaelnielsen.org/polymath1/index.php?title=Journal_publishing_reform}{michaelnielsen.org/polymath1/index.php?title=Journal\_publishing\_reform} for the reasons why.

\begingroup
\raggedright
\renewcommand\refname{\vskip -1cm}

\endgroup        
}%
\newpage
\phantomsection
\addcontentsline{toc}{section}{Index}
{\footnotesize\printindex}
\section*{List of symbols}
\addcontentsline{toc}{section}{List of symbols}
\sectionmark{List of symbols}

\begin{tabular}{ccc}
\begin{tabular}{rl}
$\C$&\rpklink{CSTAR}\\
$\II$&\rpklink{UNIT}\ \rpklink{DANIELUNIT}\\
$x^*$&\rpklink{STAR}\\
$\C^o$&\rpklink{OPPOSITE}\\
$\Aut$&\rpklink{AUT}\\
$\C_\alpha$&\rpklink{FIXEDPOINT}\\
$\leq$&\rpklink{PORDER}\ \rpklink{PHILEQ}\ \rpklink{POSET}\\
$\Proj$&\rpklink{PROJ}\\
$(\cdot)^+$&\rpklink{POSITIV}\\
$(\cdot)^\sa$&\rpklink{SA}\\
$(\cdot)^\asa$&\rpklink{ASA}\\
$(\cdot)^\uni$&\rpklink{UNI}\\
$(\cdot)^\banach$&\rpklink{BANACH}\ \rpklink{BANACH.FUNCTOR}\\
$(\cdot)_\star$&\rpklink{PREDUAL}\ \rpklink{PREDUAL.FUNCTOR}\\
$(\cdot)_0$&\rpklink{ZERO}\\
$(\cdot)_1$&\rpklink{ONE}\\
$\Scal$&\rpklink{SCAL}\\
$\Scal_0$&\rpklink{SZERO}\\
$\N$&\rpklink{N}\\
$\BH$&\rpklink{BH}\\
$\mmm_\omega$&\rpklink{mmm}\\
$\W$&\rpklink{WC}\ \rpklink{MEASURE.W}\\
$\W_0$&\rpklink{WZEROC}\ \rpklink{MEASURE.W.ZERO}\\
$\nnn_\phi$&\rpklink{nnn}\\
$\supp$&\rpklink{SUPP}\ \rpklink{SUPP.PSI}\ \rpklink{SUPP.DREI}\\
$\ll$&\rpklink{ll}\ \rpklink{MEASURE.LL}\\
$\ab{\phi}$&\rpklink{AB.PHI}\\
$\pi$&\rpklink{pi}\\
$\pi_\omega$&\rpklink{pi.omega}\ \rpklink{pi.omega.zwei}\\
$\Omega_\omega$&\rpklink{omega.omega}\\
$\H_\omega$&\rpklink{h.omega}\ \rpklink{h.omega.zwei}\\
$\s{\cdot,\cdot}_\omega$&\rpklink{scal.omega}\\
$[\cdot]_\omega$&\rpklink{rep.omega}\ \rpklink{rep.omega.zwei}\\
$\N^\comm$&\rpklink{comm}\\
$\zentr_\N$&\rpklink{zentr}\\
$\face$&\rpklink{face}\\
$\co$&\rpklink{co}\\
$\bary$&\rpklink{BARY}\\
$\fact$&\rpklink{FACT}\\
$\pvm^x$&\rpklink{pvm}\ \rpklink{pvm.nelson.nota}\\
$\povm$&\rpklink{povm}\\
$\C^\alpha_\infty$&\rpklink{CAINF}\\
$\Lin$&\rpklink{LIN}\\
$J_\omega$&\rpklink{J}\ \rpklink{J.ZWEI}\ \rpklink{J.DREI}\\
$\Delta_\omega$&\rpklink{DELTA}\ \rpklink{DELTA.ZWEI}\ \rpklink{DELTA.DREI}\\
$\sigma^\omega$&\rpklink{SIGMA}\ \rpklink{SIGMA.DREI}\ \rpklink{TT.MODALG}\\
$\N_{\sigma^\omega}$&\rpklink{NSG}\\
$K_\omega$&\rpklink{KOMEGA}\\
$\stdcone_\Omega$&\rpklink{STDCONE.OMEGA}\\
$\stdcone$&\rpklink{STDCONE}
\end{tabular}
&
\begin{tabular}{rl}
$\xi_\pi$&\rpklink{STDREP}\\
$\xi_\natural^\pi$&\rpklink{STDREP.REVERSE}\\
$J_{\phi,\omega}$&\rpklink{JREL}\\
$\Delta_{\phi,\omega}$&\rpklink{DELTAREL}\ \rpklink{DELTA.REL.ZWEI}\ \rpklink{DELTA.REL.CANON}\\
$V_{\phi,\omega}$&\rpklink{STDUNITRANS}\\
$\Connes{\phi}{\omega}{t}$&\rpklink{CONNES.COC}\ \rpklink{CONNES.COC.ZWEI}\ \rpklink{CONNES.COC.DREI}\ \rpklink{CONNES.COC.VIER}\ \rpklink{CONNES.MODALG}\\
$\connes{\psi}{\phi^\comm}$&\rpklink{CONNES.SPAT}\ \rpklink{CONNES.SPAT.CANON}\ \rpklink{CONN.SPAT.MODALG}\\
$J_\N$&\rpklink{CANON.J}\\
$\pi_\N$&\rpklink{CANON.PI}\\
$\LLL$&\rpklink{LEFT}\ \rpklink{LLL.SHERMAN}\\
$\RRR$&\rpklink{RIGHT}\ \rpklink{RRR.SHERMAN}\\
$\dimfun$&\rpklink{DIMFUN}\\
$\modspec$&\rpklink{MODSPEC}\ \rpklink{MODSPEC.FLOW}\\
$\alpha$&\rpklink{ALPHA}\\
$\der$&\rpklink{DER}\\
$V_\alpha$&\rpklink{V.ALPHA}\\
$K^\alpha$&\rpklink{K.ALPHA}\\
$\tmu^G$&\rpklink{HAAR}\\
$\mathrm{C}_\mathrm{c}$&\rpklink{Cc}\ \rpklink{CC.RIESZ.THM}\\
$\pi_\alpha$&\rpklink{PI.ALPHA}\\
$u_G$&\rpklink{UGIE}\\
$\rtimes_\alpha$&\rpklink{CROSSEDPROD}\\
$\hat{G}$&\rpklink{PONTR.DUAL}\\
$\pi_{\sigma^\psi}$&\rpklink{PI.SIGMA}\\
$u_\RR$&\rpklink{URR}\\
$\phi^\lambda$&\rpklink{PHI.ONE.HALF}\ \rpklink{PHI.IT}\ \rpklink{PSI.IT}\ \rpklink{PHIp}\ \rpklink{PHI.IT.DREI}\ \rpklink{PHI.IT.MODALG}\\
$\precore$&\rpklink{PRECORE}\\
$\N(t)$&\rpklink{NT}\\
$\H(t)$&\rpklink{HT}\\
$\core$&\rpklink{CORE}\\
$\widetilde{\H}$&\rpklink{COREH}\\
$\aff$&\rpklink{AFF}\ \rpklink{AFF.ZWEI}\\
$\MMM(\N)$&\rpklink{MMM}\\
$\tau_h$&\rpklink{TAU.H}\ \rpklink{TAU.H.ZWEI}\\
$\phi_h$&\rpklink{PHI.H}\\
$L_p(\N,\tau)$&\rpklink{LPNTAU}\ \rpklink{LPNTAU.ZWEI}\\
$\MMM(\N,\tau)$&\rpklink{MMMNTAU}\\
$\schatten_p$&\rpklink{SCHATT}\ \rpklink{SCHATT.P}\\
$\N^\ext$&\rpklink{NEXT}\\
$\tilde{\tau}$&\rpklink{TILDETAU}\\
$\tilde{\phi}$&\rpklink{TILDEPHI}\\
$\tilde{\tau}_x$&\rpklink{TILDETAUX}\\
$T$&\rpklink{T.MAP}\\
$\nnn_T$&\rpklink{nnnT}\\
$\mmm_T$&\rpklink{mmmT}\\
$\condexp$&\rpklink{condexp}\\
$\hat{T}$&\rpklink{HATT}\ \rpklink{HATT.ZWEI}\\
$\hat{\phi}$&\rpklink{HATPHI}\ \rpklink{HATPHI.ZWEI}\ \rpklink{DUALWEIGHT}\\
$h_\phi$&\rpklink{H.PHI}\\
$\MMM^p(\precore,\tau_\psi)$&\rpklink{MMMPPRECORE}
\end{tabular}
&
\begin{tabular}{rl}
$L_p(\N,\psi)$&\rpklink{HTLP}\ \rpklink{AMLP}\\
$L_p(\N,\psi^\comm)$&\rpklink{CHLP}\\
$\duality{\cdot,\cdot}_\psi$&\rpklink{AMDUALITY}\\
$L_p(\N)$&\rpklink{CANON.HILB}\ \rpklink{LPNKOSAKI}\ \rpklink{LPNFT}\ \rpklink{LITNFT}\\
$\grad$&\rpklink{GRAD}\\
$\MMM^p(\core,\taucore)$&\rpklink{MMMPCORE}\\
$\sup$&\rpklink{SUP}\\
$\inf$&\rpklink{INF}\\
$\boole$&\rpklink{BOOLE}\\
$\sp_\mathrm{S}$&\rpklink{STONESPEC}\\
$\mu$&\rpklink{MEASURE}\\
$L_p(\boole)$&\rpklink{LPBINF}\ \rpklink{LPBZERO}\ \rpklink{LPBN}\\
$L_p(\boole,\mu)$&\rpklink{LPBM.ONE}\ \rpklink{LPBM}\\
$\chi$&\rpklink{CHICHAR}\\
$\boole^\mu$&\rpklink{BOOLMEASIDEAL}\\
$\mho$&\rpklink{MHO}\\
$\mho^0$&\rpklink{MHO.ZERO}\\
$\tmu$&\rpklink{SET.MEASURE}\\
$\Meas^+$&\rpklink{SET.MEASURE.SET}\\
$\nul$&\rpklink{NULL}\\
$\mho^{\tmu}$&\rpklink{MHO.TMU}\\
$L_p(\X,\mho(\X),\tmu)$&\rpklink{LPMEASURESET}\\
$\boole_\tmu$&\rpklink{BOOLETMUALG}\\
$\mho_{\mathrm{Borel}}(\X)$&\rpklink{BOREL}\\
$\Meas^+_\star$&\rpklink{NORMAL.RADON}\\
$\Rad(\cdot)^+$&\rpklink{RADON.BARY}\ \rpklink{RADON}\\
$\sp_{\mathrm{G}}$&\rpklink{GELFSPEC}\\
$\mathrm{C}$&\rpklink{CONT.BARY}\ \rpklink{CONTIFUN}\ \rpklink{CONTI.RIESZ.THM}\\
$\Orlicz$&\rpklink{ORLICZFUN}\\
$L_\Orlicz(\X,\mho(\X),\tmu)$&\rpklink{LORLICZX}\\
$E_\Orlicz(\X,\mho(\X),\tmu)$&\rpklink{EORLICZX}\\
$\schatten_\Orlicz$&\rpklink{SCHATTENORLICZ}\\
$\rearr{x}{\tau}$&\rpklink{REARRANGEMENT}\\
$L_\Orlicz(\N,\tau)$&\rpklink{LORLICZNT}\ \rpklink{LORLICZNT.ZWEI}\\
$E_\Orlicz(\N,\tau)$&\rpklink{EORLICZNT}\\
$L_\Orlicz(\N)$&\rpklink{LORLICZN}\\
$E_\Orlicz(\N)$&\rpklink{EORLICZN}
\end{tabular}
\end{tabular}
\end{document}

%% file: chapter.2-algebraic.tex
\ifvarwstarintcompile 
\mottotakesakidc
\vspace{-1cm}
\section{Introduction}

While there exists a wide range of detailed expositions of various aspects of the theory of noncommutative algebras of operators, to mention only \cite{Kaplansky:1955,Naimark:1956,Dixmier:1957,Rickart:1960,Dixmier:1964,Schwartz:1967,Topping:1971,Sakai:1971,Bonsall:Duncan:1973,Guichardet:1974,Stratila:Zsido:1975,Arveson:1976,Stratila:Zsido:1977:1979,vanDaele:1978,Pedersen:1979,Stratila:1981,Kadison:Ringrose:1983:1986:1991:1992,Sunder:1986,Tomiyama:1987,Fell:Doran:1988,Murphy:1990,Sakai:1991,WeggeOlsen:1993,Palmer:1994,Connes:1994,Evans:Kawahigashi:1998,Takesaki:2003,Blackadar:2005}, there is no self-contained text covering the theory of noncommutative integration over arbitrary $W^*$-algebras.\footnote{The canonical references for the theory of $W^*$-algebras are \cite{Sakai:1971,Stratila:1981,Takesaki:2003}, and we definitely recommend them.} Our exposition is intended to fill this gap.\footnote{On this occasion, we fill some gaps in the literature, clarifying the functorial character of some constructions. The category theoretic formulation of equivalences \eqref{Sakai.Kosaki.duality}, \eqref{sherman.equivalence}, and the diagrams \eqref{ctft.cat.diag}, \eqref{ncLp.cat.diag} seem to be new.} The theory of $W^*$-algebras is presented in an algebraic fashion, with elimination of dependence on the Hilbert spaces, and downplaying the usual noncommutative topological ($C^*$-algebraic) point of view in favour of exposition of the role played by predualisation and relative modular theory. Our presentation of noncommutative integration covers large parts of the theory that are otherwise scattered among many papers and books. In order to provide an overview that is detailed but also has a reasonable size, we have omitted proofs. Yet, we compensate for this by discussion of notions and results under consideration and by providing quite detailed references to original papers.

The new mathematical results of this text are: the construction of the family of noncommutative Orlicz spaces $L_\Upsilon(\N)$ canonically associated with arbitrary $W^*$-algebra $\N$ and arbitrary Orlicz function $\Upsilon$, the construction of the family of commutative $L_p(\boole)$ spaces canonically associated with arbitrary mcb-algebra\footnote{That is, the Dedekind--MacNeille complete boolean algebras admitting a strictly positive semi-finite countably additive measure (see Section \ref{comm.integr.section} for definitions of these terms).} $\boole$ for $p\in\,[1,\infty[$, as well as Propositions \ref{proposition.L.gamma.is.abstract}-\ref{huge.diagram.int.theory}. The category theoretic formulation of equivalences \eqref{bLinf.B.equiv}-\eqref{Linfc.mcBc.equiv}, \eqref{mcBc.locMesp.duality}, \eqref{mcBc.hypso.duality}, \eqref{mcBIso.hypsh.duality}, \eqref{Stratila.Zsido.duality}, \eqref{long.segal.equivalence}, and the diagram \eqref{commutative.integration.diagram} are also new. These results allow us to establish an explicit relationship between noncommutative and commutative integration theories without passing to topological representations (see Section \ref{integr.compar.section}), and to provide the category theoretic description of the main structures of integration theory in terms of the diagram \eqref{canonical.integration.theory.big.stuff}.

\vspace{0.3cm}{\small\noindent\textbf{Aknowledgments:} I am very indebted to Professor Stanis{\l}aw L. Woronowicz for numerous discussions that introduced me to the worlds of modular theory and noncommutative integration theory. I would like also to thank Jan Derezi\'{n}ski, Wojciech Kami\'{n}ski, Jerzy Kijowski, W{\l}adys{\l}aw Adam Majewski, Mustafa Muratov, W{\l}odzimierz Natorf, Dmitri\u{\i} Pavlov, Aleksander Pe{\l}czy\'{n}ski, David Sherman, and Piotr So{\l}tan for valuable comments, discussions, xerocopies, and/or correspondence. Naturally, all eventual drawbacks of this text are my own responsibility. This research was supported in part by Perimeter Institute for Theoretical Physics. Research at Perimeter Institute is supported by the Government of Canada through Industry Canada and by the Province of Ontario through the Ministry of Research and Innovation.}

\vspace{0.3cm}{\small\noindent\textbf{Conventions:} $\ii$, $\ee$, $\pipi$, $\RR$, $\CC$, $\ZZ$ are standard; $\RR^+\equiv[0,\infty[$, $\NN\equiv\{1,2,\ldots\}$; $\id$ denotes identity morphism in the category of objects of a given kind, $\dom$ denotes domain, $\cod$ denotes codomain, $\ran$ denotes image, $\Span_\KK$ denotes linear span in the vector space over $\KK\in\{\RR,\CC\}$, $\INT$ denotes interior of subspace of a topological space, $\Set$ denotes the category of sets and functions; given a small category $\mathscr{C}$, $\Ob(\mathscr{C})$ denotes the set of all objects in $\mathscr{C}$, $\Mor(\mathscr{C})$ denotes the set of all arrows in $\mathscr{C}$, while $\Hom_\mathscr{C}(X,Y)$ denotes the set of all arrows between $X,Y\in\Ob(\mathscr{C})$; $\dirac_{ij}$ denotes Kronecker's delta; $\sp$ denotes spectrum of an operator on a Hilbert space.  All Cyrillic titles and names were transliterated from the original papers and books. For the Latin transliteration of the Cyrillic script (in references and surnames) we use the following modification of the system GOST 7.79-2000B: {\cyrrm{ts}} = c, {\cyrrm{ch}} = ch, {\cyrrm{kh}} = kh, {\cyrrm{zh}} = zh, {\cyrrm{sh}} = sh, {\cyrrm{shch}} = shh, {\cyrrm{yu}} = yu, {\cyrrm{ya}} = ya, {\cyrrm{\"{e}}} = \"{e}, {\cyrrm{\cdprime}} = `, {\cyrrm{\cprime}} = ', {\cyrrm{\`{e}}} = \`{e}, {\cyrrm{\u{i}}} = \u{\i}, with an exception that names beginning with {\cyrrm{Kh}} are transliterated to H.\footnote{This is required for agreement with the widespread practice to transliterate {\cyrrm{Kholevo}} as Holevo, etc.} For Russian texts: {\cyrrm{y}} = y, {\cyrrm{i}} = i; for Ukrainian: {\cyrrm{i}} = y, i = i, \"{\i} = \"{\i}.}%
%
\clearpage
\else 
\chapter{$W^*$-algebras and noncommutative integration\label{alg.chapter}}
\mottotakesaki

This Chapter contains a detailed overview of the theories of $W^*$-algebras and noncommutative integration, up to the Falcone--Takesaki theory of noncommutative $L_p$ spaces over arbitrary $W^*$-algebras. The topics under consideration include the Tomita--Takesaki modular theory, the relative modular theory (featuring bimodules, spatial quotients, and canonical representation), the theory of $W^*$-dynamical systems (featuring derivations, liouvilleans, and crossed products), noncommutative Radon--Nikod\'{y}m type theorems, and operator valued weights. We pay special attention to abstract algebraic formulation of all properties (avoiding the dependence on Hilbert spaces wherever it is possible), to functoriality of canonical structures arising in the theory, and to the relationship between commutative and noncommutative integration theories.

While there exists a wide range of detailed expositions of various aspects of the theory of noncommutative algebras of operators, to mention only \cite{Kaplansky:1955,Naimark:1956,Dixmier:1957,Rickart:1960,Dixmier:1964,Schwartz:1967,Topping:1971,Sakai:1971,Bonsall:Duncan:1973,Guichardet:1974,Stratila:Zsido:1975,Arveson:1976,Stratila:Zsido:1977:1979,vanDaele:1978,Pedersen:1979,Stratila:1981,Kadison:Ringrose:1983:1986:1991:1992,Sunder:1986,Tomiyama:1987,Fell:Doran:1988,Murphy:1990,Sakai:1991,WeggeOlsen:1993,Palmer:1994,Connes:1994,Evans:Kawahigashi:1998,Takesaki:2003,Blackadar:2005}, there is no self-contained text covering the theory of noncommutative integration over arbitrary $W^*$-algebras.\footnote{The canonical references for the theory of $W^*$-algebras are \cite{Sakai:1971,Stratila:1981,Takesaki:2003}, and we definitely recommend them.} Our exposition is intended to fill this gap.\footnote{On this occasion, we fill some gaps in the literature, clarifying the functorial character of some constructions. The category theoretic formulation of equivalences \eqref{Sakai.Kosaki.duality}, \eqref{sherman.equivalence}, and the diagrams \eqref{ctft.cat.diag}, \eqref{ncLp.cat.diag} seem to be new.} The theory of $W^*$-algebras is presented in an algebraic fashion (with elimination of dependence on the Hilbert spaces), and with downplaying the usual noncommutative topological ($C^*$-algebraic) point of view in favour of exposition of the role played by predualisation and relative modular theory. Our presentation of noncommutative integration covers large parts of the theory that are otherwise scattered among many papers and books. In order to provide an overview that is detailed but also has a reasonable size, we have omitted proofs. Yet, we compensate for this by discussion of notions and results under consideration and by providing quite detailed references to original papers.

The new mathematical results contained in this Chapter are: the construction of the family of noncommutative Orlicz spaces $L_\Upsilon(\N)$ canonically associated with arbitrary $W^*$-algebra $\N$ and arbitrary Orlicz function $\Upsilon$, the construction of the family of commutative $L_p(\boole)$ spaces canonically associated with arbitrary mcb-algebra\footnote{That is, the Dedekind--MacNeille complete boolean algebras admitting a strictly positive semi-finite countably additive measure (see Section \ref{comm.integr.section} for definitions of these terms).} $\boole$ for $p\in\,[1,\infty[$, as well as Propositions \ref{proposition.L.gamma.is.abstract}-\ref{huge.diagram.int.theory}. The category theoretic formulation of equivalences \eqref{bLinf.B.equiv}-\eqref{Linfc.mcBc.equiv}, \eqref{mcBc.locMesp.duality}, \eqref{mcBc.hypso.duality}, \eqref{mcBIso.hypsh.duality}, \eqref{Stratila.Zsido.duality}, \eqref{long.segal.equivalence}, and the diagram \eqref{commutative.integration.diagram} are also new. These results allow us to establish an explicit relationship between noncommutative and commutative integration theories without passing to topological representations (see Section \ref{integr.compar.section}), and to provide the category theoretic description of the main structures of integration theory in terms of the diagram \eqref{canonical.integration.theory.big.stuff}.
\fi 
\section{Algebras and functionals\label{alg.fun.section}}
\ifvarwstarintcompile 
The theory of abstract $C^*$- and $W^*$- algebras emerged from the works of von Neumann and Murray \cite{vonNeumann:1930:algebra,Murray:vonNeumann:1936,vonNeumann:1936:topology,Murray:vonNeumann:1937,vonNeumann:1943,Murray:vonNeumann:1943,vonNeumann:1949} on von Neumann algebras, and the works of Gel'fand, Na\u{\i}mark, Ra\u{\i}kov and Shilov on Banach algebras and $C^*$-algebras \cite{Gelfand:1939,Gelfand:1941,Gelfand:1941:Ideale,Gelfand:Shilov:1941,Gelfand:Naimark:1943,Gelfand:Raikov:1943,Raikov:1946,Gelfand:Raikov:Shilov:1946,Gelfand:Naimark:1948}. The most important results obtained at that early stage (up to 1950) were von Neumann's reduction theorem \cite{vonNeumann:1949} allowing to represent separable all von Neumann algebras in terms of factor von Neumann algebras, Gel'fand--Na\u{\i}mark \cite{Gelfand:Naimark:1943} representation theory for commutative $C^*$-algebras, and its generalisation to noncommutative $C^*$-algebras by Segal \cite{Segal:1947:irreducible}. Among key results obtained later (up to early 1970s) within the frames of a general theory, it is necessary to mention Sakai's characterisation of $W^*$-algebras in purely algebraic terms \cite{Sakai:1956}, development of the theory of weights by Combes and Pedersen \cite{Combes:1966,Combes:1967,Combes:1968,Combes:1970,Combes:1971,Combes:1971:esperances,Pedersen:1966,Pedersen:1968,Pedersen:1969:III,Pedersen:1969:IV,Pedersen:1971}, and development of the theory of integral decomposition of functionals on $C^*$-algebras by Tomita \cite{Tomita:1953,Tomita:1956}, Ruelle \cite{Ruelle:1966,Ruelle:1970:integral}, Sakai \cite{Sakai:1965:central,Sakai:1971}, Wils \cite{Wils:1968,Wils:1969,Wils:1971} and others \cite{Segal:1951:decompositions,Kastler:Robinson:1966,Skau:1975} using the Choquet theory of boundary integrals \cite{Choquet:1956:I,Choquet:1956:II,Choquet:1956:III,Choquet:1956:IV,Bishop:deLeeuw:1959,Choquet:1960}.
\else 
The theory of abstract $C^*$- and $W^*$- algebras emerged from the works of von Neumann and Murray \cite{vonNeumann:1930:algebra,Murray:vonNeumann:1936,vonNeumann:1936:topology,Murray:vonNeumann:1937,vonNeumann:1943,Murray:vonNeumann:1943,vonNeumann:1949} on von Neumann algebras, and the works of Gel'fand, Na\u{\i}mark, Ra\u{\i}kov and Shilov on Banach algebras and $C^*$-algebras \cite{Gelfand:1939,Gelfand:1941,Gelfand:1941:Ideale,Gelfand:Shilov:1941,Gelfand:Naimark:1943,Gelfand:Raikov:1943,Raikov:1946,Gelfand:Raikov:Shilov:1946,Gelfand:Naimark:1948}. The most important results obtained at that early stage (up to 1950) were von Neumann's reduction theorem \cite{vonNeumann:1949} allowing to represent all separable von Neumann algebras in terms of factor von Neumann algebras, Gel'fand--Na\u{\i}mark \cite{Gelfand:Naimark:1943} representation theory for commutative $C^*$-algebras, and its generalisation to noncommutative $C^*$-algebras by Segal \cite{Segal:1947:irreducible}. Among key results obtained later (up to early 1970s) within the frames of a general theory, it is necessary to mention Sakai's characterisation of $W^*$-algebras in purely algebraic terms \cite{Sakai:1956}, development of the theory of weights by Combes and Pedersen \cite{Combes:1966,Combes:1967,Combes:1968,Combes:1970,Combes:1971,Combes:1971:esperances,Pedersen:1966,Pedersen:1968,Pedersen:1969:III,Pedersen:1969:IV,Pedersen:1971}, and development of the theory of integral decomposition of functionals on $C^*$-algebras by Tomita \cite{Tomita:1953,Tomita:1956}, Ruelle \cite{Ruelle:1966,Ruelle:1970:integral}, Sakai \cite{Sakai:1965:central,Sakai:1971}, Wils \cite{Wils:1968,Wils:1969,Wils:1971} and others \cite{Segal:1951:decompositions,Kastler:Robinson:1966,Skau:1975} using the Choquet theory of boundary integrals \cite{Choquet:1956:I,Choquet:1956:II,Choquet:1956:III,Choquet:1956:IV,Bishop:deLeeuw:1959,Choquet:1960}.
\fi 
\subsection{$C^*$-algebras\label{C.star.section}}
A \df{$C^*$-algebra} \cite{Gelfand:Naimark:1943} is defined as an algebra $\rpktarget{CSTAR}\C$ over the field $\CC$ with unit $\rpktarget{UNIT}\II$, equipped with:
\begin{enumerate}
\item[(i)] an operation $\rpktarget{STAR}^*:\C\ra\C$ satisfying
\begin{equation}
        (xy)^*=y^*x^*,\;\;(x+y)^*=x^*+y^*,\;\;(x^*)^*=x,\;\;(\lambda x)^*=\lambda^*x^*\;\;\forall x,y\in\C\;\forall\lambda\in\CC,
\label{star.operation.def}
\end{equation}
with the action of $^*$ on $\lambda\in\CC$ given by standard complex conjugation,
\begin{equation}
        ^*:\CC\ni\lambda=(\re(\lambda)+\ii\cdot\im(\lambda))\mapsto\lambda^*:=(\re(\lambda)-\ii\cdot\im(\lambda))\in\CC,
\end{equation}
\item[(ii)] a norm map $\n{\cdot}:\C\ra[0,+\infty[$ such that $\C$ is a Banach space \cite{Banach:1922} over $\CC$ with norm $\n{\cdot}$, and 
\begin{equation}
        \n{x^*x}=\n{x}^2\;\;\forall x\in\C,
        \label{C.star.condition}
\end{equation}
or, equivalently \cite{Ono:1959,Glimm:Kadison:1960},
\begin{equation}
        \n{x^*x}=\n{x^*}\n{x}\;\;\forall x\in\C.
        \label{B.star.condition}
\end{equation}
\end{enumerate}
The properties $\n{\II}=1$ and $\n{x^*y}\leq\n{x}\n{y^*}\;\forall x,y\in\C$ follow from \eqref{C.star.condition} \cite{Araki:Elliott:1973} (see also \cite{Magyar:Sebestyen:1985}), ensuring continuity of $^*$ and of multiplication in the topology generated by $\n{\cdot}$. Algebra over $\CC$ that is equipped only with the map $^*$ defined as above (but with no norm) is called a \df{$*$-algebra}, while the algebra over $\CC$ that is equipped only with the norm $\n{\cdot}$ such that it becomes a Banach space (but with no $^*$ map) is called a \df{Banach algebra} \cite{Nagumo:1936}. A $*$-algebra equipped with a structure of a topological vector space such that the operations $\cdot$ and $^*$ are continuous will be called \df{topological} \cite{vanDantzig:1931}. (It is possible to develop a theory of $C^*$-algebras without unit $\II$, as well as the theory of \df{Banach $*$-algebras}, defined as $*$-algebras that are also Banach spaces and satisfy $\n{x^*}=\n{x}$ and $\n{xy}\leq\n{x}\n{y}$ instead of \eqref{C.star.condition}. Such theories, however, will not be useful for our purposes.) The most remarkable property of $C^*$-algebras follows from the condition \eqref{C.star.condition}: their norm structure is uniquely determined by their algebraic structure \cite{Gelfand:Naimark:1943}, which follows from the property\footnote{For the definition of invertibility see below, for more discussion of $\sup$ see Section \ref{comm.integr.section}.} 
\begin{equation}
        \n{x}=\sqrt{\sup\left\{\ab{\lambda}\mid\lambda\in\CC,\;(x^*x-\lambda\II)\mbox{ is not invertible}\right\}}\;\;\forall x\in\C.
\label{norm.determined.algebraically}
\end{equation}

Let $\lambda_1,\lambda_2\in\CC$ be arbitrary. A \df{$*$-homomorphism} of $*$-algebras $\C_1$ and $\C_2$ is defined as a map $\varsigma:\C_1\ra\C_2$ such that 
\begin{align}
                \varsigma(\lambda_1x_1+\lambda_2x_2)&=\lambda_1\varsigma(x_1)+\lambda_2\varsigma(x_2),\label{homo.lin}\\
        \varsigma(x_1x_2)&=\varsigma(x_1)\varsigma(x_2),\label{homo.multi}\\
        \varsigma(x^*)&=\varsigma(x)^*,\label{homo.star}
\end{align}
for all $x,x_1,x_2\in\C_1$. From \eqref{norm.determined.algebraically} it follows that every $*$-homomorphism $\varsigma$ of a $C^*$-algebra $\C$ preserves its topological structure, that is, it is continuous with respect to the norm $\n{\cdot}$ of $\C$, and 
\begin{equation}
        \n{\varsigma(x)}\leq\n{x}\;\forall x\in\C.
\label{Cstarhomo.norm.inequality}
\end{equation}
An \df{antilinear $*$-homomorphism} is defined as a map $\varsigma:\C_1\ra\C_2$ satisfying \eqref{homo.multi}, \eqref{homo.star}, and
\begin{equation}
        \varsigma(\lambda_1x_1+\lambda_2x_2)=\lambda_1^*\varsigma(x_1)+\lambda_2^*\varsigma(x_2)\;\;\forall x_1,x_2\in\C_1.
\end{equation}
A \df{Jordan $*$-homomorphism} is defined as a map $\varsigma:\C_1\ra\C_2$ satisfying \eqref{homo.lin}, \eqref{homo.star}, and
\begin{equation}
        \varsigma(x_1x_2+x_2x_1)=\varsigma(x_1)\varsigma(x_2)+\varsigma(x_2)\varsigma(x_1)\;\;\;\forall x_1,x+2\in\C_1,
\end{equation}
or, equivalently,
\begin{equation}
        \varsigma(xx)=\varsigma(x)\varsigma(x)\;\;\;\forall x\in\C_1.
\end{equation}
A \df{$*$-antihomomorphism} is defined as a map $\varsigma:\C_1\ra\C_2$ satisfying \eqref{homo.lin}, \eqref{homo.star}, and
\begin{equation}
        \varsigma(x_1x_2)=\varsigma(x_2)\varsigma(x_1)\;\;\;\forall x_1,x_2\in\C_1.
\end{equation}
Given any $C^*$-algebra $\C$ the \df{opposite algebra} $\rpktarget{OPPOSITE}\C^o$ is defined as a $C^*$-algebra which has the same elements and norm as $\C$, but the opposite multiplication maps (that is, if $x,y\in\C$ and $\lambda_1,\lambda_2\in\CC$, then $x^o,y^o\in\C^o$ satisfy: $(\lambda_1x+\lambda_2y)^o=\lambda_1x^o+\lambda_2y^o$, $(x^*)^o=(x^o)^*$,  $(xy)^o=y^ox^o$). Hence, given $C^*$-algebras $\C_1$ and $\C_2$, one can identify $*$-homomorphisms $\varsigma:\C_1\ra\C_2$ with $*$-antihomomorphisms $\varsigma^o:\C_1^o\ra\C_2$ by $\varsigma^o(x^o)=\varsigma(x)\;\forall x\in\C_1$. A $*$-homomorphism $\varsigma:\C_1\ra\C_2$ of $C^*$-algebras $\C_1$ and $\C_2$ is called: \df{unital} if{}f $\varsigma(\II)=\II$; a \df{$*$-isomorphism} if{}f
\begin{equation}
        0=\ker(\varsigma):=\{x\in\C_1\mid\varsigma(x)=0\}.
\label{star.iso.condition}
\end{equation}
Condition \eqref{star.iso.condition} determines analogously the notions of \df{antilinear $*$-isomorphism}, \df{Jordan $*$-iso\-mor\-ph\-ism}, and \df{$*$-antiisomorphism}. If $\B_1$ and $\B_2$ are Banach spaces, then a linear function $f:\B_1\ra\B_2$ is called: \df{norm preserving} if{}f $\n{f(x)}_{\B_2}=\n{x}_{\B_1}$ $\forall x\in\B_1$; an \df{isometry} if{}f it is norm preserving and continuous with respect to norm topologies of $\B_1$ and $\B_2$; an \df{isometric isomorphism} if{}f it is a bijective isometry. Every isometry is injective, so $f$ is an isometric isomorphism if{}f it is a surjective isometry. If $\C_1$ and $\C_2$ are $C^*$-algebras, then every $*$-isomorphism $\varsigma:\C_1\ra\C_2$ is an isometric isomorphism. The same is true for antilinear $*$-isomorphism, Jordan $*$-isomorphism, and $*$-antiisomorphism. A \df{$*$-automorphism} of a $C^*$-algebra $\C$ is defined as a $*$-isomorphism from $\C$ to $\C$. Every $*$-automorphism $\alpha$ of a $C^*$-algebra satisfies $\n{\alpha(x)}=\n{x}\;\forall x\in\C$. The set of all $*$-automorphisms of a given $C^*$-algebra $\C$ is a group, denoted by $\rpktarget{AUT}\Aut(\C)$. Given some family $\{\alpha_\iota\}\subseteq\Aut(\C)$, with $\iota$ ranging over some set $I$, the \df{fixed point} $*$-subalgebra $\C_\alpha$ of $\C$ with respect to $\{\alpha_\iota\}$ is defined as 
\begin{equation}\rpktarget{FIXEDPOINT}
        \C_\alpha:=\{x\in\C\mid\alpha_\iota(x)=x\;\;\forall\iota\in I\}.
\end{equation}

An element $x$ of a $*$-algebra $\C$ is called: \df{normal} if{}f $xx^*=x^*x$ \cite{Toeplitz:1918}; \df{self-adjoint} if{}f $x=x^*$; \df{anti-self-adjoint} if{}f $x=-x^*$; \df{isometry} if{}f $x^*x=\II$; \df{unitary} if{}f $xx^*=x^*x=\II$ \cite{Antonne:1902,Schur:1909}; \df{projection} if{}f $x=x^*=xx=:x^2$; \df{positive} if{}f $\exists y\in\C$ such that $x=y^*y$; \df{square root} if{}f it is positive and $\exists!$ positive $y\in\C$ such that $y^2=x$ (which is denoted by $y=:x^{1/2}$); \df{partial isometry} if{}f $x^*x$ is a projection; \df{invertible} if{}f $\exists y\in\C$ such that $xy=\II=yx$ (this is denoted by $y=:x^{-1}$). For a given $x\in\C$, the following implications hold:
\begin{equation}
        \xymatrix{
        x \mbox{ is unitary}
        \ar@{=>}[r]
        \ar@{=>}[d]
        &
        x \mbox{ is an isometry}
        \ar@{=>}[r]
        \ar@{=>}[d]
        &       
        x \mbox{ is a partial isometry}
        &
        x \mbox{ is projection}
        \ar@{=>}[l]
        \ar@{=>}[d]
        \\
        x \mbox{ is invertible}
        &
        x \mbox{ is normal}
        &
        x \mbox{ is self-adjoint}
        \ar@{=>}[l]
        &
        x \mbox{ is positive}.
        \ar@{=>}[l]
        }
\label{implications.elements}
\end{equation}
Moreover, $(x^{-1})^{-1}=x$, $(xy)^{-1}=y^{-1}x^{-1}$, and $(x^*)^{-1}=(x^{-1})^*$. If $x$ is a partial isometry, then $xx^*$ is a projection. The \df{absolute value} of $x\in\C$ is defined as $\ab{x}:=(x^*x)^{1/2}\in\C$, and it is always positive. If $x$ is a positive element of a $*$-algebra $\C$, then one writes $x\geq0$. A partial order relation $\rpktarget{PORDER}\geq$ on elements of $*$-algebra $\C$ is defined by\footnote{As shown in \cite{Fukamiya:1952,Kelley:Vaught:1953,Fukamiya:Misonou:Takeda:1954}, the above definition of partial order $x\geq y$ for $x,y\in\C$ on an arbitrary $C^*$-algebra $\C$ is equivalent to the partial order defined by the conditions
\begin{enumerate}
\item[(i)] $x\geq 0\iff x=x^*$ and $\sp(x)\subseteq\RR^+$,
\item[(ii)] $x\geq y\iff x-y\geq 0$.
\end{enumerate}
} 
\begin{equation}
        x\geq y\iff(x-y)\geq 0,\;\mbox{for}\;x,y\in\C.
\label{Cstar.order}
\end{equation}
Given a $*$-algebra $\C$, the set of all positive elements of $\C$ is denoted $\rpktarget{POSITIV}\C^+$, the set of all projections of $\C$ is denoted $\rpktarget{PROJ}\Proj(\C)$, the set of all self-adjoint elements of $\C$ is denoted $\rpktarget{SA}\C^\sa$, the set of all anti-self-adjoint elements of $\C$ is denoted $\rpktarget{ASA}\C^\asa$, the set of all unitary elements of $\C$ is denoted $\rpktarget{UNI}\C^\uni$, while the set of all invertible elements of $\C$ is denoted $\C^\inv$. For any $*$-algebra $\C$, the set $\C^+$ is a cone that is pointed ($\C^+\cap(-\C^+)=\{0\}$), convex, closed in norm topology, and spans linearly $\C$ ($\Span_\CC\C^+=\C$) \cite{Fukamiya:1952,Kelley:Vaught:1953}. As follows from \eqref{implications.elements}, $\Proj(\C)\subseteq\C^+\subseteq\C^\sa$ and $\C^\inv\subseteq\C^\uni$. The set $\C^\sa$ equipped with partial order $\geq$ is a lattice if{}f $\C$ is commutative \cite{Sherman:1951}.

If $I$ is an ideal in a $C^*$-algebra $\C$ that is closed in norm topology of $\C$, then $\C/I$ is also a $C^*$-algebra \cite{Segal:1949:ideals,Kaplansky:1951:rings}. The $C^*$-algebra $\C$ is said to be \df{generated} by a set $Y\subseteq\C$ if $\C$ is equal to completion of the algebra of polynomials of elements of $Y$ in the topology of a norm $\n{\cdot}$ of $\C$. In particular, if $\C_2\subseteq\C_1$ is a $C^*$-algebra generated by the set $\{x,x^*\}$, where $x,x^*\in\C_1$, then $(x-\lambda\II)^{-1}\subseteq\C_2$. Every $C^*$-algebra is generated by $\C^\sa$, and if $\C$ contains $\II$ (we consider only such cases), then it is generated by the set of all its unitary elements.

If $\C$ is a $C^*$-algebra and $Y\subseteq\C$, then a \df{right annihilator} of $Y$ and a \df{two sided annihilator} of $Y$ are defined, respectively, as \cite{Dieudonne:1942,Mackey:1945}
\begin{equation}
        \ann_{\mathrm{r}}(Y):=\{x\in\C\mid yx=0\;\forall y\in Y\},\mbox{ and }
        \ann(Y):=\{x\in\C\mid yx=0=xy\;\forall y\in Y\},
\end{equation}
with $\ann_{\mathrm{r}}(\{y\})\equiv\ann_{\mathrm{r}}(y)$ and $\ann(\{y\})\equiv\ann(y)$. A $C^*$-algebra $\C$ is called: \df{\c{S}tr\u{a}til\u{a}--Zsid\'{o}} \cite{Zsido:1973,Stratila:Zsido:1977:1979} if{}f
\begin{equation}
        \forall x\in\C^+\;\;\forall\lambda_1,\lambda_2\in]0,\infty[\;\;\exists P\in\Proj(\C)\;\;\;\lambda_1\leq\lambda_2\limp(xP\geq\lambda_1P\mbox{ and }x(1-P)\leq\lambda_2(1-P));
\end{equation}
\df{Rickart} \cite{Rickart:1946} if{}f
\begin{equation}
        \forall x\in\C\;\;\exists P\in\Proj(\C)\;\;\ann_{\mathrm{r}}(x)=P\C:=\bigcup_{x\in\C}\{Px\};
\end{equation}
an \df{$AW^*$-algebra} \cite{Kaplansky:1951,Kaplansky:1952} if{}f
\begin{equation}
        \forall Y\subseteq\C\;\;\exists P\in\Proj(\C)\;\;Y\neq\varnothing\;\limp\;\ann_{\mathrm{r}}(Y)=P\C.
\end{equation}
Each $AW^*$-algebra is a Rickart $C^*$-algebra, and each Rickart $C^*$-algebra is a \c{S}tr\u{a}til\u{a}--Zsid\'{o} $C^*$-algebra \cite{Stratila:Zsido:1977:1979}. A nonunital $C^*$-subalgebra $\C_2$ of a $C^*$-algebra $\C_1$ is called \df{hereditary} if{}f
\begin{equation}
        0\leq x\leq y\;\limp\;x\in\C_2\;\;\forall x\in\C_1\;\forall y\in\C_2.
\end{equation}
If $Y$ is a subset of $C^*$-algebra $\C$, then the hereditary $C^*$-subalgebra of $\C$ generated by $Y$ is denoted $\her_\C(Y)$. A $*$-homomorphism $\varsigma:\C_1\ra\C_2$ of $C^*$-algebras $\C_1$ and $\C_2$ will be called \df{complete} if{}f
\begin{equation}
        \ann(\varsigma(\C_3))=\her_{\C_2}(\varsigma(\ann(\C_3)))
\end{equation}
for every hereditary nonunital $C^*$-subalgebra $\C_3$ of $\C_1$ \cite{Pedersen:1986}.
\subsection{Functionals} 
Let $\B$ be a Banach space over $\KK$, where $\KK\in\{\RR,\CC\}$. Recall that the \df{norm topology} on $\B$ is defined by the set of neighbourhoods
\begin{equation}
        N_{\epsilon}(x):=\{y\in\B\mid\n{x-y}<\epsilon\},
\end{equation}
with $\epsilon>0$ and $x\in\B$. Any function $\phi:\B\ra\KK$ is called a \df{functional}. It is additionally called: \df{$\KK$-linear} if{}f 
\begin{equation}
        \phi(\lambda_1 x+\lambda_2 y)=\lambda_1\phi(x)+\lambda_2\phi(y)\;\forall\lambda_1,\lambda_2\in\KK;
\end{equation}
\df{norm continuous} if{}f it is continuous in the topology induced by the norm of $\B$. 
A \df{Banach dual} \cite{Hahn:1927} of $\B$ is defined as a set $\rpktarget{BANACH}\B^\banach$ of all norm continuous $\KK$-linear functionals $\phi$ on $\B$, equipped with the norm
\begin{equation}
        \n{\phi}:=\sup\{\ab{\phi(x)}\mid\n{x}\leq1,\;x\in\B\}.
\label{norm.on.banach.dual}
\end{equation}
The set $\B^\banach$ is complete in the topology generated by this norm, and so is a Banach space. For every Banach space $\B$ there exists a \df{canonical embedding} $j:\B\ra\B^\banach{}^\banach$, defined by 
\begin{equation}
        (j(x))(\phi)=\phi(x)\;\;\;\forall\phi\in\B^\banach\;\;\forall x\in\B,
\end{equation}
which is an isometry \cite{Hahn:1927}. The \df{weak-$\star$ topology} on $\B^\banach$ is defined as the weakest topology on $\B^\banach$ such that the $\KK$-linear functions $\B^\banach\ni\omega\mapsto\omega(x)\in\KK$ are continuous in this topology for every $x\in\B$. The neighbourhoods of the weak-$\star$ topology on $\B^\banach$ have the form
\begin{equation}
        N_{\epsilon,\{x_k\}}(\phi):=
        \{
                \omega\in\B^\banach\mid
                        \ab{\omega(x_k)-\phi(x_k)}<\epsilon      
        \},
\end{equation}
where $\{x_k\}\subseteq\B$, $k\in\{1,\ldots,m\}$, $m\in\NN$, $\epsilon>0$. This topology is locally convex, but, in general, it is not first countable.

If there exists a Banach space $\B_\star$ that satisfies $\rpktarget{PREDUAL}(\B_\star)^\banach=\B$, then it is called a \df{predual} of $\B$ \cite{Bourbaki:1938}, and it can be embedded as a subset of the space $\B^\banach$ of $\KK$-linear functionals on $\B$ by means of a canonical embedding map $j:\B_\star\ra\B^\banach$. In general, a Banach space can possess no predual or it may possess different preduals which are not isometrically isomorphic (see \cite{Godefroy:1989} for a review of this issue).

A functional $\omega:\C\ra\CC$ on a $C^*$-algebra $\C$ is called: \df{linear} if{}f it is $\CC$-linear; \df{positive} if{}f $\omega(x^*x)\geq0$; \df{faithful} if{}f $\omega(x^*x)=0\limp x=0$; \df{tracial} if{}f $\omega(xy)=\omega(yx)$; \df{normalised} if{}f $\omega(\II)=1$; \df{self-adjoint} if{}f $\omega^*=\omega$, where $\omega^*(x):=(\omega(x^*))^*$; \df{normal} if{}f $\omega(\sup\filter)=\sup_{x\in\filter}\omega(x)$ for each directed filter $\filter\subseteq\C^+$ with the upper bound $\sup\filter$, or, equivalently, if{}f $\omega(\sup_\iota\{x_\iota\})=\sup_\iota\{\omega(x_\iota)\}$ for every uniformly bounded ($\sup_\iota\n{x_\iota}<\infty$) and increasing ($\iota_1\geq\iota_2\iff x_{\iota_1}\geq x_{\iota_2}$) net $\{x_\iota\}\subseteq\C^+$. Every positive linear functional is self-adjoint. If $\omega,\phi$ are self-adjoint linear functionals on $\C$ and $(\omega-\phi)$ is positive, then one writes $\phi\leq\omega$, and says that $\omega$ \df{majorises} $\phi$ (this is equivalent to $\phi(x)\leq\omega(x)\;\forall x\in\C^+$). Self-adjoint functionals are uniquely determined by their restriction to the self-adjoint elements of $\C$. Moreover, every linear functional $\omega$ on $\C$ can be represented in the form $\omega=\omega_1+\ii\omega_2$, with linear self-adjoint $\omega_1$ and $\omega_2$. The norm of a linear functional is defined by \eqref{norm.on.banach.dual},
\begin{equation}
        \n{\omega}:=
        \sup\{
                \ab{\omega(x)}\mid
                        \n{x}\leq1,\;
                        x\in\C
        \}.
\label{dual.functional}
\end{equation}
Each positive linear functional on a $C^*$-algebra satisfies $\n{\omega}=\omega(\II)$, is continuous, and self-adjoint (hence, it can be completely determined by the values it takes on the self-adjoint elements of the algebra).

If $\C$ is a $C^*$-algebra, then the space $\C^\banach$ is a Banach space with a norm given by \eqref{dual.functional}. The space of all positive linear functionals on $\C$ is denoted $\C^\banach{}^+$ and it is a convex cone in $\C^\banach$ that is closed in the weak-$\star$ topology on $\C^\banach$. The space of all faithful elements of $\C^{\banach+}$ is denoted $\rpktarget{ZERO}\C^{\banach+}_0$. The space of all self-adjoint elements of $\C^\banach$ is denoted $(\C^\banach)^\sa$. The space of all normalised elements of $\C^{\banach+}$ is denoted $\rpktarget{SCAL}\Scal(\C)$, and the notation $\rpktarget{SZERO}\Scal_0(\C):=\Scal(\C)\cap\C_0^{\banach+}$ will be used. An element $\phi\in\C^{\banach+}$ is called: \df{pure} if{}f 
\begin{equation}
\forall\psi\in\C^{\banach+}\;\;\;\psi\leq\phi\;\limp\;\;\exists\lambda\in\RR^+\;\psi=\lambda\phi;
\end{equation}
\df{mixed} if{}f it is not pure. An element $x$ of the convex set $X$ is called an \df{extremal point} if{}f
\begin{equation}
        (\exists x_1,x_2\in X\;\exists\lambda\in\,]0,1[\;\;\;x=\lambda x_1+(1-\lambda)x_2\;\;\mbox{and}\;\;x_1\neq x_2)\mbox{ is false}.
\end{equation}
From the weak-$\star$ compactness theorem \cite{Banach:1929,Alaoglu:1940,Shmulyan:1940,Kakutani:1940,Dieudonne:1942} it follows that $\Scal(\C)$ is a convex subset of $\C^\banach$ which is compact in weak-$\star$ topology on $\C^\banach$ \cite{Segal:1947:postulates}. Any $\omega\in\Scal(\C)$ is pure if{}f it is an extremal point of $\Scal(\C)$. By the Kre\u{\i}n--Milman theorem \cite{Krein:Milman:1940,Yosida:Fukamiya:1941}, if $\B$ is a Banach space and $X\subseteq\B^\banach$ is convex and compact in weak-$\star$ topology on $\B^\banach$, then it is a closure in weak-$\star$ topology on $\B^\banach$ of the set of all finite convex combinations of its extremal points. Hence, every element of $\Scal(\C)$ can be obtained as a finite convex combination of pure elements of $\Scal(\C)$.

If for a given $C^*$-algebra $\C$ there exists a predual $\C_\star$, then it is a unique predual of $\C$, and in such case $\C$ is called a \df{$W^*$-algebra} \cite{Sakai:1956}. Every $W^*$-algebra is an $AW^*$-algebra \cite{Kaplansky:1951}, but converse is false \cite{Dixmier:1951}. If $\rpktarget{N}\N$ is a $W^*$-algebra, then any $\omega\in\N^{\banach+}$ is normal if{}f it is  continuous in the weak-$\star$ topology on $\C$. In what follows, the term `weak-$\star$ topology' will refer by default to weak-$\star$ topology on a $W^*$-algebra $\N$ with respect to its predual $\N_\star$. Other uses of the weak-$\star$ topology (e.g., on $\C^\banach$ with respect to a $C^*$-algebra $\C$) will be always explicitly stated. The predual $\N_\star$ of a $W^*$-algebra $\N$ is a norm closed vector subspace of $\N^\banach=(\N_\star)^\banach{}^\banach$. The norm on $\C_\star$ coincides with the norm on $\C^\banach$. One defines also 
\begin{equation}
        \C_\star^+:=\C^{\banach+}\cap\C_\star,\;
        \rpktarget{ONE}\C_{\star1}^+:=\Scal(\C)\cap\C_\star,\;
        \C^+_{\star0}:=\C^{\banach+}_0\cap\C_\star,\;
        \C^+_{\star01}:=\Scal(\C)\cap\C^+_{\star0},\;
        \C_\star^\sa:=(\C^\banach)^\sa\cap\C_\star,
\end{equation}
and the following embeddings hold:
\begin{equation}
\xymatrix{
        \C^+_{\star01}
        \ar@{^{(}->}[r]
        \ar@{^{(}->}[rd]&
        \C^+_{\star1}
        \ar@{^{(}->}[r]&
        \C^+_\star
        \ar@{^{(}->}[r]&
        \C^\sa_\star
        \ar@{^{(}->}[r]&
        \C_\star.\\
        &
        \C^+_{\star0}
        \ar@{^{(}->}[r]
        \ar@{^{(}->}[ru]&
        \C_{\star0}
        \ar@{^{(}->}[ru]&&
}
\end{equation}
Each element of $\C^+_\star$ will be called a \df{state}, while each element of $\C^+_{\star1}$ will be called a \df{normalised state}.\footnote{This terminological shift with respect to more traditional terminology (which defines a `state' as an element of $\Scal(\C)$ for any $C^*$-algebra $\C$, see \cite{Segal:1947:postulates,Segal:1947:irreducible,Emch:1972}) reflects the change of perspective advocated in the present work: according to it, the space $\N_\star^+$ over a $W^*$-algebra $\N$ is a more fundamental object of interest than the space $\Scal(\C)$ over a $C^*$-algebra $\C$.} A linear function between $W^*$-algebras will be called \df{normal} if{}f it is continuous with respect to their weak-$\star$ topologies. While every $*$-homomorphism of $C^*$-algebras is continuous with respect to their norm topologies, from the uniqueness of a predual of a $W^*$-algebra it follows that every $*$-isomorphism of $W^*$-algebras is also continuous with respect to their weak-$\star$ topologies. The same holds for every Jordan $*$-isomorphism of $W^*$-algebras. By a theorem of Kadison \cite{Kadison:1951}, for every Jordan $*$-isomorphism $\varsigma:\N_1\ra\N_2$ of $W^*$-algebras $\N_1$ and $\N_2$ there exists a $*$-isomorphism $\varsigma_a:\N_1\ra\N_2$ and a $*$-antiisomorphism $\varsigma_b:\N_1\ra\N_2$ such that $\varsigma$ is a `sum' of $\varsigma_a$ and $\varsigma_b$ in the sense that: $\varsigma=\varsigma_a+\varsigma_b$ as a linear map, and there exist $W^*$-algebras $\N_{1a},\N_{1b}\subseteq\N_1$ and $\N_{2a},\N_{2b}\subseteq\N_2$ such that $\N_1=\N_{1a}\oplus\N_{1b}$, $\N_{2a}\oplus\N_{2b}$, $\varsigma_a:\N_{1a}\ra\N_{1b}$ is a $*$-isomorphism, $\varsigma_b:\N_{1b}\ra\N_{2b}$ is a $*$-antiisomorphism, $\varsigma_a(\N_{1b})=0$, and $\varsigma_b(\N_{1a})=0$. This means that for $W^*$-algebras the notion of $*$-isomorphism is strictly stronger than the notion of Jordan $*$-isomorphism. The uniqueness of a predual allows to define two additional topologies on a $W^*$-algebra $\N$: the \df{ultrastrong topology}, defined by the family of semi-norms on $\N$ given by
\begin{equation}
        \N\times\N_\star^+\ni(x,\phi)\mapsto\n{x}_\phi:=(\phi(x^*x))^{1/2}\in\RR^+,
\end{equation}
and the \df{ultrastrong-$\star$ topology}, provided by the family of semi-norms $\{(\n{x}^2_\phi+\n{x^*}^2_\phi)^{1/2}\}$.\footnote{When considered in the context of $W^*$-algebras $\N\subseteq\BH$, these topologies are usually called \textit{$\sigma$-strong} and \textit{$\sigma$-strong-$\star$}, respectively. This terminology is intended to avoid the overuse of the symbol $\sigma$ (which will be reserved for modular automorphisms), and also to stress independence of the Hilbert space structure.} An example of a $W^*$-algebra is an algebra $\rpktarget{BH}\BH$ of all bounded linear operators $\H\ra\H$ on any Hilbert space $\H$ (the predual $\schatten_1(\H)=\BH_\star$ and other $\schatten_p(\H)$ spaces will be discussed in Section \ref{integration.trace.section}).

A generalisation of the concept of state is provided by the notion of \textit{weight} \cite{Dixmier:1957,Tomita:1959,Combes:1966,Combes:1967,Combes:1971,Pedersen:1966,Pedersen:1971}. A \df{weight} on a $C^*$-algebra $\C$ is defined as a function $\omega:\C^+\ra[0,+\infty]$ such that $\omega(0)=0$, $\omega(x+y)=\omega(x)+\omega(y)$, and $\lambda\geq0\limp\omega(\lambda x)=\lambda\omega(x)$, with the convention $0\cdot(+\infty)=0$. The domain of a weight $\omega$ can be extended by linearity to the topological $*$-algebra\rpktarget{mmm}
\begin{equation}
        \mmm_\omega:=
                \Span_\CC\{x^*y\mid
                        x,y\in\C,\;
                        \omega(x^*x)<\infty,\;
                        \omega(y^*y)<\infty\}=
                                                                \Span_\CC\{x\in\C^+\mid
                                                                                                \omega(x)<\infty\}
        \subseteq\C,
\label{extended.domain.of.definition.of.weight}
\end{equation}
while $\omega$ can be extended to a positive linear functional on $\mmm_\omega$, which coincides with $\omega$ on $\mmm_\omega\cap\C^+$. A weight is called: \df{normalised} if{}f $\omega(\II)=1$; \df{faithful} if{}f $\omega(x)=0\limp x=0$; \df{finite} if{}f $\omega(\II)<\infty$; \df{semi-finite} if{}f
\begin{equation}
        \forall x\in\C^+\;\exists y\in\C^+\;\;\omega(x)=\infty\;\limp\;(x\geq y\mbox{ and }0<\omega(y)<\infty);
\end{equation}
\df{trace} if{}f $\omega(u^*xu)=\omega(x)$ $\forall u\in\C^\uni$ $\forall x\in\C^+$ (this is equivalent to the condition $\omega(xx^*)=\omega(x^*x)\;\forall x\in\C$); \df{normal} if{}f $\omega(\sup\{x_\iota\})=\sup\{\omega(x_\iota)\}$ for any uniformly bounded increasing net $\{x_\iota\}\subseteq\C^+$. A space of all normal semi-finite weights on a $C^*$-algebra $\C$ is denoted $\rpktarget{WC}\W(\C)$, while the subset of all faithful elements of $\W(\C)$ is denoted $\rpktarget{WZEROC}\W_0(\C)$. The symbol $\tau$ will be used exclusively to denote traces. Every finite weight is also semi-finite. Every tracial element of $\C^{\banach+}$ is a finite trace on $\C^+$. Conversely, every finite trace on $\C^+$ can be uniquely extended by linearity to a tracial element of $\C^{\banach+}$. Given $\omega,\phi\in\W(\C)$, we say that: $\omega$ \df{majorises} $\phi$ if{}f $\phi(x)\leq\omega(x)\;\forall x\in\C^+$, which is denoted by $\rpktarget{PHILEQ}\phi\leq\omega$; $\phi$ is \df{dominated} by $\omega$ if{}f $\exists\lambda>0$ $\phi\leq\lambda\omega$. The equivalent conditions for a weight $\omega$ on a $W^*$-algebra $\N$ to be normal are \cite{Haagerup:1975:normal:weights}:
\begin{enumerate}
\item[1)] $\omega(x)=\sup_{\phi\in\N_\star^+}\{\phi(x)\mid\phi\leq\omega\}$,
\item[2)] $\omega$ is weakly-$\star$ lower semi-continuous (that is, the set $\{x\in\N^+\mid\phi(x)\leq\lambda\}$ is weakly-$\star$ closed for each $\lambda\in\RR^+$).
\end{enumerate}
For any normal weight on $\N$ there exists a family $\{\phi_i\}\subseteq\N_\star^+$ such that $\phi(x)=\sum_i\phi_i(x)$ $\forall x\in\N^+$ \cite{Pedersen:Takesaki:1973}. A weight $\phi$ on any $W^*$-algebra $\N$ is semi-finite if{}f a left ideal in $\N$ given by
\begin{equation}\rpktarget{nnn}
        \nnn_\phi:=\{x\in\N\mid\phi(x^*x)<\infty\}
\end{equation}
is weakly-$\star$ dense in $\N$, or, equivalently, if{}f $\mmm_\phi$ is weakly-$\star$ dense in $\N$. Every state is a finite normal weight, and every faithful state is a finite faithful normal state, hence the diagram
\begin{equation}
\xymatrix{
        \N^+_{\star0}
        \ar@{^{(}->}[r]
        \ar@{^{(}->}[d]&
        \W_0(\N)
        \ar@{^{(}->}[d]\\
        \N^+_\star
        \ar@{^{(}->}[r]&
        \W(\N)
}
\label{Wstar.states.weights.comm}
\end{equation}
commutes. Note that for $\omega\in\N_\star^+$ the normality is equivalent to weak-$\star$ continuity, which is stronger than weak-$\star$ lower semi-continuity of elements of $\W(\N)$.

Given a $W^*$-algebra $\N$, the set $\Proj(\N)$ is a Dedekind--MacNeille complete lattice with respect to the partial order relation $\leq$ (see Section \ref{comm.integr.section} for a definition). Moreover, $\Span_\CC\Proj(\N)$ is a norm dense subset of $\N$. For any $x\in\N$ a \df{left support} and a \df{right support} of $x$ are defined, respectively, as
\begin{align}
        \supp_L(x)&:=\II-\sup\{P\in\Proj(\N)\mid Px=x\},\\
        \supp_R(x)&:=\II-\sup\{P\in\Proj(\N)\mid xP=x\}.
\end{align}
If $y\in\N^\sa$, then $\rpktarget{SUPP}\supp_L(y)=\supp_R(y)=:\supp(y)$, which is called a \df{support} of $y$. Given $\omega\in\N^+_\star$,
\begin{align}
                \supp(\omega)&:=
                        \inf\{P\in\Proj(\N)\mid
                        \omega(P)=1\}
                        =\II-\sup\{P\in\Proj(\N)\mid\omega(P)=0\},\\
        \supp_\zentr(\omega)&:=
        \inf\{P\in\Proj(\N)\cap\zentr_\N\mid
                \omega(P)=1\},
\end{align}
are called, respectively, a \df{support} and a \df{central support} of $\omega$. If $\psi$ is a normal weight on $\N$, then there exist $P_1,P_2\in\Proj(\N)$ such that $\ker(\psi)=\N P_2$ and $\overline{\nnn_\psi}=\N P_1$, where bar denotes a closure in the weak-$\star$ topology. As a result, $\overline{\nnn_\psi}=P_1\N P_1$ and
\begin{equation}
        \psi(x)=\psi(P_1xP_1)=\psi((P_1-P_2)x(P_1-P_2))\;\;\forall x\in\N.
\end{equation}
The \df{support} of $\psi$ is defined as $\rpktarget{SUPP.PSI}\supp(\psi):=P_1-P_2$. For $\psi\in\W(\N)$,
\begin{equation}
        \supp(\psi)=\II-\sup\{P\in\Proj(\N)\mid\psi(P)=0\}.
\end{equation}
For $\omega,\phi\in\N_\star^+$ we will write $\rpktarget{ll}\omega\ll\phi$ if{}f $\supp(\omega)\leq\supp(\phi)$.\footnote{If $\N=\BH$ and $\omega=\tr(\rho_\omega\cdot)$ for $\rho_\omega\in\schatten_1(\H)^+$, then $\supp(\omega)=\ran(\rho_\omega)$, so for any $\phi=\tr(\rho_\phi\cdot)$ with $\rho_\phi\in\schatten_1(\H)^+$ one has $\omega\ll\phi$ if{}f $\ran(\rho_\omega)\subseteq\ran(\rho_\phi)$.} An element $\omega\in\N^{\banach+}$ is faithful if{}f $\supp(\omega)=\II$. If $\phi$ is a normal weight on a $W^*$-algebra $\N$ (which includes $\omega\in\N^+_\star$ as a special case), then the restriction of $\phi$ to a \df{reduced} $W^*$-algebra,
\begin{equation}
        \N_{\supp(\phi)}:=\{x\in\N\mid\supp(\phi)x=x=x\,\supp(\phi)\}=\bigcup_{x\in\N}\{\supp(\phi)x\,\supp(\phi)\},
\end{equation}
is a faithful normal weight (respectively, an element of $(\N_{\supp(\phi)})^+_{\star0}$). If $\phi$ is semi-finite, then $\phi|_{\N\supp(\phi)}\in\W_0(\N_{\supp(\phi)})$. Hence, given $\psi\in\W(\N)$ and $P\in\Proj(\N)$, $P=\supp(\psi)$ if{}f $\psi|_{\N_P}\in\W_0(\N_P)$ and $\psi(P)=\psi(PxP)\;\forall x\in\N^+$. In particular, for $\omega,\phi\in\N^+_\star$ and $\omega\ll\phi$, we have $\omega|_{\N_{\supp(\phi)}}\in\W_0(\N_{\supp(\phi)})$.

One of the most important properties of a $W^*$-algebra is the existence of unique \textit{polar decompositions} of their elements, as well as of elements of their preduals (when considered only for elements of an algebra, unique polar decompositions exist for all Rickart $C^*$-algebras \cite{Ara:Goldstein:1993,Goldstein:1995}). For any $x\in\N$ there exists a unique partial isometry $v\in\N$ and a unique $y\in\N^+$ such that $x=vy$, where $y=(x^*x)^{1/2}$, while $v$ satisfies $v^*v=\supp(\ab{x})$ and $vv^*=\supp(\ab{x^*})$ \cite{vonNeumann:1932:adjungierte}. On the other hand, if $\phi\in\N_\star$, then there exists a unique partial isometry $v\in\N$ and a unique $\omega\in\N^+_\star$ such that $\phi(\cdot)=\omega(\,\cdot\,v)$, where $\n{\omega}=\n{\phi}$, $\supp(\phi)=v^*v$, and $\supp(\ab{\phi^*})=vv^*$, with $\rpktarget{AB.PHI}\ab{\phi}:=\omega$ \cite{Sakai:1958,Tomita:1960}. Moreover, $\ab{\phi}=\phi(v^*\,\cdot\,)$. The equations $x=v\ab{x}$ and $\phi=\ab{\phi}(\,\cdot\,v)$ are called \df{polar decomposition} of, respectively, $x$ and $\phi$. 
\subsection{Representations\label{representations.section}}
A \df{representation} of a $C^*$-algebra $\C$ is defined as a pair $(\H,\pi)$ of a Hilbert space $\H$ and a $*$-homomorphism $\rpktarget{pi}\pi:\C\ra\BH$. An \df{antirepresentation} of $\C$ is defined as a pair $(\H,\pi)$ of a Hilbert space $\H$ and a $*$-antihomomorphism $\pi:\C\ra\BH$. From \eqref{Cstarhomo.norm.inequality} it follows that every representation $(\H,\pi)$ is continuous with respect to the norm topologies of $\C$ and $\BH$ and satisfies $\n{\pi(x)}\leq\n{x}\;\forall x\in\C$. A representation $\pi:\C\ra\BH$ is called: \df{nondegenerate} if{}f $\{\pi(x)\xi\mid (x,\xi)\in\C\times\H\}$ is dense in $\H$; \df{nonzero} if{}f $\pi(\C)\neq\{0\}$; \df{normal} if{}f it is continuous with respect to the weak-$\star$ topologies of $\C$ and $\BH$; \df{faithful} if{}f it satisfies any of the equivalent conditions:
\begin{enumerate}
\item[i)] $\ker(\pi)=\{0\}$ (i.e., it is a $*$-isomorphism),
\item[ii)] $\pi(x)\geq0$ and $\pi(x)\neq0$ $\forall x\in\C^+\setminus\{0\}$,
\item[iii)] $\n{\pi(x)}=\n{x}\;\forall x\in\C$.  
\end{enumerate}
For any representation $\pi:\C\ra\BH$ the space $\ker(\pi)$ is an ideal in $\C$, and $(\H,\pi)$ is a faithful representation of a quotient $C^*$-algebra $\C/\ker(\pi)$. Any representation $\pi:\C\ra\BH$ of a $C^*$-algebra is nondegenerate if{}f $\pi(\II)=\II$. From the fact that every $*$-isomorphism of $W^*$-algebras is weakly-$\star$ continuous it follows that every faithful representation of a $W^*$-algebra is also normal. The representations $\pi_1:\C\ra\BBB(\H_1)$ and $\pi_2:\C\ra\BBB(\H_2)$ are called \df{unitarily equivalent} if{}f there exists a unitary operator $u:\H_1\ra\H_2$ such that $\pi_2(x)=u\pi_1(x)u^{-1}\;\forall x\in\C$. If the representations $\pi_1:\C\ra\BBB(\H_1)$ and $\pi_2:\C\ra\BBB(\H_2)$ are not unitarily equivalent, they are called \df{unitarily inequivalent}. 

An element $\xi\in\H$ is called \df{cyclic} for a $C^*$-algebra $\C\subseteq\BH$ if{}f $\C\xi:=\bigcup_{x\in\C}\{x\xi\}$ is norm dense in $\BH$. Hence, if $\xi\in\H$ is cyclic for $\C$, then
\begin{equation}
        \H=\overline{\C\xi}=\overline{\bigcup_{x\in\C}\{x\xi\}},
\end{equation}
where the completion is provided in norm topology of $\H$. An element $\xi\in\H$ is called \df{separating} for a $C^*$-algebra $\C\subseteq\BH$ if{}f 
\begin{equation}
        x\xi=0\limp x=0\;\forall x\in\C.
\end{equation}
A representation $\pi:\C\ra\BH$ of a $C^*$-algebra $\C$ is called \df{cyclic} if{}f there exists $\Omega\in\H$ that is cyclic for $\pi(\C)$. Every cyclic representation is nondegenerate. Every nondegenerate representation $\pi(\C)$ is a (countable or noncountable) direct sum of cyclic representations. According to the Gel'fand--Na\u{\i}mark--Segal theorem \cite{Gelfand:Naimark:1943,Segal:1947:irreducible} for every pair $(\C,\omega)$ of a $C^*$-algebra $\C$ and $\omega\in\C^{\banach+}$ there exists a triple $(\H_\omega,\pi_\omega,\Omega_\omega)$ of a Hilbert space $\H_\omega$ and a cyclic representation $\rpktarget{pi.omega}\pi_\omega:\C\ra\BH$ with a cyclic vector $\rpktarget{omega.omega}\Omega_\omega\in\H_\omega$\rpktarget{h.omega}, and this triple is unique up to unitary equivalence. The proof of this theorem is provided by the following explicit construction. For a $C^*$-algebra $\C$ and $\omega\in\C^{\banach+}$, one defines the scalar form $\rpktarget{scal.omega}\s{\cdot,\cdot}_\omega$ on $\C$,
\begin{equation}
        \s{x,y}_\omega := \omega(x^*y)\;\;\forall x,y\in\C,
\end{equation}
and the \df{Gel'fand ideal} 
\begin{equation}
        \I_\omega:=\{x\in\C\mid\omega(x^*x)=0\}=\{x\in\C\mid\omega(x^*y)=0\;\forall y\in\C\},
\end{equation}
which is a left ideal of $\C$, closed in the norm topology (it is also closed in the weak-$\star$ topology if $\omega\in\C^{\banach+}_\star$). The form $\s{\cdot,\cdot}_\omega$ is hermitean on $\C$ and it becomes a scalar product $\s{\cdot,\cdot}_\omega$ on $\C/\I_\omega$. The Hilbert space $\H_\omega$ is obtained by the completion of $\C/\I_\omega$ in the topology of norm generated by $\s{\cdot,\cdot}_\omega$. Consider the morphisms\rpktarget{rep.omega}
\begin{align}
        [\cdot]_\omega:\C\ni x&\longmapsto [x]_\omega\in\C/\I_\omega,\\
        \pi_\omega(y):[y]_\omega&\longmapsto[xy]_\omega.
\end{align}
From $\II\n{x}^2\geq x^*x$ one obtains the inequality
\begin{equation}
        \n{\pi_\omega(x)[y]_\omega}^2=
        \n{[xy]_\omega}^2=
        \omega(y^*x^*xy)\leq\n{x}^2\omega(y^*y)=
        \n{x}^2\n{[y]_\omega}^2
        \;\;\forall x,y\in\C,
\label{GNS.boundedness}
\end{equation}
which shows that $\pi_\omega(x)$ is a bounded operator on $\H_\omega$ for each $x\in\C$. On the other hand,
\begin{equation}
        \s{[y]_\omega,\pi_\omega(x)[z]_\omega}_\omega=
        \omega(y^*xz)=
        \omega((x^*y)^*z)=
        \s{\pi_\omega(x^*)[y]_\omega,[z]_\omega}_\omega
        \;\;\forall x,y,z\in\C
\label{GNS.star.preservation}
\end{equation}
shows that $\pi_\omega(x^*)=\pi_\omega(x)^*$. Hence, the map $\pi_\omega:\C\ni x\mapsto\pi_\omega(x)\in\BBB(\H_\omega)$ is a $*$-representation. 

A triple $(\H_\omega,\pi_\omega,\Omega_\omega)$ is called the \df{Gel'fand--Na\u{\i}mark--Segal representation}. The space $\C/\I_\omega$ is a $C^*$-algebra, so $\pi_\omega(\C)$ is a $C^*$-algebra too. The GNS representation is nondegenerate. If $\omega\in\N_\star^+$ for a given $W^*$-algebra $\N$, then $(\H_\omega,\pi_\omega,\Omega_\omega)$ is faithful if{}f $\supp_\zentr(\omega)=\II$ (in such case $\omega$ is called \df{centrally faithful}). The element $\omega\in\C^{\banach+}$ is uniquely represented in terms of $\H_\omega$ by the vector $[\II]_\omega=:\Omega_\omega\in\H_\omega$, which is cyclic for $\pi_\omega(\C)$ and satisfies $\n{\Omega_\omega}=\n{\omega}$. Hence
\begin{align}
        \omega(x)&=\s{\Omega_\omega,\pi_\omega(x)\Omega_\omega}_\omega
        \;\;\forall x\in\C,
\label{dense.omega}\\
                                \omega(y^*x)&=\s{\pi_\omega(y)\Omega_\omega,\pi_\omega(x)\Omega_\omega}_\omega\;\;\forall x,y\in\C.
\end{align}
Given $C^*$-algebra $\C$, if $\omega\in\C^{\banach+}_0$ then its cyclic GNS representative $\Omega_\omega\in\H_\omega$ is also separating for $\pi_\omega(\C)$. Hence, $\pi_\omega(x)\Omega_\omega\neq0\;\forall\pi_\omega(x)\neq0$, because such GNS representation is faithful. Every cyclic representation $(\H,\pi)$ of a $C^*$-algebra $\C$ with a cyclic vector $\xi$ is unitarily equivalent to a GNS representation $(\H_\phi,\pi_\phi,\xi)$ of $\C$, where $\phi\in\C^{\banach+}$ satisfies $[\II]_\phi=\xi$. Moreover, every representation of a $C^*$-algebra can be decomposed as a (countable or noncountable) direct sum of representations that are unitarily equivalent to the GNS representation. 

An analogue of the Gel'fand--Na\u{\i}mark--Segal representation theorem for weights follows the similar construction, but lacks cyclicity. If $\C$ is a $C^*$-algebra, and $\omega$ is a weight on $\C$, then there exists the Hilbert space $\H_\omega$, defined as the completion of $\nnn_\omega/\ker(\omega)$ in the topology of a norm generated by the scalar product $\s{\cdot,\cdot}_{\omega}:\nnn_\omega\times\nnn_\omega\ni(x,y)\mapsto\omega(x^*y)\in\CC$,
\begin{equation}\rpktarget{h.omega.zwei}
        \H_\omega:=\overline{\nnn_\omega/\ker(\omega)}=\overline{\{x\in\C\mid\omega(x^*x)<\infty\}/\{x\in\C\mid\omega(x^*x)=0\}}=\overline{\nnn_\omega/\I_\omega},
\end{equation}
and there exist the maps\rpktarget{rep.omega.zwei}\rpktarget{pi.omega.zwei} 
\begin{align}
        [\cdot]_\omega:\nnn_\omega\ni x&\mapsto [x]_\omega\in\H_\omega,
        \label{GNS.class.weight}\\
        \pi_\omega:\C\ni x&\mapsto([y]_\omega\mapsto[xy]_\omega)\in\BBB(\H_\omega),
        \label{GNS.rep.weight}
\end{align}
such that $[\cdot]_\omega$ is linear, $\ran([\cdot]_\omega)$ is dense in $\H_\omega$, and $(\H_\omega,\pi_\omega)$ is a representation of $\C$ (which follows from \eqref{GNS.boundedness} and \eqref{GNS.star.preservation} with domains of variables substituted accordingly to \eqref{GNS.class.weight} and \eqref{GNS.rep.weight}). If $\omega\in\W(\N)$ for a $W^*$-algebra $\N$, then $\pi_\omega$ is nondegenerate and normal. If $\omega\in\W_0(\N)$, then $\pi_\omega$ is also faithful.
\subsection{von Neumann algebras\label{vNa.section}}
The \df{commutant} of a subalgebra $\N$ of any algebra $\C$ is defined as 
\begin{equation}\rpktarget{comm}
        \N^\comm:=\{y\in\C\mid xy=yx\;\forall x\in\N\},
\end{equation}
while the \df{center} of $\N$ is defined as $\rpktarget{zentr}\zentr_\N:=\N\cap\N^\comm$. The commutant operation satisfies 
\begin{equation}
        \N_1\subseteq\N_2\limp\N_1^\comm\subseteq\N_2^\comm,\;
        \N\subseteq\N^\comm{}^\comm,\;
        \N^\comm{}^\comm{}^\comm=\N^\comm.
\end{equation}
If $\N\subseteq\N^\comm$, then $\N$ is a commutative algebra. A subalgebra $\N$ of a unital algebra $\C$ over $\CC$ is called: \df{irreducible} if{}f $\N^\comm=\CC\II$; \df{reducible} if{}f $\N^\comm\neq\CC\II$; a \df{factor} if{}f $\zentr_\N=\CC\II$. Every unital commutative algebra is a factor. A unital $*$-subalgebra $\N$ of an algebra $\BH$ is called the \df{von Neumann algebra} \cite{vonNeumann:1930:algebra,Murray:vonNeumann:1936} if{}f $\N=\N^\comm{}^\comm$. From von Neumann's double commutant theorem \cite{vonNeumann:1930:algebra} it follows that this is equivalent with any of the conditions: $\N$ is weakly-$\star$ closed, $\N$ is ultrastrongly closed, $\N$ is ultrastrongly-$\star$ closed. In particular, $\BH$ is a von Neumann algebra. 

An image $\pi(\N)$ of any representation $(\H,\pi)$ of a $W^*$-algebra $\N$ is a von Neumann algebra if{}f $\pi$ is normal and nondegenerate. In particular, this is always the case for GNS representation $\pi_\omega(\N)$ generated by $\omega\in\N^+_\star$ or $\omega\in\W(\N)$ for any $W^*$-algebra $\N$. By the Sakai theorem \cite{Sakai:1956,Sakai:1971}, for every $W^*$-algebra $\N$ there exists a faithful representation $(\H,\pi)$ such that $\pi(\N)$ is a von Neumann algebra on $\H$. Because any faithful representation is a $*$-isomorphism, this means that the topological structure of von Neumann algebras is determined by their algebraic structure \cite{Dixmier:1953,Takesaki:2003:entrance}. Every pair of $C^*$-algebra $\C$ and $\omega\in\C^{\banach+}$ generates the von Neumann algebra $(\pi_\omega(\C))^\comm{}^\comm\subseteq\BH$, which is called an \df{enveloping von Neumann algebra}. If $\omega\in\N^+_{\star0}$ for a given $W^*$-algebra $\N$, then $\N\iso\pi_\omega(\N)\iso\pi_\omega(\N)^\comm{}^\comm$. For any $C^*$-algebra $\C$, $\omega\in\C^{\banach+}$ determines a unique centrally faithful $\widetilde{\omega}\in(\pi_\omega(\C)^\comm{}^\comm)^+_\star$ such that 
\begin{equation}
        \widetilde{\omega}|_{\pi_\omega(\C)}=
        \s{\Omega_\omega,\,\cdot\,\Omega_\omega}_\omega\in(\pi_\omega(\C))^{\banach+}
\end{equation}
and $\pi_\omega(\C)$ is dense in $\pi_\omega(\C)^\comm{}^\comm$. With an abuse of notation, $\widetilde{\omega}$ is often denoted $\omega$, an it is called a \df{normal extension} of $\omega$ to $\pi_\omega(\C)^\comm{}^\comm$. If $\omega\in\C^{\banach+}$ is tracial, then $\widetilde{\omega}$ is tracial too.

The representation $(\H,\pi)$ of a $C^*$-algebra $\C$ is called \df{irreducible} if{}f $\pi(\C)$ is irreducible, hence, if{}f $\pi(\C)^\comm=\CC\II$. This is equivalent to the condition: (every $\H\ni\xi\neq0$ is cyclic for $\pi(\C)$) or ($\pi(\C)=\{0\}$ and $\H=\CC$). The GNS representation $\pi_\omega$ and the algebra $\pi_\omega(\C)$ are irreducible if{}f $\omega$ is pure \cite{Gelfand:Raikov:1943,Segal:1947:irreducible}. Thus, a nonzero representation is irreducible if{}f it is unitarily equivalent to a GNS representation associated with a pure $\omega$. 

If $\omega(x):=\s{\xi_\omega,x\xi_\omega}_\H\forall x\in\N\subseteq\BH$ then $\xi_\omega$ is cyclic if{}f $\supp(\omega)^\comm=\II$. On the other hand, an element $\xi\in\H$ is separating for $\N$ if{}f $\omega_\xi(\cdot):=\s{\xi,\cdot\;\xi}$ is faithful for $\N$. From this it follows that $\omega$ is faithful for $\N$ if{}f $\xi_\omega$ is cyclic for $\N^\comm$ and
\begin{equation}
        (\xi\mbox{ is separating for }\N)\iff(\xi\mbox{ is cyclic for } \N^\comm)\;\;\forall\N\subseteq\BH\;\forall\xi\in\H.
\end{equation}
Hence, a vector $\xi$ is cyclic and separating for a von Neumann algebra $\N$ if $\N\xi$ and $\N^\comm\xi$ are dense in $\H$. Moreover, the following conditions for a $W^*$-algebra $\N$ are equivalent: 
\begin{enumerate}
\item[i)] $\N^+_{\star0}\neq\varnothing$,
\item[ii)] $\N$ is $*$-isomorphic to a von Neumann algebra possessing a cyclic and separating vector,
\item[iii)] every family $\{P_n\}\subseteq\Proj(\N)$ satisfying $P_iP_j=\dirac_{ij}\II$ is countable.
\end{enumerate}
If any of these conditions is satisfied, then $\N$ is called \df{countably finite}. In particular, every von Neumann algebra $\N\subseteq\BH$ is countably finite if $\H$ is separable, and $\BH$ is countably finite if{}f $\H$ is separable.\footnote{A topological space $X$ is called \df{separable} if{}f it contains a countable dense subset, that is, if there exists a sequence $\{x_i\}\subseteq X$ such that every open nonempty subset of $X$ contains an element of this sequence.}
\subsection{Barycentric decompositions\label{barycentric.decomp.section}}
\ifvarwstarintcompile 
\textit{Note: this section is included for the purpose of completeness of exposition, but its contents will not be used in the rest of this text, with an exception of the last paragraph of Section \ref{KMS.section}.}
\else 
\textit{Note: this section is included for the purpose of completeness of exposition, but its contents will not be used in the rest of this Chapter, with an exception of the last paragraph of Section \ref{KMS.section}.}
\fi 

{\vskip 1em}\noindent Let $K$ be a convex subset of a real topological vector space $X$. An element $x\in K$ is called \df{extreme} if{}f
\begin{equation}
        \exists x_1,x_2\in K\;\;\;(x_1\neq x_2\;\;\mbox{and}\;\;\exists\lambda\in\,]0,1[\;\;\;x=(1-\lambda)x_1+\lambda x_2)
\end{equation}
is false. The set of all extreme elements of $K$ is denoted $\ex(K)$. A subset $F\subseteq K$ is called a \df{face} if{}f 
\begin{equation}
        \forall x\in F\;\;
        \exists n\in\NN\;\;\left(
                \exists\{\lambda_i\}_{i=1}^n\subseteq\RR^+\;\;
                x=\sum_{i=1}^n\lambda_ix_i,\;\;
                \sum_{i=1}^n\lambda_i=1
        \right)
        \;\;\limp\;\;
                \{x_i\}_{i=1}^n\subseteq F.     
\end{equation}
If $Y\subseteq K$, then the smallest (with respect to set embedding) among the faces of $K$ that contain $Y$ is called to be \df{generated} by $Y$ in $K$, and is denoted  $\face(Y)$. If $x\in K$, then one also defines
\begin{equation}\rpktarget{face}
        \face(x):=\{y\in K\mid\exists\epsilon>0\;\;x+\epsilon(x-y)\in K\}.
\end{equation}
A \df{convex hull} of $Y\subseteq X$ is defined by
\begin{equation}\rpktarget{co}
        \co(Y):=\left\{x\in X\mid\exists n\in\NN\;\;\exists\{x_i\}_{i=1}^n\subseteq Y\;\;\exists\{\lambda_i\}_{i=1}^n\subseteq\RR^+\;\;\sum_{i=1}^n\lambda_i=1,\;\;x=\sum_{i=1}^n\lambda_ix_i\right\}.
\end{equation} 
The elements $y,z\in K$ will be called \df{strongly disjoint} and denoted $y\strdisj z$ if{}f
\begin{equation}
\left\{
                                \begin{array}{l}
                                        \face(\face(y)\cup\face(z))=\co(\face(y)\cup\face (z))\\
                                        \Span_\RR\face(y)\cap\Span_\RR\face(z)=\{0\}.
                                \end{array}
                        \right.
\end{equation}
A point $x\in K$ will be called \df{primary} if{}f
\begin{equation}
        x=\lambda y+(1-\lambda) z
        \;\limp\;
                \left(
                        (y\strdisj z)
                        \mbox{ is false}
        \right)
        \;\;\;
        \forall y,z\in K\;\;\forall \lambda\in\,]0,1[.
\end{equation}

In what follows, we will use some of the notions introduced and discussed in Section \ref{comm.integr.section}, marking them by \textit{italics}. We will assume that $X$ is \textit{locally compact} \textit{Hausdorff} real topological vector space, and that $K$ is its \textit{compact} convex subset. The set $\Rad(K)^+\rpktarget{RADON.BARY}$ of \textit{Radon measures} $\tmu:\mho_{\mathrm{Borel}}(K)\ra[0,+\infty]$ is \textit{order preserving} isometrically isomorphic to a subset of a Banach dual $\mathrm{C}(K)^\banach$ of a space $\mathrm{C}(K)$ of all continuous functions on $K$. As a result, $\Rad(K)^+$ and $\Rad(K)^+_1:=\{\tmu\in\Rad(K)^+\mid\n{\tmu}=1\}$ can be equipped with a weak-$\star$ topology on $\mathrm{C}(K)^\banach$ with respect to $\mathrm{C}(K)\rpktarget{CONT.BARY}$. If $\tmu\in\Rad(K)^+_1$, then $\tmu$ is called: \df{supported} by $C\subseteq K$ if{}f $\tmu(C)=1$; \df{pseudosupported} by $C\subseteq K$ if{}f $\tmu(Z)=0$ for every compact countable intersection $Z$ of open sets in $K$ such that $Z\cap C=\varnothing$. A \df{barycenter} of nonzero $\tmu\in\Rad(K)^+$ is defined as
\begin{equation}\rpktarget{BARY}
        \bary_K(\tmu):=\frac{1}{\tmu(K)}\int_K\tmu(x)x.
\end{equation}
Every $\tmu\in\Rad(K)^+_1$ has a unique barycenter. The set of all elements of $\Rad(K)^+_1$ with barycenter $x\in K$ will be denoted $\Rad_x(K)^+_1$. A measure $\tmu\in\Rad(K)^+_1$ is called \df{central} if{}f
\begin{equation}
        \face\left(
                \bary_K\left(
                        \lambda^{-1}\tmu|_{\Y}
                \right)
        \right)
        \strdisj
        \face\left(
                \bary_K\left(
                        (1-\lambda)^{-1}\tmu|_{K\setminus\Y}
                \right)
        \right)
\end{equation}
for all $\Y\in\mho_{\mathrm{Borel}}(K)$ such that $\lambda=\tmu(\Y)\in\,]0,1[$. Consider an order relation $\prec$ on $\Rad(K)^+$ defined by \cite{Choquet:1960}
\begin{equation}
        \tmu_1\prec\tmu_2\;\;:\iff\;\;\int\tmu_1f\leq\int\tmu_2f\;\;\;\;\forall\mbox{ convex }f\in\mathrm{C}(K;\RR).
\end{equation}
A measure $\tmu\in\Rad_\omega(K)^+_1$ is called \df{maximal central} if{}f it is central and is maximal in terms of $\prec$ among all central measures in $\Rad_\omega(K)^+_1$. According to Wils' theorem \cite{Wils:1968,Wils:1969}, which generalises earlier result of Sakai \cite{Sakai:1965}, for every $x\in K$ there exists a unique maximal central measure $\tmu\in\Rad(K)^+_1$ such that $x$ is the barycenter of $\tmu$, and $\tmu$ is pseudosupported by the set of primary points of $K$.\footnote{See e.g. \cite{Alfsen:1971} for a detailed exposition.}

A nonempty compact convex subset $K$ of a locally convex vector space $X$ is called a \df{Choquet simplex} \cite{Choquet:1956:I,Choquet:1956:II,Choquet:1956:III,Choquet:1956:IV} if{}f is is contained in a closed hyperplane\footnote{A vector subspace $Y$ of a vector space $X$ is called a \df{hyperplane} if{}f $X/Y$ is one-dimensional.} not containing the origin of $X$ and the set $\{\lambda x\mid\lambda\geq0,\;\;x\in K\}$ is a lattice with respect to the ordering defined by
\begin{equation}
        x\geq y\;\;:\iff\;\;x-y\in K\;\;\;\forall x,y\in X.
\end{equation}
According to the Choquet representation theorem \cite{Choquet:1956:I,Choquet:1956:II,Choquet:1956:III,Choquet:1956:IV}, a compact convex metrisable\footnote{A topological space $X$ (with topology $\mathcal{T}$) is called \df{metrisable} if{}f there exists a metrical distance function $d:X\times X\ra\RR^+$ such that the topology induced on $X$ by $d$ is $\mathcal{T}$.} subset $K$ of a locally convex vector space $X$ is a Choquet simplex if{}f
\begin{equation}
        \forall x\in K\;\;\exists!\tmu\in\Rad(K)^+_1\;\;\;\;\bary_K(\tmu)=x\;\;\;\mbox{and}\;\;\;\tmu(\ex(K))=1.
\end{equation}

Let $\C$ be a $C^*$-algebra. An element $\omega\in\C^{\banach+}$ will be called \df{factorial} if{}f $\pi_\omega(\C)^\comm{}^\comm$ is a factor, or, equivalently, if{}f
\begin{equation}
        \zentr_{\pi_\omega(\C)^\comm{}^\comm}\equiv{\pi_\omega(\C)^\comm{}^\comm}\cap{\pi_\omega(\C)^\comm{}}=\CC\II.
\end{equation}
The set of all factorial elements of $\C^{\banach+}$ will be denoted $\fact(\C)\rpktarget{FACT}$. A pair $\omega_1,\omega_2\in\C^{\banach+}$ will be called \df{orthogonal} (and denoted by $\omega_1\perp\omega_2$) if{}f
\begin{equation}
        (\omega\leq\omega_1,\;\;\omega\leq\omega_2)
        \;\;\limp\;\;
        \omega=0
        \;\;\;\forall\omega\in\C^{\banach+}.
\end{equation}
Recall that $\Scal(\C)$ is a convex set that is compact in the weak-$\star$ topology on $\C^\banach$. If $\tmu$ is a \textit{compactly inner regular} \textit{Borel measure} on $\mho_{\mathrm{Borel}}(\Scal(\C))$, and
\begin{equation}
        \left(
                \int_\Y\tmu(\omega)\omega
        \right)
        \perp
        \left(
                \int_{\Scal(\C)\setminus\Y}\tmu(\omega)\omega
        \right)
        \;\;\;\forall\Y\in\mho_{\mathrm{Borel}}(\Scal(\C)),
\end{equation}
then $\tmu$ is called \df{orthogonal}. The set of all orthogonal elements of $\Rad_\omega(\Scal(\C))^+_1$ will be denoted by $\Rad^\perp_\omega(\Scal(\C))^+_1$. Given $x\in\C$, let $\hat{x}$ denote an affine continuous function $\Scal(\C)\ra\CC$ defined by $\hat{x}(\phi):=\phi(x)$. According to the Tomita--Ruelle theorem \cite{Tomita:1956,Ruelle:1970:integral}, for any $\omega\in\Scal(\C)$ there is a bijection between:
\begin{enumerate}
\item[1)] the elements $\tmu\in\Rad_\omega^\perp(\Scal(\C))^+_1$,
\item[2)] the commutative von Neumann subalgebras $\N\subseteq\pi_\omega(\C)^\comm$,
\item[3)] the elements $P\in\Proj(\BBB(\H_\omega))$ such that $P\Omega_\omega$ and $P\pi_\omega(\C)P\subseteq\{P\pi_\omega(\C)P\}^\comm$.
\end{enumerate}
It is provided by the following relations:
\begin{enumerate}
\item[i)] $P$ is a projection onto $[\N\Omega_\omega]$,
\item[ii)] $\N=\{\pi_\omega(\C)\cup P\}^\comm$,
\item[iii)] $\tmu(\widehat{x_1}\cdots\widehat{x_n})=\s{\Omega_\omega,\pi_\omega(x_1)P\cdots\pi_\omega(x_n)P\Omega_\omega}_\omega$,
\item[iv)] there exists a $*$-isomorphism between $\N$ and $\ran(w_\omega^{\tmu})$, where
\begin{equation}
        w_\omega^{\tmu}:L_\infty(\Scal(\C),\mho_{\mathrm{Borel}}(\Scal(\C)),\tmu)\ra\pi_\omega(\C)^\comm
\label{w.omega.tmu.map}
\end{equation}
is a positive map defined by
\begin{align}
        \s{\Omega_\omega,w_\omega^\tmu(f)\pi_\omega(x)\Omega_\omega}_\omega
        &=\int_{\Scal(\C)}\tmu(\phi)f(\phi)\hat{x}(\phi)\;\;\;\forall x\in\C\;\forall f\in L_\infty(\Scal(\C),\mho_{\mathrm{Borel}}(\Scal(\C)),\tmu),
        \\
        w_\omega^\tmu(\hat{y})\pi_\omega(x)\Omega_\omega
        &=\pi_\omega(x)P\pi_\omega(y)\Omega_\omega\;\;\forall x,y\in\C.
\end{align}
\end{enumerate}

Let us recall that a \df{spectral measure} is defined as a function $\rpktarget{pvm}\pvm:\mho_{\mathrm{Borel}}(\RR)\ra\Proj(\BH)$ such that
\begin{enumerate}
\item[(i)] $\pvm(\Y)\geq\pvm(\varnothing)=0$ $\forall\Y\in\mho_{\mathrm{Borel}}(\RR)$,
\item[(ii)] $\pvm(\RR)=\II$,
\item[(iii)] $\pvm(\bigcup^{\infty}_{i=1}\Y_i)=\sum_{i=1}^\infty\pvm(\Y_i)$ for each countable sequence $\{\Y_i\mid i\in\NN\cup\{+\infty\}\}\subseteq\mho_{\mathrm{Borel}}(\RR)$ of mutually disjoint sets, where the series converges in the weak-$\star$ topology,
\end{enumerate}
while a \df{semi-spectral measure} \cite{Carleman:1923,Naimark:1940} is defined as a function $\rpktarget{povm}\povm:\mho(\X)\ra\BH^+$ such that
\begin{enumerate}
\item[(i)] $\povm(\Y)\geq\povm(\varnothing)=0$ $\forall\Y\in\mho(\X)$,
\item[(ii)] $\povm(\X)=\II$,
\item[(iii)] $\povm(\bigcup^{\infty}_{i=1}\Y_i)=\sum_{i=1}^\infty\povm(\Y_i)$ for each countable sequence $\{\Y_i\mid i\in\NN\cup\{+\infty\}\}\subseteq\mho(\X)$ for mutually disjoint sets, where the series converges in the weak-$\star$ topology.
\end{enumerate}
Let $\chr(\Y)$ denote the characteristic function of $\Y\subseteq\RR$. For arbitrary $\xi,\zeta\in\H$, consider a $*$-homomorphism $Q$ from the set of bounded $\mho_{\mathrm{Borel}}(\RR)$-measurable functions on $\RR$ to $\BH$, defined by
\begin{equation}
        \s{\xi,Q(f)\zeta}_\H=\int_\RR\mu_{\xi,\zeta}f\;\;\;\forall f\in\mathrm{C}(\RR),
\end{equation}
where $\mu_{\xi,\zeta}$ is unique by the Riesz representation theorem (see Section \ref{comm.integr.section}). According to the spectral representation theorem \cite{vonNeumann:1930:Allgemeine,vonNeumann:1930:algebra,Riesz:1930:Hilbert,Stone:1932:linear,vonNeumann:1932:grundlagen}, there is a bijection between self-adjoint operators $x$ on abstract separable Hilbert space $\H$ and spectral measures $\pvm:\mho_{\mathrm{Borel}}(\RR)\ra\Proj(\BH)$, provided by
\begin{equation}
        x=\int_{\dom(\pvm)}\pvm(\lambda)\lambda,\;\;\;\sp(x)=\dom(P),\;\;\;P(\lambda)=Q(\chr([0,\lambda])).
\label{spectral.theorem.bijection}
\end{equation}
This determines a unitary isomorphism between $\H$ and a particular Hilbert space $\H_{x,\xi}$,
\begin{equation}
        U_{x,\xi}:\H\ra\H_{x,\xi}:=L_2(\sp(x),\mho_{\mathrm{Borel}}(\sp(x)),\s{\xi,\pvm^x(\cdot)\xi}_\H),
\end{equation}
such that, whenever $x\in\BH^\sa$, the commutative algebra generated by $x$ is represented as a subalgebra of $L_\infty(\sp(x),\mho_{\mathrm{Borel}}(\sp(x)),\s{\xi,\pvm^x(\cdot)\xi}_\H)$ acting on $\H_{x,\xi}$ by right multiplication. The element $\xi\in\H$ is required to be cyclic for $x$ (but otherwise arbitrary), $\sp(x)$ is the spectrum of $x$, while $\pvm^x$ denotes the spectral measure associated to $x$ by means of \eqref{spectral.theorem.bijection}. More generally, one can begin with a commutative subalgebra $\C\subseteq\BH^\sa$, and consider a unitary isomorphism of $\H$ with $L_2(\sp_{\mathrm{G}}(\C),\mho_{\mathrm{Borel}}(\sp_{\mathrm{G}}(\C)),\tr_\H(\rho\pvm^x(\cdot)))$, for $x\in\C^\sa$, $\rho\in\schatten_1(\H)^+$, and $\sp_{\mathrm{G}}(\C)$ denoting the Gel'fand spectrum of $\C$ (see Section \ref{integr.compar.section}). Hence, different choices of $(x,\xi)$ (or $(\C,\rho)$, or $(\C,\xi)$, or $(x,\rho)$) provide different `commutative snapshots' of the noncommutative algebra $\BH$.

As observed in \cite{Halvorson:2004:remote}, for a given choice of a normalised measure on $\Scal(\C)$ with barycenter $\omega$, the Tomita--Ruelle theorem determines a unique semi-spectral measure on $\Scal(\C)$ valued in $\pi_\omega(\C)^\comm$. More precisely, given a $C^*$-algebra $\C$ and $\tmu\in\Rad_\omega^\perp(\Scal(\C))^+_1$, there exists a unique semi-spectral measure $\povm_\omega^\tmu:\mho_{\mathrm{Borel}}(\Scal(\C))\ra\BBB(\H_\omega)^+$ such that 
\begin{equation}
        \s{\left(\povm_\omega^\tmu(\Y)\right)^{1/2}\Omega_\omega,\pi_\omega(x)\left(\povm_\omega^\tmu(\Y)\right)^{1/2}\Omega_\omega}_\omega=\int_{\Y}\tmu(\phi)\phi(x)\;\;\forall x\in\C\;\;\forall\Y\in\mho_{\mathrm{Borel}}(\Scal(\C)).
\label{tomita.povm}
\end{equation}
By definition, $\povm_\omega^\tmu$ is supported on the same set that $\tmu$ is.

If $K=\Scal(\C)$ for a given $C^*$-algebra $\C$, then the unique measure associated to any $\omega\in\Scal(\C)$ by Wils' theorem is given by
\begin{equation}
        \tmu_{\zentr_{\pi_\omega(\C)^\comm{}^\comm}}\in\Rad^\perp_\omega(\Scal(\C))^+_1,
\end{equation}
which is determined by the Tomita--Ruelle theorem as the measure corresponding to $\N=\zentr_{\pi_\omega(\C)^\comm{}^\comm}$. The measure $\tmu_{\zentr_{\pi_\omega(\C)^\comm{}^\comm}}$ is pseudosupported by the set $\fact(\C)\cap\Scal(\C)$. If $\C$ is separable in norm topology, then $\fact(\C)\cap\Scal(\C)\in\mho_{\mathrm{Borel}}(\Scal(\C))$ \cite{Sakai:1971}. Hence, in such case $\tmu_{\zentr_{\pi_\omega(\C)^\comm{}^\comm}}$ is supported on $\fact(\C)\cap\Scal(\C)$. (On the other hand, if $\C$ is an arbitrary $C^*$-algebra, $\omega\in\Scal(\C)$, and $\H_\omega$ is separable, then there exists a $\tmu_{\zentr_{\pi_\omega(\C)^\comm{}^\comm}}$-measurable subset $\Y\in\fact(\C)\cap\Scal(\C)$ such that $\tmu_{\zentr_{\pi_\omega(\C)^\comm{}^\comm}}(\Y)=1$.) If $\omega$ is not factorial, then $\zentr_{\pi_\omega(\C)^\comm{}^\comm}\neq\CC\II$. So, if $\C$ is separable in norm topology and $\omega\not\in\fact(\C)$, then the spectral decomposition\footnote{The symbol $\int^\oplus$ denotes direct integral in von Neumann's \cite{vonNeumann:1949} sense. See any exposition of spectral representation theorem for a detailed definition.}
\begin{equation}
        U:\H_\omega\iso\int^\oplus_{\sp_{\mathrm{G}}\left(\zentr_{\pi_\omega(\C)^\comm{}^\comm}\right)}\tmu_{\zentr_{\pi_\omega(\C)^\comm{}^\comm}}(\Y)\H_\Y,
\label{spec.decomp.zentr}
\end{equation}
is supported on the factorial elements of $\Scal(\C)$. This means that every operator acting on $\H_\omega$ can be spectrally represented as a direct integral of operators acting on the spaces $\H_\phi$ corresponding to, in general unitarily inequivalent, GNS representations for $\phi\in\fact(\C)\cap\Scal(\C)$,
\begin{equation}
        U\pi_\omega(x)U^*=\int^\oplus_{\sp_{\mathrm{G}}\left(\zentr_{\pi_\omega(\C)^\comm{}^\comm}\right)}\tmu_{\zentr_{\pi_\omega(\C)^\comm{}^\comm}}(\phi)\pi_\phi(x)\;\;\forall x\in\C.
\label{direct.integral.of.operators.primary}
\end{equation}
\section{Modular theory\label{modular.theory.section}}
The Hilbert space $\H_\omega$ of GNS representation of any finite dimensional $C^*$-algebra $\C$ has a structure of the Hilbert--Schmidt space $\schatten_2(\K)$, which can be considered as a noncommutative analogue of the $L_2(\X,\mho(\X),\tmu)$ space, or, more precisely, of \boldmath$\ell$\unboldmath$_2$ space (see Section \ref{integration.trace.section}). It is quite remarkable that in this case every $\omega\in\C^{\banach+}$ corresponds uniquely to a vector $\xi_\omega$ in the cone $\schatten_2(\K)^+$ of positive elements of $\schatten_2(\K)$: 
\begin{equation}
\xi_\omega=\rho^{1/2}_\omega\in\schatten_2(\K)^+
\;\;\iff\;\;
\tr_{\BBB(\H_\omega)}(\rho_\omega\pi_\omega(\,\cdot\,))=\omega\in\C^{\banach+}\;\;\iff\;\;
\rho_\omega\in\schatten_1(\K),
\end{equation}
where $\K$ is a Hilbert space, $\schatten_1(\K)$ is the space of all trace class (nuclear) operators on $\K$, and $\K\otimes\K^\banach=\schatten_2(\K)\iso\H_\omega$ is equipped with the inner product $\s{\xi_1,\xi_2}_{\schatten_2(\K)}:=\tr_{\BBB(\K)}(\xi_2^*\xi_1)$. 

In general case, the requirement that the elements of an algebra should be considered as dual to the functionals (and not the other way round) restricts the considerations from pairs $(\C,\C^{\banach+})$ for arbitrary $C^*$-algebras $\C$ to pairs $(\N_\star^+,\N)$ for arbitrary $W^*$-algebras $\N$. Quite remarkably, this setting allows to construct an analogue of the Hilbert--Schmidt space for an arbitrary $W^*$-algebra and to provide a similar representation of the space $\N_\star^+$ in terms of the positive cone inside this space in a way which allows to consider self-adjoint and positive elements, as well as polar decomposition.

For the standard trace $\tr$ on $\BH$, the ideal $\nnn_\tr$ is equal to the space $\schatten_2(\H)$ of all Hilbert--Schmidt operators, while the ideal $\mmm_\tr$ is equal to $\schatten_1(\H)$. In the case of $W^*$-algebra $\N$ admitting faithful normal semi-finite trace $\tau$, a proper generalisation of these properties can always be obtained by choosing a representation $(\H,\pi)$ which is \textit{semi-standard}, that is, there exists a conjugation operator\footnote{A linear operator $J:\dom(J)\ra\H$, where $\dom(J)\subseteq\H$ and $\H$ is a Hilbert space, is called a \df{conjugation} if{}f it is antilinear, isometric, and involutive ($J^2=\II$).} $J$ on $\H$ such that $J\pi(\N)J=\pi(\N)^\comm$ and $JxJ=x^*$ $\forall x\in\zentr_{\pi(\N)}$ \cite{Dixmier:1952}. The existence of an operator $J$ follows from the fact that the operation $^*:\N\ra\N$ is isometric with respect to the inner product on $\pi(\N)$ defined by $\tau$. However, if $\N$ does not admit such trace, this construction fails. A major breakthrough in the theory of operator algebras provided by Tomita \cite{Tomita:1967:a,Tomita:1967:b} can be summarised (in more modern terms) as: every faithful normal semi-finite weight $\psi$ on arbitrary $W^*$-algebra $\N$ determines a group of $*$-automorphisms $\{\Ad(\Delta^{\ii t}_\psi)\mid t\in\RR\}\subseteq\Aut(\pi_\psi(\N))$, with unitary $\Delta^{\ii t}_\psi$, and an antilinear $*$-isomorphism $j_\psi:\pi_\psi(\N)\ra J_\psi\pi_\psi(\N)J_\psi=\pi_\psi(\N)^\comm$, with a conjugation operator $J_\psi$, such that
\begin{equation}
        \Ad(\Delta_\psi^{\ii t})|_{\zentr_{\pi_\psi(\N)}}=
        \id_{\zentr_{\pi_\psi(\N)}},
        \;\;\;
        J_\psi xJ_\psi=x^*\;\forall x\in\zentr_{\pi_\psi(\N)}.
\end{equation}
The group $\pi_\psi^{-1}\circ\Ad(\Delta^{\ii t}_\psi)$ of $*$-automorphisms of $\N$ characterises the nontracial behaviour of $\psi$ ($\Delta_\psi=\pi_\psi(\II)$ if{}f $\psi$ is a trace). Among a huge amount of structural results and applications, Tomita's result has lead also to construction of an analogue of the Hilbert--Schmidt representation, developed for countably additive $W^*$-algebras by Araki \cite{Araki:1974:modular:conjugation} and Connes \cite{Connes:1974}, and for arbitrary $W^*$-algebras by Haagerup \cite{Haagerup:1973,Haagerup:1975:standard:form} (the latter is called a \textit{standard} representation). The novel property of these representations is that they determine not only a conjugation operator $J:\pi(\N)\ra\pi(\N)^\comm$, but also a convex cone $\stdcone\subseteq\H$ which admits an order preserving homeomorphism from $\N_\star^+$, thus forming a generalisation of $L_2(\X,\mho(\X),\tmu)^+$ and $\schatten_2(\K)^+$ spaces. Further generalisation of Tomita's theory by Connes \cite{Connes:1973:poids:normaux,Connes:1973:classification,Connes:1980} and Araki \cite{Araki:1973:relative:hamiltonian,Araki:1974:modular:conjugation} resulted in relative modular operators $\Delta_{\phi,\psi}$, Connes' cocycles $\Connes{\phi}{\psi}{t}$ and Connes' spatial quotients $\connes{\phi}{\psi}$, which characterise the relationships between two normal semi-finite weights on $\N$ (with Connes' cocycle playing the most important role). This allowed Kosaki \cite{Kosaki:1980:PhD} to construct a canonical $L_2(\N)$ space and a \textit{canonical} representation $(L_2(\N),\pi_\N,J_\N,L_2(\N)^+)$ associated to any $W^*$-algebra $\N$. They are constructed without invoking any auxiliary Hilbert spaces, and using the equivalence classes of the elements of $\N_\star^+$ that are induced by Connes' cocycle. In particular, $L_2(\N)^+$ is a self-polar convex cone defined by an embedding of $\N_\star^+$ into $L_2(\N)$ which preserves positivity, additivity and multiplication. This construction defines a functor from the category of $W^*$-algebras with $*$-isomorphisms to the category of standard representations with standard unitary equivalences.
\subsection{Kubo--Martin--Schwinger condition\label{KMS.section}}
If $\C$ is a $C^*$-algebra, then a group homomorphism $\alpha:\RR\ni t\mapsto\alpha_t\in\Aut(\C)$ and a group $\{\alpha_t\mid t\in\RR\}$ are called \df{strongly continuous} if{}f $t\mapsto\alpha_t(x)$ is continuous in norm topology of $\C$ for all $x\in\C$, that is,
\begin{equation}
        \lim_{t\ra0}\n{\alpha_t(x)-x}=0\;\;\forall x\in\C.
\end{equation}
If $\N$ is a $W^*$-algebra, then a group homomorphism $\alpha:\RR\ni t\mapsto\alpha_t\in\Aut(\N)$ and a group $\{\alpha_t\mid t\in\RR\}$ are called \df{weakly-$\star$ continuous} if{}f $t\mapsto\phi(\alpha_t(x))$ is a continuous function of $t$ for all $x\in\N$ and for all $\phi\in\N_\star$. If $\C$ is a $C^*$-algebra or a $*$-subalgebra of some $C^*$-algebra, and $\alpha:\RR\ra\Aut(\C)$ is strongly or weakly-$\star$ continuous group homomorphism, then\rpktarget{CAINF}
\begin{equation}
        \C^\alpha_\infty:=\{x\in\C\mid\exists!\mbox{ extension of }\alpha\mbox{ to an analytic function }\CC\ni z\mapsto\alpha_z(x)\in\C\}
\label{A.alpha.infty}
\end{equation}
is a $*$-subalgebra of $\C$.

Let $\C$ be (1a) a $C^*$-algebra or (1b) a $W^*$-algebra; let $\alpha:\RR\ni t\mapsto\alpha_t\in\Aut(\C)$ be (2a) strongly continuous or (2b) weakly-$\star$ continuous; let (3a) $\omega\in\C^{\banach+}$ or (3b) $\omega\in\W(\N)$; and let $\beta\in\RR\setminus\{0\}$. Then $\omega$ satisfies the \df{Kubo--Martin--Schwinger condition} for $\alpha$ and $\beta$ \cite{Kubo:1957,Martin:Schwinger:1959,Haag:Hugenholtz:Winnink:1967} (see also \cite{Winnink:1968,KPP:1969,Winnink:1970}) if{}f any of the following equivalent conditions holds:
\begin{enumerate}
\item[i)] $\C^\alpha_\infty\neq\varnothing$ and
\begin{equation}
        \omega(y\alpha_{z+\ii\beta}(x))=\omega(\alpha_z(x)y)\;\;\forall x\in\C^\alpha_\infty\;\forall y\in\C\;\forall z\in\CC;
\label{KMS.condition.analytic}
\end{equation}
\item[ii)]
for any $x,y\in\C$ (in the case (1a,2a,3a)) or for any $x,y\in\nnn_\omega\cap\nnn_\omega^*$ (in the case (1b,2b,3b)) there exists a map $F_{x,y}:\RR\ni t\mapsto F_{x,y}(t)\in\CC$ that can be continued for $t\in\CC$ in such a way that it is holomorphic for $\im(t)\in\,]0,\beta[$ (if $\beta\geq0$) or $\im(t)\in\,]\beta,0[$ (if $\beta<0$), as well as bounded and continuous in the closure of this domain, and
\begin{align}
                                F_{x,y}(t)&=\omega(x\alpha_t(y))\;\;\forall t\in\RR,\nonumber\\
        F_{x,y}(t+\ii\beta)&=\omega(\alpha_t(y)x)\;\;\forall t\in\RR;
\label{KMS.condition}
\end{align}
\item[iii)] 
\begin{equation}
        \int_{-\infty}^{+\infty}\dd t\,f(t)\omega(x\alpha_t(y))=\int_{-\infty}^{+\infty}\dd t\,f(t+\ii\beta)\omega(\alpha_t(y)x)\;\;\forall x,y\in\C,
\end{equation}
for all such $f$ that their inverse Fourier transform $\hat{f}$, given by
\begin{equation}
        f(t+\ii\beta)=\int_{-\infty}^{+\infty}\dd\lambda\,\hat{f}(\lambda)\ee^{-\ii\lambda t-\beta\lambda},
\end{equation}
is an infinitely differentiable function on $\RR$ with a compact support.
\end{enumerate}
Moreover, with the above notation and assumptions (1a,2a,3a) or (1b,2b,3b), $\omega$ is said to satisfy the KMS condition for $\alpha$ and:
\begin{enumerate}
\item[$0$)] $\beta=0$ if{}f $\omega$ is a trace and 
\begin{equation}
        \omega(\alpha_t(x))=\omega(x)\;\;\forall x\in\C\;\forall t\in\RR;
\end{equation}
\item[$\infty$)] $\beta=\infty$ if{}f $\C_\infty^\alpha\neq\varnothing$ and
\begin{equation}
        \im(z)\geq0\;\limp\;\ab{\omega(y\alpha_z(x))}\leq\n{x}\n{y}\;\;\forall x\in\C^\alpha_\infty\;\forall y\in\C\;\forall z\in\CC.
\end{equation}
\end{enumerate}
The basic properties implied by the KMS condition at $\beta\neq\infty$ are:
\begin{enumerate}
\item[1)] if $\omega$ satisfies the KMS condition for $\alpha$, then $\omega(\alpha_t(x))=\omega(x)\;\forall x\in\C\;\forall t\in\RR$,
\item[2)] if $\omega$ satisfies the KMS condition for $\alpha$, then $\alpha_t(x)=x\;\forall x\in\supp_{\zentr_\C}(\omega)\zentr_\C$,
\item[3)] if $\omega$ satisfies the KMS condition for $\alpha$ and $\beta\neq0$, then it satisfies the KMS condition for $\alpha_{\lambda t}$ and $\frac{\beta}{\lambda}$, where $\lambda\in\RR\setminus\{0\}$,
\item[4)] $\omega$ satisfies the KMS condition and is faithful $\iff$ $\pi_{\omega}$ is faithful,
\item[5)] the $*$-subalgebra $\C^\alpha_\infty$ is norm dense in $\C$ if{}f $\alpha$ is strongly continuous, and it is weakly-$\star$ dense in $\C$ if{}f $\alpha$ is weakly-$\star$ continuous,
\item[6)] if $\C$ has a tracial $\phi\in\Scal(\C)$ and $\alpha$ is given by $\alpha_t=\exp(t[\ii h,x])$ for some $h\in\C$, then 
\begin{equation}
        \omega=\frac{\phi(\ee^{-\beta h}\,\cdot\,)}{\phi(\ee^{-\beta h})}
\label{KMS.exponential.form1}
\end{equation}
satisfies the KMS condition for $\alpha$ and every $\beta\in\RR$. On the other hand, if $\N=\BH$ for $\dim\H\in\NN$, and there is given a strongly continuous group homomorphism $\alpha:\RR\ni t\mapsto\alpha_t:=\Ad(\ee^{\ii tH})\in\Aut(\N)$ with $H\in\N^\sa$, then $\omega\in\Scal(\N)$ satisfies the KMS condition for $\alpha$ and $\beta$ if{}f 
\begin{equation}
        \omega=\frac{\tr_\H(\rho\,\cdot\,)}{\tr_\H(\rho)}\;\;\mbox{and}\;\;\rho=\ee^{-\beta H}\in\schatten_1(\H)^+.
\label{KMS.exponential.form2}
\end{equation}
\end{enumerate}
See \cite{Kastler:1976,Sakai:1991} for review of the properties implied by the KMS condition, and see  \cite{HKTP:1974,Roepstorff:1976,Haag:TrychPohlmeyer:1977,Fannes:Verbeure:1977,Sewell:1977,Araki:Sewell:1977,Araki:1978,Pusz:Woronowicz:1978:passive,Woronowicz:1985} for equivalent characterisations of the KMS condition.

The set $\Scal_\beta^\alpha(\C)\subseteq\Scal(\C)$ of all normalised positive functionals on a $C^*$-algebra $\C$ satisfying the KMS condition for a fixed automorphism $\alpha$ and a fixed value of $\beta$ is a convex space, compact in weak-$\star$ topology on $\C^\banach$. Moreover, it is a Choquet simplex. From Choquet's representation theorem it follows that every element of $\Scal_\beta^\alpha(\C)$ can be uniquely decomposed, by means of barycentric decomposition, as a convex combination of extremal elements of this space \cite{Ruelle:1970,Takesaki:1970}. Any $\omega\in\Scal_\beta^\alpha(\C)$ is an extremal element of $\Scal_\beta^\alpha(\C)$ if{}f $\zentr_{\pi_\omega(\C)^\comm{}^\comm}=\CC\II$ \cite{Araki:Miyata:1968,Araki:1969,Lanford:1970,Takesaki:1970,Winnink:1970}, that is, if{}f $\omega\in\fact(\C)$. If $\omega$ is not extremal and $\C$ is separable in norm topology (or $\H_\omega$ is separable), then the spectrum of $\zentr_{\pi_{\omega}(\C)^\comm{}^\comm}$ is labelled by extremal elements of $\Scal_\beta^\alpha(\C)$, and the corresponding unique maximal central measure $\tmu_{\zentr_{\pi_{\omega}(\C)^\comm{}^\comm}}$ on $\Scal_\beta^\alpha(\C)$ with barycenter $\omega$ is supported on $\fact(\C)$, providing a unique integral decomposition of $\omega$ in terms of \eqref{spec.decomp.zentr}. If $\Scal_\beta^\alpha(\C)$ contains only one element, $\omega$, then $\zentr_{\pi_\omega(\C)^\comm{}^\comm}=\CC\II$. If a $W^*$-algebra $\N$ is finite (see Section \ref{classification.section}), then $\Scal_\beta^\alpha(\N)=\{\omega\}$. 
\subsection{Tomita--Takesaki modular theory}
For the purpose of the following discussion, let us recall some facts about polar decompositions of possibly unbounded linear operators on a Hilbert space $\H$. A linear operator $x:\dom(x)\ra\H$ is called: \df{densely defined} if{}f $\dom(x)$ is a dense subset of $\H$; \df{closable} if{}f it is densely defined and the closure of a graph $\bigcup_{\xi\in\dom(x)}(\xi,x\xi)$ is a closed subset of $\H\oplus\H$. Given a closable $x$, an operator defined by the closure of its graph is called a \df{closure} of $x$, and is denoted by $\bar{x}$. A densely defined operator is called \df{closed} if{}f it is closable and $x=\bar{x}$. For every closable densely defined unbounded operator $x$ its \df{adjoint} $x^*$ is defined as a unique densely defined operator $x^*:\dom(x^*)\ra\H$ such that 
\begin{equation}
        \s{x^*\xi,\zeta}=
        \s{\xi,x\zeta}\;\;
        \forall\xi\in\dom(x^*)\;
        \forall\zeta\in\dom(x), 
\end{equation}
where
\begin{align}
\dom(x^*):&
        =\{
        \xi\in\H\mid
                \dom(x)\ni\zeta\mapsto
                \s{\xi,x\zeta}\in\CC\mbox{ is bounded}
        \}
        \\
        &=
        \left\{
                \xi\in\H\mid
                \sup_{\zeta\in\dom(x)}
                \left\{
                        \frac{\ab{\s{\xi,x\zeta}}}{\n{\zeta}}
                \right\}
                <\infty
        \right\}.
\end{align}
It satisfies $x^*=(\overline{x})^*$. According to the Hellinger--Toeplitz theorem \cite{Hellinger:Toeplitz:1910}, if $\dom(x)=\H$ and $x=x^*$, then $x\in\BH$. Hence, an unbounded operator $x$ can satisfy $x=x^*$ only on $\dom(x)\subsetneq\H$. A \df{support} $\rpktarget{SUPP.DREI}\supp(x)$ of a linear operator $x:\dom(x)\ra\H$ is defined as a projection in $\Proj(\BH)$ corresponding to a closed subspace $\overline{\dom(x)}\ominus\overline{\ker(x)}$.
 
A closed, densely defined linear operator $x:\dom(x)\ra\H$ is called: \df{self-adjoint} if{}f $x=x^*$ with $\dom(x)=\dom(x)^*$; \df{anti-self-adjoint} if{}f $x=-x^*$ with $\dom(x)=\dom(x)^*$; \df{positive} if{}f $\s{x\xi,\xi}\geq0\;\forall\xi\in\dom(x)$. If $x$ is positive, then one writes $x\geq0$. If $x\geq0$, then the operator 
\begin{equation}
        y:=\int_{\sp(x)}\pvm^{x}(\lambda)\sqrt{\lambda}
\end{equation}
is a unique positive operator satisfying $y^2=x$, and is denoted by $x^{1/2}$. For every closed densely defined linear operator $x:\dom(x)\ra\H$, the operator $x^*x$ is self-adjoint and positive, hence 
\begin{equation}
        \ab{x}:=\sqrt{x^*x}=\int_{\sp(x^*x)}\pvm^{x^*x}(\lambda)\sqrt{\lambda}
\end{equation}
is a unique positive operator satisfying $\ab{x}^2=x$. Moreover, every closed densely defined operator $x$ can be uniquely decomposed in a form $x=v\ab{x}$, called a \df{polar decomposition} of $x$, where $v$ is a partial isometry in $\BH$ such that $vv^*=\supp(x)$ and $v^*v=\overline{\ran(\ab{x})}$. The spectral representation of $\ab{x}$ reads $\ab{x}=\int_{\sp(\ab{x})}\pvm^{\ab{x}}(\lambda)\lambda$. The set of all self-adjoint (respectively, anti-self-adjoint) closed, densely defined linear operators on a given Hilbert space $\H$ will be denoted $\rpktarget{LIN}(\Lin(\H))^\sa$ (respectively, $(\Lin(\H))^\asa$).

Let $\N$ be a von Neumann algebra acting on the Hilbert space $\H$, and let $\Omega\in\H$ be cyclic and separating for $\N$. Consider a densely defined antilinear operator $R_\Omega$, acting on a dense subspace of $\H$ generated by action of the von Neumann algebra $\N$ on $\Omega$, and given by
\begin{equation}
        R_\Omega:x\Omega\mapsto x^*\Omega.
\end{equation}
This operator is always closable. Its closure $\overline{R}_\Omega$ has a unique polar decomposition $\overline{R}_\Omega=J_\Omega\Delta_\Omega^{1/2}$, where the strictly positive self-adjoint operator $\rpktarget{DELTA}\Delta_\Omega$ is called \df{modular operator}, and the conjugation operator $\rpktarget{J}J_\Omega$ is called \df{modular conjugation}. The operator $\Delta_\Omega$ is characterised by the equation $\Delta_\Omega= R^*_\Omega\overline{R}_\Omega$, while $J_\Omega$ satisfies $J^*_\Omega=J^{-1}_\Omega$, so it is antiunitary. These two operators satisfy also
\begin{equation}
        \begin{array}{ccccc}
\overline{R}_\Omega=\overline{R}^{-1}_\Omega=\Delta^{-1/2}_\Omega J^*_\Omega,&
R_\Omega^*=J\Delta_\Omega^{-1/2},&
\Delta_\Omega\Omega=\Omega=J_\Omega\Omega,&
J_\Omega\Delta_\Omega J_\Omega=\Delta_\Omega^{-1},&
J_\Omega\Delta^{\ii t}_\Omega=\Delta^{-\ii t}_\Omega J_\Omega.
\end{array}
\label{Tomita.modular.relations}
\end{equation}

The Tomita theorem \cite{Tomita:1967:a,Tomita:1967:b} states that every pair $(\N,\Omega)$, where $\N$ is a von Neumann algebra acting on a Hilbert space $\H$, and $\Omega\in\H$ is cyclic and separating for $\N$, determines a weakly-$\star$ continuous group homomorphism $\rpktarget{SIGMA}\sigma^\omega:\RR\ra\Aut(\N)$ and an antilinear $*$-isomorphism $j_\omega:\N\ra\N^\comm$ such that
\begin{enumerate}
\item[1)] $\sigma_t^\omega:x\mapsto\Delta^{\ii t}_\Omega x\Delta^{-\ii t}_\Omega$, and $\{\Delta_\Omega^{\ii t}\mid t\in\RR\}$ is a strongly continuous group of unitaries in $\BH$,
\item[2)] $j_\omega:x\mapsto J_\Omega xJ_\Omega$, and $J_\Omega\N J_\Omega=\N^\comm$,
\item[3)] the map $\RR\ni t\mapsto\Delta_\Omega^{\ii t}x\Omega\in\H$ has an analytic continuation to the strip $\{z\in\CC\mid-\frac{1}{2}<\im(z)<0\}$, and $\Delta_\Omega^{1/2}x\Omega=J_\Omega x^*\Omega$,
\item[4)] $x\in\zentr_\N$ $\limp$ $j_\omega(x)=x^*$, $\sigma^\omega_t(x)=x$.
\end{enumerate}
The Winnink--Takesaki theorem \cite{Winnink:1970,Takesaki:1970} (see also \cite{Takesaki:1971,Takesaki:1973:states}) states in addition that
\begin{enumerate}
\item[5)] $\omega(x):=\s{\Omega,x\Omega}\;\forall x\in\N$ is a unique element of $\N^{\banach+}$ that satisfies the KMS condition for $\sigma^\omega$ and $\beta=1$, and $\sigma^\omega$ is a unique strongly continuous one parameter group for which $\omega$ satisfies the KMS condition with $\beta=1$.\footnote{Sometimes it is stated that $\omega$ satisfies the KMS condition for $\sigma^\omega$ with $\beta=-1$. This is only a matter of convention (see e.g. \cite{Pedersen:1979}).}
\end{enumerate}
The bijective correspondence between cyclic and separating vectors $\Omega$ and elements $\omega\in\N^{\banach+}$ allows to use notation $\rpktarget{DELTA.ZWEI}\Delta_\omega:=\Delta_\Omega$ and $\rpktarget{J.ZWEI}J_\omega:= J_\Omega$.\footnote{Rieffel and van Daele \cite{Rieffel:vanDaele:1977} showed that the modular operators $\Delta$ and $J$ (as well as the KMS condition with respect to $\Ad(\Delta^{\ii t})$ and $\beta=1$) can be also characterised in terms of the properties of two real Hilbert subspaces of a given Hilbert space (see also \cite{Skau:1980}).} The group $\sigma^\omega:\RR\ni t\mapsto\Ad(\Delta_\omega^{\ii t})\in\Aut(\N)$ is called the \df{modular automorphisms} group. The KMS condition implies the invariance properties
\begin{align}
        \omega(x)&=\omega(\sigma_t^\omega(x))\;\;\forall x\in\N,\\
        \rpktarget{NSG}\N_{\sigma^\omega}=
        \{x\in\N\mid\sigma^\omega_t(x)=x\}&=
        \{x\in\N\mid\omega(xy)=
        \omega(yx)
        \;\forall y\in\N\}.
\end{align}
The unitary operators $\Delta^{\ii t}_\omega$ leave $\Omega$ invariant: $\Delta^{\ii t}_\omega\Omega=\Omega$, so the \df{modular hamiltonian} $\rpktarget{KOMEGA}K_\omega$, defined according to
\begin{equation}
        \ee^{-K_\omega}:=\Delta_\omega,
\end{equation}
satisfies $K_\omega^*=K_\omega$ and $K_\omega\Omega=0$. If $\omega=\s{\Omega,\,\cdot\,\Omega}$ is tracial, then $j_\omega(x)=x^*\;\forall x\in\N$. 

Every $\omega\in\N^+_{\star0}$ on a $W^*$-algebra $\N$ generates a GNS representation $(\H_\omega,\pi_\omega,\Omega_\omega)$ with cyclic and separating $\Omega_\omega$, hence every such $\omega$ satisfies the KMS condition with respect to uniquely determined weakly-$\star$ continuous group $\sigma^\omega:\RR\ni t\mapsto\sigma^\omega_t\in\Aut(\N)$ and $\beta=1$ with
\begin{equation}
        \sigma^\omega:\RR\ni t
        \mapsto
        \pi_\omega^{-1}
        \left(
        \Delta_\omega^{\ii t}\pi_\omega(x)\Delta^{-\ii t}_\omega
        \right)
        \in\N\;\;\forall x\in\N.
\end{equation}
Thus, \cite{Takesaki:1970}
\begin{equation}
        \phi\in\N^+_{\star0}
        \limp
        (x\in\N_{\sigma^\phi}
        \iff
        \phi([x,y])=0\;\forall y\in\N)
\end{equation}
A functional $\omega\in\C^{\banach+}$ on a $C^*$-algebra $\C$ is called \df{modular} if{}f its cyclic GNS representative $\Omega_\omega\in\H_\omega$ is a separating vector for the von Neumann algebra $\pi_\omega(\C)^\comm{}^\comm$, that is, if{}f $\omega$ has a faithful normal extension on $\pi_\omega(\C)^\comm{}^\comm$. The GNS representative $\Omega_\omega$ of every $\omega\in\C^{\banach+}_0$ on a $C^*$-algebra $\C$ is cyclic and separating for $\pi_\omega(\C)^\comm{}^\comm$ on $\H_\omega$, so such $\omega$'s are modular. Hence, every pair of a $C^*$-algebra $\C$ and $\omega\in\C^{\banach+}_0$ generates a unique Hilbert space $\H_\omega$ equipped with a modular automorphism $\sigma_t^\omega$ of the von Neumann algebra $\pi_\omega(\C)^\comm{}^\comm$, such that the normal extension of $\omega$ satisfies the KMS condition for $\sigma^\omega$ on $\pi_\omega(\C)^\comm{}^\comm$ and $\beta=1$.

The Tomita--Takesaki theory applies also to weights \cite{Tomita:1967:b,Takesaki:1970:notes,Combes:1971,Combes:1971:esperances} (see also \cite{vanDaele:1974,Rieffel:vanDaele:1977}). If $\N$ is a $W^*$-algebra and $\omega\in\W_0(\N)$, then there exists a closeable operator 
\begin{equation}
        R_\omega:
                [\nnn_\omega\cap\nnn_\omega^*]_\omega\ni[x]_\omega
        \mapsto
                [x^*]_\omega\in\H_\omega.
\end{equation}
Its closure $\overline{R}_\omega$ has a polar decomposition $\overline{R}_\omega=J_\omega\Delta_\omega^{1/2}=\Delta_\omega^{-1/2}J_\omega$ with self-adjoint, positive, invertible operator $\rpktarget{DELTA.DREI}\Delta_\omega$ and conjugation operator $\rpktarget{J.DREI}J_\omega$. They satisfy $\Delta_\omega=R_\omega^*\bar{R}_\omega$, $J_\omega^*=J_\omega^{-1}=J_\omega$ and all of equations in \eqref{Tomita.modular.relations}. The map $j_\omega:\pi_\omega(\N)\ni x\mapsto J_\omega x J_\omega\in\pi_\omega(\N)^\comm$ is an antilinear $*$-isomorphism. Moreover, 
$\omega$ satisfies the KMS condition for $\sigma^\omega$ and $\beta=1$, where the group $\sigma^\omega(\RR)\subseteq\Aut(\N)$ is given by 
\begin{equation}\rpktarget{SIGMA.DREI}
        \sigma^\omega:\RR\ni t\mapsto
        \sigma^\omega_t(x)=
        \pi_\omega^{-1}(
                \Delta^{\ii t}_\omega\pi_\omega(x)\Delta^{-\ii t}_\omega
        )\in\N\;\;\;\forall x\in\N.
\end{equation}
In addition,
\begin{align}
\Delta^{\ii t}_\omega[x]_\omega&=[\sigma^\omega_t(x)]_\omega\;\;\forall t\in\RR\;\forall x\in\nnn_\omega,\\
\sigma^\omega_t(\nnn_\omega\cap\nnn_\omega^*)&=\nnn_\omega\cap\nnn_\omega^*,\\
j_\omega(\pi_\omega(\N))&=\pi_\omega(\N)^\comm,\\
\omega(\sigma^\omega_t(x))&=\omega(x)\;\;\forall x\in\N^+\;\forall t\in\RR.
\label{omega.sigma.invariance}
\end{align}
The weight $\phi$ on a $W^*$-algebra $\N$ is a trace if{}f $\Delta_\phi=\II$, and every semi-finite faithful normal trace $\tau$ on $\N$ satisfies $\sigma^\tau_t=\id_\N\;\forall t\in\RR$ \cite{Hugenholtz:1967}. 

From faithfulness of the normal state on the reduced algebra it follows that every $\phi\in\N_\star^+$ satisfies the KMS condition with respect to $\sigma^\phi$ and $\beta=1$ on the reduced algebra $\N_{\supp(\phi)}$. More generally, every normal weight $\phi$ on $\N$ (not necessary semi-finite or faithful) determines a unique strongly continuous group $\{\sigma^\phi_t\mid t\in\RR\}\subseteq\Aut(\N)$ such that \eqref{omega.sigma.invariance} holds, and $\phi$ satisfies the KMS condition with respect to $\sigma^\phi:\RR\ni t\mapsto\sigma^\omega_t\in\Aut(\N)$ and $\beta=1$. However, this group can be identified with the group of modular automorphisms only on the reduced algebra $\N_{\supp(\phi)}$.
\subsection{Relative modular theory\label{relative.modular.theory.section}}
Refining earlier works by Takesaki \cite{Takesaki:1970} and Perdrizet \cite{Perdrizet:1970,Perdrizet:1971} on the dual pair of cones in Hilbert space that arises from the modular theory, Woronowicz \cite{Woronowicz:1972}, Connes \cite{Connes:1972,Connes:1974} and Araki \cite{Araki:1974:modular:conjugation} have introduced the new representation of countably finite $W^*$-algebras, that preserves certain relationship between modular conjugation operator and a positive cone in the representing Hilbert space. This representation has several remarkable properties that are \textit{natural} for dealing with the structures of modular theory over countably finite $W^*$-algebras. Independently, Haagerup \cite{Haagerup:1973,Haagerup:1975:standard:form} found a general case of this construction, called the \textit{standard representation}, which is valid for all $W^*$-algebras. It can be considered as a refinement (or replacement) of the GNS representation, particularly adapted to the setting of $W^*$-algebras.

A subspace $\D\subseteq\H$ of a complex Hilbert space $\H$ is called a \df{cone} if{}f $\lambda\xi\in\D$ $\forall\xi\in\D$ $\forall\lambda\geq0$. A cone $\D\subseteq\H$ is called \df{self-polar} if{}f
\begin{equation}
        \D=\{\zeta\in\H\mid\s{\xi,\zeta}_\H\geq0\;\forall\xi\in\D\}.
\end{equation}
Every self-polar cone $\D\subseteq\H$ is pointed ($\D\cap(-\D)=\{0\}$), spans linearly $\H$ ($\Span_\CC\D=\H$), and determines a unique conjugation $J$ in $\H$ such that $J\xi=\xi\;\forall\xi\in\H$ \cite{Haagerup:1973}, 
as well as a partial order on the set $\H^\sa:=\{\xi\in\H\mid J\xi=\xi\}$ given by,
\begin{equation}
        \xi\leq\zeta\;\iff\;\xi-\zeta\in\D\;\;\forall\xi,\zeta\in\H^\sa.
\end{equation}
See \cite{Haagerup:1973,Boes:1978,Iochum:1984} for additional discussion of self-polar cones in Hilbert spaces. A closed convex self-polar cone in $\H$ is called a \df{natural cone} \cite{Woronowicz:1972,Connes:1972,Connes:1974,Araki:1974:modular:conjugation}, and is denoted $\rpktarget{STDCONE.OMEGA}\stdcone_\Omega$ if{}f, for a given vector $\Omega\in\H$ cyclic and separating with respect to a given von Neumann algebra $\N$ on $\H$,
\begin{equation}
        \stdcone_\Omega=
        \overline{\bigcup_{x\in\N^+}\left\{xJ_\Omega x J_\Omega\Omega\right\}}=
        \overline{\Delta^{1/4}_\Omega\N^+\Omega}=
        \overline{\Delta^{-1/4}_\Omega(\N^\comm)^+\Omega},
\label{TTcone}
\end{equation}
where $(\Delta_\Omega,J_\Omega)$ are modular operator and modular conjugation associated with $(\N,\Omega)$. 
Every $\xi\in\stdcone_\Omega$ is cyclic with respect to $\N$ if{}f it is separating with respect to $\N$. In such case $J_\xi=J_\Omega$ and $\stdcone_\xi=\stdcone_\Omega$. However, if $\N$ is not countably finite, then it admits no cyclic and separating vector in $\H$. In such case a more general notion is needed.

If $\N$ is a $W^*$-algebra, $\H$ is a Hilbert space, $\stdcone\subseteq\H$ is a self-polar cone, $\pi$ is a nondegenerate faithful normal representation of $\N$ on $\H$, and $J$ is conjugation on $\H$, then the quadruple $(\H,\pi,J,\stdcone)$ is called \df{standard representation} of $\N$ and $(\H,\pi(\N),J,\stdcone)$ is called \df{standard form} of $\N$ if{}f \cite{Haagerup:1975:standard:form}\rpktarget{STDCONE}
\begin{enumerate}
        \item[1)] $J\pi(\N)J=\pi(\N)^\comm$,
        \item[2)] $\xi\in\stdcone\limp J\xi=\xi$,
        \item[3)] $\pi(x)J\pi(x)J\stdcone\subseteq\stdcone$,
        \item[4)] $\pi(x)\in\zentr_{\pi(\N)}\limp J\pi(x)J=\pi(x)^*$.
\end{enumerate}
In such case $J$ is called \df{standard conjugation}, while $\stdcone$ is called \df{standard cone}. If the elements of $\N$ are identified with the elements of $\pi(\N)$ acting on $\H$, then $\N$ is called to be in a standard form, or to \df{act standartly} on $\H$. In such case, we will also use the notation $\H=\H(\N)$ instead of $(\H,\pi,J,\stdcone)$. The standard representation satisfies following properties:
\begin{enumerate}
\item[1)] \begin{equation}
        \forall\phi\in\N_\star^+\;\exists !\xi_\pi(\phi)\in\stdcone\;\forall x\in\N\;\;\phi(x)=\s{\xi_\pi(\phi),\pi(x)\xi_\pi(\phi)}_\H.
\label{std.vector.representative}
\end{equation}
A vector $\rpktarget{STDREP}\xi_\pi(\phi)$ is called \df{standard vector representative} of $\phi$;
\item[2)] the map $\rpktarget{STDREP.REVERSE}\xi_\natural^\pi:\stdcone\ni\xi\mapsto\phi_\xi\in\N_\star^+$ defined by the condition $\phi_\xi(x)=\s{\xi,\pi(x)\xi}_\H$ $\forall x\in\N$ is a bijective norm continuous homeomorphism with $(\xi^\pi_\natural)^{-1}=\xi_\pi$. The map $\N_\star^+\ni\phi\mapsto\xi_\pi(\phi)\in\stdcone$ preserves the order $\leq$, and
\begin{equation}
\n{\xi-\zeta}^2\leq\n{\phi_\xi-\omega_\zeta}\leq\n{\xi-\zeta}\n{\xi-\zeta}\;\;\forall\xi,\zeta\in\stdcone;
\end{equation}
\item[3)] $\xi\in\stdcone\limp(J\xi\in\stdcone$ and $xj(x)\xi\in\stdcone\;\forall x\in\pi(\N))$;
\item[4)] for all $\zeta\in\H$ there exists a unique $\xi\in\stdcone$ and a unique partial isometry $v\in\pi(\N)$ such that $\zeta=v\xi$ and $v^*v=P(\xi)$, where $P(\xi)$ denotes a projection onto $\overline{\pi(\N)^\comm\xi}$. The equation $\zeta=v\ab{\zeta}$ with $\ab{\zeta}:=\xi$ is called a \df{polar decomposition} of $\zeta$;
\item[5)] every vector in $\H$ can be represented as a complex linear combination of four elements of $\stdcone$;
\item[6)] $\stdcone$ is closed and convex;
\item[7)] $\phi\in\N_{\star0}^+\limp\bigcup_{x\in\pi(\N)}\{xj(x)\xi_\pi(\phi)\}$ is dense in $\stdcone$, $\xi_\pi(\phi)$ is cyclic and separating for $\pi(\N)$, $J=J_{\xi_\pi(\phi)}$, and $\stdcone$ is a natural cone, $\stdcone=\stdcone_{\xi_\pi(\phi)}$, which satisfies 
\begin{equation}
        \Delta_{\xi_\pi(\phi)}^{\ii t}\stdcone_{\xi_\pi(\phi)}=\stdcone_{\xi_\pi(\phi)}\;\;\forall t\in\RR.
\end{equation}
\end{enumerate}
Every $W^*$-algebra has a faithful representation $\pi$ such that $\pi(\N)$ is in standard form, and this representation is unique up to a unitary equivalence: if $(\H_1,\pi_1,J_1,\stdcone_1)$ and $(\H_2,\pi_2,J_2,\stdcone_2)$ are two standard representations, and $\varsigma:\pi_1(\N)\ra\pi_2(\N)$ is a $*$-isomorphism, then there exists a unique unitary $u_\varsigma:\H_1\ra\H_2$ such that \cite{Haagerup:1973,Haagerup:1975:standard:form}:
\begin{enumerate}
\item[i)] $\varsigma(x)=u_\varsigma xu^*_\varsigma\;\;\forall x\in\pi_1(\N)$,
\item[ii)] $J_2=u_\varsigma J_1u^*_\varsigma$,
\item[iii)] $\stdcone_2=u_\varsigma\stdcone_1$.
\end{enumerate} 
This unitary will be called \df{standard unitary equivalence}. For any $W^*$-algebra $\N$, every $\omega\in\W_0(\N)$ determines a standard representation given by $(\H_\omega,\pi_\omega,J_\omega,\stdcone_\omega)$, where $J_\omega$ is a modular conjugation of $\omega$, and
\begin{equation}
        \stdcone_\omega=\overline{\bigcup_{x\in\nnn_\omega\cap\nnn_\omega^*}\{\pi_\omega(x)J_\omega[x]_\omega\}}=\overline{\Delta^{1/4}_\omega[\mmm_\omega^+]_\omega}.
\end{equation}
If $\N$ is countably finite, then every $\omega\in\N_{\star0}^+$ determines a standard representation of $\N$ given by the faithful normal GNS representation $(\H_\omega,\pi_\omega,\Omega_\omega)$, equipped with a modular conjugation $J_\omega$ and with a natural cone 
$\stdcone_{\Omega_\omega}$. If $\omega$ is tracial, then
\begin{equation}
        \stdcone_{\Omega_\omega}=\overline{\pi_\omega(\N)^+\Omega_\omega}.
\end{equation}
In what follows, we will always assume that the von Neumann algebras under consideration are in standard form  (unless explicitly stated otherwise).

For a given $W^*$-algebra $\N$, $\phi\in\W(\N)$, and $\omega\in\W_0(\N)$ the map
\begin{equation}
        R_{\phi,\omega}:[x]_\omega\mapsto[x^*]_\phi\;\;\forall x\in\nnn_\omega\cap\nnn_\phi^*
        \label{relative.modular.weights}
\end{equation}
is a densely defined, closable antilinear operator. Its closure admits a unique polar decomposition
\begin{equation}
        \overline{R}_{\phi,\omega}=J_{\phi,\omega}\Delta^{1/2}_{\phi,\omega},
\end{equation}
where $\rpktarget{JREL}J_{\phi,\omega}$ is a conjugation operator, called \df{relative modular conjugation}, while $\rpktarget{DELTAREL}\Delta_{\phi,\omega}$ is a positive self-adjoint operator on $\dom(\Delta_{\phi,\omega})\subseteq\H_\omega$ with $\supp(\Delta_{\phi,\omega})=\supp(\phi)\H_\omega$, called a \df{relative modular operator} \cite{Araki:1973:relative:hamiltonian,Connes:1974,Digernes:1975}. These operators satisfy
\begin{equation}
        \Delta_{\phi,\omega}=J_{\phi,\omega}\Delta_{\omega,\phi}^{-1}J_{\phi,\omega},\;\;
        \Delta_{\lambda_1\phi,\lambda_2\omega}=\frac{\lambda_1}{\lambda_2}\Delta_{\phi,\omega}\;\;\forall\lambda_1,\lambda_2\in\RR.
\end{equation}
Moreover, for any $\omega\in\W_0(\N)$
\begin{equation}
        \Delta_{\omega,\omega}=\Delta_\omega,\;\;J_{\omega,\omega}=J_\omega.
\end{equation}
The relative modular operator can be equivalently characterised by 
\begin{equation}
        \Delta_{\phi,\omega}=R^*_{\phi,\omega}\bar{R}_{\phi,\omega}.
        \label{relative.modular.for.natural.cone}
\end{equation}
The self-adjoint operator $K_{\phi,\omega}=-\log\Delta_{\phi,\omega}$ is called  \df{relative modular hamiltonian} \cite{Araki:1973:relative:hamiltonian,Araki:1976:relham:relent}. For every $\phi,\omega\in\W_0(\N)$ the relative modular conjugation $J_{\phi,\omega}$ determines a unique unitary operator $\rpktarget{STDUNITRANS}J_\phi J_{\phi,\omega}=:V_{\phi,\omega}:\H_\omega\ra\H_\phi$, such that
\begin{align}
        \pi_\phi(x)&=V_{\phi,\omega}\pi_\omega(x)V_{\phi,\omega}^*,\\
        V_{\phi,\omega}(\stdcone_\omega)&=\stdcone_\phi,\\
        V_{\phi,\omega}J_\omega&=J_\phi V_{\phi,\omega}.
\end{align}
Thus, $V_{\phi,\omega}$ is a standard unitary equivalence of a $*$-isomorphism $\varsigma_{\phi,\omega}:\pi_\omega(\N)\ra\pi_\omega(\N)$ determined by the condition $\varsigma_{\phi,\omega}\circ\pi_\omega=\pi_\phi$. We will call $V_{\phi,\omega}$ \df{standard unitary transition} between $\H_\omega$ and $\H_\phi$. By definition of $V_{\phi,\omega}$, $J_{\phi,\omega}$ satisfies $J_{\phi,\omega}=V_{\phi,\omega}J_\omega=J_\phi V_{\phi,\omega}$. If $\omega,\phi\in\N^+_{\star0}$, then the operators $\Delta_{\phi,\omega}$ can be defined by the unique polar decomposition $\overline{R}_{\phi,\omega}=J_{\phi,\omega}\Delta^{1/2}_{\phi,\omega}$ of the closure of a densely defined antilinear operator $R_{\phi,\omega}$,
\begin{equation}
    R_{\phi,\omega}:\H_\omega\ni x\Omega_\omega\mapsto x^*\Omega_{\phi}\in\H_\phi,
\label{haag.relative.modular}
\end{equation}
where $\Omega_\omega$ and $\Omega_{\phi}$ are cyclic and separating GNS vector representatives of $\omega$ and $\phi$, respectively. However, this holds only if the $W^*$-algebra $\N$ is countably finite, which is one of the main reasons for introducing standard representation.

The relative modular operators allow to define an ultrastrongly-$\star$ continuous one-parameter family of partial isometries in $\supp(\phi)\N$, called \df{Connes' cocycle} \cite{Connes:1973:classification},\rpktarget{CONNES.COC}
\begin{equation}
\RR\ni t\mapsto\Connes{\phi}{\omega}{t}:=\Delta^{\ii t}_{\phi,\psi}\Delta^{-\ii t}_{\omega,\psi}=\Delta^{\ii t}_{\phi,\omega}\Delta^{-\ii t}_\omega\in\supp(\phi)\N,
\label{Connes.cocycle.def}
\end{equation}
where $\psi\in\W_0(\N)$ is arbitrary, so it can be set equal to $\omega$. If $\phi\in\W_0(\N)$, then $\Connes{\phi}{\omega}{t}$ becomes an ultrastrongly continuous family of unitary elements of $\N$. As shown by Araki and Masuda \cite{Araki:Masuda:1982} (see also \cite{Masuda:1984}), the definition of $\Delta_{\phi,\omega}$ and $\Connes{\phi}{\omega}{t}$ can be further extended to the case when $\phi,\omega\in\W(\N)$, by means of a densely defined closable antilinear operator
\begin{equation}
        R_{\phi,\omega}:[x]_\omega+(\II-\supp(\overline{[\nnn_\phi]_\omega}))\zeta\mapsto\supp(\omega)[x^*]_\phi\;\;\forall x\in\nnn_\omega\cap\nnn_\phi^*\;\forall\zeta\in\H,
\label{relative.modular.for.normal.weights}
\end{equation}
where $(\H,\pi,J,\stdcone)$ is a standard representation of a $W^*$-algebra $\N$, and $\H_\phi\subseteq\H\supseteq\H_\omega$. For $\phi,\omega\in\N_\star^+$ this becomes a closable antilinear operator \cite{Araki:1977:relative:entropy:II,Kosaki:1980:PhD}
\begin{equation}
        R_{\phi,\omega}:x\xi_\pi(\omega)+\zeta\mapsto\supp(\omega)x^*\xi_\pi(\phi)\;\;\forall x\in\pi(\N)\;\forall\zeta\in(\pi(\N)\xi_\pi(\omega))^\bot,
        \label{relative.modular.for.normal.states}
\end{equation}
acting on a dense domain $(\pi(\N)\xi_\pi(\omega))\cup(\pi(\N)\xi_\pi(\omega))^\bot\subseteq\H$, where $(\pi(\N)\xi_\pi(\omega))^\bot$ denotes a complement of the closure in $\H$ of the linear span of the action $\pi(\N)$ on $\xi_\pi(\omega)$. In both cases, the relative modular operator is determined by the polar decomposition of the closure $\overline{R}_{\phi,\omega}$ of $R_{\phi,\omega}$,\rpktarget{DELTA.REL.ZWEI}
\begin{equation}
        \Delta_{\phi,\omega}:=R^*_{\phi,\omega}\overline{R}_{\phi,\omega}.
\label{RR.Delta.relative}
\end{equation}
If \eqref{relative.modular.for.normal.weights} or \eqref{relative.modular.for.normal.states} is used instead of \eqref{relative.modular.weights}, then the formula \eqref{Connes.cocycle.def} has to be replaced by\rpktarget{CONNES.COC.ZWEI}
\begin{equation}
        \RR\ni t\mapsto\Connes{\phi}{\omega}{t}\supp(\overline{[\nnn_\phi]_\psi}):=\Delta^{\ii t}_{\phi,\psi}\Delta^{-\ii t}_{\omega,\psi},
\label{Connes.for.ns.weights}
\end{equation}
and $\Connes{\phi}{\omega}{t}$ is a partial isometry in $\supp(\phi)\N\supp(\omega)$ whenever $[\supp(\phi),\supp(\omega)]=0$.

Consider a von Neumann algebra $\N$, acting on a Hilbert space $\H$, a weight $\phi^\comm\in\W_0(\N^\comm)$, a \df{lineal} set \cite{Sherstnev:1974,Sherstnev:1977}
\begin{equation}
        \D(\H,\phi):=
        \left\{
                \zeta\in\H\mid
                \exists\lambda\geq0\;
                \forall x\in\nnn_{\phi^\comm}\subseteq\N^\comm\;\;
                \n{x\zeta}^2\leq\lambda\phi^\comm(x^*x)
        \right\},
\label{lineal.set}
\end{equation}
and a bounded operator
\begin{equation}
        R_{\phi^\comm}(\xi):
        \H_{\phi^\comm}\ni[x]_{\phi^\comm}
        \mapsto
         x\xi\in\H
         \;\;\forall x\in\nnn_{\phi^\comm}
         \;\;\forall\xi\in\D(\H,\phi),
\end{equation}
where $(\H_{\phi^\comm},\pi_{\phi^\comm})$ is a GNS representation of $\N^\comm$ with respect to $\phi^\comm$. The map $\xi\mapsto R_{\phi^\comm}(\xi)$ is linear, and $R_{\phi^\comm}(\xi)R_{\phi^\comm}(\xi)^*\in\N^\comm$. For any $\psi\in\W(\N)$ the function
\begin{equation}
        \xi\mapsto\psi(R_{\phi^\comm}(\xi)R_{\phi^\comm}(\xi)^*)
\end{equation}
is closable and bounded on
\begin{equation}
        \D(\H,\phi,\psi):=
        \{
                \xi\in\D(\H,\phi)
                \mid 
                R_{\phi^\comm}(\xi)\in\nnn_\psi
        \}.
\end{equation}
The \df{Connes spatial quotient}\footnote{We use the term `quotient' instead of `derivative', in order to reserve the term `derivative' for the notions defined by means of differential or smooth structures. By the same reason we use the term `Radon--Nikod\'{y}m quotient' instead of `Radon--Nikod\'{y}m derivative'. The Radon--Nikod\'{y}m quotient and Connes' spatial quotient are defined in the absence of any differential structure, and they \textit{are not} derivatives (apart from very special cases), so they shall not be called `derivatives'. This is especially important in the context of the quantum information geometry, where one considers explicitly the derivatives on spaces of integrals and spaces of quantum states.} \cite{Connes:1980} is defined as the largest positive self-adjoint operator $\rpktarget{CONNES.SPAT}\connes{\psi}{\phi^\comm}:\D(\H,\phi,\psi)\ra\H$ satisfying
\begin{equation}
        \n{\left(\connes{\psi}{\phi^\comm}\right)^{1/2}\xi}^2=\psi(R_{\phi^\comm}(\xi)R_{\phi^\comm}(\xi)^*)\;\;\forall\xi\in\D(\H,\phi,\psi).
\label{connes.spatial.def}
\end{equation}
The operator $\left(\connes{\psi}{\phi^\comm}\right)^{1/2}$ is  unbounded, but it is essentially self-adjoint on $\D(\H,\phi,\psi)$. It satisfies \cite{Connes:1980,Terp:1981,Falcone:2000}
\begin{align}
\sup_\iota\{\psi_\iota\}=\psi&\limp\sup_\iota\left\{\connes{\psi_\iota}{\phi^\comm}\right\}=\connes{\psi}{\phi^\comm}\;\;\forall\psi,\psi_\iota\in\W(\N),\\
\left(\connes{\psi_2}{\phi^\comm}\right)^{\ii t}&=\Connes{\psi_2}{\psi_1}{t}\left(\connes{\psi_1}{\phi^\comm}\right)^{\ii t}
\;\;\forall\psi_1,\psi_2\in\W_0(\N)\;\forall t\in\RR,
\label{csd.prop.one}\\
\left(\connes{\psi}{\phi^\comm}\right)^{-1}&=\connes{\phi^\comm}{\psi}
\;\;\forall\psi\in\W_0(\N),
\label{csd.prop.two}\\
\connes{\psi(x\,\cdot\,x^*)}{\phi^\comm}&=x\connes{\psi}{\phi^\comm}x^*\;\;\forall\psi\in\W(\N)\;\forall x\in\N,\\
\sigma^\psi_t(x)&=\left(\connes{\psi}{\phi^\comm}\right)^{\ii t}x\left(\connes{\psi}{\phi^\comm}\right)^{-\ii t}
\;\;\forall x\in\N_{\supp(\psi)}\;\forall\psi\in\W(\N),
\label{csd.prop.three}\\
\connes{(\psi_1+\psi_2)}{\phi^\comm}&=\connes{\psi_1}{\phi^\comm}+\connes{\psi_2}{\phi^\comm}\;\;\forall\mbox{ finite }\psi_1,\psi_2\in\W(\N),\\
\psi_1\leq\psi_2\;&\iff\;\connes{\psi_1}{\phi^\comm}\leq\connes{\psi_2}{\phi^\comm}\;\;\forall\psi_1,\psi_2\in\W(\N).
\end{align}
The extension of $\connes{\psi}{\phi^\comm}$ to all normal weights $\psi$ on $\N$ was constructed by Terp \cite{Terp:1981}, and it also satisfies the above properties. Connes' spatial quotient is a generalisation of the relative modular operator $\Delta_{\phi,\psi}$. For a given choice of a standard representation, and for $\psi\in\W_0(\N)$ and $\phi\in\W_0(\N)$ or $\phi\in\N_\star^+$ one has
\begin{equation}
        \connes{\phi}{\psi^\comm}\xi=\Delta_{\phi,\psi}\xi,
\end{equation}
where $\xi\in\supp(\phi)\H$ and $\psi^\comm:=\psi(J\cdot J)\in\W_0(\N^\comm)$, so the properties \eqref{csd.prop.one}-\eqref{csd.prop.three} turn into the corresponding properties of $\Delta_{\phi,\psi}$. By \eqref{csd.prop.one} and \eqref{csd.prop.two}, if $\varphi,\phi,\psi\in\W_0(\N)$, then
\begin{equation}
        \left(\connes{\phi}{\varphi(J\cdot J)}\right)^{\ii t}
        \left(\connes{\psi}{\varphi(J\cdot J)}\right)^{-\ii t}=
        \Connes{\phi}{\psi}{t}\;\;\forall t\in\RR.
\end{equation}
If $\N$ is semi-finite, $\phi\in\W_0(\N)$, and $\tau$ is a faithful normal semi-finite trace on $\N$, then
\begin{equation}
        \left(\connes{\phi}{\tau(J\,\cdot\,J)}\right)^{\ii t}=\Connes{\phi}{\tau}{t}\;\;\forall t\in\RR.
\end{equation}
If $\Omega\in\H$ is cyclic and separating for $\N$, $\omega(x)=\s{\Omega,x\Omega}\;\forall x\in\N$ and $\omega^\comm(x^\comm)=\s{\Omega,x^\comm\Omega}\;\forall x^\comm\in\N^\comm$, then 
\begin{equation}
        \connes{\omega}{\omega^\comm}=\Delta_\omega.
\end{equation}
\subsection{Canonical representation and bimodules\label{cano.rep.bimod}}
We will consider now the canonical construction of relative modular structures over arbitrary $W^*$-algebra. The starting point is a characterisation of Connes' cocycle in terms of analytic properties of a pair of weights. Next we will follow Kosaki's \cite{Kosaki:1980:PhD} construction of a \textit{canonical} standard representation that is uniquely associated with a given $W^*$-algebra. Finally, we will consider Connes' idea \cite{Connes:1980} of right- and bi- $W^*$-modules (correspondences) \cite{Connes:1980,Popa:1986,Connes:1994,Yamagami:1994,Falcone:1996,Falcone:2000} in the context of Kosaki's canonical representation, and its special cases, as well as Falcone's \cite{Falcone:1996,Falcone:2000} definition of Connes' spatial quotient based on right $W^*$-modules.

For any $W^*$-algebra $\N$, $\phi\in\W_0(\N)$ and $\psi\in\W(\N)$, a \df{Connes cocycle} $\rpktarget{CONNES.COC.DREI}\RR\ni t\mapsto\Connes{\psi}{\phi}{t}\in\supp(\psi)\N$ can be characterised \cite{Connes:1973:classification} independently of any representation, as a unique ultrastrongly-$\star$ continuous family $\{u_t\mid t\in\RR\}$ of partial isometries in $\supp(\psi)\N$ such that for all $s,t\in\RR$
\begin{enumerate}
\item[1)] \begin{equation}
        u_{s+t}=u_s\sigma^\phi_s(u_t),
        \label{Connes.first.cond}
\end{equation}
\item[2)] $u_su_s^*=\supp(\psi)$,
\item[3)] $u_s^*u_s=\sigma^\phi_s(\supp(\phi))$,
\item[4)] $u_s\sigma^\phi_s(\nnn_\psi^*\cap\nnn_\phi)\subseteq\nnn_\psi^*\cap\nnn_\phi$,
\item[5)] \begin{equation}
        \sigma^\psi_t(x)=u_t\sigma^\phi_t(x)u_t^*\;\;\forall x\in\N_{\supp(\psi)},
        \label{Connes.second.cond}
\end{equation}
\item[6)] for all $x\in\nnn_\psi\cap\nnn_\phi^*$ and for all $y\in\nnn_\phi\cap\nnn_\psi^*$ there exists a function $F_{x,y}$ that is bounded and continuous on the strip $\{z\in\CC\mid\im(z)\in[0,1]\}$ and holomorphic in its interior such that 
\begin{align}
        F_{x,y}(t)&=\psi(u_t\sigma^\phi_t(y)x),\\
        F_{x,y}(t+\ii)&=\phi(xu_t\sigma^\phi_t(y)).
\end{align}
\end{enumerate}
Hence, as opposed to relative modular operators, Connes' cocycle is a canonical object associated to any pair $(\psi,\phi)\in\W(\N)\times\W_0(\N)$ of weights on any $W^*$-algebra $\N$. If $\psi\in\W_0(\N)$, then $t\mapsto\Connes{\psi}{\phi}{t}$ becomes an ultrastrongly continuous group of unitaries in $\N$, the conditions 2) and 3) are satisfied trivially, the relation $\subseteq$ in 4) becomes $=$, and $\N_{\supp(\psi)}$ in 5) turns to $\N$ \cite{Connes:1973:classification}. If $\psi\in\N_\star^+$ and $\phi\in\N^+_{\star0}$, then the condition 4) is satisfied trivially, while the domains of $x$ in $y$ in the condition 6) become $x\in\N\supp(\psi)$ and $y\in\supp(\psi)\N$ \cite{Kosaki:1980:PhD}. 

Conversely, if $\phi\in\W_0(\N)$ and $\{u_t\mid t\in\RR\}$ is an ultrastrongly-$\star$ continuous one parameter family of partial isometries in $\N$ such that for all $s,t\in\RR$
\begin{enumerate}
\item[1')] $u_{s+t}=u_s\sigma^\phi_s(u_t)$,
\item[2')] $u_su_s^*\in\Proj(\N)$,
\item[3')] $u_s^*u_s=\sigma^\phi(u_su_s^*)$,
\end{enumerate}
then there exists a unique $\psi\in\W(\N)$ such that $\Connes{\psi}{\phi}{t}=u_t$ and $\supp(\psi)=u_tu_t^*$ $\forall t\in\RR$. If $\{u_t\mid t\in\RR\}$ is assumed to be an ultrastrongly continuous family of unitaries in $\N$ satisfying condition 1'), then there exists a unique $\psi\in\W_0(\N)$ such that $\Connes{\psi}{\phi}{t}=u_t\;\forall t\in\RR$. The main properties of Connes' cocycle for $\omega_1,\omega_2,\omega_3\in\W_0(\N)$ are:
\begin{align}
\Connes{\omega_1}{\omega_2}{0}&=\II,\\
\Connes{\omega_1}{\omega_2}{t}=\II\;\forall t\in\RR
&\limp\;\omega_1=\omega_2,\\
\Connes{\omega_1}{\omega_3}{t}=\Connes{\omega_2}{\omega_3}{t}
&\iff\;\omega_1=\omega_2,\\
\Connes{\omega_1}{\omega_2}{t}\Connes{\omega_2}{\omega_3}{t}&=
\Connes{\omega_1}{\omega_3}{t},\label{Connes.cocycle}\\
\Connes{\omega_1}{\omega_2}{t}^*&=\Connes{\omega_2}{\omega_1}{t},
\label{Connes.cocprop.star}\\
\Connes{\omega_1}{\omega_2}{t}(\sigma_t^{\omega_1}(x))&=(\sigma_t^{\omega_1}(x))\Connes{\omega_1}{\omega_2}{t},\\
\omega_1(\cdot)=\omega_2(u\cdot u^*)&\iff\Connes{\omega_1}{\omega_2}{t}=u^*\sigma^{\omega_2}_t(u)\;\;\forall u\in\N^\uni.
\end{align}
If $\omega_1,\omega_2\in\W(\N)$ and $\Connes{\omega_1}{\omega_2}{t}$ defined by \eqref{Connes.for.ns.weights} satisfies
\begin{align}
        \Connes{\omega_1}{\omega_2}{t}\Connes{\omega_1}{\omega_2}{t}^*
        &=\sigma^{\omega_1}(\supp(\omega_1)\supp(\omega_2)),\\
        \Connes{\omega_1}{\omega_2}{t}^*\Connes{\omega_1}{\omega_2}{t}
        &=\sigma^{\omega_2}_t(\supp(\omega_1)\supp(\omega_2)),\\
        [\supp(\omega_1),\supp(\omega_2)]
        &=0,
\end{align}
then \eqref{Connes.cocprop.star} holds, \eqref{Connes.first.cond} holds for $\supp(\omega_2)\geq\supp(\omega_1)$, \eqref{Connes.second.cond} holds for $x\in\supp(\omega_1)\N\supp(\omega_2)$, while \eqref{Connes.cocycle} holds if either $\supp(\omega_2)\geq\supp(\omega_1)$ or $\supp(\omega_2)\geq\supp(\omega_3)$.

Following Kosaki \cite{Kosaki:1980:PhD}, consider new addition and multiplication structure on $\N_\star^+$,
\begin{align}
        \lambda\sqrt{\phi}&=
        \sqrt{\lambda^2\phi}
        \;\;\forall\lambda\in\RR^+
        \;\forall\phi\in\N_\star^+,
        \label{Kosaki.new.multiplication}\\
        \sqrt{\phi}+\sqrt{\psi}&=
        \sqrt{(\phi+\psi)(y^*\,\cdot\,y)}
        \;\;\forall\phi,\psi\in\N_\star^+,
        \label{Kosaki.new.addition}
\end{align}
where\footnote{For the theorem of Connes specifying conditions guaranteeing existence and boundedness of an analytic continuation of $\Connes{\phi}{\psi}{t}$ to $t=\ii/2$ see Section \ref{nc.RN.section}.}
\begin{equation}
        y:=\Connes{\phi}{(\phi+\psi)}{-\ii/2}+\Connes{\psi}{(\phi+\psi)}{-\ii/2},
\end{equation}
and $\sqrt{\phi}$ is understood as a \textit{symbol} denoting the element $\phi$ of $\N_\star^+$ whenever it is subjected to the above operations instead of `ordinary' addition and multiplication on $\N_\star^+$. A `noncommutative Hellinger integral' on $\N_\star^+$,
\begin{equation}
        \kosaki{\phi}{\psi}:=(\phi+\psi)\left(\Connes{\psi}{(\phi+\psi)}{-\ii/2}^*\Connes{\phi}{(\phi+\psi)}{-\ii/2}\right)
\end{equation}
is a positive bilinear symmetric form on $\N_\star^+$ with respect to the operations defined by \eqref{Kosaki.new.multiplication} and \eqref{Kosaki.new.addition}. Consider an equivalence relation $\sim_\surd$ on pairs $(\sqrt{\phi},\sqrt{\psi})\in\N_\star^+\times\N_\star^+$,
\begin{equation}
        (\sqrt{\phi_1},\sqrt{\psi_1})\sim_\surd(\sqrt{\phi_2},\sqrt{\psi_2})\;\;\iff\;\;\sqrt{\phi_1}+\sqrt{\psi_2}=\sqrt{\phi_2}+\sqrt{\psi_1}.
\end{equation}
The set of equivalence classes $\N_\star^+\times\N_\star^+/\sim_\surd$ can be equipped with a real vector space structure, provided by
\begin{align}
(\sqrt{\phi_1},\sqrt{\phi_2})_\surd+(\sqrt{\psi_1},\sqrt{\psi_2})_\surd&:=(\sqrt{\phi_1}+\sqrt{\psi_1},\sqrt{\phi_2}+\sqrt{\psi_2})_\surd,\\
\lambda\cdot(\sqrt{\phi},\sqrt{\psi})_\surd&:=
        \left\{
                \begin{array}{ll}
                        (\lambda\sqrt{\phi},\lambda\sqrt{\psi})_\surd&:\lambda\geq0\\
                        ((-\lambda)\sqrt{\psi},(-\lambda)\sqrt{\phi})_\surd&:\lambda<0,
                \end{array}
        \right.
\end{align}
where $(\sqrt{\phi},\sqrt{\psi})_\surd$ denotes an element of $\N_\star^+\times\N_\star^+/\sim_\surd$. The real vector space $(\N_\star^+\times\N_\star^+/\sim_\surd,+,\cdot)$ will be denoted $V$. The map
\begin{equation}
        \N_\star^+\ni\phi\mapsto(\sqrt{\phi},0)_\surd\in V
\label{positive.predual.canonical.embedding}
\end{equation}
is injective, positive and preserves addition and multiplication by positive scalars. Its image in $V$ will be denoted by $L_2(\N)^+$. A function
\begin{equation}
        \s{\cdot,\cdot}_\surd:V\times V\ra\RR,
\end{equation}
\begin{equation}
        \s{(\sqrt{\phi_1},\sqrt{\phi_2})_\surd,(\sqrt{\psi_1},\sqrt{\psi_2})_\surd}_\surd:=
        \kosaki{\phi_1}{\psi_1}+
        \kosaki{\phi_1}{\psi_2}+
        \kosaki{\phi_2}{\psi_1}+
        \kosaki{\phi_2}{\psi_2},
\end{equation}
is an inner product on $V$, and $(V,\s{\cdot,\cdot}_\surd)$ is a real Hilbert space with respect to it, denoted $L_2(\N;\RR)$. The \df{canonical Hilbert space} is defined as a complexification of the Hilbert space $L_2(\N;\RR)$,
\begin{equation}
        L_2(\N):=L_2(\N;\RR)\otimes_\RR\CC.\rpktarget{CANON.HILB}
\end{equation}
The space $L_2(\N)^+$ is a self-polar convex cone in $L_2(\N)$, and, by \eqref{positive.predual.canonical.embedding}, it is an embedding of $\N_\star^+$ into $L_2(\N)$. The elements of $L_2(\N)^+$ will be denoted $\rpktarget{PHI.ONE.HALF}\phi^{1/2}$, where $\phi\in\N^+_\star$. Every element of $L_2(\N)$ can be expressed as a linear combination of four elements of $L_2(\N)^+$. The antilinear conjugation $\rpktarget{CANON.J}J_\N:L_2(\N)\ra L_2(\N)$ is defined by
\begin{equation}
        J_\N(\xi+\ii\zeta)=\xi-\ii\zeta\;\;\forall\xi,\zeta\in L_2(\N;\RR).
\end{equation}
The quadruple $(L_2(\N),\N,J_\N,L_2(\N)^+)$ is a standard form of $\N$, called a \df{canonical standard form} of $\N$. 

A bounded generator\footnote{See Section \ref{derivations.section} for a definition of this notion.} $\der_\N(x):=\frac{\dd}{\dd t}\left(f(\ee^{tx})\right)|_{t=0}$ of the norm continuous one parameter group of automorphisms
\begin{equation}
        \RR\ni t\mapsto f(\ee^{tx})\in\BBB(L_2(\N))\;\;\forall x\in\N,
\end{equation}
where $\N\ni x\mapsto f(x)\in\BBB(L_2(\N))$ is defined as a unique extension of the bounded linear function
\begin{equation}
        L_2(\N)^+\ni\phi^{1/2}\mapsto
        \left(\phi\left(
        \sigma^\phi_{+\ii/2}(x)x^*\,\cdot\,x\sigma^\phi_{-\ii/2}(x^*)
        \right)\right)^{1/2}\in L_2(\N)^+\;\;
        \forall x\in\N\;\forall\phi\in\N_\star^+,
\end{equation}
determines a map
\begin{equation}
        \der_\N:\N\ni x\mapsto\der_\N(x)\in\BBB(L_2(\N)),
\end{equation}
which is a homomorphism of real Lie algebras: for all $x,y\in\N$ and for all $\lambda\in\RR$,
\begin{align}
        [\der_\N(x),\der_\N(y)]&=\der_\N([x,y]),\\
        \lambda\der_\N(x)&=\der_\N(\lambda x).
\end{align}
The faithful normal representation $\rpktarget{CANON.PI}\pi_\N:\N\ra\BBB(L_2(\N))$,
\begin{equation}
        \pi_\N(x):=\textstyle\frac{1}{2}(\der_\N(x)-\ii\der_\N(\ii x)),
\end{equation}
determines a standard representation $(L_2(\N),\pi_\N,J_\N,L_2(\N)^+)$ of a $W^*$-algebra $\N$, called a \df{canonical representation} of $\N$. From the properties of standard representation it follows that every $*$-isomorphism $\varsigma:\N_1\ra\N_2$ of $W^*$-algebras $\N_1,\N_2$ determines a unique standard unitary equivalence $u_\varsigma:L_2(\N_2)\ra L_2(\N_1)$ satisfying $u_\varsigma(L_2(\N_2)^+)=L_2(\N_1)^+$, $u_\varsigma^*J_{\N_2}u_\varsigma=J_{\N_1}$, and such that $\Ad(u_\varsigma^*)$ is a unitary implementation\footnote{See Section \ref{automorphisms.section} for a definition of this notion.} of $\varsigma$. This means that Kosaki's construction of canonical representation defines a functor $\CanRep$ from the category $\WsIso$ of $W^*$-algebras with $*$-isomorphisms to the category $\StdRep$ of standard representations with standard unitary equivalences. It also allows to define a functor $\CanVN$ from the category $\Wsn$ of $W^*$-algebras with normal $*$-homomorphisms to the category $\VNn$ of von Neumann algebras with normal $*$-homomorphisms. The functor $\CanVN$ assigns $\pi_\N(\N)$ to each $\N$, and normal $*$-homomorphism $\pi_\N\circ\varsigma\circ\pi_\N^{-1}:\N_1\ra\N_2$ to each normal $*$-homomorphism $\varsigma:\N_1\ra\N_2$ (which is well defined due to faithfulness of $\pi_\N$). Let $\FrgHlb:\VNn\ra\Wsn$ be the forgetful functor which forgets about Hilbert space structure that underlies von Neumann algebras and their normal $*$-homomorphisms. Due to Sakai's theorem (see Section \ref{vNa.section}), $\CanVN$ and $\FrgHlb$ form the equivalence of categories,
\begin{equation}
        \FrgHlb\circ\CanVN\iso\id_{\Wsn},\;\;\CanVN\circ\FrgHlb\iso\id_{\VNn}.
\label{Sakai.Kosaki.duality}
\end{equation}

Given a $W^*$-algebra $\N$, a \df{right $\N$ module} $\H_{\pi(\N)}$ is defined as a Hilbert space $\H$ equipped with a normal antirepresentation $\pi^o:\N\ra\BH$, or, equivalently, a normal representation $\pi:\N^o\ra\BH$. If $\N_1$ and $\N_2$ are $W^*$-algebras, then an $\N_1$-$\N_2$ \df{bimodule} (or an $\N_1$-$\N_2$ \df{correspondence}) ${}_{\pi_1(\N_1)}\H_{\pi_2(\N_2)}$ is defined as a Hilbert space $\H$ equipped with a normal representation $\pi_1:\N_1\ra\BH$ and a normal representation $\pi_2:\N_2^o\ra\BH$ such that $\pi_1$ and $\pi_2$ commute:
\begin{align}
        \pi_1(\N_2)&\subseteq\pi_2(\N_2^o)^\comm,\\
        \pi_2(\N_2^o)&\subseteq\pi_1(\N_1)^\comm.
\end{align}
This allows to use an associative notation
\begin{equation}
        \pi_1(x)\pi_2(y^o)\xi=:x\xi y=(x\xi)y=x(\xi y)\;\;\forall x\in\N_1\;\forall y\in\N_2\;\;\;\forall\xi\in\H.
\end{equation}
The canonical representation $(L_2(\N),\pi_\N,J_\N,L_2(\N)^+)$ equips $L_2(\N)$ with $\N$-$\N$ bimodule structure provided by the left action of $\pi_\N:\N\ra\BBB(L_2(\N))$ and the right action of
\begin{equation}
        \pi^o_\N(\cdot):=J_\N\pi_\N(\cdot)^*J_\N=\textstyle\frac{1}{2}(\der_\N((\cdot)^*)+\ii\der_\N(\ii(\cdot)^*))^*.
\end{equation}
A \df{canonical right $\N$ module} and a \df{canonical $\N$-$\N$ bimodule} (the latter called also an \df{identity correspondence} \cite{Connes:1980}) are defined, respectively, as
\begin{equation}
        L_2(\N)_\N:=L_2(\N)_{J_\N\pi_\N(\N)^*J_\N},\;\;\;
        {}_\N L_2(\N)_\N:={}_{\pi_\N(\N)}L_2(\N)_{J_\N\pi_\N(\N)^*J_\N}.
\end{equation}

Let us consider a few special cases of this construction that correspond to cases discussed in the introduction of Section \ref{modular.theory.section}. Given a standard form $(\H,\pi,J,\stdcone)$ of a $W^*$-algebra $\N$, we can define a right action of $\N$ on $\H$ by a normal antirepresentation
\begin{equation}
        \xi x:=Jx^*J\xi\;\;\forall x\in\pi(\N)\;\forall\xi\in\H.
\end{equation}
Since $J\pi(\N)J=\pi(\N)^\comm$, $\pi^o(x):=J\pi(x^*)J$ defines a $*$-antiisomorphism $\pi(\N)\mapsto\pi^o(\N)$ and a normal representation $\pi:\N^o\ra\BH$, turning $\H$ into an $\N$-$\N$ bimodule. 
If $\tau$ is a faithful normal semi-finite trace on $\N$, then the standard representation of $\N$ can be specified as a GNS representation $(\H_\tau,\pi_\tau)$, equipped with the conjugation $J$ on $\H_\tau$, defined as a unique extension of the antilinear isometry $J:[x]_\omega\mapsto[x^*]_\omega\;\forall x\in\nnn_\omega$. The map 
\begin{equation}
\pi^o_\tau(x)[y]_\tau:=[yx]_\tau
\end{equation}
defines an antirepresentation of $\N$, whose boundedness follows from
\begin{equation}
        \n{[yx]_\tau}^2=\phi(x^*y^*yx)=\tau(yxx^*y^*)\leq\n{x}^2\tau(yy^*)=\n{x}^2\tau(y^*y)=\n{x}[y]^2_\tau.
\end{equation}
It satisfies
\begin{align}
        J\pi_\tau(x)J&=\pi_\tau^o(x^*)\;\;\forall x\in\N,\\
        (\pi_\tau^o(\N))^\comm&=\pi_\tau(\N).
\end{align}
If $\N\iso\BBB(\K)$ for some Hilbert space $\K$, then the standard representation is given by the Hilbert--Schmidt space $\schatten_2(\K)=\K\otimes\K^\banach=:\H_\HS$ (see Section \ref{integration.trace.section}), its positive cone $\stdcone=\H_\HS^+$, and
\begin{align}
        \rpktarget{LEFT}\pi(x)=\LLL(x):\H_\HS\ni\xi&\mapsto x\xi\in\H_\HS,\\
        J:\H_\HS\ni\xi&\mapsto\xi^*\in\H_\HS.
\end{align}
while the right action of $\N$ is given by the antilinear map
\begin{equation}
        \rpktarget{RIGHT}\RRR(x):\H_\HS\ni\xi\mapsto\xi x^*\in\H_\HS.
\end{equation}
These operations satisfy \cite{Murray:vonNeumann:1937,Haag:Hugenholtz:Winnink:1967}
\begin{equation}
\LLL(\N)\cap\RRR(\N)=\CC\II,\;\;\LLL(\N)^\comm=J\LLL(\N)J=\RRR(\N),\;\;\RRR(\N)^\comm=\LLL(\N).
\end{equation}

Finally, let us recall Falcone's  construction of Connes' spatial quotient in terms of right $\N$-modules \cite{Falcone:1996,Falcone:2000,Takesaki:2003} (see also \cite{Yamagami:1994}). For the Hilbert spaces $\H,\H_1,\H_2$, a $W^*$-algebra $\N$, and $\phi,\psi\in\W(\N)$, let the right $\N$ modules ${\H_1}_\N$ and ${\H_2}_\N$ be defined by normal  representations $\pi_1:\N^o\ra\BBB(\H_1)$ and $\pi_2:\N^o\ra\BBB(\H_2)$, respectively. Let
\begin{align}
        \BBB({\H_1}_\N,{\H_2}_\N)&:=\{x:\H_1\ra\H_2\mid x\mbox{ is a bounded linear map and }x(\xi y)=(x\xi y)\;\;\forall y\in\pi_1(\N)\},\\
        \nnn_\phi(\H)&:=\{x\in\BBB(\H_\N,L_2(\N)_\N)\mid\phi(x^*x)<\infty\}.
\end{align}
Then the antilinear operator
\begin{equation}
        R_{\phi,\psi}[x]_\psi:=[x^*]_\phi\;\;\forall x\in\nnn_\psi(\H)\cap\nnn_\phi(\H)^*
\end{equation}
is closable, and \df{Connes' spatial quotient} can be defined as\rpktarget{CONNES.SPAT.CANON}
\begin{equation}
        \left(\connes{\phi}{\psi^\comm}\right)^{1/2}:=R_{\phi,\psi}^*\bar{R}_{\phi,\psi}.
        \label{Connes.spatial.by.Falcone}
\end{equation}
See \cite{Falcone:1996,Falcone:2000,Takesaki:2003} for a proof of equivalence of definitions \eqref{Connes.spatial.by.Falcone} and \eqref{connes.spatial.def}. This allows to introduce a \df{canonical relative modular operator} on $L_2(\N)$ \cite{Kosaki:1980:PhD},\rpktarget{DELTA.REL.CANON}
\begin{equation}
        \Delta_{\phi,\psi}:=\connes{\phi}{\psi(J_\N\cdot J_\N)}.
\label{canonical.rel.mod.op}
\end{equation}
\subsection{Classification of $W^*$-algebras\label{classification.section}}
The problem of classification of von Neumann algebras has played a key role in the development of theory of operator algebras. As a consequence of the von Neumann double commutant theorem \cite{vonNeumann:1930:algebra}, the set $\Proj(\N)$ generates every von Neumann algebra $\N$ by $\Proj(\N)^\comm{}^\comm=\N$. This has to be contrasted with general $C^*$-algebras, which may have no nonzero projectors. Von Neumann and Murray developed a classification of factor von Neumann algebras $\N$ based on the analysis of the properties of $\Proj(\N)$ in terms of a \textit{dimension function} \cite{Murray:vonNeumann:1936,Murray:vonNeumann:1937,vonNeumann:1939,vonNeumann:1943,Murray:vonNeumann:1943,vonNeumann:1949}. Due to von Neumann's reduction theorem \cite{vonNeumann:1949}, every von Neumann algebra over a separable Hilbert space is uniquely isomorphic to a direct integral of factor von Neumann algebras (the construction of a direct integral of von Neumann algebras is based on the construction of the direct integral of Hilbert spaces). As a result, the classification of factor von Neumann algebras allows for a complete classification of all von Neumann algebras. Dixmier \cite{Dixmier:1949,Dixmier:1951:reduction}, Kaplansky \cite{Kaplansky:1950,Kaplansky:1951,Kaplansky:1952} and others developed equivalent classification of arbitrary von Neumann algebras $\N$ in terms of properties of projections in $\Proj(\N)$, without invoking dimension function, without using decomposition to factors, and without requirement of separability of the Hilbert space. Both classifications divide the collection of all von Neumann algebras into disjoint types I$_n$, I$_\infty$, II$_1$, II$_n$, and III. Von Neumann and Murray constructed two nonisomorphic type II$_1$ factors \cite{Murray:vonNeumann:1943} and proved \cite{vonNeumann:1943,Murray:vonNeumann:1943} existence of one example of type III factor. Next two examples of type III factors were found in \cite{Pukanszky:1956,Schwartz:1963}, while next seven examples of type II$_1$ factors were found in \cite{Schwartz:1963:typeII,Sakai:1968,Ching:1969,Dixmier:Lance:1969,ZellerMeier:1969}. The breakthrough came with discovery of continuum of type III$_\lambda$ factors by Powers \cite{Powers:1967}, a continuum of type III$_1$ factors by Araki and Woods \cite{Araki:Woods:1968}, and another  continuum of type III factors by Sakai \cite{Sakai:1970:typeIII}. Soon also a continuum of nonisomorphic factors of type II$_1$ and II$_\infty$ was discovered \cite{McDuff:1969:countable,McDuff:1969,Sakai:1970}. This has led to refined classification of type III von Neumann algebras, into types III$_0$, III$_1$ and III$_\lambda$ with $\lambda\in\,]0,1[$, developed by Connes and Takesaki \cite{Connes:1973:classification,Takesaki:1973:duality,Takesaki:1973:structure,Connes:Takesaki:1974,Connes:Takesaki:1977,Connes:1982,Takesaki:1983}. The Dixmier--Kaplansky classification was translated to abstract $W^*$-algebras by Sakai \cite{Sakai:1962,Sakai:1964,Sakai:1971}. Thanks to Kosaki's construction of canonical representation, the Connes--Takesaki refined classification of type III von Neumann algebras can be applied also to $W^*$-algebras in a canonical sense.

Given any $W^*$-algebra $\N$, consider a \df{dimension function}, defined as a map $\rpktarget{DIMFUN}\dimfun:\Proj(\N)\ra[0,+\infty]$ satisfying
\begin{enumerate}
\item[i)] $P=0\;\iff\;\dimfun(P)=0$,
\item[ii)] $P_1P_2=0\;\limp\;\dimfun(P_1+P_2)=\dimfun(P_1)+\dimfun(P_2)$,
\item[iii)] $P_1\sim P_2$ $\iff$ $\dimfun(P_1)=\dimfun(P_2)$,
\item[iv)] (($P_2\leq P_1$ and $P_2\sim P_1$) $\limp$ $P_2=P_1$) $\iff$ $\dimfun(P_1)<\infty$,
\end{enumerate}
where for any $P_1,P_2\in\Proj(\N)$ one writes $P_1\sim P_2$ if{}f there exists a partial isometry $v\in\N$ such that $P_1=v^*v$ and $P_2=vv^*$.

For any factor von Neumann algebra $\N$ acting on a separable Hilbert space there exists a dimension function $\dimfun_\N$ that is unique up to multiplication by $\lambda\in\,]0,\infty[$ \cite{Murray:vonNeumann:1936}, and every factor von Neumann algebra $\N$ can be classified as one and only one among of the following types:
\begin{itemize}
\item \df{type I$_n$} if{}f $\ran(\dimfun_\N)=\{0,1,2,\ldots,n\}$,
\item \df{type I$_\infty$} if{}f $\ran(\dimfun_\N)=\{0,1,2,\ldots,\infty\}$,
\item \df{type II$_1$} if{}f $\ran(\dimfun_\N)=[0,1]$,
\item \df{type II$_\infty$} if{}f $\ran(\dimfun_\N)=[0,\infty]$,
\item \df{type III} if{}f $\ran(\dimfun_\N)=\{0,\infty\}$.
\end{itemize}

If $\N=\BH$, $\H$ is separable and $P_\D\in\Proj(\N)$ is a projection operator onto a subspace $\D\subseteq\H$, then $\dimfun_\N(P_\D)=\dim\D$. Murray and von Neumann showed that for type II$_1$ factors $\N$ the dimension function $\dimfun_\N$ can be extended to a trace $\dimfun_\N:\N^+\ra[0,+\infty]$, and it is equal to a \textit{unique} trace on $\N$ that satisfies $\tau(\II)=1$ \cite{Murray:vonNeumann:1937}. This trace is normal, hence it corresponds to a distinguished fixed element of $\N_\star^+$ (see \cite{Yeadon:1971} for a proof based on the Ryll-Nardzewski theorem).

The `global' analogue of the above classification for an arbitrary $W^*$-algebra $\N$ is based on the analysis of the internal properties of $\Proj(\N)$.  A projection $P\in\Proj(\N)$ is called: \df{nonzero} if{}f $P\neq0$; \df{central} if{}f $P\in\zentr_\N$; \df{abelian} if{}f it is nonzero and $P\N P$ is commutative; \df{finite} if{}f $\forall P_1\in\Proj(\N)$ $(P_1\leq P$ and $P_1\sim P)\;\limp\;P_1=P$; \df{infinite} if{}f it is not finite; \df{purely infinite} if{}f it does not contain any nonzero finite projection; \df{properly infinite} if{}f $\forall P_1\in\Proj(\N)\cap\zentr_\N$ $P_1P$ is finite $\limp$ $P_1P=0$. A $W^*$-algebra is called to be: \df{finite} if{}f $\II$ is finite; \df{purely infinite} if{}f $\II$ is purely finite; \df{properly infinite} if{}f $\II$ is properly infinite; \df{semi-finite} if{}f the unique maximal purely infinite central projection in $\N$ is zero (or, equivalently, if{}f any nonzero central projection contains a nonzero finite projection); \df{type I} if{}f every nonzero central projection contains a nonzero abelian projection; \df{type I$_\infty$} if{}f it is of type I and is not finite; \df{type I$_n$} if{}f there exists a family $\{P_i\}\subseteq\Proj(\N)$, $i\in\{1,\ldots,n\}$, of abelian projections such that
\begin{equation}
        \sum_{i=1}^nP_i=\II,\;\;\;P_iP_j=0\;\;\forall i\neq j\in\{1,\ldots,n\},\;\;\;P_i\sim P_j\;\;\forall i,j\in\{1,\ldots,n\};
\end{equation}
\df{type II} if{}f $\N$ is semi-finite and does not contain any nonzero abelian projection; \df{type II$_1$} if{}f it is of type II and is finite; \df{type II$_\infty$} if{}f it is of type II and is properly infinite; \df{type III} if{}f it is purely infinite; \df{discrete} if{}f it is of type I; \df{continuous} if{}f it is not discrete. According to classification theorem for $W^*$-algebras, every $W^*$-algebra can be uniquely decomposed to a direct sum of five $W^*$-algebras, which are of type I$_n$, type I$_\infty$, type II$_1$, type II$_\infty$, and type III, respectively \cite{vonNeumann:1949,Dixmier:1957,Sakai:1964,Sakai:1971,Zsido:1973}.

Global classification of $W^*$-algebras based on the properties of $\Proj(\N)$ can be equivalently provided in terms of types of functionals they admit. A $W^*$-algebra is: semi-finite if{}f it admits a faithful normal semi-finite trace; finite if{}f it admits faithful normal finite trace. Type I$_n$ and type II$_1$ $W^*$-algebras admit faithful normal finite traces; type I$_\infty$ and type II$_\infty$ $W^*$-algebras do not admit them, but they admit faithful normal semi-finite traces; for type III $W^*$-algebras no faithful normal semi-finite trace exists. Hence, a $W^*$-algebra is semi-finite if{}f it does not contains any type III $W^*$-algebra. Finally, type II and type III $W^*$-algebras do not admit pure states \cite{Dye:1952,Dixmier:1953}\footnote{Recall that the GNS representation $\pi_\omega$ is irreducible if{}f $\omega$ is normalised and pure \cite{Segal:1947:irreducible}.}. The reasons for working with three classes of functionals on $W^*$-algebras can be summarised as follows:
\begin{enumerate}
\item[1)] only countably finite $W^*$-algebras admit faithful states,
\item[2)] only semi-finite $W^*$-algebras admit faithful normal semi-finite traces,
\item[3)] every $W^*$-algebra admits a faithful normal semi-finite weight.
\end{enumerate} 

All commutative $W^*$-algebras are type I factors. Every type I$_n$ factor is isomorphic to the algebra $\MNC$ of complex $n\times n$ matrices, while every type I$_\infty$ factor is isomorphic to the algebra $\BH$ on a separable infinite dimensional Hilbert space $\H$. In general, a $W^*$-algebra $\N$ is of type I$_n$ with $n\in\NN\cup\{+\infty\}$ if{}f it is $*$-isomorphic to $\BH$ for some $\H$ with $\dim\H=n$. Every normal trace on $\BH$ is either proportional to $\tr$, or takes the value $+\infty$ for all strictly positive operators \cite{Dixmier:1957}. Every type II$_\infty$ factor is a tensor product of some type I$_\infty$ factor and some type II$_1$ factor. Moreover, every type III factor is equal to a crossed product $\N\rtimes_{\sigma^\psi}\RR$ of a type II$_\infty$ factor with $\RR$ with respect to the group of modular automorphisms \cite{Takesaki:1973:duality} (see Section \ref{crossed.product.section}).

Given arbitrary $W^*$-algebra $\N$ and $\omega\in\W_0(\N)$, the spectrum $\sp(\Delta_\omega)$ measures the periodicity of the modular automorphism group $\sigma_t^\omega=\pi_\omega^{-1}\circ\Ad(\Delta^{\ii t}_\omega)$. If $\sp(\Delta_\omega)=\{1\}$ then $\sigma_t^\omega=\id_\N\;\forall t\in\RR$. If $\N$ is a countably finite $W^*$-algebra, then
\begin{equation}
\bigcap_{\omega\in\N^+_{\star0}}\sp(\Delta_\omega)=\{1\}\;\;\iff\;\;\N\mbox{ is finite}.
\end{equation}
The \df{modular spectrum} 
\begin{equation}\rpktarget{MODSPEC}
        \modspec(\N):=\bigcap_{\omega\in\W_0(\N)}\sp(\Delta_\omega)
\end{equation}
is an algebraic invariant characterising the factor $W^*$-algebra. If $\N$ is a type III factor, then it is called \cite{Connes:1973:classification}
\begin{itemize}
\item \df{type III$_0$} if{}f $\modspec(\N)=\{0,1\}$,
\item \df{type III$_\lambda$} if{}f $\modspec(\N)=\{0\}\cup\{\lambda^n\mid n\in\ZZ\}$ with $\lambda\in\,]0,1[$,
\item \df{type III$_1$} if{}f $\modspec(\N)=[0,\infty[$. 
\end{itemize}

The group $\Inn(\C)$ of \df{inner $*$-automorphisms} of a $C^*$-algebra $\C$ is defined by
\begin{equation}
        \Inn(\C):=\{\alpha\in\Aut(\C)\mid\exists u\in\C^\uni\;\;\forall x\in\C\;\;\;\alpha(x)=uxu^*\}.
\end{equation}
$\Inn(\C)$ is a normal subgroup of $\Aut(\C)$.\footnote{A subgroup $G_2$ of a group $G_1$ is called \df{normal} if{}f $g_1g_2g_1^{-1}\in G_2$ $\forall g_1\in G_1$ $\forall g_2\in G_2$.} Although $\Delta_\omega^{\ii t}x\Delta_\omega^{-\ii t}\in\N$ $\forall t\in\RR$ $\forall x\in\N$ for any $W^*$-algebra $\N$ and any $\omega\in\W_0(\N)$, the relationship $\Delta_\omega^{\ii t}\not\in\N$ holds also if $\N$ is semi-finite, unless $\Delta_\omega=\II$. 
Given
\begin{equation}
        \varepsilon_\N:\Aut(\N)\ra\Out(\N):=\Aut(\N)/\Inn(\N),
\end{equation}
one has 
\begin{equation}
        \varepsilon_\N(\sigma^\phi_t)=\varepsilon_\N(\sigma^\psi_t)\;\;\forall\psi,\phi\in\W_0(\N)\;\forall t\in\RR,
\end{equation}
and the map $\hat{\varepsilon}_\N:\RR\ni t\mapsto\varepsilon_\N(\sigma^\phi_t)\in\Out(\N)$ is a homeomorphism of the additive group $\RR$ onto the center of the group $\Out(\N)$, which does not depend on $\phi$. Connes \cite{Connes:1973:classification} showed that the $*$-automorphism groups $\sigma^\phi$ associated with the elements $\phi\in\W_0(\N)$ can be characterised as such arrows $\alpha$ for which the diagram
\begin{equation}
\xymatrix{
        &
        \RR
        \ar[dl]_{\alpha}
        \ar[dr]^{\hat{\varepsilon}_\N}
        &
        \\
        \Aut(\N)
        \ar[rr]_{\varepsilon_\N}
        &
        &
        \Out(\N).
}
\end{equation}
commutes. In this sense, modular automorphisms classify the noninner $*$-automorphisms of $W^*$-algebras. For type III$_\lambda$ factors $\N$ one has
\begin{equation}
        \ker(\hat{\varepsilon}_\N)=-\frac{2\pipi}{\log\lambda}\ZZ.
\end{equation}
The conditions
\begin{itemize}
\item[i)] $0\not\in\modspec(\N)$,
\item[ii)] $\modspec(\N)=\{1\}$,
\item[iii)] $\ker(\hat{\varepsilon}_\N)=\RR$,
\item[iv)] $\sigma^\phi\in\Inn(\N)\;\forall\phi\in\W_0(\N)$,
\item[v)] $\N$ is semi-finite,
\end{itemize}
are equivalent \cite{Takesaki:1970,Pedersen:Takesaki:1973,Connes:1973:classification}. 

The Connes--Takesaki refined classification theory leads to a question about possibility of complete classification of type III factors, and more generally, of all $W^*$-algebras. This problem \cytat{remains elusive and appears hardly more realizable than a full classification for infinite discrete groups} \cite{Segal:1996}. In particular, the space of isomorphism classes of factors is not countably separable \cite{Woods:1973}. However, there are still some properties which allow further refinement of the classification. A factor $W^*$-algebra $\N$ with a separable predual $\N_\star$ is called \df{approximately finite dimensional} if{}f it is generated by a sequence $\{\N_i\}$ of type I$_{n_i}$ factor $W^*$-algebras with $n_1\leq n_2\leq\ldots$, which holds if{}f $\{\bigcup_i\N_i\}$ is dense in $\N$ in weak-$\star$ topology. By von Neumann's bicommutant theorem, this implies that approximately finite dimensional factor von Neumann algebras satisfy $\N=\{\bigcup_i\N_i\}^\comm{}^\comm$. Von Neumann and Murray \cite{Murray:vonNeumann:1943} \ showed that there exists a unique approximately finite dimensional factor of type II$_1$ (see also \cite{Connes:1975,Connes:1976:injective,Connes:1977}). Connes \cite{Connes:1975,Connes:1976:injective} showed that there exists a unique approximately finite dimensional factor of type II$_\infty$ and of type III$_\lambda$ for each $\lambda\in\,]0,1[$ separately. Finally, Haagerup proved the uniqueness of an approximately finite dimensional type III$_1$ factor \cite{Haagerup:1987} (see also \cite{Connes:1985}). There is no unique approximately finite dimensional factor of type III$_0$ \cite{Krieger:1976,Connes:1975:hyperfinite}. A unique approximately finite dimensional type II$_\infty$ factor is equal to a tensor product of a unique approximately finite dimensional type I$_\infty$ and type II$_1$ factors, and a unique approximately finite dimensional type III$_1$ factor (which was defined in \cite{Araki:Woods:1968}) is a crossed product of a unique approximately finite dimensional type II$_\infty$ factor with $\RR$. 
\section{Automorphisms and their representations\label{automorphisms.section}}
Within the frames of a Hilbert space based approach to quantum theory, the symmetries of quantum theoretic models are usually investigated in terms of representations of groups $G$ by unitary operators in $\BH$ \cite{Weyl:1927,Weyl:1928,Wigner:1931,vanderWaerden:1932,Wigner:1939,Barut:Raczka:1977}. The algebraic perspective shifts considerations from the abstract Hilbert space $\H$ to the abstract $C^*$-algebra $\C$. This leads to description of symmetries of algebraic quantum models in terms of $*$-automorphisms $\rpktarget{ALPHA}\alpha\in\Aut(\C)$ and representations $G\ni g\mapsto\alpha_g\in\Aut(\C)$ of groups $G$.

Given any group\footnote{For an extension of the theory of $W^*$- and $C^*$- dynamical systems to groupoids, see \cite{Masuda:1984:I,Masuda:1984:II,Masuda:1985}.} $G$, a \df{representation} of $G$ in the group $\Aut(\C)$ of $*$-automorphisms of a $C^*$-algebra $\C$ is a map $\alpha:G\ni g\mapsto\alpha(g)=:\alpha_g\in\Aut(\C)$ which is a \df{group homomorphism}, that is,
\begin{enumerate}
        \item[1)] $\alpha(e)=\id_\C$,
        \item[2)] $\alpha(g_1)\circ\alpha(g_2)=\alpha(g_1\circ g_2)$ $\forall g_1,g_2\in G$,
\end{enumerate}
where $e$ denotes the neutral element of $G$. A group $G$ is called: \df{topological} if{}f it is also a topological space and a map $G\times G\ni(g_1,g_2)\mapsto g_1\circ g_2^{-1}\in G$ is continuous for all $g_1,g_2\in G$; \df{locally compact} if{}f it is topological and $e\in G$ has a compact\footnote{See Section \ref{comm.integr.section} for the definitions of compact and locally compact topological spaces.} topological neighbourhood. The group $\Aut(\C)$ can be considered as a topological group with respect to various topologies. In particular, the \df{norm topology} on $\Aut(\C)$ is defined by the convergence in the norm
\begin{equation}
        \n{\alpha_1-\alpha_2}_{\Aut(\C)}:=\sup_{\n{x}\leq 1}\n{\alpha_1(x)-\alpha_2(x)}\;\;\forall x\in\C\;\forall\alpha_1,\alpha_2\in\Aut(\C),
\end{equation}
while the \df{strong topology} on $\Aut(\C)$ is defined by the collection of neighbourhoods
\begin{equation}
        N_{\{x_i\},\epsilon}(\alpha):=\{\varsigma\in\Aut(\C)\mid\n{\varsigma(x_i)-\alpha(x_i)}<\epsilon\},
\end{equation}
where $\{x_i\}\subseteq\C$, $i\in\{1,\ldots,n\}$, $n\in\NN$, $\epsilon>0$. If $\N$ is a $W^*$-algebra, then $\Aut(\N)$ is a topological group also with respect to \df{weak-$\star$ topology} on $\Aut(\N)$, defined by the collection of neighbourhoods \cite{Takesaki:1983}
\begin{equation}
        N_{\{\omega_i\}}(\alpha):=
        \{\varsigma\in\Aut(\N)\mid
        \n{\omega_i\circ\alpha-\omega_i\circ\varsigma}_{\N_\star}<1,\;
        \n{\omega_i\circ\alpha^{-1}-\omega_i\circ\varsigma^{-1}}_{\N_\star}<1\},
\end{equation}
where $\{\omega_i\}\subseteq\N_\star$, $i\in\{1,\ldots,n\}$, $n\in\NN$. A triple $(\C,G,\alpha)$ of a $C^*$-algebra $\C$, locally compact group $G$, and a representation $\alpha:G\ra\Aut(\C)$ is called a \df{$C^*$-dynamical system} (or a \df{$C^*$-covariant system}) if{}f $\alpha$ is continuous in the strong topology of $\Aut(\C)$. This condition is equivalent to the continuity of the map $G\ni g\mapsto\alpha_g(x)\in\C$ in the norm topology of $\C$,
\begin{equation}
        \lim_{g\ra h}\n{\alpha_g(x)-\alpha_h(x)}=0\;\;\forall x\in\C\;\forall h\in G,
\label{strong.continuous.group}
\end{equation}
and such $\alpha$ is called a \df{strongly continuous} representation. A triple $(\N,G,\alpha)$ of a $W^*$-algebra, locally compact group $G$, and a representation $\alpha:G\ra\Aut(\N)$ is called a \df{$W^*$-dynamical system} (or a \df{$W^*$-covariant system}) if{}f $\alpha$ is continuous in the weak-$\star$ topology of $\Aut(\N)$. This condition is equivalent to the continuity of the map $G\ni g\mapsto\alpha_g(x)\in\N$ in the weak-$\star$ topology of $\N$ for any $x\in\N$, that is, to 
\begin{equation}
        G\ni g\mapsto\phi(\alpha_g(x))\in\CC\mbox{ is a continuous function}\;\;\forall x\in\N,
\label{weak.star.continuous.group}
\end{equation}
and such $\alpha$ is called a \df{weakly-$\star$ continuous} representation. Uniqueness of a predual of a $W^*$-algebra $\N$ allows to define isometries $\alpha_\star$ of $\N_\star$ that uniquely correspond to the elements $\alpha\in\Aut(\N)$, and to define the isometries of $\N_\star$ uniquely corresponding to representations $\alpha:G\ra\Aut(\N)$:
\begin{equation}
        \duality{(\alpha_g)_\star(\phi),x}_{\N_\star\times\N}=\duality{\phi,\alpha_g(x)}_{\N_\star\times\N}=\phi(\alpha_g(x))\;\;\forall x\in\N\;\forall\phi\in\N_\star.
\end{equation}
The above equivalence can be shown (see e.g. \cite{Sakai:1991}) by proving that \eqref{weak.star.continuous.group} implies continuity of $\alpha_\star$ in the norm of $\N_\star$,
\begin{equation}
        \lim_{g\ra e}\n{(\alpha_g)_\star(\phi)-\phi}_{\N_\star}=0\;\;\forall\phi\in\N_\star\;\forall g\in G.
\end{equation}

For example, for the Hilbert space $\H$ and a self-adjoint operator $H$ on $\H$, a triple $(\N,\RR,\alpha)$, with $\N:=\BH$ and $\alpha_t(x):=\ee^{\ii tH}x\ee^{-\ii tH}$ can be considered as a $C^*$-dynamical system and as a $W^*$-dynamical system. (See \cite{Simon:1976} for a detailed mathematical derivation of a Schr\"{o}dinger equation on a Hilbert space from a suitable one parameter group of automorphisms.)

Consideration of groups of $*$-automorphisms of a $C^*$-algebra $\C$ instead of groups of unitary operators on a Hilbert space is a nontrivial generalisation. For example, for a given group of $*$-automorphisms $\{\alpha_g\mid g\in G\}\subseteq\Aut(\C)$ there may be no strongly  continuous group of unitary elements $u(g)\in\C^\uni$ such that $\alpha_g(x)=u(g)xu(g)^*\;\forall x\in\C$. Moreover, if such group exists, it might be nonunique. On the other hand, a \df{unitary implementation} of a representation $\alpha:G\ra\Aut(\C)$ in a given representation $\pi:\C\ra\BH$ is defined as a map $u:G\ni g\mapsto u(g)\in\BH^\uni$ that determines a family $\{u(g)\mid g\in G\}$ of unitary operators satisfying the \df{covariance equation}
\begin{equation}
        \pi(\alpha_g(x))=u(g)\pi(x)u(g)^*\;\;\forall x\in\C\;\forall g\in G.
\label{covariance.equation.group.G}
\end{equation}
The condition \eqref{covariance.equation.group.G} alone does not determine $\{u(g)\mid g\in G\}$ uniquely. This leads to a question under which conditions one can guarantee the existence and uniqueness of corresponding groups of unitaries (and their generators) acting either on the level of algebra $\C$ or on the level of the representation of $\C$.
\subsection{Derivations\label{derivations.section}}
A \df{derivation} of a $C^*$-algebra $\C$ is defined \cite{Kaplansky:1953} as a map $\rpktarget{DER}\der:\dom(\der)\ra\C$, where $\dom(\der)\subseteq\C$ is a $^*$-subalgebra such that, for every $\lambda_1,\lambda_2\in\CC$ and every $x,y\in\dom(\der)$,
\begin{enumerate}
\item $\der(\lambda_1x+\lambda_2y)=\lambda_1\der(x)+\lambda_2\der(y)$,
\item $\der(xy)=\der(x)y+x\der(y)$,
\item $\der(x)^*=\der(x^*)$ and $x\in\dom(\der)\limp x^*\in\dom(\der)$.
\end{enumerate}
In this sense, a derivation is an algebraic derivative (differential) over a $C^*$-algebra. On the other hand, an algebraic analogue of a hamiltonian is provided by an (infinitesimal) \df{generator} of a strongly continuous representation $\alpha:\RR\ra\Aut(\C)$, which is defined as a map $k:\dom(k)\ra\C$, where
\begin{align}
        \dom(k)&:=\{x\in\C\mid\lim_{t\ra 0}\frac{1}{t}(\alpha_t(x)-x)\mbox{ exists in norm}\},\\
        k(x)&:=\lim_{t\ra 0}\frac{1}{t}(\alpha_t(x)-x)\;\;\forall x\in\dom(k).
\end{align}
If $k$ is a generator, then from the properties \eqref{homo.lin}-\eqref{homo.star} it follows that it is a derivation. But when a derivation is a generator? In other words, when it is `integrable' to a $*$-automorphism?

The unique relationship between derivations and generators can be established without introducing any additional conditions only if the derivation is bounded and if the $*$-automorphism is norm continuous. For any $C^*$-algebra $\C$ if $\dom(\der)=\C$, then $\der$ is bounded (and if $\C$ is also commutative, then $\der=0$) \cite{Kaplansky:1953,Sakai:1960}. In this case, there is a bijection between norm continuous groups $\RR\ni t\mapsto\alpha_t\in\Aut(\C)$ of $*$-automorphisms, their generators, and bounded derivations on $\C$, given by
\begin{align}
        \alpha_t&=\exp(t\der):=\sum_{n=0}^\infty\frac{(t\der)^n}{n!}\;\;\forall t\in\RR,\\
        \der&=\lim_{t\ra0}\frac{1}{t}(\alpha_t-\id_\N).
\end{align}

In case if $\der$ is a derivation on a $W^*$-algebra $\N$ and $\dom(\der)=\N$, then $\der$ is bounded and \cite{Kadison:1966,Sakai:1966} 
\begin{equation}
        \exists h\in\N^\sa\;\forall x\in\N\;\;\der(x)=\ii[h,x],
\label{Kadison.Sakai.thm}
\end{equation}
with $\n{h}\leq\frac{1}{2}\n{\der}$. For example, if $U:\RR\ni t\mapsto U(t)\in\BH^\uni$ is a strongly continuous group of unitary operators such that $U(t)\N U(t)^*\subseteq\N\;\forall t\in\RR$, then the family of maps
\begin{equation}
        \N\ni x\mapsto\alpha_t(x):=U(t)xU(t)^*\in\N
\end{equation}
is a weak-$\star$ continuous group of $*$-automorphisms of $\N$ whose generator $\der$ has the form $\der(x)=\ii[H,x]$, where $H$ is a self-adjoint generator of the group $\{U(t)\mid t\in\RR\}$. From \eqref{Kadison.Sakai.thm} it follows that for any derivation $\der:\C\ra\C$ on a $C^*$-algebra, and for any representation $\pi$ of $\C$,
\begin{equation}
       \exists h\in(\pi(\C)^\comm{}^\comm)^\sa\;\forall x\in\C\;\;\;\pi(\der(x))=\ii[h,\pi(x)]
\end{equation}
with $\n{h}\leq\frac{1}{2}\n{\der}$, hence there exists a group $\{\alpha_t\mid t\in\RR\}\subseteq\Aut(\C)$ such that
\begin{equation}
        \pi(\alpha_t(x))=\ee^{\ii th}\pi(x)\ee^{-\ii th}\;\;\forall x\in\C\;\forall t\in\RR.
\end{equation}

The conditions of norm continuity of $*$-automorphisms and boundedness of derivations put strong limitations on possible $*$-automorphisms that can be specified uniquely using derivations without assuming any additional properties \cite{Bratteli:Robinson:1979,Pillet:2006}. This leads to three (intertwined) directions: consideration of unbounded derivations on $C^*$-algebras (see e.g. \cite{Bratteli:1986,Sakai:1991}), consideration of weakly-$\star$ continuous automorphisms of $W^*$-algebras (see e.g. \cite{Borchers:1983:symm,Borchers:1983,Pillet:2006}), and imposing additional properties on derivations that guarantee the existence and uniqueness of their extensions to generators of groups of $*$-automorphisms (see e.g. \cite{Bratteli:Robinson:1979,Bratteli:1986,Sakai:1991}). In all these cases the pair $(\C,\der)$ faces not only the problem of unique determination of a corresponding $*$-automorphism $\alpha$, but also the problem of unique specification of a generator of a unitary implementation of $\{\alpha_t\mid t\in\RR\}$ in terms of unitary operators acting on $\BH$ for a given representation $(\H,\pi)$. In order to solve both problems, additional restrictions must be imposed on derivations.

For example, a derivation $\der$ on a $C^*$-algebra (or, respectively, a $W^*$-algebra) $\C$ determines uniquely a generator of a strongly continuous (or, respectively, weakly-$\star$ continuous) group $\{\alpha_t\mid t\in\RR\}\subseteq\Aut(\C)$ if{}f it satisfies the following conditions \cite{Hille:1948,Yosida:1948,Bratteli:Robinson:1976:II}:\footnote{If  $\C$ is a $C^*$-algebra and strong continuity is considered, then 2) is not necessary.}
\begin{enumerate}
\item[1)] $\der$ is densely defined and closed in the norm (or, respectively, weak-$\star$) topology,
\item[2)] $\II\in\dom(\der)$,
\item[3)] $(\II+\lambda\der)\dom(\der)=\C\;\;\forall\lambda\in\RR\backslash\{0\}$,
\item[4)] $(\II+\lambda\der)x\geq0\limp x\geq0\;\;\forall\lambda\in\RR\;\forall x\in\dom(\der)$,
\end{enumerate}
or, equivalently,
\begin{enumerate}
\item[4')]$ \n{(\II+\lambda\der)(x)}\geq\n{x}\;\;\forall\lambda\in\RR\;\forall x\in\dom(\der)$.
\end{enumerate}
However, there are various other theorems of this type. In particular \cite{Powers:Sakai:1975,Bratteli:Robinson:1976:II}, a derivation on a $C^*$-algebra $\C$ determines uniquely a generator of a strongly continuous group $\{\alpha_t\mid t\in\RR\}\subseteq\Aut(\C)$ if{}f it satisfies the conditions
\begin{enumerate}
\item[1)] $(\II\pm\der)\dom(\der)$ are dense in $\C$,
\item[2)] $\forall x\in\dom(\der)^\sa$ $\exists\phi_x\in\Scal(\C)$ such that $\phi_x(x)=\n{x}$ and $\phi_x(\der(x))=0$.
\end{enumerate}

Construction of a unique generator of unitary implementation can be provided in different inequivalent ways (see e.g. \cite{Bratteli:Robinson:1976,Bratteli:Robinson:1979,Sakai:1991}). For example \cite{Bratteli:Robinson:1979}, if $\der$ is a derivation of $C^*$-algebra $\C$ that is a generator of a strongly continuous representation $\alpha:\RR\ni t\mapsto\alpha_t\in\Aut(\C)$, and $\omega\in\Scal(\C)$ satisfy $\omega(\der(x))=0\;\forall x\in\dom(\der)$, then
\begin{equation}
        \left\{
        \begin{array}{l}
                \pi_\omega(\der(x))\xi=\ii[H,\pi_\omega(x)]\xi\;\;\forall\xi\in\pi_\omega(\dom(\der))\Omega_\omega,\\
                \pi_\omega(\dom(\der))\Omega_\omega\subseteq\dom(H),
        \end{array}
        \right. 
\end{equation}
where $(\H_\omega,\pi_\omega,\Omega_\omega)$ is a GNS representation of $\C$ generated by $\omega$ and $H$ is a self-adjoint generator of a one-parameter group of unitary transformations $U(t)$ that uniquely corresponds to $\alpha_t$ by
\begin{equation}
        \left\{
        \begin{array}{l}
                \pi_\omega(\alpha_t(x))=U(t)\pi_\omega(x)U(t)^*,\\
                U(t)\Omega_\omega=\Omega_\omega,
        \end{array}
        \right.
\label{unique.invariant.unitary.rep.auto}
\end{equation}
which holds for all $x\in\C$ and all $t\in\RR$.
\subsection{Standard liouvilleans\label{standard.liouvilleans.section}}
Characterisation of the generator of unitary implementation with the help of functionals over an algebra opens the problem of justification of the choice of particular properties of functionals that are used for this characterisation. The setting of $W^*$-algebras admits a remarkable solution to this problem: every pair of a $W^*$-dynamical system $(\N,\RR,\alpha)$ and a standard representation $(\H,\pi,J,\stdcone)$ determines uniquely a corresponding unitary implementation \textit{together with} a unique self-adjoint generator of this family of unitaries. This generator is called a \textit{standard liouvillean}\footnote{It would be however more precise to call it \textit{quantum koopmanian}, because in the commutative setting (of statistical mechanics and probability measures) the `liouvillean operator' (defined by the Poisson bracket) acts on elements of $L_1(\X,\mho(\X),\tmu)$, while it is the `koopmanian operator' \cite{Koopman:1931,vonNeumann:1932:zur,vonNeumann:1932:zusaetze} that acts on the positive cone of $L_1(\X,\mho(\X),\tmu)$.}. It is not called `hamiltonian', because in general its spectrum may be not bounded from any side, while the notion of  `hamiltonian' is usually understood as referring to a self-adjoint operator that generates a strongly continuous group of unitary operators \textit{and} has a nonnegative (or at least bounded from below) spectrum\footnote{E.g., \cytat{one of the most important principles of quantum field theory, ensuring the stability, demands that the energy should have a lower bound} \cite{Haag:1992}.}. Moreover, as opposed to hamiltonian, the construction of standard liouvillean does not require any additional analytic conditions that constrain derivation to an `integrable' infinitesimal generator. This way the $W^*$-algebraic approach makes the notion of a hamiltonian less relevant than the notion of a liouvillean. 

For any $W^*$-algebra $\N$, the unique predualisation of action of $\alpha\in\Aut(\N)$ can be connected with the uniqueness property of representation of elements of $\N_\star^+$ in terms of a standard cone of a standard representation $(\H,\pi,J,\stdcone)$ of $\N$: any $\alpha\in\Aut(\N)$ defines a unique map $u:\stdcone\ra\stdcone$ by 
\begin{equation}
        u\xi_\pi(\phi):=\xi_\pi(\alpha_\star(\phi))\;\;\;\forall\phi\in\N_\star^+.
\end{equation} 
This map is linear, can be extended to a unitary operator on all $\H$, and satisfies 
\begin{equation}
        u\pi(x)u^*=\pi(\alpha(x))\;\;\forall x\in\N.
\end{equation}
This leads to a question, whether it is possible to generate this way a \textit{standard} unitary implementation of a given representation $\alpha:G\ra\Aut(\N)$. The answer is in the affirmative, and was established by Haagerup \cite{Haagerup:1975:standard:form} (the special cases of this result were obtained earlier in \cite{Kadison:1965,Kallman:1971,Henle:1970,Halpern:1972,Pedersen:Takesaki:1973}). If $(\H,\pi,J,\stdcone)$ is a standard representation of a $W^*$-algebra $\N$, then there exists a unique  strongly continuous unitary implementation $\rpktarget{V.ALPHA}V_\alpha(g)$ of $\alpha$ satisfying
\begin{align}
        V_\alpha(g)\stdcone&=\stdcone,\\
        JV_\alpha(g)&=V_\alpha(g)J.
\end{align}
Such family $\{V_\alpha(g)\mid g\in G\}$ is called a \df{standard} unitary implementation of $\alpha$.

Thus, if $(\N,\RR,\alpha)$ is a $W^*$-dynamical system with $\N$ in standard form $(\H,\pi(\N),J,\stdcone)$, then from the theorems of Haagerup and Stone \cite{Stone:1930,Stone:1932,vonNeumann:1932:Stone} it follows that there exists a unique strongly continuous group of unitaries $\{V_\alpha(t)\mid t\in\RR\}\subseteq\BH^\uni$, and a unique self-adjoint operator $\rpktarget{K.ALPHA}K^\alpha$ on $\H$, called \df{standard liouvillean}, such that $V_\alpha(t)$ is a strongly continuous unitary implementation of $\alpha$ and for every $t\in\RR$
\begin{enumerate}
\item[i)] $V_\alpha(t)=\ee^{\ii tK^\alpha}$,
\item[ii)] $\ee^{\ii tK^\alpha}\stdcone=\stdcone$,
\item[iii)] $JK^\alpha+K^\alpha J=0$.
\end{enumerate}
Moreover, $V_\alpha(t)$ satisfy also
\begin{equation}
        \ee^{\ii tK^\alpha}\N^\comm\ee^{-\ii tK^\alpha}=\N^\comm\;\;\forall t\in\RR.
\end{equation}
The definition of a standard liouvillean $K^\alpha$ does not depend on any choice of $\omega\in\N^{\banach+}$ or $\omega\in\W(\N)$: it depends only on a $W^*$-dynamical system and a standard representation of $W^*$-algebra. For every $W^*$-dynamical system $(\N,\RR,\alpha)$ we define a \df{canonical liouvillean} as a standard liouvillean associated with a canonical representation $(L_2(\N),\pi_\N,J_\N,L_2(\N)^+)$.

As an example \cite{Bratteli:Robinson:1979,JOPP:2012}, consider a $W^*$-algebra $\N$ and its standard representation $(\H,\pi,J,\stdcone)$, equipped with a group of unitary operators $U(t):=\ee^{\ii tH}\in\pi(\N)$, with $t\in\RR$ and $H\in\pi(\N)^\sa$, which is a unitary implementation of $\RR\ni t\mapsto\alpha_t\in\Aut(\N)$,
\begin{equation}
        \pi(\alpha_t(x))=\ee^{\ii tH}\pi(x)\ee^{-\ii tH}\;\;\;\forall x\in\N\;\forall t\in\RR.
\end{equation}
If $\omega\in\N_\star^+$ and $\xi_\pi(\omega)$ is its standard vector representative, then
\begin{equation}
        \omega(\alpha_t(x))=\s{\ee^{-\ii tH}\xi_\pi(\omega),\pi(x)\ee^{-\ii tH}\xi_\pi(\omega)}\;\;\;\forall x\in\N\;\forall t\in\RR.
\end{equation}
However, generally,
\begin{equation}
        \exists(\omega,t)\in\N_\star^+\times\RR\;\;\;\left(\xi_\pi(\omega)\right)(t):=\ee^{-\ii tH}\xi_\pi(\omega)\not\in\stdcone.
\end{equation}
On the other hand, the group $\{V_\alpha(t)\mid t\in\RR\}$ of unitary operators uniquely determined by the condition $V_\alpha(t)\stdcone=\stdcone$, also implements $\alpha_t$ in $\pi$. As a result
\begin{equation}
        V_\alpha(t)\pi(x)V_\alpha(t)^*=\pi(\alpha_t(x))=U(t)xU(t)^*\;\;\forall x\in\N\;\forall t\in\RR,
\end{equation}
but $V_\alpha(t)\neq U(t)$:
\begin{equation}
        V_\alpha(t)=U(t)JU(t)J.
\label{VUJUJ.eq}
\end{equation}
If $\N$ is semi-finite, \eqref{VUJUJ.eq} implies that the standard liouvillean $K^\alpha$ of $V_\alpha(t)$ is related to the self-adjoint generator $H$ of U(t) by
\begin{equation}
        K^\alpha=H-JHJ=[H,\,\cdot\;].
\label{hamiltonian.liouvillean.J.eqn}
\end{equation}
Hence, $K^\alpha\not\in\pi(\N)^\sa$, $V_\alpha(t)\not\in\pi(\N)$. If $\N$ is finite, then, given $H\in\N^\sa$, the unitary implementation $U(t)$ is
\begin{equation}
        \LLL(\alpha_t(x))=\LLL(\ee^{\ii tH}x\ee^{-\ii tH})=\LLL(\ee^{\ii tH})\LLL(x)\LLL(\ee^{-\ii tH})=\ee^{\ii t\LLL(H)}\LLL(x)\ee^{-\ii t\LLL(H)}=\ee^{\ii tH}x\ee^{-\ii tH},
\end{equation}
while the unitary implementation $V_\alpha(t)$ reads \cite{Haag:Hugenholtz:Winnink:1967}
\begin{align}
        \LLL(\alpha_t(x))&=\LLL(\ee^{\ii tH})\RRR(\ee^{\ii tH})\LLL(x)\RRR(\ee^{-\ii tH})\LLL(\ee^{-\ii tH})=\ee^{\ii t(\LLL(H)-\RRR(H))}\LLL(x)\ee^{-\ii t(\LLL(H)-\RRR(H))}\nonumber\\&=\ee^{\ii t[H,\,\cdot\,]}x\ee^{\ii t[H,\,\cdot\,]}=\ee^{\ii tK^\alpha}x\ee^{-\ii tK^\alpha},
\end{align}
where the second equality follows from $\ee^{t(x+JxJ)}=\ee^{tx}J\ee^{tx}J$. The spectrum of $K^\alpha$ is given by
\begin{equation}
        \sp(K^\alpha)=\{\lambda_1-\lambda_2\mid\lambda_1,\lambda_2\in\sp(H)\},
\end{equation}
while the standard vector representative of the evolved state $\omega(t)$ is given by
\begin{equation}
        \xi_{\omega(t)}=(\rho_\omega(t))^{1/2}=\ee^{-\ii tH}\rho_\omega^{1/2}\ee^{\ii tH}=\LLL(\ee^{-\ii tH})\RRR(\ee^{-\ii tH})\xi_\omega\neq\xi_\omega(t).
\end{equation}
If $\N$ is a type I $W^*$-algebra, then the standard liouvillean takes the form $K^\alpha=H\otimes\II+\II\otimes H$ and acts on $\schatten_2(\H)\iso L_2(\BH,\tr)\iso\H\otimes\H^\banach$.

Another example of a standard liouvillean is provided by the modular theory. If $\omega\in\C^{\banach+}_0$ for a $C^*$-algebra $\C$, and $(\H_\omega,\pi_\omega,\Omega_\omega)$ is a GNS representation, then $\Omega_\omega$ is cyclic and separating for $\pi_\omega(\C)^\comm{}^\comm$. From the Tomita theorem it follows that the modular automorphism $\sigma^\omega$ of $\pi_\omega(\C)^\comm{}^\comm$, 
\begin{equation}
        \sigma_t^\omega(x):=\Delta^{\ii t}_\omega x\Delta^{-\ii t}_\omega\;\;\;\forall x\in\pi_\omega(\C)^\comm{}^\comm\;\forall t\in\RR,
\end{equation}
forms a $W^*$-dynamical system $(\pi_\omega(\C)^\comm{}^\comm,\RR,\sigma^\omega)$, and the modular hamiltonian is equal to a standard liouvillean of $(\H_\omega,\pi_\omega(\N)^\comm{}^\comm,J_\omega,\stdcone_{\Omega_\omega})$, 
\begin{equation}
        K^\alpha=K_\omega=-\log\Delta_\omega.
\end{equation}
The same holds also for $(\pi_\omega(\N),\RR,\sigma^\omega)$ and $(\H_\omega,\pi_\omega(\N),J_\omega,\stdcone_\omega)$ if $\N$ is a $W^*$-algebra and $\omega\in\N^+_{\star0}$.
\subsection{Crossed products\label{crossed.product.section}}
A \df{covariant representation} of a $C^*$-dynamical system $(\C,G,\alpha)$ is defined as a triple $(\H,\pi,U)$, where $\H$ is a Hilbert space, $\pi:\C\ra\BH$ is a nondegenerate representation of $\C$ on $\H$ and $U:G\ni g\mapsto U(g)\in\BH$ is a strongly continuous unitary representation of $G$ on $\H$ that is a unitary implementation of $\alpha$. A \df{covariant representation} of a $W^*$-dynamical system $(\N,G,\alpha)$ is defined as a triple $(\H,\pi,U)$ satisfying the above conditions with respect to $(\N,G,\alpha)$, under an additional requirement that the representation $(\H,\pi)$ is normal. The notions of $*$-dynamical system, covariant representation and crossed product are closely related: each $C^*$-dynamical system defines a unique corresponding $C^*$-crossed product algebra, each $W^*$-dynamical system defines a unique corresponding crossed product algebra, and there is a one-to-one correspondence between nondegenerate representations of crossed products and covariant representations of the dynamical systems in both $C^*$- and $W^*$- cases.

Let $(\C,G,\alpha)$ be a $C^*$-dynamical system, let $\tmu^G_L$ be a left Haar measure\footnote{A \df{right Haar measure} $\rpktarget{HAAR}\tmu_R^G$ (respectively, a \df{left Haar measure} $\tmu_L^G$) is defined as a nonzero Radon measure on a locally compact topological group $G$ such that $\int_G \tmu_R^G(g)f(g)=\int_G \tmu_R^G(g)f(bg)$ (respectively, $\int_G \tmu^G_L(g)f(g)=\int_G \tmu^G_L(g)f(gb)$) $\forall f\in \mathrm{C}_\mathrm{c}(G;\CC)\;\forall g,b\in G$ \cite{Haar:1933}, where $\rpktarget{Cc}\mathrm{C}_\mathrm{c}(G;\CC)$ denotes the set of continuous functions $f:G\ra\CC$ with compact supports (for a definition of Radon measure, see Section \ref{comm.integr.section}). According to the Haar--von Neumann--Weil theorem \cite{Haar:1933,vonNeumann:1935:Haarschen,vonNeumann:1936,Weil:1940}, for every locally compact topological group there exists a unique left and a unique right Haar measure, up to a nonzero multiplication constant. A \df{modular function} $M:G\ra\RR^+\setminus\{0\}$ is defined by
\begin{equation}
\int_G\tmu^G_L(h)f(hg)=M(g)\int_G\tmu^G_L(h)f(h)\;\forall f\in L_1(G,\tmu^G_L).
\end{equation}
If left and right Haar measures on $G$ coincide, then $M(g)=1$. The \df{Hurwitz measure} \cite{Hurwitz:1897} can be characterised as the Haar measure with $G$ given by the Lie group. The \df{Lebesgue measure} \cite{Lebesgue:1902,Lebesgue:1910} can be characterised as such Hurwitz measure on $\RR^n$ which is invariant with respect to a topological group of differentiable maps $f:\RR^n\ra\RR^n$ with jacobians $\det(\partial f/\partial x)=\pm 1$. The left and right invariant Lebesgue integrals are equal. The inequalities $\int_G\tmu_R^G(g)<\infty$ and $\int_G\tmu_L^G(g)<\infty$ hold if{}f $G$ is compact.} on a locally compact topological group $G$, and let $M:G\ra\,]0,\infty[$ be the modular function of $G$.
The complex vector space $\mathrm{C}_\mathrm{c}(G;\C)$ of norm continuous functions $f:G\ra\C$ with compact support can be equipped with the structure of normed $*$-algebra:
\begin{align}
                (xy)(g)&:=\int_G \tmu^G_L(b) x(b)\alpha_b(y(b^{-1}g)),
                \label{group.convolution.product}\\
                x^*(g)&:=M(g)^{-1}\alpha_g(x(g^{-1})^*),\\
                \n{x}_1&:=\int_G \tmu^G_L(b)\n{x(b)},
\end{align}
for all $x,y\in \mathrm{C}_\mathrm{c}(G;\C)$ and for all $g,b\in G$, where the integral in \eqref{group.convolution.product} is defined as the Bochner integral \cite{Bochner:1933} over the compact set $g(\supp(y))^{-1}\cap\supp(x)$. The completion of this space in the topology of the norm $\n{\cdot}_1$ is a Banach $*$-algebra, denoted by $L_1(G;\C)$ \cite{Doplicher:Kastler:Robinson:1966}. This space can be equipped also with an alternative norm
\begin{equation}
        \n{x}:=\sup\n{\pi(x)}\;\;\forall x\in \mathrm{C}_\mathrm{c}(G;\C),
\label{cscp.norm}
\end{equation}
where $\pi$ varies over all representations of $L_1(G;\C)$ on a given Hilbert space $\H$. The \df{$C^*$-crossed product} of $\C$ and $G$ with respect to $\alpha$ is defined as the completion of $\mathrm{C}_\mathrm{c}(G;\C)$ in the norm \eqref{cscp.norm}, and is denoted by $\C\rtimes_\alpha G$. If $(\H,\pi,U)$ is a covariant representation of a $C^*$-dynamical system $(\C,G,\alpha)$, then the corresponding representation $(\H,\tilde{\pi})$ of $C^*$-crossed product $\C\rtimes_\alpha G$ is given uniquely by \cite{Doplicher:Kastler:Robinson:1966}:
\begin{equation}
        \tilde{\pi}(x)=\int_G \tmu^G_L(g)\pi(x(g))U(g)\;\;\forall x\in \mathrm{C}_\mathrm{c}(G;\C),
\end{equation}
and the map $(\H,\pi,U)\ra(\H,\tilde{\pi})$ is a bijection onto the set of nondegenerate representations of $L_1(G;\C)$. For further information on $C^*$-dynamical systems and $C^*$-crossed products, see \cite{Pedersen:1979,Soltan:2007}.

Let $(\N,G,\alpha)$ be a $W^*$-dynamical system, let $\tmu^G_L$ be a left Haar measure on a locally compact topological group $G$, and let $\H$ be a Hilbert space. The complex vector space $\mathrm{C}_\mathrm{c}(G;\H)$ of continuous functions $f:G\ra\H$ with compact support can be equipped with the inner product
\begin{equation}
                \mathrm{C}_\mathrm{c}(G;\H)\times \mathrm{C}_\mathrm{c}(G;\H)\ni(\xi_1,\xi_2)             \mapsto\s{\xi_1,\xi_2}_{\tmu^G_L}:=\int_G \tmu^G_L(g)\s{\xi_1(g),\xi_2(g)}_\H\in\CC.
\end{equation}
The completion of $\mathrm{C}_\mathrm{c}(G;\H)$ in the topology of the norm defined by this inner product is a Hilbert space $L_2(G,\tmu^G_L;\H)$, which can be identified with $\H\otimes L_2(G,\tmu_L^G)$ by means of the unitary isomorphism $u:\H\otimes L_2(G,\tmu_L^G)\ra L_2(G,\tmu^G_L;\H)$ satisfying
\begin{equation}
        (u(\xi_0\otimes f))(g)=f(g)\xi_0\;\;\;\;\forall g\in G\;\forall \xi_0\in\H\;\forall f\in L_2(G,\tmu_L^G).
\end{equation}
The operators\rpktarget{PI.ALPHA}\rpktarget{UGIE}
\begin{align}
                (\pi_\alpha(x)\xi)(b)&:=\alpha_b^{-1}(x)\xi(b)
                \;\;\forall x\in\N\;\forall b\in G\;\forall\xi\in L_2(G,\tmu^G_L;\H),
                \label{pi.covariant.rep}\\
                (u_G(g)\xi)(b)&:=\xi(g^{-1}b)
                \;\;\forall g,b\in G\;\forall\xi\in L_2(G,\tmu^G_L;\H),
                \label{u.covariant.rep}
\end{align}
define a faithful normal representation $\pi_\alpha:\N\ra\BBB(L_2(G,\tmu^G_L;\H))$ and a strongly continuous unitary representation $u_G$ of $G$ in $L_2(G,\tmu^G_L;\H)$ that satisfy the \df{covariance equation}
\begin{equation}
        u_G(g)\pi_\alpha(x)u_G(g)^*=\pi_\alpha(\alpha_g(x))\;\;\forall x\in\N\;\forall g\in G.
\label{crossed.product.covariance.eq}
\end{equation}
Hence, $(L_2(G,\tmu^G_L;\H),\pi_\alpha,u_G)$ is a covariant representation of a $W^*$-dynamical system $(\N,G,\alpha)$. The \df{crossed product} of $\N$ by $G$ with respect to $\alpha$, denoted by $\rpktarget{CROSSEDPROD}\N\rtimes_\alpha G$, is defined as a von Neumann algebra that acts on $L_2(G,\tmu^G_L;\H)$ and is generated by $\pi_\alpha(\N)$ and $u_G(G)$ \cite{Takesaki:1973:duality}. The algebra $\N\rtimes_\alpha G$ can be equivalently defined as a von Neumann subalgebra of $\N\otimes\BBB(L_2(G,\tmu_L^G))$ generated by $\pi_\alpha(\N)$ and $\{\II\otimes u_g\mid g\in G\}$, where $u_g\in\BBB(L_2(G,\tmu_L^G))$ is a (left regular) representation of $G$ in $L_2(G,\tmu_L^G)$ given by 
\begin{equation}
        (u_g\xi)(b):=\xi(g^{-1}b)\;\;\forall\xi\in L_2(G,\tmu_L^G)\;\;\forall g,b\in G.
\end{equation}

If $\N\rtimes_\alpha G$ is a crossed product, then the corresponding covariant representation of a $W^*$-dynamical system $(\N,G,\alpha)$ is explicitly given in the definition of the crossed product $\N\rtimes_\alpha G$. The corresponding nondegenerate representation $(\H,\tilde{\pi})$ is provided in terms of nondegenerate representation $(\H,\pi_\alpha,u_G)$:
\begin{equation}
        \tilde{\pi}(x)=\int_G \tmu^G_L(g)\pi_\alpha(x(g))u_G(g)\;\;\;\forall x\in \mathrm{C}_\mathrm{c}(G;\H).
\end{equation}
For further information on $W^*$-dynamical systems and crossed products, see \cite{Guichardet:1974,Digernes:1975,vanDaele:1978}.

If a locally compact group $G$ is abelian, then its \df{Pontryagin dual group} $\rpktarget{PONTR.DUAL}\hat{G}$ \cite{Pontryagin:1934,vanKampen:1935} is defined as a set of all continuous group homomorphisms $G\ra\TT:=\{z\in\CC\mid\ab{z}=1\}\iso\RR/\ZZ$, equipped with: a unit element $\hat{e}$ given by a function identically equal to $1$, a group operation $\circ$ given by positive multiplication, an inverse operation $(\cdot)^{-1}$ given by the complex conjugate $(\cdot)^*$, and a topology given by a topology of uniform convergence on compact subsets of $G$. The group $\hat{G}$ is also locally compact and abelian. For $\hat{b}\in\hat{G}$, $g\in G$, the \df{Pontryagin duality map} is given by
\begin{equation}
        \duality{\cdot,\cdot}_{G\times\hat{G}}:G\times\hat{G}\ni(g,\hat{b})\mapsto\hat{b}(g)\in\TT.
\end{equation}
The strongly continuous representation $\hat{u}_{\hat{G}}$ of $\hat{G}$ in $L_2(G,\tmu^G_L;\H)$,
\begin{equation}
        \hat{u}_{\hat{G}}(\hat{b})\xi(g):=\left(\duality{g,\hat{b}}_{G\times\hat{G}}\right)^*\xi(g)\;\;\forall\xi\in L_2(G,\tmu^G_L;\H)\;\forall g\in G\;\forall\hat{b}\in\hat{G},
\end{equation}
satisfies
\begin{align}
        \hat{u}_{\hat{G}}(\hat{b})\pi_\alpha(x)\hat{u}_{\hat{G}}(\hat{b}^{-1})&=\pi_\alpha(x)\;\;\forall x\in\N\;\forall\hat{b}\in\hat{G},\\
        \hat{u}_{\hat{G}}(\hat{b})\hat{u}_{\hat{G}}(g)\hat{u}_{\hat{G}}(\hat{b}^{-1})&=\left(\duality{g,\hat{b}}_{G\times\hat{G}}\right)^*u_G(g)\;\;\forall\hat{b}\in\hat{G}\;\forall g\in G,
\end{align}
so
\begin{equation}
        \hat{u}_{\hat{G}}(\hat{b})\left(\N\rtimes_\alpha G\right)\hat{u}_{\hat{G}}(\hat{b}^{-1})=\N\rtimes_\alpha G\;\;\forall\hat{b}\in\hat{G}.
\end{equation}
This makes it possible to define the \df{dual action} $\hat{\alpha}$ of $\hat{G}$ on $\N\rtimes_\alpha G$ as a continuous map \cite{Takesaki:1973:duality}
\begin{equation}
        \hat{\alpha}:
        \hat{G}\times(\N\rtimes_\alpha G)
        \ni
        (\hat{b},\hat{x})
        \mapsto
        \hat{\alpha}_{\hat{b}}(\hat{x}):=
        \hat{u}_{\hat{G}}(\hat{b})\hat{x}\hat{u}_{\hat{G}}(\hat{b})^*=
        (\II\otimes{\hat{u}}_{\hat{b}})
        \hat{x}
        (\II\otimes{\hat{u}}_{\hat{b}})^*
        \in\N\rtimes_\alpha G,
\end{equation}
where the unitary operator ${\hat{u}}_{\hat{b}}\in\BBB(L_2(G,\tmu^G_L))$, given by
\begin{equation}
        (\hat{u}_{\hat{b}})f(g):=\left(\duality{g,\hat{b}}_{G\times\hat{G}}\right)^*f(g)\;\;\forall f\in \mathrm{C}_\mathrm{c}(G;\CC),
\end{equation}
defines a representation of $\hat{G}$ in $L_2(G,\tmu^G_L)$ that is continuous in the topology of uniform convergence on compact sets in $\hat{G}$ \cite{Haagerup:1978:dualweights:I}. The \df{fixed point} subalgebra of $\N\rtimes_\alpha G$ under the action $\hat{\alpha}$ satisfies \cite{Landstad:1979}
\begin{equation}
        (\N\rtimes_\alpha G)_{\hat{\alpha}}:=\{\hat{x}\in\N\rtimes_\alpha G\mid\hat{\alpha}_{\hat{b}}(\hat{x})=\hat{x}\;\forall\hat{b}\in\hat{G}\}=\pi_\alpha(\N).
\label{fixed.point.algebra.crossprod}
\end{equation}
The triple $(\N\rtimes_\alpha G,\hat{G},\hat{\alpha})$ is a $W^*$-dynamical system. Its covariant representation takes a form
\begin{align}
        \hat{\alpha}_{\hat{b}}(\pi_\alpha(x))&=\pi_\alpha(x)\;\;\forall x\in\N\;\forall\hat{b}\in\hat{G},\\
        \hat{\alpha}_{\hat{b}}(u_G(g))&=\left(\duality{g,\hat{b}}_{G\times\hat{G}}\right)^*u_G(g)\;\;\forall g\in G\;\forall\hat{b}\in\hat{G},
\end{align}
where $\pi_\alpha$ and $u_G$ are given by \eqref{pi.covariant.rep} and \eqref{u.covariant.rep}, respectively. As a result, one obtains the crossed product
\begin{equation}
        (\N\rtimes_\alpha G)\rtimes_{\hat{\alpha}}\hat{G}.
\end{equation}
According to Takesaki's theorem \cite{Takesaki:1973:duality}, if the multiplicative constant of the left Haar measure $\tmu_L^{\hat{G}}$ on $\hat{G}$ is chosen such that the \df{Plancherel formula} \cite{Plancherel:1910}
\begin{equation}
        \int_G\tmu_L^G(g)\ab{f(g)}^2=\int_{\hat{G}}\tmu_L^{\hat{G}}(\hat{b})\ab{\int_G\tmu_L^G(g)\left(\duality{g,\hat{b}}_{G\times\hat{G}}\right)^*f(g)}^2\;\;\forall f\in L_1(G,\tmu_L^G)\cap L_2(G,\tmu_L^G)
\end{equation}
holds, then there exists a unique $*$-isomorphism
\begin{equation}
        U:(\N\rtimes_\alpha G)\rtimes_{\hat{\alpha}}\hat{G}\iso\N\otimes\BBB(L_2(G,\tmu_L^G))
\label{unique.iso.of.dual.crossprod}
\end{equation}
such that, for all $x\in\N$, $\xi\in L_2(G,\tmu^G_L;\H)$, $g,g_1,g_2\in G$, $\hat{b}\in\hat{G}$,
\begin{align}
        (U(\pi_{\hat{\alpha}}\circ\pi_\alpha(x))\xi)(g)&=
        \alpha_g^{-1}(x)\xi(g),\\
        (U(\pi_{\hat{\alpha}}\circ u_G(g_2))\xi)(g_1)&=\xi(g_1g_2^{-1}),\\
        (U(u_{\hat{G}}(\hat{b}))\xi)(g)&=\left(\duality{g,\hat{b}}_{G\times\hat{G}}\right)^*\xi(g).
\end{align}
The extension of the above duality theory to $C^*$-crossed products was provided by Takai \cite{Takai:1974,Takai:1975}, while its further extension to nonabelian groups was provided by Nakagami \cite{Nakagami:1975,Nakagami:1976} (see also \cite{Olesen:Pedersen:1978,Landstad:1979,Nakagami:Takesaki:1979}).

For a key example of the above constructions, consider a crossed product $\N\rtimes_{\sigma^\psi}\RR$, defined as the von Neumann algebra acting on the Hilbert space $L_2(\RR,\dd t;\H)\iso\H\otimes L_2(\RR,\dd t)$ and generated by the operators $\pi_{\sigma^\psi}(x)$ and $u_\RR(t)$, which are defined by\rpktarget{PI.SIGMA}\rpktarget{URR}
\begin{align}   
(\pi_{\sigma^\psi}(x)\xi)(t)&:=\sigma_{-t}^\psi(x)\xi(t),
\label{pi.sigma}\\
        (u_\RR(t_2)\xi)(t_1)&:=\xi(t_1-t_2),
\label{lambda.tau}
\end{align}
for all $x\in\N$, $t,t_1,t_2\in\RR$, $\xi\in L_2(\RR,\dd t;\H)$. These two operators satisfy the covariance equation
\begin{equation}
        u_\RR(t)\pi_{\sigma^\psi}(x)u_\RR^*(t)=\pi_{\sigma^\psi}(\sigma_t^\psi(x)).
\label{covariance.eqn}
\end{equation}
The equation \eqref{pi.sigma} can be written as
\begin{equation}
        (\pi_{\sigma^\psi}(x)\xi)(t)=\Delta_\psi^{\ii t}x\Delta_\psi^{-\ii t}\xi(t)=\ee^{-\ii K_\psi t}x\ee^{\ii K_\psi t}\xi(t),
\end{equation}
where $K_\psi$ is a modular hamiltonian of the modular operator $\Delta_\psi$. So, the covariance equation \eqref{covariance.eqn} translates between family of unitaries that partially generate the crossed product algebra $\N\rtimes_{\sigma^\psi}\RR$ and the modular automorphism of the underlying von Neumann algebra $\N$:
\begin{equation}
        u_\RR(t)\pi_{\sigma^\psi}(x)u_\RR(t)^*=
        \pi_{\sigma^\psi}(\ee^{-\ii t K_\psi}x\ee^{\ii t K_\psi}).
\label{modular.histories.covariance}
\end{equation}
Given any $\phi,\psi\in\W_0(\N)$, the corresponding crossed products are unitarily isomorphic,
\begin{equation}
        U_{\phi,\psi}\left(\N\rtimes_{\sigma^\phi}\RR\right)U_{\phi,\psi}^*=\N\rtimes_{\sigma^\psi}\RR,
\label{precore.isomorphisms}
\end{equation}
where
\begin{align}
        (U_{\phi,\psi}\xi)(s)&:=\Connes{\psi}{\phi}{-s}\xi(s)\;\;\forall\xi\in \mathrm{C}_\mathrm{c}(\RR;\H)\;\forall s\in\RR,\\
        (U_{\phi,\psi}^*\xi)(s)&:=\Connes{\psi}{\phi}{-s}^*\xi(s)\;\;\forall\xi\in \mathrm{C}_\mathrm{c}(\RR;\H)\;\forall s\in\RR.
\end{align}
The crossed product algebra can be also denoted as
\begin{equation}
        \N\rtimes_{\sigma^\psi}\RR=\{x\otimes\II,\Delta^{\ii t}_\psi\otimes u_\RR(t)\mid x\in\N,t\in\RR\}^\comm{}^\comm\subseteq\BBB(L_2(\N)\otimes L_2(\RR,\dd\lambda)).
\end{equation}
Introducing the notation
\begin{align}
        \tilde{x}&:=x\otimes\II,\\
        \psi^{\ii t}&:=\Delta^{\ii t}_\psi\otimes u_\RR(t),\\
        \rpktarget{PHI.IT}\phi^{\ii t}&:=\Delta^{\ii t}_{\phi,\psi}\otimes u_\RR(t)=(\Connes{\phi}{\psi}{t}\otimes\II)(\Delta^{\ii t}_\psi\otimes u_\RR(t)),
\end{align}
we obtain
\begin{equation}
\psi^{\ii t}\tilde{x}\psi^{-\ii t}=\sigma^\psi_t(\tilde{x}),\;\;\phi^{\ii t}\tilde{x}\phi^{-\ii t}=\sigma^\phi_t(\tilde{x})\;\;\forall x\in\N.
\label{sv1.eq}
\end{equation}

The dual group $\hat{\RR}$ of $\RR$ is equal to $\RR$, and the canonical Pontryagin duality in this case has the form $\duality{t,s}_{\RR\times\RR}=\ee^{\ii ts}$. The dual action $\hat{\sigma}^\psi:\RR\ni s\mapsto\hat{\sigma}^\psi_s\in\Aut(\precore)$ of $\RR$ on $\rpktarget{PRECORE}\precore:=\N\rtimes_{\sigma^\psi}\RR$ is characterised by
\begin{align}
        \hat{\sigma}^\psi_s(\pi_{\sigma^\psi}(x))&=\pi_{\sigma^\psi}(x)\;\;\forall x\in\N\;\forall s\in\RR,\\
        \hat{\sigma}^\psi_s(u_\RR(t))&=\ee^{-\ii st}u_\RR(t)\;\;\forall t,s\in\RR.\label{quasiinv.under.sigma.hat}
\end{align}
The dual unitary representation $\hat{u}_\RR$ of $\RR$ in $L_2(\RR,\dd\lambda;\H)$ is given by
\begin{equation}
        \hat{u}_\RR(s)\xi(t)=\ee^{-\ii st}\xi(t),
\end{equation}
while the unitary $\hat{u}_t\in\BBB(L_2(\RR,\dd\lambda))$ reads $(\hat{u}_tf)(s)=\ee^{-\ii st}f(s)\;\;\forall f\in \mathrm{C}_\mathrm{c}(\RR;\CC)$, so
\begin{equation}
\hat{\sigma}^\psi_s(\hat{x})=\hat{u}_\RR(s)\hat{x}\hat{u}_\RR(s)^*=(\II\otimes\hat{u}_s)\hat{x}(\II\otimes\hat{u}_s)\;\;\forall\hat{x}\in\precore.
\end{equation}
As a result, $(\precore,\RR,\hat{\alpha})$ is a $W^*$-dynamical system, and \eqref{unique.iso.of.dual.crossprod} in this case becomes
\begin{equation}
        (\N\rtimes_{\sigma^\psi}\RR)\rtimes_{\hat{\sigma}^\psi}\RR\iso\N\otimes\BBB(L_2(\RR,\dd\lambda)).
\label{isomorphism.of.double.dual.psi}
\end{equation}
The von Neumann algebra $\N$ is a subalgebra of $\precore$, with an embedding $\N\ra\precore$ characterised in terms of \eqref{fixed.point.algebra.crossprod} as a fixed point subalgebra of $\precore$ that is a centraliser with respect to $\hat{\sigma}^\psi$,
\begin{equation}
\N\iso\precore_{\hat{\sigma}^\psi}=\{x\in\N\rtimes_{\sigma^\psi}\RR\mid\hat{\sigma}^\psi_s(x)=x\;\forall s\in\RR\}.
\label{fixed.point.algebra.crossprod.psi}
\end{equation}
If $\N\rtimes_{\sigma^\psi}\RR$ is of type I, then $\N$ is semi-finite \cite{vanDaele:1978}. On the other hand, if $\N$ of type III, then $\N\rtimes_{\sigma^\psi}\RR$ is of type II$_\infty$ \cite{Takesaki:1973:duality}. More generally, for any von Neumann algebra $\N$ the algebra $\N\rtimes_{\sigma^\psi}\RR$ is semi-finite \cite{Takesaki:1973:duality}. From \eqref{fixed.point.algebra.crossprod.psi} it follows that every von Neumann algebra $\N$ can be represented as a crossed product $\N\iso\precore\rtimes_{\hat{\sigma}^\psi}\RR$ with a semi-finite von Neumann algebra $\precore$ \cite{Landstad:1979}. Moreover, if $\N$ is a type III factor, then
\begin{equation}
        \N\iso(\N\rtimes_{\sigma^\psi}\RR)\rtimes_{\hat{\sigma}^\psi}\RR\;\;\forall\psi\in\W_0(\N).
\end{equation}
The triple $(\zentr_\precore,\RR,\hat{\sigma}^\psi|_{\zentr_{\precore}})$ (or, equivalently, $\left(\zentr_\precore,\RR^+\setminus\{0\},\left(t\mapsto\hat{\sigma}^\psi_{-\log(t)}\right)|_{\zentr_\precore}\right)$) is a $W^*$-dy\-na\-mi\-cal system, called \df{flow of weights}. As shown by Connes and Takesaki \cite{Connes:Takesaki:1974,Connes:Takesaki:1977}, the flow of weights is independent of the choice of $\psi$, in the sense that it is uniquely determined by the underlying von Neumann algebra, up to a $*$-isomorphism 
\begin{equation}
        \varsigma:\zentr_{\N\rtimes_{\sigma^{\psi_1}}\RR}\ra
                        \zentr_{\N\rtimes_{\sigma^{\psi_2}}\RR}
\end{equation}
that satisfies
\begin{equation}
        \varsigma\circ\hat{\sigma}_t^{\psi_1}|_{\zentr_{\N\rtimes_{\sigma^{\psi_1}}\RR}}=\hat{\sigma}_t^{\psi_2}|_{\zentr_{\N\rtimes_{\sigma^{\psi_2}}\RR}}\circ\varsigma\;\;\forall t\in\RR.
\end{equation}
This defines a functor $\CTflow$ from the category $\VNfIIIIso$ of type III factor von Neumann algebras with $*$-isomorphisms to the category $\WstarCovR$ of $W^*$-dynamical systems $(\N,\RR,\alpha)$ with arrows $(\N_1,\RR,\alpha^1)\ra(\N_2,\RR,\alpha^2)$ given by such $*$-isomorphisms $\varsigma:\N_1\ra\N_2$ that satisfy 
\begin{equation}
        \varsigma\circ\alpha_t^1=\alpha_t^2\circ\varsigma\;\;\forall t\in\RR.
\end{equation}
The codomain category of $\CTflow$ can be further restricted to a subcategory of \textit{properly ergodic flows}, see \cite{Connes:Takesaki:1974,Connes:Takesaki:1977,Takesaki:1978,Connes:1994,Takesaki:2003} for details. Connes and Takesaki used flow of weights to provide a refined classification of type III factors with separable preduals, which is equivalent with the classification in terms of Connes' modular spectrum $\modspec(\N)$:
\begin{equation}\rpktarget{MODSPEC.FLOW}
        \modspec(\N)=\exp\left(\ker\left(\hat{\sigma}^\psi|_{\zentr_\precore}\right)\right)\cup\{0\}.
\end{equation}
\subsection{Canonical core algebra\label{FT.core.algebra.section}}
Consider a $W^*$-algebra $\N$ and a relation $\sim_t$ on $\N\times\W_0(\N)$ defined by \cite{Falcone:Takesaki:2001}
\begin{equation}
        (x,\psi)\sim_t(y,\phi)\iff y=x\Connes{\psi}{\phi}{t}\;\;\forall x,y\in\N\;\forall\psi,\phi\in\W_0(\N).
\label{core.equivalence}
\end{equation}
The property \eqref{Connes.cocycle} of Connes' cocycle implies that $\sim_t$ is an equivalence relation in $\N\times\W_0(\N)$. The equivalence class $(\N\times\W_0(\N))/\sim_t$ is denoted by $\rpktarget{NT}\N(t)$, and its elements are denoted by $\rpktarget{PSI.IT}x\psi^{\ii t}$. The operations
\begin{align}
        x\psi^{\ii t}+y\psi^{\ii t}&:=(x+y)\psi^{\ii t},\\
        \lambda(x\psi^{\ii t})&:=(\lambda x)\psi^{\ii t}\;\forall\lambda\in\CC,\\
        \n{x\psi^{\ii t}}&:=\n{x},
\end{align}
equip $\N(t)$ with the structure of the Banach space, which is isometrically isomorphic to $\N$, considered as a Banach space. By definition, $\N(0)$ a $W^*$-algebra that is trivially $*$-isomorphic to $\N$. However, for $t\neq0$ the spaces $\N(t)$ are not $W^*$-algebras. 

The operations
\begin{align}
        \cdot\,:\N(t_1)\times\N(t_2)\ni(x\psi^{\ii t_1},y\psi^{\ii t_2})&\mapsto x\sigma_{t_1}^{\psi}(y)\psi^{\ii(t_1+t_2)}\in\N(t_1+t_2),\\
        ^*\,:\N(t)\ni x\psi^{\ii t}&\mapsto\sigma_{-t}^\psi(x^*)\psi^{-\ii t}\in\N(-t),
\end{align}
equip the disjoint sum $\fell(\N):=\coprod_{t\in\RR}\N(t)$ over $\N\times\RR$ with the structure of $*$-algebra. The bijections 
\begin{equation}
        \N(t)\ni x\psi^{\ii t}\mapsto(x,t)\in\N\times\RR
\label{core.pairs.bijection}
\end{equation}
allow to endow $\fell(\N)$ with the topology induced by \eqref{core.pairs.bijection} from the product topology on $\N\times\RR$ of the weak-$\star$ topology on $\N$ and the usual topology on $\RR$. This provides the Fell's Banach $^*$-algebra bundle structure on $\fell(\N)$ (see \cite{Fell:1969,Fell:1977,Fell:Doran:1988} for a general theory of the Fell bundles). One can consider the Fell bundle $\fell(\N)$ as a natural algebraic structure which enables to translate between elements of $\N(t)$ at different $t\in\RR$. In order to recover an element of $\fell(\N)$ at a given $t\in\RR$, one has to select a section $\widetilde{x}:\RR\ra\fell(\N)$ of $\fell(\N)$:
\begin{equation}
        t\mapsto x(t)\psi^{\ii t}=:\widetilde{x}(t).
\end{equation}
Consider the set $\Gamma^1(\fell(\N))$ of such cross-sections of $\fell(\N)$ that are \df{integrable} in the following sense:
\begin{enumerate}
\item[i)] for any $\epsilon>0$ and any bounded interval $I\subseteq\RR$ there exists a compact subset $Y\subseteq I$ such that $\ab{I-Y}<\epsilon$ and the restriction $Y\ni t\mapsto x(t)\in\fell(\N)$ is continuous relative to the topology induced in $\fell(\N)$ by \eqref{core.pairs.bijection},
\item[ii)] $\int_\RR\dd r\,\n{x(r)}<\infty$.
\end{enumerate}
The set $\Gamma^1(\fell(\N))$ can be endowed with a multiplication, involution, and norm,
\begin{align}
        (\widetilde{x}\widetilde{y})(t)
        &:=\int_\RR\dd r\, x(r)y(t-r)
        =\left(\int_\RR\dd r\, x(r)\sigma_r^\psi(y(t-r))\right)\psi^{\ii t},\\
  \widetilde{x}^*(t)
  &:=\widetilde{x}(-t)^*
  =\sigma_t^\psi(x(-t)^*)\psi^{\ii t},\\
  \n{\widetilde{x}}
  &:=\int_\RR\dd r\,\n{x(r)},
\end{align}
thus forming a Banach $*$-algebra, denoted by $\bundlealg(\N)$. 

Falcone and Takesaki \cite{Falcone:2000,Falcone:Takesaki:2001} constructed also a suitably defined `bundle of Hilbert spaces' over $\RR$. Let $\N=\pi(\C)$ be a von Neumann algebra representing a $W^*$-algebra $\C$ in terms of a standard representation $(\H,\pi,J,\stdcone)$. The space $\H$ can be considered as a $\N$-$(\N^\comm)^o$ bimodule, with the left action of $\N$ on $\H$ given by ordinary multiplication from the left, and with the right action of $(\N^\comm)^o$ on $\H$ defined by 
\begin{equation}
        \xi x^o:=x\xi\;\;\;\forall\xi\in\H\;\;\forall x^o\in(\N^\comm)^o.
\end{equation}
Thus, the left action of $\N$ on $\H$ is just an action of a standard representation of the underlying $W^*$-algebra $\C$, while the right action of $(\N^\comm)^o$ is provided by the corresponding antirepresentation of $\C$ (that is, by commutant of a standard representation of $\C^o$). Given arbitrary $r_1,r_2\in\RR$, $\zeta_1,\zeta_2\in\H$, $\phi_1,\phi_2\in\W_0(\N)$, and $\varphi_1,\varphi_2:\W_0((\N^\comm)^o)$ the condition
\begin{equation}      
\left(\connes{\phi_1}{\varphi_1}\right)^{\ii r_1}\zeta_1=\left(\connes{\phi_2}{\varphi_2}\right)^{\ii r_2}\zeta_2\Connes{\varphi_2}{\varphi_1}{t},
\end{equation}
defines an equivalence relation
\begin{equation}
        (r_1,\phi_1,\zeta_1,\varphi_1)\;\;\sim_t\;\;(r_2,\phi_2,\zeta_2,\varphi_2)
\label{hilb.equiv}
\end{equation}
on the set $\RR\times\W_0(\N)\times\H\times\W_0((\N^\comm)^o)$. The equivalence class of the relation \eqref{hilb.equiv} is denoted by $\rpktarget{HT}\H(t)$, and its elements have the form
\begin{equation}
        \phi^{\ii t}\xi=\left(\connes{\phi}{\varphi}\right)^{\ii t}\xi\varphi^{\ii t},
\end{equation}
which is equivalent to
\begin{equation}
        \phi^{\ii t}\xi\varphi^{-\ii t}=\left(\connes{\phi}{\varphi}\right)^{\ii t}\xi.
\label{connes.cocycle.as.generator}
\end{equation}
Falcone and Takesaki show that $\H(t)$ is a Hilbert space independent of the choice of weights $\phi_1,\phi_2,\varphi_1,\varphi_2$ and of the choice of $r\in\RR$. This enables to form the Hilbert space bundle over $\RR$,
$\coprod_{t\in\RR}\H(t)$, and to form the Hilbert space of square-integrable cross-sections of this bundle,\rpktarget{COREH}
\begin{equation}
        \widetilde{\H}:=\Gamma^2\left(\coprod\limits_{t\in\RR}\H(t)\right).
\end{equation}
The Hilbert space bundle $\coprod_{t\in\RR}\H(t)$ is homeomorphic to $\H\times\RR$ for \textit{any} choice of weight $\psi\in\W_0(\N)$. The left action of $\bundlealg(\N)$ on $\widetilde{\H}$ generates a von Neumann algebra $\rpktarget{CORE}\core$ called \df{standard core} \cite{Falcone:Takesaki:2001}. For type III$_1$ factors $\N$ the standard core $\core$ is a type II$_\infty$ factor, but in general case $\core$ is not a factor. The structure of $\core$ is independent of the choice of weight on $\N$. However, for any choice of $\psi\in\W_0(\N)$ there exists a unitary map
\begin{equation}
\coreiso_\psi:\widetilde{\H}=\Gamma^2\left(\coprod_{t\in\RR}\H(t)\right)\ra L_2(\RR,\dd t;\H)\iso\H\otimes L_2(\RR,\dd t),
\end{equation}
such that
\begin{equation}
        \coreiso_\psi(\xi)(t)=\psi^{-\ii t}\xi(t)\in\H\;\;\;\forall\xi\in\widetilde{\H}.
\end{equation}
It satisfies
\begin{align}
        (\coreiso_\psi x \coreiso_\psi^*)(\xi)(t)
        &=\sigma^\psi_{-t}(x)\xi(t),\\
        (\coreiso_\psi\psi^{\ii s}\coreiso_\psi^*)(\xi)(t)
        &=\xi(t-s),\\
        (\coreiso_\psi\phi^{-\ii s}\coreiso_\psi^*)(\xi)(t)
        &=\left(\connes{\psi}{\phi^\comm}\right)^{\ii s}.
\end{align}
for all $\xi\in L_2(\RR,\dd t;\H)$, $x\in\N$, $\phi,\psi\in\W_0(\N)$, $s,t\in\RR$. This map provides a $*$-isomorphism between the standard core $\core$ on $\widetilde{\H}$ and the crossed product $\N\rtimes_{\sigma^\psi}\RR$ on $\H\otimes L_2(\RR,\dd t)$,
\begin{equation}
       \coreiso_\psi^*\core\coreiso_\psi=\N\rtimes_{\sigma^\psi}\RR.
\label{core.iso.map}
\end{equation}
The equation \eqref{core.iso.map} is a canonical analogue of \eqref{precore.isomorphisms}. Using the uniqueness of the standard representation up to unitary equivalence, Falcone and Takesaki \cite{Falcone:Takesaki:2001} proved that the map $\N\mapsto\core$ extends to a functor $\VNCore$ from the category $\VNIso$ of von Neumann algebras with $*$-isomorphisms to its own subcategory $\VNsfIso$ of semi-finite von Neumann algebras with $*$-isomorphisms. The functoriality of Kosaki's construction of canonical representation $\pi_\C$ of any $W^*$-algebra $\C$ turns the assignment
\begin{equation}
        \C\mapsto\pi_\C(\C)=:\N\mapsto\core=\widetilde{\pi_\C(\C)}
\end{equation}
to a functor
\begin{equation}
        \WstarCore:\WsIso\ra\VNsfIso,
\end{equation}
where $\WstarCore:=\VNCore\,\circ\,\CanVN$. For any $W^*$-algebra $\N$, the object $\WstarCore(\N)\in\Ob(\VNsfIso)$ will be called \df{canonical core} of $\N$.
\section{Noncommutative integration\label{noncommutative.integration.section}}
The construction of a standard representation by Connes, Araki and Haagerup and its further canonical refinement by Kosaki allows one to assign a canonical $L_2(\N)$ space to every $W^*$-algebra $\N$, without any choice of weight on $\N$ involved. This leads to a question: is it possible to develop the theory of noncommutative integration and $L_p(\N)$ spaces for arbitrary $W^*$-algebras $\N$ along the lines of analogy between Hilbert--Schmidt space as a member $\mathfrak{G}_2(\H)$ of the family of von Neumann--Schatten $\mathfrak{G}_p(\H)$ spaces and a Hilbert $L_2(\X,\mho(\X),\tmu)$ space as a member of a family of Riesz--Radon $L_p(\X,\mho(\X),\tmu)$ spaces? The answer is in the affirmative.

The theory of integration on $W^*$-algebras has a long history, which can be divided roughly into three stages. It started from results and methods developed for the analysis of factor von Neumann algebras by Murray and von Neumann \cite{Murray:vonNeumann:1936,vonNeumann:1943} (this includes first noncommutative Radon--Nikod\'{y}m type theorem, called \textit{BT-theorem}). Dye \cite{Dye:1952} proved more general Radon--Nikod\'{y}m type theorem for finite von Neumann algebras, while Segal \cite{Segal:1953} and Puk\'{a}nszky \cite{Pukanszky:1954} extended it to all semi-finite von Neumann algebras $\N$. Segal \cite{Segal:1953} proposed the foundations of noncommutative integration theory based on the notion of measurability of an unbounded (but affiliated) operator with respect to a faithful normal semi-finite trace $\tau$, and used it to define noncommutative $L_1(\N,\tau)$, $L_2(\N,\tau)$, and $L_\infty(\N,\tau)$ spaces (see also \cite{Segal:1965}). This definition was extended to the full range of noncommutative $L_p(\N,\tau)$ spaces by Ogasawara and Yoshinaga \cite{Ogasawara:Yoshinaga:1955,Ogasawara:Yoshinaga:1955:II} and Kunze \cite{Kunze:1958}. Further refinements of the theory, based on the stronger notion of measurability with respect to a trace, were introduced by Stinespring \cite{Stinespring:1959} and Nelson \cite{Nelson:1974}. Independently, Dixmier \cite{Dixmier:1953} has also introduced a family of $L_p(\N,\tau)$ spaces, based on the collection of abstract ideals in $\N$ and their completion with respect to $p$-norms defined by trace $\tau$ (see also \cite{Dixmier:1950,Dixmier:1952:remarques,Dixmier:1957}). The alternative (but equivalent) approach, based on Grothendieck's rearrangements of operators \cite{Grothendieck:1955}, was developed by Yeadon \cite{Yeadon:1973,Yeadon:1975} and Fack \& Kosaki \cite{Fack:Kosaki:1986}. The isometries between $L_p(\N,\tau)$ spaces were analysed in \cite{Broise:1966,Russo:1968,Arazy:1975,Katavolos:1976,Tam:1979,Katavolos:1981,Katavolos:1982}, and their complete description was provided by Yeadon \cite{Yeadon:1981}. For additional results in the foundations of noncommutative integration on semi-finite von Neumann algebras, and $L_p(\N,\tau)$ space theory, see also \cite{Padmanabhan:1967,Saito:1969,Saito:1970,Ovchinnikov:1970,Christensen:1972,Padmanabhan:1979,Cecchini:1978,Trunov:1979,Trunov:1980,Trunov:1981,Yeadon:1981,Kosaki:1981:Lorentz,Leinert:1986,Leinert:1992}. The von Neumann--Schatten spaces $\schatten_p(\H)$ \cite{vonNeumann:1937:Tomsk,Schatten:1946,Schatten:vonNeumann:1946,Schatten:vonNeumann:1947,Schatten:1950,Schatten:1960}, associated by definition with type I von Neumann algebras $\BH$, are just a special case of this theory, obtained by $\schatten_p(\H)=L_p(\BH,\tr)$.

Extension of the noncommutative integration theory to an arbitrary von Neumann algebra $\N$ equipped with an arbitrary 
$\psi\in\W_0(\N)$ became possible only after development of the Tomita--Takesaki modular theory \cite{Tomita:1967:a,Tomita:1967:b,Takesaki:1970}, which has lead in particular to noncommutative Radon--Nikod\'{y}m type theorems by Pedersen--Takesaki \cite{Pedersen:Takesaki:1973} and Connes \cite{Connes:1973:poids:normaux,Connes:1973:classification}, valid for all von Neumann algebras. First construction of a complete range of noncommutative $L_p(\N,\psi)$ spaces for $\psi\in\W_0(\N)$ and $p\in[1,\infty]$ was provided by Haagerup \cite{Haagerup:1979:ncLp} (its complete exposition was given by Terp \cite{Terp:1981}), using Takesaki's duality theory for crossed products of the von Neumann algebras by locally compact abelian groups \cite{Takesaki:1973:duality,Digernes:1974,Digernes:1975,Haagerup:1978:dualweights:I,Haagerup:1978:dualweights:II} as well as Haagerup's theory of operator valued weights \cite{Haagerup:1979:ovw1,Haagerup:1979:ovw2}. In this approach, the noncommutative $L_p(\N,\psi)$ spaces are represented as Banach spaces of closed densely defined operators that are affiliated with the crossed product $\N\rtimes_{\sigma^\psi}\RR$ and satisfying certain additional conditions. The extension of Haagerup's construction to $p\in\CC$ was proposed by Yamagami \cite{Yamagami:1992}. An alternative approach was developed by Connes \cite{Connes:1980} and Hilsum \cite{Hilsum:1981}, who have defined $L_p(\N,\psi^\comm)$ spaces for $\psi^\comm\in\W_0(\N^\comm)$ as Banach spaces of closed densely defined unbounded operators $x$ acting on a single Hilbert space $\H$, and such that $\exists\phi\in\N_\star^+\;\;\ab{x}^p=\connes{\phi}{\psi^\comm}$, Cauchy completed in the norm $\n{x}_p:=(\phi(\ab{\II}))^{1/p}$. Kosaki \cite{Kosaki:1984:ncLp} defined the families of $L_p(\N,\psi)$ spaces for $\psi\in\N^+_{\star0}$ as Calder\'{o}n complex interpolation Banach spaces \cite{Calderon:1964} between $\N_\star$ and $\N$. This construction was extended to $\psi\in\W_0(\N)$ by Terp \cite{Terp:1982}. Izumi \cite{Izumi:1996,Izumi:1997,Izumi:1998,Izumi:1998:PhD,Izumi:2000} completed Kosaki--Terp approach and provided further extension of this construction to $p\in\CC$. Araki and Masuda \cite{Araki:Masuda:1982} constructed $L_p(\N,\psi)$ spaces for $\psi\in\N^+_{\star0}$ using standard form of $\N$ on a single Hilbert space $\H$ and the properties of $p$-family of positive cones in $\H$ introduced by Araki \cite{Araki:1974:modular:conjugation} (which were also investigated in this context by Kosaki \cite{Kosaki:1980:cones,Kosaki:1981:positive,Kosaki:1982:RN,Kosaki:1982:Tt,Kosaki:1983:RN2}). Masuda \cite{Masuda:1983} extended this construction to $\psi\in\W_0(\N)$. Another approach to the  theory of noncommutative integration for arbitrary von Neumann algebras was developed by Sherstn\"{e}v and Trunov \cite{Sherstnev:1974,Sherstnev:1977,Sherstnev:1978,Trunov:Sherstnev:1978:I,Trunov:Sherstnev:1978:II,Trunov:1979:L2,Sherstnev:1982,Trunov:1983,Zolotarev:1985:konus}. However, they have constructed only noncommutative $L_1(\N,\psi)$ and $L_2(\N,\psi)$ spaces (Trunov's $L_p(\N,\psi)$ spaces \cite{Trunov:1979,Trunov:1981,Zolotarev:1988} were restricted to semi-finite $\N$). An extension of this approach to full theory of noncommutative $L_p(\N,\psi)$ spaces was provided later, and in different ways, by Zolotar\"{e}v \cite{Zolotarev:1982,Zolotarev:1985,Zolotarev:1986} and Cecchini \cite{Cecchini:1984,Cecchini:1985,Cecchini:1986,Cecchini:1988}, in both cases using Calder\'{o}n's complex interpolation method. Yet another different constructions of a range of $L_p(\N,\psi)$ spaces were provided by Tikhonov \cite{Tikhonov:1982} and Leinert \cite{Leinert:1991}. A construction of $L_2(\N,\psi)$ spaces dependent on $\psi\in\N^+_{\star0}$ was proposed also by Holevo \cite{Holevo:1976,Holevo:1977}.

Most of the above constructions have been shown to be equivalent. The isometric isomorphism between Connes--Hilsum $L_p(\N,\psi^\comm)$ spaces and Kosaki--Terp $L_p(\N,\psi)$ spaces was established by Terp \cite{Terp:1982}. The isometric isomorphism between Haagerup--Terp $L_p(\N,\psi)$ spaces and Kosaki--Terp $L_p(\N,\psi)$ spaces was established by Kosaki \cite{Kosaki:1984:ncLp}. The isometric isomorphism between Haagerup--Terp $L_p(\N,\psi)$ spaces and Connes--Hilsum $L_p(\N,\psi^\comm)$ spaces was established by Terp \cite{Terp:1981}. The isometric isomorphism between Araki--Masuda and Connes--Hilsum spaces is stated in Araki--Masuda \cite{Araki:Masuda:1982} and Masuda \cite{Masuda:1983}. The isometric isomorphisms between Cecchini's $L_p(\N,\psi)$ spaces, Zolotar\"{e}v's $L_p(\N,\psi)$ spaces, and Connes--Hilsum $L_p(\N,\psi^\comm)$ spaces was established by Cecchini \cite{Cecchini:1986}. Leinert proved \cite{Leinert:1991} that his $L_p(\N,\psi)$ spaces are isometrically isomorphic to Kosaki--Terp $L_p(\N,\psi)$ spaces. 

These equivalence results, together with the independence of the $L_p(\N,\psi)$ spaces on the choice of $\psi\in\W_0(\N)$, reflected by isometric isomorphisms $L_p(\N,\psi)\iso L_p(\N,\phi)$ for arbitrary $\psi,\phi\in\W_0(\N)$, lead to the problem of construction of the canonical theory of noncommutative $L_p(\N)$ spaces, which would unify all these approaches, and would be independent of the choice of representation of a $W^*$-algebra in terms of a von Neumann algebra. Such theory was developed by Kosaki \cite{Kosaki:1980:PhD}, who constructed the range of $L_p(\N)$ spaces over arbitrary $W^*$-algebra $\N$ by introducing and applying the \textit{canonical} form $(L_2(\N),\N,J_\N,L_2(\N)^+)$. The alternative (but equivalent by means of an isometric isomorphism) construction of the canonical $L_p(\N)$ space theory has been provided by Falcone and Takesaki \cite{Falcone:1996,Falcone:Takesaki:1999,Falcone:2000,Falcone:Takesaki:2001,Takesaki:2003}. The crucial underlying notion of this approach is the \textit{core} von Neumann algebra $\core$ which is associated functorially to any $W^*$-algebra $\N$, and provides a weight-independent analogue of the crossed product $\N\rtimes_{\sigma^\psi}\RR$ used in Haagerup's construction of $L_p(\N,\psi)$. The structure of $\core$ is designed to include the elements of $\N$ and $\W_0(\N)$ on equal footing, what develops the ideas contained in earlier works of Woronowicz \cite{Woronowicz:1979}, Connes \cite{Connes:1980:correspondences} and Yamagami \cite{Yamagami:1992}, who considered the grading of weights and bimodule structures as a method of construction of $L_1(\N,\psi)$, $L_2(\N,\psi)$, and $L_p(\N,\psi)$ spaces, respectively. It gives rise to a \textit{modular algebra} generated by the elements $x\phi^z$, where $x\in\N$, $\phi\in\W_0(\N)$ (or $\N_\star^+$), and $z\in\CC$ with $\re(z)\geq0$, as well as to a canonical integral that acts on all elements of a modular algebra that have a form $x_1\phi_1^{z_1}\cdots x_n\phi_n^{z_n}$ with $\sum_{i=1}^nz_i=1$. Both canonical constructions of $L_p(\N)$ spaces (by Kosaki and by Falcone \& Takesaki) do not require any underlying Hilbert space and are functorial over the category of $W^*$-algebras with $*$-isomorphisms. Kosaki showed this explicitly by providing all constructions in terms of the canonical representation, and by showing functoriality of assignment of the canonical representation to a $W^*$-algebra. In the Falcone--Takesaki case this follows from functorial dependence of the standard core algebra on the underlying von Neumann algebra, together with the functorial association of a von Neumann algebra to any $W^*$-algebra by means of Kosaki's canonical representation.

For (selective) reviews of noncommutative integration theory, see \cite{Trunov:Sherstnev:1985,Goldstein:2002,Takesaki:2003,Pisier:Xu:2003,Zumbraegel:2004,Sherstnev:2008,Flattot:2010}. For further development of the theory of $L_p(\N)$ spaces using the modular algebras and bimodules, see \cite{Yamagami:1994,Sherman:2001,Sherman:2003,Sherman:2005,Junge:Sherman:2005,Sherman:2006,Sherman:2006:new,Pavlov:2011}. The Banach space properties of noncommutative $L_p(\N,\psi)$ spaces (with a special attention paid to the problem of classification of isometries of these spaces) are discussed in  \cite{Watanabe:1992,Watanabe:1995,Watanabe:1996,Sukochev:1996,Watanabe:1999,Raynaud:Xu:2001,Raynaud:Xu:2003,Haagerup:Rosenthal:Sukochev:2003,Junge:Ruan:2004,Junge:Ruan:Sherman:2004,Junge:Ruan:Xu:2005,Randrianantonina:2008,Haagerup:Junge:Xu:2010,Junge:Parcet:2010}. As Banach spaces, $L_p(\N,\psi)$ spaces can be considered as special cases of the more general notions of noncommutative Banach spaces and operator spaces \cite{Ovchinnikov:1970,Dodds:Dodds:dePagter:1989,Xu:1991,Dodds:Dodds:dePagter:1992,Dodds:Dodds:dePagter:1993,Dodds:Dodds:1995,Pisier:2003,dePagter:2007}. For some extensions of a theory of noncommutative $L_p$ spaces to $C^*$-algebras see \cite{Majewski:Zegarlinski:1995,Majewski:Zegarlinski:1996:MPRF,Majewski:Zegarlinski:1996:APP,Phan:1999,Goldstein:Phan:2000}. For elements of noncommutative integration theory over $^*$-algebras see \cite{Gudder:Hudson:1978,Gudder:1979:RN}.
\subsection{Noncommutative Radon--Nikod\'{y}m type theorems\label{nc.RN.section}}
Let $\N$ be a von Neumann algebra acting on $\H$. A linear operator $x:\dom(x)\ra\H$ is called \df{affiliated} to $\N$ if{}f $[x,u]=0$ for every unitary element $u\in\N^\comm$ \cite{Murray:vonNeumann:1936}. The set of all operators affiliated to $\N$ will be denoted $\rpktarget{AFF}\aff(\N)$, while the set of all positive self-adjoint elements of $\aff(\N)$ will be denoted $\aff(\N)^+$. A closed linear operator $x$ is affiliated to $\N$ if{}f $x(\II+x^*x)^{1/2}\in\N$. A closed densely defined linear operator $x:\dom(x)\ra\H$ with polar decomposition $x=v\ab{x}$ is affiliated with $\N$ if{}f any of the following equivalent conditions holds:
\begin{enumerate}
\item[1)] $[u,x]=0$ $\forall$ unitary $u\in\N^\comm$,
\item[2)] $[u,\ab{x}]=0$ and $[u,v]=0$ $\forall$ unitary $u\in\N^\comm$,
\item[3)] $v\in\N$ and all spectral projections of $\ab{x}$ belong to $\N$.
\end{enumerate}
The space of all closed densely defined linear operators affiliated with $\N$ will be denoted by $\rpktarget{MMM}\MMM(\N)$. 

A generalisation of the notion of affiliation to arbitrary $W^*$-algebras was alluded already in \cite{Stratila:Zsido:1975}, but it was actually provided by Derezi\'{n}ski, Jak\v{s}i\'{c} and Pillet \cite{DJP:2003}, using earlier insights of \cite{Baaj:Jungl:1983,Woronowicz:1991}. For a given $W^*$-algebra $\N$, a linear map $x:\dom(x)\ra\N$, where $\dom(x)\subseteq\N$, is called to be \df{affiliated} to $\N$ if{}f there exists $y\in\N$ such that 
\begin{enumerate}
\item[1)] $\n{y}\leq1$,
\item[2)] $(\II-yy^*)\N$ is weakly-$\star$ dense in $\N$,
\item[3)] $\forall y_1,y_2\in\N$ $(y_1\in\dom(x),\;xy_1=u_2)$ $\iff$ $yy_1=(1-yy^*)^{1/2}y_2$.
\end{enumerate}
If such $y$ exists, then it is unique. The set of all operators affiliated to $\N$ will be denoted $\rpktarget{AFF.ZWEI}\aff(\N)$. The set of all elements $x\in\aff(\N)$ such that $(1-yy^*)^{-1/2}\in\N^+$ will be denoted $\aff(\N)^+$. If $(\H,\pi)$ is a normal representation of $\N$ such that $\pi(\II)=\II$, then there exists a unique extension $\hat{\pi}:\aff(\N)\ra\aff(\pi(\N))$ such that
\begin{equation}
(1+\pi(x)\pi(x)^*)^{-1/2}\pi(x)=\pi(y),
\end{equation}
where $y$ is determined by $x\in\aff(\N)$ as above. If $(\H,\pi)$ is faithful, then $\hat{\pi}$ is faithful. The set $\aff(\pi(\N))$ of maps $x:\dom(x)\ra\pi(\N)$ coincides with the set $\aff(\pi(\N))$ of operators $x:\dom(x)\ra\H$ affiliated to $\pi(\N)$, and the same holds for the sets $\aff(\pi(\N))^+$.

Recall that any weight on a $W^*$-algebra $\N$ can be uniquely extended to a linear functional on $\mmm_\phi$ which coincides with $\phi$ on $\N^+\cap\mmm_\phi$. Given a semi-finite trace $\tau:\N^+\ra[0,\infty]$ on a semi-finite $W^*$-algebra $\N$, its extension to a two-sided ideal $\mmm_\tau$ of $\N$ satisfies
\begin{equation}
        \tau(yx)=\tau(xy)\;\;\forall x\in\mmm_\tau\;\forall y\in\N.
\label{commutative.trace.mmm}
\end{equation}
In addition, if $\tau$ is normal, then for any $x\in\mmm_\tau$ the map 
\begin{equation}
        y\mapsto\omega_x(y):=\tau(xy)
\label{omega.x.map}
\end{equation}
is an element of $\N_\star^+$ \cite{Dye:1952}. Moreover, 
\begin{equation}
        \tau(yx)=\tau(x^{1/2}yx^{1/2})=\tau(y^{1/2}xy^{1/2})\;\;\forall x\in\mmm_\tau^+\;\forall y\in\N^+.
\end{equation}
So, the formula
\begin{equation}
        \omega_x(y):=\tau(x^{1/2}yx^{1/2})\;\;\forall y\in\N
\label{omega.from.tau.ideal.mt}
\end{equation}
gives rise to $\omega_x\in\N_\star^+$ with $\n{\omega_x}=\tau(\ab{x})$ for each $x\in\mmm_\tau$. 

Given any $h\in\aff(\N)^+$,
\begin{align}
        \pvm^h\left(\left]\textstyle\frac{1}{n},n\right[\right)&\in\mmm_\tau^+\;\;\forall n\in\NN,\\
        \pvm^h\left(\left]\textstyle\frac{1}{n},n\right[\right)h&\in\N^+\;\;\forall n\in\NN,
\end{align}
so the formulas \eqref{commutative.trace.mmm}-\eqref{omega.from.tau.ideal.mt} can be applied also to $\pvm^h\left(\left]\textstyle\frac{1}{n},n\right[\right)$. Setting $\epsilon:=\frac{1}{n}$, define
\begin{equation}\rpktarget{TAU.H}
        \tau_h:=\lim_{\epsilon\ra^+0}\tau
        \left(
                \left(\pvm^h\left(\left]\epsilon,\textstyle\frac{1}{\epsilon}\right[\right)\right)^{1/2}
                \,\cdot\,
                \left(\pvm^h\left(\left]\epsilon,\textstyle\frac{1}{\epsilon}\right[\right)\right)^{1/2}
        \right).
\end{equation}
According to the Segal--Puk\'{a}nszky theorem \cite{Segal:1953,Pukanszky:1954} (see also \cite{Perdrizet:1971}),
\begin{equation}
        \forall\phi\in\N_\star^+\;\exists!h\in\aff(\N)^+\;\forall x\in\N^+\;\;\;\phi(x)=\tau_h(x).
\label{Segal.Pukanszky.theorem}
\end{equation}
This theorem implies the Dye--Segal theorem \cite{Dye:1952,Segal:1953},
\begin{equation}
        \forall\psi,\phi\in\N_\star^+\;\;
        \phi\leq\psi\;
        \limp\;
        \exists!h\in\aff(\N)^+\;\;
        \phi(x)=\lim_{\epsilon\ra^+0}\psi
        \left(
        \left(\pvm^h\left(\left]\epsilon,\textstyle\frac{1}{\epsilon}\right[\right)\right)^{1/2}
        h^{1/2}xh^{1/2}
        \left(\pvm^h\left(\left]\epsilon,\textstyle\frac{1}{\epsilon}\right[\right)\right)^{1/2}
        \right).
\label{Dye.Segal.thm}
\end{equation}

If $h:\dom(x)\ra\H$ is a positive self-adjoint linear operator on $\H$, then $h(1+\epsilon h)^{-1}$ is self-adjoint and bounded for any $\epsilon>0$. If $h\in\aff(\N)^+$ then $h(1+\epsilon h)^{-1}\in\N^+$ $\forall\epsilon>0$. Using these properties, Pedersen and Takesaki \cite{Pedersen:Takesaki:1973} showed that the `regularised perturbation' of the trace $\tau$ by $h\in\aff(\N)^+$,\rpktarget{TAU.H.ZWEI}
\begin{equation}
        \tau_h(\cdot):=\lim_{\epsilon\ra^+0}\tau\left((h(1+\epsilon h)^{-1})^{1/2}\,\cdot\,(h(1+\epsilon h)^{-1})^{1/2}\right)
\label{omega.from.tau}
\end{equation}
is a semi-finite normal weight, $\tau_h\in\W(\N)$, and 
\begin{equation}
\forall\phi\in\W(\N)\;\exists! h\in\aff(\N)^+\;\;\phi=\tau_h.
\label{Dye.Segal.weights}
\end{equation}
Moreover, the map $h\mapsto\tau_h$ is a bijection between $\aff(\N)^+$ and $\W(\N)$. In this sense, the equations \eqref{Dye.Segal.thm} and \eqref{Dye.Segal.weights} provide the noncommutative analogue of the Radon--Nikod\'{y}m theorem. (For a derivation of \eqref{Dye.Segal.thm} using \eqref{Dye.Segal.weights}, see \cite{Pedersen:1979}. In such case the Dye--Segal theorem is valid for any $W^*$-algebra. See also a discussion of this theorem in \cite{Gudder:Marchand:1972}.)

As noted by Segal \cite{Segal:1953,Segal:1965}, the key property responsible for the Dye--Segal--Puk\'{a}nszky analogue of the Radon--Nikod\'{y}m theorem is the invariance property of a trace: $\tau(x)=\tau(uxu^*)$ for all unitary $u\in\N$. For general weights, this property no longer holds. A weaker condition is a `relative invariance' property: $\psi=\psi\circ\sigma^\phi$, which turns out to correspond to the use of such $h$ that are affiliated to the subset of $\N$ invariant with respect to the action of $\sigma^\phi$, namely $h\in\aff(\N_{\sigma^\phi})^+$. If $\N$ is not semi-finite, then there exists no normal semi-finite trace on it. This suggests to consider traces on $\N_{\sigma^\phi}$. However, for $\phi\in\W(\N)$, $\N_{\sigma^\phi}$ may be not semi-finite (it may be of type III even if $\N$ is of type II$_\infty$ \cite{Haagerup:1977}). Moreover, given $\phi\in\W(\N)$, the restriction $\phi|_{\N_{\sigma^\phi}}$ is a trace on $\N_{\sigma^\phi}\cap\mmm_\phi$, but it is not semi-finite unless $\phi$ is \textit{strictly semi-finite} (i.e., $\phi$ can be expressed as a sum of $\{\phi_i\}\subseteq\N_\star^+$ with $\supp(\phi_i)\supp(\phi_j)=\dirac_{ij}\supp(\phi_i)\;\forall i,j$) \cite{Combes:1971:esperances}. To handle more general cases, it is necessary to establish the method of comparison between all normal weights. The property
\begin{equation}
        h\in\N_{\sigma^\phi}\;\;\iff\;\;h\mmm_\phi\subseteq\mmm_\phi,\;\mmm_\phi h\subseteq\mmm_\phi,\;\phi(hx)=\phi(xh)\;\;\forall x\in\mmm_\phi
\end{equation}
allowed Pedersen and Takesaki \cite{Pedersen:Takesaki:1973} to construct the `perturbations' of weights similar to \eqref{omega.from.tau}, and to provide the extension of the theorems \eqref{Segal.Pukanszky.theorem} and \eqref{Dye.Segal.weights} to the case of weights on \textit{arbitrary} $W^*$-algebras. 

If $\phi\in\W(\N)$ and $h\in\N^+_{\sigma^\phi}$, then
\begin{equation}\rpktarget{PHI.H}
        \phi_h:\N^+\ni x\mapsto\phi_h(x):=\phi(h^{1/2}x h^{1/2})\in[0,\infty]
\end{equation}
satisfies
\begin{enumerate}
\item[1)] $\phi_h\in\W(\N)$,
\item[2)] $\supp(\phi_h)=\supp(h)$,
\item[3)] $\phi_h\in\W_0(\N)\;\iff\;\supp(h)=\II$,
\item[4)] $\phi_{h_1+h_2}=\phi_{h_1}+\phi_{h_2}$,
\item[5)] $h_1\leq h_2\;\limp\;\phi_{h_1}\leq\phi_{h_2}$,
\item[6)]  $\sup_\iota\{h_\iota\}=h\;\limp\;\sup_\iota\{\phi_{h_\iota}\}=\phi_h$,
\end{enumerate}
for all $h_1,h_2,h_\iota\in\N^+_{\sigma^\phi}$. If $h_\epsilon:=h\in\aff(\N_{\sigma^\phi})^+$ then $h(1+\epsilon h)^{-1}\in\N_{\sigma^\phi}^+\;\forall\epsilon>0$ and $\lim_{\epsilon\ra^+0}h_\epsilon=h$. This allows to define the extension of $\phi_h$ for $h\in\aff(\N_{\sigma^\phi})^+$,
\begin{equation}
        \phi_h:=\lim_{\epsilon\ra^+0}\phi\left((h(1+\epsilon h)^{-1})^{1/2}\,\cdot\,(h(1+\epsilon h)^{-1})^{1/2}\right)\in\W(\N).
\label{Pedersen.Takesaki.perturbed.weight}
\end{equation}
It satisfies
\begin{enumerate}
\item[1)] $\phi_h\in\W(\N)$,
\item[2)] $\supp(\phi_h)=\supp(h)$,
\item[3)] $\phi_h\in\W_0(\N)\;\iff\;\supp(h)=\II$,
\item[4)] $h_1\leq h_2\;\limp\;\phi_{h_1}\leq\phi_{h_2}$,
\item[5)] $h_1=h_2\;\iff\;\phi_{h_1}=\phi_{h_2}$,
\item[6)] $\sigma^{\phi_h}_t(x)=h^{\ii t}\sigma^\phi_t(x)h^{-\ii t}\;\;\forall x\in\N_{\supp(h)}\;\forall t\in\RR$,
\item[7)] $\phi\in\W_0(\N),\;[h_1,h_2]=0\;\limp\;(\phi_{h_1})_{h_2}=(\phi_{h_2})_{h_1}$,
\item[8)] $\Connes{\phi_h}{\phi}{t}=h^{\ii t}\;\forall t\in\RR$.
\end{enumerate}
Equation 6) holds for all $x\in\N$ if $h$ is invertible. One says that $h\in\aff(\N_{\sigma^\phi})^+$ is \df{nonsingular} if{}f $\supp(h)=\II$, or, equivalently, if{}f $h\xi\neq0\;\forall\xi\in\dom(h)\setminus\{0\}$. According to the Pedersen--Takesaki theorem \cite{Pedersen:Takesaki:1973} (see also \cite{Elliott:1975,Stratila:1981}), if $\phi\in\W_0(\N)$ and $\psi\in\W(\N)$ then the following conditions are equivalent:
\begin{enumerate}
\item[i)] $\psi=\psi\circ\sigma^\phi_t\;\forall t\in\RR$,
\item[ii)] $\Connes{\psi}{\phi}{t}\in\N_{\sigma^{\psi}}\;\forall t\in\RR$,
\item[iii)] $\Connes{\psi}{\phi}{t}\in\N_{\sigma^{\phi}}\;\forall t\in\RR$,
\item[iv)] $t\mapsto\Connes{\psi}{\phi}{t}$ is an ultrastrongly continuous group of unitary elements in $\N_{\supp(\psi)}$,
\item[v)] $\exists h\in\aff(\N_{\sigma^{\phi}})^+$ such that $\psi=\phi_h$.
\end{enumerate}
If any of these conditions holds, then $h$ is unique and $\supp(\psi)\in\N_{\sigma^\phi}$. Such $h$ will be called a \df{Pedersen--Takesaki density}. For $\psi\in\W_0(\N)$, $h$ is assumed to be nonsingular, and then the above conditions are also equivalent to
\begin{enumerate}
\item[vi)] $\phi=\phi\circ\sigma^\psi_t\;\forall t\in\RR$.
\item[vii)] $\exists!$ nonsingular $k\in\aff(\N_{\sigma^\psi})^+$ such that $\phi=\psi_k$.
\end{enumerate}
Moreover, if $\psi\in\W_0(\N)$, then $h^{\ii t}\in\N_{\sigma^\phi}$. Conversely, if $h\in\aff(\N_{\sigma^\phi})^+$ is nonsingular, then the equation $\psi=\phi_h$ determines $\psi\in\W_0(\N)$ and
\begin{equation}
        \sigma^\psi_t=\Ad(h^{\ii t})\sigma_t^\phi=\Connes{\psi}{\phi}{t}\sigma^\phi_t\Connes{\psi}{\phi}{t}^*\;\;\forall t\in\RR.
\end{equation}
If $\phi,\psi\in\N^+_{\star0}$, then $k,h\in\aff(\N_{\sigma^\phi}\cap\N_{\sigma^\psi})^+$ \cite{Takesaki:1970}. From 8), iv) and v) above it follows that the description of relationship of two weights (`integrals') in terms of perturbation by some operator (`function') is equivalent to the specification of this operator in terms of Connes' cocycle (`derivative'). This relationship becomes explicit for $\phi,\psi\in\W_0(\N)$. In such case, Connes' theorem \cite{Connes:1973:poids:normaux,Connes:1973:classification} states that the following conditions are equivalent:
\begin{enumerate}
\item[i)] $\exists\lambda>0\;\;\psi\leq\lambda\phi$,
\item[ii)] $x\in\nnn_\phi\;\limp\;x\in\nnn_\psi$,
\item[iii)] $\rpktarget{CONNES.COC.VIER}t\mapsto\Connes{\psi}{\phi}{t}$ can be extended to a map that is valued in $\N$, bounded (by $\lambda^{1/2}$) and weakly-$\star$ continuous on a strip $\{z\in\CC\mid\im(z)\in[-\frac{1}{2},0]\}$, holomorphic in interior of this strip, and satisfying the boundary condition
\begin{equation}
        \psi(x)=\phi\left(\Connes{\psi}{\phi}{-\ii/2}^*x\Connes{\psi}{\phi}{-\ii/2}\right)\;\;\forall x\in\mmm_\psi.
\end{equation}
\end{enumerate}
This theorem extends to $\psi\in\W(\N)$, with $\RR\ni t\mapsto\Connes{\psi}{\phi}{t}\in\supp(\psi)\N\;\forall t\in\RR$ \cite{Kosaki:1980:PhD}. Thus, whenever the condition i) is satisfied, the analytic continuation of Connes' cocycle 
\begin{equation}
        h^{1/2}=\Connes{\psi}{\phi}{-\ii/2}
\end{equation}
plays the role of a noncommutative Radon--Nikod\'{y}m quotient. If $\exists\lambda>0\;\;\psi\leq\lambda\phi$ is not satisfied, then $\Connes{\psi}{\phi}{-\ii/2}$ is an unbounded operator which is not closable unless $\N$ is finite \cite{Kosaki:1985}. The existence of inverse of the noncommutative Radon--Nikod\'{y}m quotient is equivalent to a condition that both weights are faithful (from the perspective of the GNS representation, this condition states that the corresponding Gel'fand ideals are empty).

Every faithful normal semi-finite trace $\tau$ on a semi-finite $W^*$-algebra $\N$ satisfies $\sigma^\tau_t=\id_\N\;\forall t\in\RR$, so the Pedersen--Takesaki theorem reduces in this case to the generalised Dye--Segal--Puk\'{a}nszky theorem \eqref{omega.from.tau} with $h^{\ii t}=\Connes{\phi}{\tau}{t}$.

Pedersen and Takesaki \cite{Pedersen:Takesaki:1973} proved also that for $\phi\in\W_0(\N)$ and $\psi\in\W(\N)$ the following conditions are equivalent:
\begin{enumerate}
\item[i)] $\psi$ satisfies the KMS condition with respect to $\sigma^\phi$ and $\beta=1$,
\item[ii)] $\supp(\psi)=\zentr_\N$ and $\sigma_t^\psi=\sigma^\phi_t|_{\N\supp(\psi)}\;\forall t\in\RR$,
\item[iii)] $\exists h\in\aff(\zentr_\N)^+$ such that $\psi=\phi_h$.
\end{enumerate}
If any of these conditions holds, then $h$ is unique and $h^{\ii t}\in\zentr_\N\subseteq\N_{\sigma^\phi}$. This is the case, in particular, if $\phi$ and $\psi$ are faithful normal semi-finite traces on a semi-finite $W^*$-algebra $\N$.

The noncommutativity of $\N$ allows to formulate also other noncommutative Radon--Nikod\'{y}m type theorems. In particular, there is a collection of theorems which express a functional $\psi$ majorised by $\phi$ as a `perturbation' $\psi=\phi(h\,\cdot\,h)$ with $0\leq h\leq\II$, starting from Segal's theorem \cite{Segal:1947:irreducible} for any $C^*$-algebra $\C$,
\begin{equation}
        \phi,\psi\in\C^{\banach+},\;\psi\leq\phi\;\;\limp\;\;\exists! h\in\pi_\phi(\C)^\comm\;\;\;\psi=\s{\pi_\phi(\,\cdot\,)h\Omega_\phi,\Omega_\phi}_\phi\mbox{ and }0\leq h\leq\II,
\end{equation}
and Sakai's theorem \cite{Sakai:1965} for any $W^*$-algebra $\N$,
\begin{equation}
        \phi,\psi\in\N_\star^+,\;\psi\leq\phi\;\;\limp\;\;\exists! h\in\N^+\;\;\psi=\phi(h\,\cdot\,h)\mbox{ and }0\leq h\leq\II,
\end{equation}
see also  \cite{Pedersen:Takesaki:1973,Haagerup:1973,Araki:1974:modular:conjugation,vanDaele:1975,Pedersen:1979}. Among these theorems we want to mention a theorem by \c{S}tr\u{a}til\u{a} and Zsid\'{o} \cite{Stratila:Zsido:1975}: if $\phi,\psi\in\N_\star^+$, $\supp(\psi)\leq\supp(\phi)$, and $\supp(\phi)$ is finite in $\N$, then
\begin{equation}
        \exists! h\in\aff(\N)^+\;\;\psi=\phi(h\,\cdot\,h)\;\mbox{ and }\;\supp(h)\leq\supp(\phi).
\end{equation}
For a generalisation of noncommutative Radon--Nikod\'{y}m theorem to arbitrary $*$-algebras see \cite{Gudder:1979:RN}.
\subsection{Integration relative to a trace\label{integration.trace.section}}
Let $\tau$ be a faithful normal semi-finite trace on a $W^*$-algebra $\N$. The map 
\begin{equation}
        \n{\cdot}_p:\N\ni x\mapsto\n{x}_p:=\tau(\ab{x}^p)^{1/p}\in[0,\infty]
\end{equation}
for $p\in[1,\infty[$ is a norm on a vector space $\{x\in\N\mid\n{x}_p<\infty\}$. Denote the Cauchy completion of this normed vector space by $\rpktarget{LPNTAU}L_p(\N,\tau)$. Equivalently, $L_p(\N,\tau)$ can be defined as a Cauchy completion of $\{x\in\N\mid\tau(\ab{x})<\infty\}$ in the norm given by $\n{\cdot}_p$ \cite{Nelson:1974}, or as a Cauchy completion of $\Span_\CC\{x\in\N^+\mid\tau(\supp(x))<\infty\}$ in $\n{\cdot}_p$ \cite{Pisier:Xu:2003}. The space $L_1(\N,\tau)$ can be equivalently defined also as a Cauchy completion of $\mmm_\tau$ in $\n{\cdot}_1$, while $L_2(\N,\tau)$ as a Cauchy completion of $\nnn_\tau$ in $\n{\cdot}_2$ \cite{Dixmier:1953,Takesaki:2003}. The property $\ab{\tau(x)}\leq\n{x}_1\;\forall x\in\mmm_\tau$ allows the unique continuous extension of $\tau$ from a linear functional on $\mmm_\tau$ to a linear functional on $L_1(\N,\tau)$. This extends a bilinear form
\begin{equation}
        \mmm_\tau\times\N\ni(h,x)\mapsto\tau(h^{1/2}x h^{1/2})\in\CC
\end{equation}
to the bilinear form $L_1(\N,\tau)\times\N\ra\CC$, which defines a duality between $L_1(\N,\tau)$ and $\N$, and makes $L_1(\N,\tau)$ isometrically isomorphic to $\N_\star$ \cite{Dixmier:1953}. Extending the notation $\omega_x$ of \eqref{omega.x.map} to all elements of $\N_\star$ corresponding to $x\in L_1(\N,\tau)$, we have
\begin{equation}
        \omega_x(y)=\tau(yx)=\tau(xy)\;\;\forall y\in\N\;\forall x\in L_1(\N,\tau),
\end{equation}
and \cite{Dye:1952,Segal:1953}
\begin{equation}
         \forall\omega\in\N_\star^+\;\;\exists! x\in L_1(\N,\tau)^+\;\;\forall y\in\N\;\;\omega(y)=\tau(xy)=\tau(x^{1/2}yx^{1/2}).
\label{Dye.Segal.density}
\end{equation}
Such $x$ will be called a \df{Dye--Segal density} of $\omega$ with respect to $\tau$. If $\psi,\phi\in\W_0(\N)$ are traces with corresponding Dye--Segal densities $\rho_\psi$ and $\rho_\psi$, then
\begin{equation}
	\Connes{\psi}{\phi}{t}=\rho^{\ii t}_\psi\rho^{-\ii t}_\phi\;\;\forall t\in\RR.
\end{equation}

Using the notion of measurability with respect to a trace $\tau$, the above range of $L_p(\N,\tau)$ spaces can be represented in terms of operators affiliated to a von Neumann algebra $\N$ acting on $\H$. Let $\tau$ be a fixed faithful normal semi-finite trace on $\N$. A closed densely defined linear operator $x:\dom(x)\ra\H$ is called \df{$\tau$-measurable} \cite{Segal:1953,Nelson:1974} if{}f any of the following equivalent conditions holds:\footnote{We give here the set of conditions that are equivalent for Nelson's notion of $\tau$-measurability \cite{Nelson:1974}, which is stronger than Segal's notion \cite{Segal:1953}. For a discussion of differences between these two notions, see \cite{Cecchini:1978}. Instead of $\rpktarget{pvm.nelson.nota}P^{\ab{x}}(]\lambda,+\infty[)$ the notation $\II-P^{\ab{x}}_\lambda:=\II-P^{\ab{x}}(]-\infty,\lambda])$ is also used.}
\begin{enumerate}
\item[1)] $\exists\lambda>0\;\;\tau(\pvm^{\ab{x}}(]\lambda,+\infty[))<\infty$, 
\item[2)] $\forall\epsilon_1>0\;\exists\epsilon_2>0\;\;\tau(\pvm^{\ab{x}}(]\epsilon_2,\infty[))\leq\epsilon_1$,
\item[3)] $\forall\epsilon>0\;\exists P\in\Proj(\N)$ such that $\tau(\II-P)<\epsilon$, and $P\H\subseteq\dom(x)$, 
\item[4)] $\lim_{\lambda\ra\infty}\tau(\pvm^{\ab{x}}(]\lambda,+\infty[))=0$. 
\end{enumerate}
The space of all $\tau$-measurable operators affiliated with $\N$ will be denoted by $\rpktarget{MMMNTAU}\MMM(\N,\tau)$. For $x,y\in\MMM(\N,\tau)$ the algebraic sum $x+y$ and algebraic product $xy$ may not be closed, hence in general they do not belong to $\MMM(\N,\tau)$. However, their closures (denoted with the abuse of notation by the same symbol) belong to $\MMM(\N,\tau)$. Moreover, $\II\in\MMM(\N,\tau)$. 

The space $\MMM(\N,\tau)$ can be equipped with a Hausdorff metrisable Cauchy complete topology, called \df{$\tau$-topology} \cite{Stinespring:1959,Nelson:1974}, given by the set of neighbourhoods of $0\in\MMM(\N,\tau)$,
\begin{align}
        N_{\epsilon_1,\epsilon_2}(0):&=\{x\in\MMM(\N,\tau)\mid\exists P\in\Proj(\N),\;\tau(\II-P)\leq\epsilon_1,\;P\H\subseteq\dom(x),\;\n{xP}\leq\epsilon_2\}\nonumber\\&=\{x\in\MMM(\N,\tau)\mid\tau(\pvm^{\ab{x}}(]\epsilon_2,\infty[))\leq\epsilon_1\},
\end{align}
where $\epsilon_1>0$, $\epsilon_2>0$. This turns $\MMM(\N,\tau)$ into a unital topological $*$-algebra, with sum and multiplication defined by closures of algebraic sum and multiplication \cite{Nelson:1974}. The algebra $\N$ is a $*$-subalgebra of $\MMM(\N,\tau)$ that is dense in $\tau$-topology. Moreover, the addition, multiplication, and conjugation in $\N$ have unique extensions to $\MMM(\N,\tau)$. It follows that given any semi-finite von Neumann algebra $\N$ and a faithful normal semi-finite trace $\tau$, the topological $*$-algebra $\MMM(\N,\tau)$ is defined uniquely as a completion of $\N$ in $\tau$-topology. In general, $\tau$-topology is not locally convex. However, the subspace
\begin{equation}
        \MMM_0(\N,\tau):=\{x\in\MMM(\N,\tau)\mid\tau(\pvm^{\ab{x}}(]\lambda,\infty[))<\infty\;\forall\lambda>0\}
\end{equation}
is a sequential space, that is, the convergence in $\tau$-topology coincides on $\MMM_0(\N,\tau)$ with the sequential convergence in $\tau$. The latter is defined as follows. A sequence $\{x_i\}\subseteq\MMM(\N,\tau)$ is called to \df{converge in $\tau$} to $x\in\MMM(\N,\tau)$ if{}f $\exists\{P_i\}\subseteq\Proj(\N)$ such that
\begin{enumerate}
\item[i)] $\lim_i\n{(x_i-x)P_i}=0$,
\item[ii)] $\lim_i\tau(\II-P_i)=0$.
\end{enumerate}
The space $\MMM_0(\N,\tau)$ is a two-sided ideal in $\MMM(\N,\tau)$. If $\N$ is a finite von Neumann algebra, then $x\in\MMM(\N)\limp x\in\MMM(\N,\tau)$. If $\N=\BH$ and $\tau$ is a standard trace $\tr$ on $\BH$, then $\MMM(\BH,\tr)=\BH$. For a more detailed study of $\MMM(\N,\tau)$ and various topologies on it, see \cite{Muratov:Chilin:2007}.

Consider the extension of a trace $\tau$ from $\N^+$ to $\aff(\N)^+$ given by
\begin{equation}
        \tau:\aff(\N)^+\ni x\mapsto\tau(x):=
        \sup_{n\in\NN}\left\{\tau\left(
                \int_0^n\pvm^x(\lambda)\lambda
        \right)\right\}
        \in[0,\infty],
\label{trace.extension.to.affN.plus}
\end{equation}
the map
\begin{equation}
        \n{\cdot}_p:\MMM(\N,\tau)\ni x\mapsto\n{x}_p:=(\tau(\ab{x}^p))^{1/p}\in[0,\infty],
\label{trace.p.norm}
\end{equation}
and the family of vector spaces\rpktarget{LPNTAU.ZWEI}
\begin{equation}
        L_p(\N,\tau):=\{x\in\MMM(\N,\tau)\mid\n{x}_p<\infty\},
\label{Lp.measurable}
\end{equation}
where $p\in[1,\infty[$. The map \eqref{trace.p.norm} is a norm on \eqref{Lp.measurable} \cite{Yeadon:1975}, and $L_p(\N,\tau)$ are Cauchy complete with respect to the topology of this norm. In addition, one defines $L_\infty(\N):=\N$. The Banach spaces $L_p(\N,\tau)$ defined this way coincide with the $L_p(\N,\tau)$ spaces defined at the beginning of this Section. The spaces $L_p(\N,\tau)$ embed continuously into $\MMM(\N,\tau)$, and are subsets of $\MMM_0(\N,\tau)$ \cite{Nelson:1974}. The space $\N\cap L_p(\N,\tau)$ is dense in $L_p(\N,\tau)$ with respect to the topology of $\n{\cdot}_p$ norm. For all $\gamma\in\,]0,1]$ \cite{Yeadon:1975}
\begin{equation}
	(x,y)\in L_{1/\gamma}(\N,\tau)\times L_{1/(1-\gamma)}(\N,\tau)\;\;\limp\;\;xy\in L_1(\N,\tau),
\end{equation}
and the duality
\begin{equation}
        L_{1/\gamma}(\N,\tau)\times L_{1/(1-\gamma)}(\N,\tau)\ni(x,y)
        \mapsto\duality{x,y}
        :=\tau(xy)\in\RR
\end{equation}
determines an isometric isomorphism of Banach spaces
\begin{equation}
        L_{1/\gamma}(\N,\tau)^\banach\iso L_{1/(1-\gamma)}(\N,\tau).
\end{equation}
The noncommutative analogue of the Rogers--H\"{o}lder inequality reads \cite{Yeadon:1975}
\begin{equation}
	\n{xy}_1\leq\n{x}_{1/\gamma}\n{y}_{1/(1-\gamma)}\;\;\forall(x,y)\in L_{1/\gamma}(\N,\tau)\times L_{1/(1-\gamma)}(\N,\tau).
\label{RH.ncLNtau.ineq}
\end{equation}
The special case of \eqref{RH.ncLNtau.ineq} were obtained in \cite{Grothendieck:1955,Garling:1967} for $(\N,\tau)=(\BH,\tr)$.

The space of \df{Riesz--Schauder} \cite{Riesz:1917,Schauder:1930} (or \df{compact}) operators over a Hilbert space $\H$,
\begin{equation}\rpktarget{SCHATT}
        \schatten_0(\H):=\overline{\{x\in\BH\mid\dim\ran(x)\leq\infty\}},
\end{equation}
where bar denotes the Cauchy completion in the norm of $\BH$, allows to define the space $\schatten_1(\H)$ of \df{trace class} (or \df{nuclear}) operators \cite{Schatten:vonNeumann:1946,Schatten:vonNeumann:1947} and the space $\schatten_2(\H)$ of \df{Hilbert--Schmidt} operators \cite{Schmidt:1907:a,vonNeumann:1927:Mathematische,Stone:1932} as a Cauchy completion of $\schatten_0(\H)$ in the norm  $\n{x}_1:=\tr(\ab{\sqrt{x^*x}})$ and $\n{x}_2:=\tr(x^*x)$, respectively. More generally, the spaces $\schatten_p(\H)$ of \df{von~Neumann--Schatten $p$-class} operators over a Hilbert space $\H$ are defined as \cite{vonNeumann:1937:Tomsk,Schatten:1946,Schatten:vonNeumann:1946,Schatten:vonNeumann:1947,Schatten:1950,Schatten:1960,Gokhberg:Krein:1965,Simon:1979}
\begin{equation}\rpktarget{SCHATT.P}
        \schatten_p(\H):=
        \{x\in\schatten_0(\H)\mid\n{x}_p:=
        \tr((x^*x)^{p/2})^{1/p}<\infty\},
\end{equation}
for $p\in[1,\infty[$, and they are Banach spaces with respect to the norm $\n{\cdot}_p$ for $p\in[1,\infty[$. In addition, one sets $\schatten_\infty(\H):=\BH$ with $\n{x}_\infty:=\n{x}_\BH$. The spaces $\schatten_p(\H)$ are uniformly convex and uniformly Fr\'{e}chet differentiable for $p\in\,]1,\infty[$ \cite{Dixmier:1953,McCarthy:1967,Lai:1973}, and the following Banach space dualities hold \cite{Schatten:1960,McCarthy:1967}:
\begin{equation}
        \schatten_0(\H)^\banach\iso\schatten_1(\H),\;\;\;
        \schatten_1(\H)^\banach\iso\schatten_\infty(\H),\;\;\;
        \schatten_{1/\gamma}(\H)^\banach\iso\schatten_{1/(1-\gamma)}(\H),
\end{equation}
for $\gamma\in\,]0,1]$. If $\N\subseteq\BH$, then \cite{Dixmier:1953,Dixmier:1954}
\begin{equation}
        \forall\omega\in\N^\banach\;\;
        \left(
                \omega\in\N_\star\;\;
                \iff\;\;
                \exists x\in\schatten_1(\H)\;\;\;
                \omega(\cdot)=\tr_\BH(x\,\cdot\,)
        \right).
\label{normal.state.is.trace.class.op}
\end{equation}
In such case $\n{\omega}=\tr(x)$. This theorem holds also for $(\omega,x)\in\N_\star^+\times\schatten_1(\H)^+$, as well as for $(\omega,x)\in\N^+_{\star1}\times\schatten_1(\H)^+_1$. However, the uniqueness of $x$ in \eqref{normal.state.is.trace.class.op}, as well as in its positive and normalised cases, holds only for $\N=\BH$, because in such case \eqref{normal.state.is.trace.class.op} defines a linear isometry $L_1(\BH,\tr)\iso\schatten_1(\H)\iso\BH_\star$ \cite{Dixmier:1950,Schatten:1950}. More generally, if $\N\subseteq\BH$, then \cite{Dixmier:1953}
\begin{equation}
        \N_\star\iso\schatten_1(\H)/\{x\in\schatten_1(\H)\mid\tr(xy)=0\;\forall y\in\N\}.
\end{equation}
The space $\schatten_2(\H)$ can be equipped with the inner product 
\begin{equation}
        \s{x,y}_{\schatten_2(\H)}:=\tr(y^*x)\;\;\forall x,y\in\schatten_2(\H),
\end{equation}
which turns it into a Hilbert space, called the \df{Hilbert--Schmidt space}\footnote{A spectral theory of bounded operators associated with this space was a subject of analysis in \cite{Hilbert:1904,Hilbert:1905,Hilbert:1906:a:b:c,Schmidt:1907:a,Schmidt:1907:b,Schmidt:1907:c,Schmidt:1908,Hilbert:1910}.}. The von~Neumann--Schatten $\schatten_p(\H)$ spaces can be characterised by
\begin{equation}
        \schatten_p(\H)=L_p(\BH,\tr)\;\;\forall p\in[1,\infty].
\end{equation}
\subsection{Operator valued weights}
Given a $W^*$-algebra $\N$, the \df{extended positive cone} $\rpktarget{NEXT}\N^\ext$ is defined as set of maps $m:\N_\star^+\ra[0,\infty]$ such that for all $\phi,\psi\in\N_\star^+$
\begin{enumerate}
\item[1)] $m(\lambda\phi)=\lambda m(\phi)\;\;\forall\lambda\geq0$,
\item[2)] $m(\phi+\psi)=m(\phi)+m(\psi)$,
\item[3)] $m$ is weakly lower semi-continuous, that is,
\begin{equation}
        \sup_\iota\{\omega_\iota(x)\}=
        \omega(x)\;\;\limp\;\;
        m(\omega)\leq\lim\inf_\iota\{m(\omega_\iota)\}
        \;\;\forall\omega,\omega_\iota\in\N_\star^+,
\end{equation}
or, equivalently,
\begin{equation}
        \mbox{the sets }\{\omega\in\N_\star^+\mid m(\omega)>\lambda\}\mbox{ are weakly open }\forall\lambda\in\RR.
\end{equation}
\end{enumerate}
The set $\N^\ext$ can be considered as the `set of normal weights on $\N_\star$'. It contains $\N^+$, and is closed under addition, multiplication by nonnegative scalars, and increasing limits of nets. If $\N$ is a von Neumann algebra acting on some Hilbert space $\H$, then $\aff(\N)^+\subseteq\N^\ext$. For all $m_1,m_2\in\N^\ext$, $x\in\N$, $\phi\in\N_\star^+$ and $\lambda\in\RR^+$ one defines
\begin{align}
        (\lambda m)(\phi)&:=\lambda m(\phi),\\
        (m_1+m_2)(\phi)&:=m_1(\phi)+m_2(\phi),\label{add.m.ext}\\
        (x^*mx)(\phi)&:=m(\phi(x\,\cdot\,x^*)).\label{pert.m.ext}
\end{align}
Every $m\in\N^\ext$ has a unique spectral decomposition
\begin{equation}
        m(\phi)=
        \int_0^\infty\phi(P^m(\lambda))+\infty\cdot\phi(P^m)
        \;\;\forall\phi\in\N_\star^+,
\label{spec.decomp.m}
\end{equation}
where $\{P^m(\lambda)\in\Proj(\N)\mid\lambda\in\RR^+\}$ is an increasing family which is ultrastrongly continuous from the right, and $P^m=\II-\lim_{\lambda\ra\infty}P^m(\lambda)$. Moreover,
\begin{align}
P^m(0)=0&\iff m\mbox{ is \df{faithful}  }(\mbox{i.e. }m(\phi)>0\;\;\forall\phi\in\N_\star^+\setminus\{0\}),\\
P^m=0&\iff m\mbox{ is \df{semi-finite} }(\mbox{i.e. }\{\phi\in\N_\star^+\mid m(\phi)<\infty\}\mbox{ is dense in }\N_\star^+).\label{m.semifinite}
\end{align}
Every normal weight on $\N$ has a unique extension $\rpktarget{TILDEPHI}\tilde{\phi}$ to $\N^\ext$ satisfying
\begin{enumerate}
\item[1)] $\tilde{\phi}(\lambda x)=\lambda\tilde{\phi}(x)\;\;\forall\lambda\geq0\;\forall x\in\N^\ext$,
\item[2)] $\tilde{\phi}(x+y)=\tilde{\phi}(x)+\tilde{\phi}(y)\;\;\forall x,y\in\N^\ext$,
\item[3)] $\tilde{\phi}(\sup_\iota\{m_\iota\})=\sup_\iota\{\tilde{\phi}(m_\iota)\}\;\;\forall$ increasing $\{m_\iota\}\subseteq\N^\ext$,
\end{enumerate} 
where $\{m_\iota\}$ is said to be \df{increasing} if{}f $\{m_\iota(\omega)\}$, with $\{m_\iota(\omega)\}\subseteq[0,\infty]$, is increasing for all $\omega\in\N_\star^+$. The proof of this proposition is based on the fact that every normal weight satisfies $\phi=\sum_i\phi_i$ for some family $\{\phi_i\}\subseteq\N_\star^+$ \cite{Pedersen:Takesaki:1973}, which allows to write
\begin{equation}
        m(\phi)
        =\lim_{n\ra\infty}\phi_i\left(
        \int_o^nP^m(\lambda)\lambda+nP^m
        \right)
        =:\lim_{n\ra\infty}\sum_i\phi_i(x_n)
        =\sum_im(\phi_i),
\label{m.as.a.sum}
\end{equation}
where $P^m(\lambda)$ and $P^m$ are given by \eqref{spec.decomp.m}. The formula \eqref{m.as.a.sum} provides a generalisation of \eqref{trace.extension.to.affN.plus}. It allows also to prove that for every faithful normal semi-finite trace $\tau$ the function
\begin{equation}
        \N^+\times\N^+\ni(x,y)\mapsto\tau(x^{1/2}yx^{1/2})=\tau(y^{1/2}xy^{1/2})\in[0,\infty]
\end{equation}
can be extended to $\N^\ext\times\N^\ext$ by a formula\rpktarget{TILDETAU}
\begin{equation}
        \N^\ext\times\N^\ext\ni(x,y)\mapsto\tilde{\tau}(x^{1/2}yx^{1/2}):=\sup_{i,j}\{\tau(x_j^{1/2}y_ix^{1/2}_j)\},
\label{tau.ext.extension}
\end{equation}
which is characterised as a unique extension of $\tau$ to $\N^\ext\times\N^\ext$ such that
\begin{enumerate}
\item[1)] $\tilde{\tau}(\lambda x)=\lambda\tilde{\tau}(x)$,
\item[2)] $\tilde{\tau}(x+y)=\tilde{\tau}(x)+\tilde{\tau}(y)$,
\item[3)] $(\sup_{\iota_1}\{x_{\iota_1}\}=x$, $\sup_{\iota_2}\{y_{\iota_2}\}=y)$ $\limp$ $\sup_{\iota_1,\iota_2}\{\tilde{\tau}(x^{1/2}_{\iota_1}y_{\iota_2}x^{1/2}_{\iota_1})\}=\tilde{\tau}(x^{1/2}yx^{1/2})$,
\item[4)] $\tilde{\tau}(x^{1/2}yx^{1/2})=\tilde{\tau}(y^{1/2}xy^{1/2})$,
\end{enumerate}
for all $\lambda\in\RR^+$, all $x,y\in\N^\ext$, and all increasing nets $\{x_{\iota_1}\},\{y_{\iota_2}\}\subseteq\N^\ext$. For every $y\in\N^+$, the extended trace $\tilde{\tau}$ defines a normal weight $\tilde{\tau}_x$ on $\N$ given by\rpktarget{TILDETAUX}
\begin{equation}
        \tilde{\tau}_x(y):=\sup_i\{\tilde{\tau}(x_i^{1/2}yx_i^{1/2})\}.
\label{weight.from.ext.trace}
\end{equation}
The map
\begin{equation}
        \N^\ext\ni x\mapsto
        \tilde{\tau}_x\in\{\mbox{all normal weights on }\N\}
\label{ext.normal.weights.bijection}
\end{equation}
is a bijection that preserves additivity, multiplication by $\lambda\in\RR^+$, ordering, and suprema. If $x\in\aff(\N)^+$ and $y\in\N^+$, then the definition \eqref{tau.ext.extension} coincides with \eqref{omega.from.tau}, and by \eqref{m.semifinite} the bijection \eqref{ext.normal.weights.bijection} turns to a bijection
\begin{equation}
        \aff(\N)^+\ni x\mapsto\tilde{\tau}_x\in\W(\N).
\end{equation}
This way \eqref{weight.from.ext.trace} leads to Haagerup's generalisation \cite{Haagerup:1979:ovw1} of the noncommutative Radon--Nikod\'{y}m theorem \eqref{Dye.Segal.weights}:
\begin{equation}
        \forall\mbox{ normal weight }\phi\;\mbox{ on }\N\;\;
        \exists!h\in\N^\ext\;\;\;\phi=\tilde{\tau}_h.
\label{Haagerup.ncRN.ext}
\end{equation}

Given $W^*$-algebras $\N_1$, $\N_2$ such that $\N_2\subseteq\N_1$, an \df{operator valued weight} from $\N_1$ to $\N_2$ is a map $\rpktarget{T.MAP}T:\N_1^+\ra\N_2^\ext$ satisfying \cite{Haagerup:1979:ovw1,Haagerup:1979:ovw2,Hirakawa:1992,Falcone:Takesaki:1999} 
\begin{enumerate}
\item[1)] $T(\lambda x)=\lambda T(x)\;\;\forall\lambda\geq0\;\forall x\in\N^+_1$,
\item[2)] $T(x+y)=T(x)+T(y)\;\;\forall x,y\in\N_1^+$,
\item[3)] $T(y^*xy)=y^*T(x)y\;\;\forall x\in\N_1^+\;\forall y\in\N_2$.
\end{enumerate}
It is called: \df{normal} if{}f $\sup_\iota\{x_\iota\}= x$ $\limp$ $\sup_\iota\{T(x_\iota)\}=T(x)$ $\forall x_\iota,x\in\N_1^+$; \df{faithful} if{}f $T(x^*x)=0$ $\limp$ $x=0$; \df{semi-finite} if{}f the set
\begin{equation}
        \rpktarget{nnnT}\nnn_T:=\{x\in\N_1\mid\n{T(x^*x)}\leq\infty\}
\end{equation}
is weakly-$\star$ dense in $\N_1$. If $T:\N_1^+\ra\N^\ext_2$ is a normal (respectively, faithful, or semi-finite) operator valued weight and $\phi$ is a normal (respectively, faithful, or semi-finite) weight on $\N_2$, then $\tilde{\phi}\circ T$ is a normal (respectively, faithful, or semi-finite) weight on $\N_1$. While Connes' cocycles are noncommutative analogues of the Radon--Nikod\'{y}m quotients, \cytat{application of an operator valued weight (...) can and should be thought of as ``partial integration''} \cite{Falcone:1996}. If $T:\N^+_1\ra\N_2^\ext$ is a faithful normal semi-finite operator valued weight then \cite{Haagerup:1979:ovw1}
\begin{align}
        \sigma^{\tilde{\psi}\circ T}_t(x)&=\sigma^\phi_t(x)\;\;\forall x\in\N_2\;\forall\phi\in\W_0(\N_1),\label{pre.nc.tonelli}\\
        \Connes{(\tilde{\psi}\circ T)}{(\tilde{\phi}\circ T)}{t}&=\Connes{\psi}{\phi}{t}\;\;\forall\phi,\psi\in\W_0(\N_1).
\end{align}
The property \eqref{pre.nc.tonelli} can be used to characterise $T$. According to Haagerup's theorem \cite{Haagerup:1979:ovw1,Falcone:Takesaki:1999}, there exists a faithful normal semi-finite operator valued weight $T:\N_1^+\ra\N_2^\ext$ if{}f there exist $\phi\in\W_0(\N_1)$ and $\psi\in\W_0(\N_2)$ such that
\begin{equation}
        \sigma^\phi_t(x)=\sigma^\psi_t(x)\;\;\forall x\in\N_2.
\end{equation}
If this condition is satisfied, then $T$ is uniquely determined by the `noncommutative Tonelli theorem'
\begin{equation}
        \phi=\tilde{\psi}\circ T.
\end{equation}
In a special case, when $\N_1$ and $\N_2$ are semi-finite von Neumann algebras, $\N_2\subseteq\N_1$, $\tau_1$ is a faithful normal semi-finite trace on $\N_1$, while $\tau_2$ is a faithful normal semi-finite trace on $\N_2$, then there exists a unique faithful normal semi-finite operator valued weight $T:\N_1^+\ra\N_2^\ext$ such that
\begin{equation}
        \tau_1=\tilde{\tau}_2\circ T.
\end{equation}

Given $W^*$-algebras $\N_1,\N_2$ such that $\N_2\subseteq\N_1$, the \df{conditional expectation} from $\N_1$ to $\N_2$ is defined as a map $\rpktarget{condexp}\condexp:\N_1\ra\N_2$ such that \cite{vonNeumann:1943,Dixmier:1953,Moy:1954,Nakamura:Turumaru:1954,Umegaki:1954,Umegaki:1956,Umegaki:1959,Umegaki:1962}\footnote{See \cite{Stoermer:1997,Stoermer:2013} for an overview.}
\begin{enumerate}
\item[1)] $\condexp(\lambda_1x_1+\lambda_2x_2)=\lambda_1\condexp(x_1)+\lambda_2\condexp(x_2)\;\;\forall x_1,x_2\in\N_1,\;\forall\lambda_1,\lambda_2\in\CC$,
\item[2)] $\condexp(x)=x$ $\forall x\in\N_2$,
\item[3)] $x\geq0\;\;\limp\;\;\condexp(x)\geq0$.
\end{enumerate}
Instead of 2) a weaker condition can be equivalently used:
\begin{enumerate}
\item[2')] $\condexp(\II)=\II$.
\end{enumerate}
From \cite{Tomiyama:1957,Tomiyama:1958,Tomiyama:1959,Choi:Effros:1974} it follows that the conditions 1)-3) imply the following equivalent properties
\begin{enumerate}
\item[4)] $\condexp(y_1xy_2)=y_1\condexp(x)y_2\;\;\forall x\in\N_1\;\forall y_1,y_2\in\N_2$,
\item[4')] $\condexp(y^*xy)=y^*\condexp(x)y\;\;\forall x\in\N_1\;\forall y\in\N_2$,
\item[4'')] $\condexp(x\condexp(y))=\condexp(x)\condexp(y)\;\;\forall x,y\in\N_1$,
\item[4''')] $\condexp(xy)=\condexp(x)y\;\;\forall(x,y)\in\N_1\times\N_2$.
\end{enumerate}
A conditional expectation $\condexp$ is called: \df{faithful} if{}f $\condexp(x)=0$ $\limp$ $x=0$ $\forall x\in\N_1^+$; \df{normal} if{}f $\sup_\iota\{x_\iota\}=x$ $\limp$ $\sup_\iota\{\condexp(x_\iota)\}=\condexp(x)$ for every bounded increasing net $\{x_\iota\}\subseteq\N_1^+$; \df{$\omega$-stable} (or \df{$\omega$-invariant}) for $\omega\in(\N_1)_\star^+$ if{}f $\omega|_{\N_2}\circ\condexp=\omega$. If $\omega\in(\N_1)^+_{\star01}$ and $\condexp$ is $\omega$-stable, then $\condexp$ is faithful normal. According to the Takesaki--Golodec theorem \cite{Takesaki:1972,Golodec:1972}, if $\N_1,\N_2$ are $W^*$-algebras such that $\N_2\subseteq\N_1$ and $\phi\in\N^+_{\star0}$, then there exists a $\phi$-stable conditional expectation $\condexp:\N_1\ra\N_2$ if{}f $\sigma^\phi|_{\N_2}(\N_2)=\N_2$. If such $\condexp$ exists, then it is unique. For any operator valued weight $T:\N_1^+\ra\N_2^\ext$ satisfying $T(\II)=\II$ there exists a conditional expectation $\condexp:\N_1\ra\N_2$ such that $T$ is a restriction of $\condexp$ to $\N_1^+$. Conversely, every operator valued weight $T$ has a unique linear extension $T:\mmm_T\ra\N_2$, where
\begin{align}
        \rpktarget{mmmT}\mmm_T:&=\Span_\CC\{x^*y\mid x,y\in\N_1,\;\;\n{T(x^*x)}<\infty,\;\;\n{T(y^*y)}<\infty\}\\
        &=\Span_\CC\{x\in\N^+_1\mid\n{T(x)}<\infty\}=\Span_\CC\mmm^+_T,
\end{align}
which satisfies
\begin{equation}
        T(y_1xy_2)=y_1T(x)y_2\;\;\forall x\in\mmm_T\;\forall y_1,y_2\in\N_2.
\end{equation}
The sets $\nnn_T$ and $\mmm_T$ are bimodules over $\N_2$, while $T(\mmm_T)$ is a weakly-$\star$ dense two sided ideal of $\N_2$. If $T(\II)=\II$ then the above extension of $T$ is a conditional expectation from $\N_1$ to $\N_2$. (The notions of an operator valued weight generalises this way the \textit{unbounded} conditional expectations of \cite{Combes:Delaroche:1975}.)

Given a $W^*$-dynamical system $(\N,G,\alpha)$ with a locally compact abelian group $G$, 
the formula
\begin{equation}
        T(x):=\int_G\tmu_L^G(g)\alpha_g(x)
\end{equation}
defines a normal operator valued weight $T:\N^+\ra\N_\alpha^\ext$. We say that $\alpha$ is \df{integrable} if{}f the set $\{x\in\N\mid T(x^*x)\in\N^+\}$ is weakly-$\star$ dense in $\N$. For $x\in(\N\rtimes_\alpha G)^+$, the map $x\mapsto\int_\D\tmu_L^{\hat{G}}(\hat{b})\hat{\alpha}_{\hat{b}}(x)$ is weakly-$\star$ continuous for any compact $\D\subseteq\hat{G}$ and\rpktarget{HATT}
\begin{equation}
        \hat{T}(x):=\int_{\hat{G}}\tmu_L^{\hat{G}}(\hat{b})\hat{\alpha}_{\hat{b}}(x)\;\;\forall x\in(\N\rtimes_\alpha G)^+
\label{nsf.ovw.cp.G}
\end{equation}
defines a faithful normal semi-finite operator valued weight $\hat{T}:(\N\rtimes_\alpha G)^+\ra(\N\rtimes_\alpha G)^\ext_{\hat{\alpha}}=(\pi_\alpha(\N))^\ext$ which is characterised by
\begin{equation}
        (\hat{T}(x))(\omega)=\int_{\hat{G}}\tmu_L^{\hat{G}}(b)(\hat{\alpha}_{\hat{b}}(x))(\omega)\;\;\forall\omega\in(\pi_\alpha(\N))^+_\star.
\end{equation}
It satisfies \cite{Haagerup:1978:dualweights:II}
\begin{equation}
        \hat{T}(u_G(g)xu_G(g)^*)=u_G(g)\hat{T}(x)u_G(g)^*\;\;\forall x\in(\N\rtimes_\alpha G)^+\;\forall g\in G.
\end{equation}
Given $\phi\in\W(\N)$, the \df{dual weight} $\rpktarget{HATPHI}\hat{\phi}$ on $\N\rtimes_\alpha G$ is defined as\footnote{Such definition of a dual weight is sufficient for our purposes. However, it is only a special case of the theory of dual weights on crossed products, developed in \cite{Takesaki:1973:duality,Digernes:1974,Digernes:1975,Sauvageot:1974,Sauvageot:1977,Haagerup:1976,Haagerup:1979:ovw2,Stratila:Voiculescu:Zsido:1976,Stratila:Voiculescu:Zsido:1977,Zsido:1979,Stratila:1981}.} \cite{Haagerup:1978:dualweights:I,Haagerup:1978:dualweights:II}
\begin{equation}
        \hat{\phi}:=\tilde{\phi}\circ\pi_\alpha^{-1}\circ\hat{T}.
\label{dual.weight.general.def}
\end{equation}

Given a $W^*$-algebra $\N$ and $\psi\in\W_0(\N)$, consider  $\precore:=\N\rtimes_{\sigma^\psi}\RR$ where $\sigma^\psi$ is a modular automorphism group induced by $\psi$ on $\N$. Formula \eqref{nsf.ovw.cp.G} turns in this case to
\begin{equation}\rpktarget{HATT.ZWEI}
        \hat{T}:\precore^+\ni x\mapsto\hat{T}(x):=\int_\RR\dd s\,\hat{\sigma}^\psi_s(x)\in\precore^\ext_{\hat{\sigma}^\psi}=(\pi_{\sigma^\psi}(\N))^\ext.
\end{equation}
Hence, the map
\begin{equation}
        \precore^+\ni x\mapsto\pi^{-1}_{\sigma^\psi}\left(\int_\RR\dd s\,\hat{\sigma}^\psi_s(x)\right)\in\N^\ext,
\end{equation}
where $\pi_{\sigma^\psi}$ is a faithful normal representation determined by \eqref{pi.sigma}, is a faithful normal semi-finite operator valued weight \cite{Haagerup:1979:ovw2}. So, given any normal weight $\phi$ on a $W^*$-algebra $\N$, its dual weight $\hat{\phi}$ on $\N\rtimes_{\sigma^\psi}\RR$ is given by \eqref{dual.weight.general.def},
\begin{equation}\rpktarget{HATPHI.ZWEI}
        \hat{\phi}:=\tilde{\phi}\circ\pi^{-1}_{\sigma^\psi}\left(\int_\RR\dd s\,\hat{\sigma}^\psi_s(\cdot)\right):\precore^+\ra[0,\infty].
\end{equation}
The dual weight $\hat{\phi}$ is also normal. If $\phi$ is faithful (respectively, semi-finite), then $\hat{\phi}$ is also faithful (respectively, semi-finite). From $\hat{T}(x)\in\precore^\ext_{\hat{\alpha}}\;\forall x\in\precore^+$ one has $\hat{\sigma}^\psi_s\circ\hat{T}=\hat{T}\;\forall s\in\RR$, where the extension of $\hat{\sigma}^\psi$ to $\precore^\ext$ is given by 
\begin{equation}
        \tilde{\psi}(\hat{\sigma}^\psi(x))=
        (\tilde{\psi}\circ\hat{\sigma}^\psi)(x)
        \;\;\forall\precore^\ext\;\forall\psi\in\W_0(\precore).
\end{equation}
Hence,
\begin{equation}
        \hat{\phi}\circ\hat{\sigma}^\psi_s=
        \hat{\phi}\;\forall s\in\RR.
        \label{hat.psi.invariance}
\end{equation}
Moreover, the map
\begin{equation}
        \W(\N)\ni\phi\mapsto\hat{\phi}\in\{\varphi\in\W_0(\precore)\mid\varphi\circ\hat{\sigma}^\psi_s=\varphi\;\;\forall s\in\RR\}
\end{equation}
is a bijection such that
\begin{align}
        \widehat{\phi+\psi}&=\hat{\phi}+\hat{\psi},\label{dual.weight.add}\\
        \widehat{\phi(x\,\cdot\,x^*)}&=\hat{\phi}(x\,\cdot\,x^*),\\
        \supp(\hat{\phi})&=\supp(\phi).\label{dual.weight.supp}
\end{align}
\subsection{Integration relative to a weight}
Now we can consider Haagerup's approach to construction of noncommutative $L_p(\N,\psi)$ spaces. A starting point of this approach is an observation that only semi-finite von Neumann algebras admit faithful normal semi-finite traces, and the crossed product $\precore:=\N\rtimes_{\sigma^\psi}\RR$ of an arbitrary von Neumann algebra $\N$ with the action of modular automorphism group $\sigma^\psi$ of $\psi\in\W_0(\N)$ is a semi-finite von Neumann algebra. So, one can try to integrate over $\N$ using a suitably defined trace on $\hat{\N}$, provided that this procedure can cover all $\N$. Because $\N$ is the fixed point subalgebra $\precore_{\hat{\sigma}^\psi}$ of $\precore$ under the dual action $\hat{\sigma}^\psi$, it seems that one can use the Pedersen--Takesaki noncommutative Radon--Nikod\'{y}m type theorem to integrate over full $\N$. However, in order to determine a \textit{natural} choice of trace $\tau_\psi$ over $\precore$ that corresponds to $\psi\in\W_0(\N)$, one needs to use dual weights, which requires to use operator valued weights and Haagerup's version of the noncommutative Radon--Nikod\'{y}m theorem.

A faithful normal semi-finite trace $\tau_\psi$ on $\precore$, uniquely determined by the equation
\begin{equation}
\Connes{\hat{\psi}}{\tau_\psi}{t}=u_\RR(t)=:h^{\ii t}\;\;\forall t\in\RR,
\label{nat.trace.Connes}
\end{equation}
will be called a \df{natural trace}. From \eqref{quasiinv.under.sigma.hat} one has $\hat{\sigma}^\psi_s(h)=\ee^{-s}h$. Together with $\hat{\psi}\circ\hat{\sigma}^\psi=\hat{\psi}$, which follows from \eqref{hat.psi.invariance}, this gives
\begin{align}
        \tau_\psi\circ\hat{\sigma}^\psi_s(x)
        &=\tilde{\psi}\circ\hat{T}(h^{-1/2}\hat{\sigma}^\psi_s(x)h^{-1/2})
        =\tilde{\psi}\circ\hat{T}
        \left(
                (\hat{\sigma}^\psi_{-s}(h^{-1}))^{1/2}
                x
                (\hat{\sigma}^\psi_{-s}(h^{-1}))^{1/2}
        \right)\\
        &=e^{-s}\tilde{\psi}\circ\hat{T}(h^{-1/2}xh^{-1/2})
        =e^{-s}\tau_\psi(x)\;\;\forall x\in\precore^+,
\end{align}
hence
\begin{equation}
        \tau_\psi\circ\hat{\sigma}^\psi_s=\ee^{-s}\tau_\psi\;\;\forall s\in\RR.
\end{equation}
By \eqref{Haagerup.ncRN.ext}, for any normal weight $\phi$ on $\N$ there exists a unique operator $\rpktarget{H.PHI}h_\phi\in\N^\ext$ that satisfies
\begin{equation}
        \tilde{\tau}_\psi(h_\phi^{1/2}xh_\phi^{1/2})=\hat{\phi}(x)\;\;\forall x\in\precore^+,
\label{h.phi.radon.nikodym.PT}
\end{equation}
which is equivalent to a uniqueness of the bijection between $\phi$ and a family of partial isometries given by Connes' cocycle,
\begin{equation}
        \Connes{\hat{\phi}}{\tau_\psi}{t}=h^{\ii t}_\phi,
\label{h.phi.radon.nikodym.Connes}
\end{equation}
whenever $\phi\in\W(\N)$. The operator $h_\phi$ plays a role of a noncommutative Radon--Nikod\'{y}m quotient of a normal weight $\hat{\phi}$ on $\precore$ with respect to a natural trace $\tau_\psi$. The map $\phi\mapsto h_\phi$ defines a bijection between the set $\W(\N)$ and a subset $\NNN^1_\psi(\precore)$ of such elements of $\aff(\precore)^+$ that satisfy $\hat{\sigma}^\psi_s(h_\phi)=\ee^{-s}h_\phi$. For any normal weight $\phi$ on $\N$, $\hat{\phi}=(\tilde{\tau}_\psi)_{h_\phi}$ determined by \eqref{h.phi.radon.nikodym.PT} satisfies \eqref{dual.weight.add}-\eqref{dual.weight.supp}, which can be expressed in terms of noncommutative Radon--Nikod\'{y}m quotient as \cite{Haagerup:1979:ovw1}
\begin{align}
        (\tilde{\tau}_\psi)_{h_1+h_2}
        &=(\tilde{\tau}_\psi)_{h_1}+(\tilde{\tau}_\psi)_{h_2},\\
        (\tilde{\tau}_\psi)_{xhx}
        &=(\tilde{\tau}_\psi)_{h}(x\,\cdot\,x^*),\\
        \supp((\tilde{\tau}_\psi)_h)
        &=\supp(h),
\end{align}
where $h_1+h_2$ and $xhx^*$ are understood in terms of \eqref{add.m.ext} and \eqref{pert.m.ext}, respectively. As a result, the map $\phi\mapsto h_\phi$ satisfies
\begin{align}
        h_{\phi+\varphi}&=h_\phi+h_\varphi,\\
        h_{\phi(x\,\cdot\,x^*)}&=xh_\phi x^*,\\
        \supp(h_\phi)&=\supp(\phi).
\end{align}

From the property (see \cite{Kosaki:1980:PhD} or \cite{Terp:1981} for a proof)
\begin{equation}
        \phi\in\W(\N)\;\;\limp\;\;\tau_\psi\left(\pvm^{h_\phi}(]\lambda,\infty[)\right)=\frac{1}{\lambda}\phi(\II)\;\;\forall\lambda>0
\end{equation}
it follows that the operator $h_\phi$ defined by \eqref{h.phi.radon.nikodym.PT} or \eqref{h.phi.radon.nikodym.Connes} is $\tau_\psi$-measurable if{}f $\phi\in\N_\star^+$. Let $\MMM(\precore,\tau_\psi)$ denote the completion of $\precore$ in $\tau_\psi$-topology. The space $\MMM(\precore,\tau_\psi)$ is a topological $*$-algebra, and can be always represented as a space of all closed densely defined $\tau_\psi$-measurable operators affiliated with $\precore$. Let the extension of $\hat{\sigma}^\psi_s$, $s\in\RR$, from $\precore$ to $\MMM(\precore,\tau_\psi)$ and $\N^\ext$ be denoted, with the abuse of notation, by the same symbol. Then the algebraic component of the structure of noncommutative $L_p(\N,\psi)$ space, where $p\in[1,\infty]$, is defined by\rpktarget{MMMPPRECORE}
\begin{equation}
        \MMM^p(\precore,\tau_\psi):=\left\{
        x\in\MMM(\precore,\tau_\psi)\mid 
                \hat{\sigma}^\psi_s(x)=\ee^{-\frac{s}{p}}x
        \right\}.
\label{nc.Lp}
\end{equation}
This way $\MMM(\precore,\tau_\psi)$ becomes a `container' for all noncommutative $L_p(\N,\psi)$ spaces. Every space $\MMM^p(\precore,\tau_\psi)$ is a self-adjoint linear subspace of $\MMM(\precore,\tau_\psi)$, closed under left and right multiplication by the elements of $\N$. If $x=v\ab{x}\in\MMM(\precore,\tau_\psi)$, then
\begin{equation}
        x\in \MMM^p(\precore,\tau_\psi)
        \iff
        (v\in\N,\;\;\ab{x}\in \MMM^p(\precore,\tau_\psi))
        \iff
        (v\in\N,\;\;\ab{x}^p\in\MMM^1(\precore,\tau_\psi)).
\end{equation}
The second equivalence is provided by the \df{Mazur map} $\N^+\ni x\mapsto x^p\in\N^+$, $p\in\,]0,\infty[$, extended by continuity to $\MMM(\precore,\tau_\psi)\ni x\mapsto x^p\in\MMM(\precore,\tau)$,
\begin{equation}
 x\in \MMM^p(\precore,\tau_\psi)\iff 
 x^p\in\MMM^1(\precore,\tau_\psi)\;\;
 \forall x\in\MMM(\precore,\tau_\psi)^+.
\end{equation}
By definition, $\N=\MMM^\infty(\precore,\tau_\psi)$. On the other hand, all elements of $\MMM^p(\precore,\tau_\psi)$ for all $p\neq\infty$ are unbounded \cite{Haagerup:1979:ncLp,Terp:1981}. All spaces $\MMM^p(\precore,\tau_\psi)$ inherit the $\tau_\psi$-topology of $\MMM(\precore,\tau_\psi)$, and all are sequential spaces with respect to it. 

Every $\phi\in\N_\star$, considered as a linear form on $\N$, has a unique polar decomposition $\phi=\ab{\phi}(\,\cdot\,u)$, and a polar decomposition of $h_\phi$ is $h_\phi=u\ab{h_\phi}=uh_{\ab{\phi}}$. This defines a unique extension of the map $\N_\star^+\ni\phi\mapsto h_\phi\in\MMM(\precore,\tau_\psi)^+$ to a linear bijection
\begin{equation}
        \N_\star\ni\phi\mapsto h_\phi\in\MMM^1(\precore,\tau_\psi)=\NNN^1_\psi(\precore)\cap\MMM(\precore,\tau_\psi).
\end{equation}
preserving positivity, conjugation, ordering, polar decomposition and the action of $\N$ \cite{Terp:1981}. For $\phi\in\W_0(\N)$ the operator $h_\phi$ is strictly positive (invertible). The space $\MMM^1(\precore,\tau_\psi)$ can be equipped with a bounded positive linear functional $\Tr:\MMM^1(\precore,\tau_\psi)\ra\CC$,
\begin{equation}
       \Tr(h_\phi):=\phi(\II)\;\;\forall\phi\in\N_\star.
\end{equation}
By polar decomposition,
\begin{equation}
        \Tr(\ab{h_\phi})
        =\Tr(h_{\ab{\phi}})
        =\ab{\phi}(\II)
        =\n{\phi}_{\N_\star}
        \;\;\forall\phi\in\N_\star,
\end{equation}
so
\begin{equation}
        \ab{\Tr(x)}\leq\Tr(\ab{x})\;\;\forall x\in\MMM^1(\precore,\tau_\psi).
\label{mod.Tr.leq.Tr.mod}
\end{equation}
As a result, the map
\begin{equation}
        x\mapsto\n{x}_1:=\Tr(\ab{x})
\label{Haagerup.L.one.norm}
\end{equation}
is a norm on $\MMM^1(\precore,\tau_\psi)$. The Mazur map enables one to define the corresponding norms for $p\in\,]1,\infty[$ by
\begin{equation}                \n{x}_p:=\n{\ab{x}^p}_1^{1/p}=(\Tr(\ab{x}^p))^{1/p}\;\;\forall x\in\MMM^p(\precore,\tau_\psi).
\end{equation}
For $p=\infty$, we define 
\begin{equation}
        \n{x}_\infty=\n{x}_\N\;\;\forall x\in\N=\MMM^\infty(\precore,\tau_\psi).
\end{equation}
For each $p\in[1,\infty]$ the space $\MMM^p(\precore,\tau_\psi)$ is Cauchy complete in the topology generated by $\n{\cdot}_p$. This topology coincides with the topology induced on $\MMM^p(\precore,\tau_\psi)$ from $\MMM(\precore,\tau_\psi)$. The Banach spaces $(\MMM^p(\precore,\tau_\psi),\n{\cdot}_p)$ will be denoted by $\rpktarget{HTLP}L_p(\N,\psi)$, and called the \df{Haagerup--Terp spaces}. From \eqref{Haagerup.L.one.norm} it follows that the map $\phi\mapsto h_\phi$ defines an isometric isomorphism $\N_\star\iso L_1(\N,\phi)$. We have also $L_\infty(\N,\psi)=\N$. For $\gamma\in\,]0,1]$, the linear form
\begin{equation}
        L_{1/\gamma}(\N,\psi)\times L_{1/(1-\gamma)}(\N,\psi)\ni(x,y)\mapsto \Tr(xy)=\Tr(yx)\in\CC
\label{duality.Lp}
\end{equation}
defines the isometric isomorphism between $L_{1/\gamma}(\N,\psi)$ and $L_{1/(1-\gamma)}(\N,\psi)$, given by the duality pairing $\duality{\cdot,\cdot}$
\begin{equation}
        ^\dual:L_{1/\gamma}(\N,\psi)\ni x\mapsto x^\dual:=\duality{x,\cdot}:=\Tr(x\;\cdot)\in L_{1/(1-\gamma)}(\N,\psi)^\banach.
\end{equation}
Moreover, one has also the noncommutative analogue of Rogers--H\"{o}lder inequality \cite{Rogers:1888,Hoelder:1889},
\begin{equation}
        \n{xy}_1\leq\n{x}_{1/\gamma}\n{y}_{1/(1-\gamma)}\;\;
        \forall\gamma\in\,]0,1]\;
        \forall x\in L_{1/\gamma}(\N,\psi)\;
        \forall y\in L_{1/(1-\gamma)}(\N,\psi).
\end{equation}
The space $L_2(\N,\psi)$ is a Hilbert space with respect to the inner product
\begin{equation}
        \s{x,y}_{L_2(\N,\psi)}:=\Tr(y^*x)=\Tr(xy^*)\;\;\forall x,y\in L_2(\N,\psi).
\end{equation}
If $\N$ has no minimal projection and if $1\leq p<q<\infty$, then 
\begin{equation}
        L_p(\N,\psi)\cap L_q(\N,\psi)=\{0\}.
\end{equation}
A Haagerup--Terp space $L_p(\N,\psi)$ is isometrically isomorphic to some commutative Riesz--Radon space $L_p(\X,\mho(\X),\tmu)$ \cite{Riesz:1910,Radon:1913} only if either $p=2$ or $\N$ is a commutative $W^*$-algebra (for semi-finite $\N$ and $\psi$ given by faithful normal semi-finite trace this was shown in \cite{Katavolos:1981,Katavolos:1982}). For semi-finite $\N$ and faithful normal semi-finite trace $\psi=\tau$ on $\N$ the Haagerup--Terp spaces $L_p(\N,\psi)$ are isometrically isomorphic and order isomorphic to spaces $L_p(\N,\tau)$ spaces defined in Section \ref{integration.trace.section}. The quadruple $(L_2(\N,\psi),\pi_L(\N),J,L_2(\N,\psi)^+)$ with $\pi_L(x)y:=xy$ and $Jy:=y^*$ $\forall x\in\N$ $\forall y\in L_2(\N,\psi)$ is a standard form of $\N$.

We will now define the Connes--Hilsum spaces. Let $\N$ be a von Neumann algebra on a Hilbert space $\H$, and let $\psi^\comm\in\W_0(\N^\comm)$. Define $\NNN^p(\N,\psi^\comm)$ as a set of all closed densely defined operators $x$ on $\H$ such that, given polar decomposition $x=u\ab{x}$, $u\in\N$ and there exists $\phi\in\W(\N)$ such that $\ab{x}^p=\connes{\phi}{\psi^\comm}$. Define
\begin{align}
        \int\psi^\comm\ab{x}
        &:=\phi(\II),\\
        \MMM^p(\N,\psi^\comm)
        &:=\{x\in\NNN^p(\N,\psi^\comm)\mid\int\psi^\comm\ab{x}^p<\infty\},\\
        \n{x}_p
        &:=\left(\int\psi^\comm\ab{x}^p\right)^{1/p}\;\;\forall x\in\MMM^p(\N,\psi^\comm),
\end{align}
for $p\in[1,\infty[$, as well as $\MMM^\infty(\N,\psi^\comm):=\N$ and $\n{x}_\infty:=\n{x}_\N$ $\forall x\in\N$. For $p\in[1,\infty]$, the space $(\MMM^p(\N,\psi^\comm),\n{\cdot}_p)$ is a Banach space if the additive structure of this space is given by the \textit{strong} sum, defined as a closure of an algebraic sum. These Banach spaces will be denoted by $\rpktarget{CHLP}L_p(\N,\psi^\comm)$ and called the \df{Connes--Hilsum spaces}.

Finally, let us define the Araki--Masuda spaces. For a $W^*$-algebra $\N$, $\psi\in\W_0(\N)$, $p\in[1,\infty[$ and
\begin{equation}
        \nnn^\infty_\psi:=(\nnn_\psi\cap\nnn_\psi^*)^{\sigma^\psi}_\infty
\end{equation}
(c.f. \eqref{A.alpha.infty}), consider the sets $\MMM^p(\N,\psi)$ of all closed operators $x$ on $\H_\psi$ satisfying
\begin{enumerate}
\item[1)] $xJ_\psi\sigma^\psi_{-\ii/p}(y)J_\psi\supseteq J_\psi yJ_\psi x\;\;\forall y\in\nnn^\infty_\psi$,
\item[2)] $\n{x}_p:=\left(\sup_{\{y\in\nnn^\infty_\psi\mid\n{y}\leq1\}}\left\{\n{\ab{x}^{p/2}[y]_\psi}\right\}\right)^{2/p}<\infty$.
\end{enumerate}
The sets $\MMM^p(\N,\psi)$ can be equipped with the structure of a vector space over $\CC$, with the addition operation given by
\begin{enumerate}
\item[1)] $\overline{(x_1+x_2)|_{\dom(x_1)\cap\dom(x_2)}}$ if{}f $x_1,x_2\in\MMM^p(\N,\psi)$ are densely defined and for $p\in[2,\infty[$,
\item[2)] $\duality{x_1,\cdot}_\psi+\duality{x_2,\cdot}_\psi$ for $p\in\,]1,\infty[$, where\rpktarget{AMDUALITY}
\begin{equation}
        \MMM^p(\N,\psi)\times\MMM^q(\N,\psi)\ni(x_p,x_q)\mapsto\duality{x_p,x_q}_\psi:=\lim_{y\ra\II}\s{x_p[y]_\psi,x_q[y]_\psi}_\psi\in\CC,
\label{Araki.Masuda.duality}
\end{equation}
for $\frac{1}{p}+\frac{1}{q}=1$, $y\in\nnn^\infty_\psi$ such that $\n{y}\leq1$, and $\lim$ denoting a limit in ultrastrong topology,
\item[3)] the linear structure of $\N_\star$ for $p=1$.
\end{enumerate}
The map $\n{\cdot}_p$ is a norm on a vector space $\MMM^p(\N,\psi)$, with respect to which $\MMM^p(\N,\psi)$ is Cauchy complete. The resulting Banach spaces will be denoted $\rpktarget{AMLP}L_p(\N,\psi)$ and called the \df{Araki--Masuda spaces}. The space $L_\infty(\N,\psi)$ is defined as $\N$, with $\n{x}_\infty:=\n{x}_\N$. The map $\duality{\cdot,(\cdot)^*}_\psi$ is a bilinear form on $L_p(\N,\psi)\times L_q(\N,\psi)$. For every $x\in L_p(\N,\psi)$ and every $p\in[1,\infty[$ there exists a unique polar decomposition
\begin{equation}
        x=u\Delta^{1/p}_{\phi,\psi},
\label{Araki.Masuda.Lp.polar}
\end{equation}
where $\phi\in\N_\star^+$ and $u$ is a partial isometry such that $\supp(\phi)=u^*u$. Moreover, every operator of the form \eqref{Araki.Masuda.Lp.polar} belongs to $L_p(\N,\psi)$, and
\begin{align}
        \n{u\Delta^{1/p}_{\phi,\psi}}_p&=(\phi(\II))^{1/p},\\
        L_p(\N,\psi)^+&=\{\Delta^{1/p}_{\phi,\psi}\mid\phi\in\N_\star^+\}.
\end{align}
The mapping
\begin{equation}
        \N_\star\ni\phi\mapsto\phi(xu)=\duality{u\Delta_{\phi,\psi},x^*}_\psi\in\CC\;\;\forall x\in\N
\end{equation}
defines an isometric isomorphism $\N_\star\iso L_1(\N,\psi)$. In addition, $\H_\psi\iso L_2(\N,\psi)$ by means of
\begin{equation}
        L_2(\N,\psi)\ni u\Delta^{1/2}_{\phi,\psi}\mapsto u\xi_\psi(\phi)\in\H_\psi,
\end{equation}
where $\xi_\psi(\phi)$ is a standard vector representative of $\phi\in\N_\star^+$ in the natural cone of $\H_\psi$.

The Haagerup--Terp, Araki--Masuda and Kosaki--Terp spaces $L_p(\N,\psi)$ are uniformly convex and uniformly Fr\'{e}chet differentiable for $p\in\,]1,\infty[$ (for proofs, see \cite{Terp:1981}, \cite{Araki:Masuda:1982,Masuda:1983}, and \cite{Kosaki:1984:ncLp}, respectively). Hence, by isometric isomorphisms, this holds also for the corresponding spaces of Connes--Hilsum, Zolotar\"{e}v, Cecchini, and Leinert.

The structure of $\precore$ is independent of the choice of $\psi$ \textit{up to a $*$-isomorphism} $\varsigma:\N\rtimes_{\sigma^{\psi_1}}\RR\ra\N\rtimes_{\sigma^{\psi_2}}\RR$ such that \cite{vanDaele:1978,Woronowicz:1979}
\begin{align}
        \varsigma\circ\hat{\sigma}^{\psi_1}_t
        &=\hat{\sigma}^{\psi_2}_t\circ\varsigma
        \;\;\forall t\in\RR,\\
        \tau_{\psi_1}
        &=\tau_{\psi_2}\circ\varsigma.
\end{align}
the map $\varsigma$ can be extended to a topological $*$-isomorphism $\bar{\varsigma}:\MMM(\precore_1,\tau_{\psi_1})\ra\MMM(\precore_2,\tau_{\psi_2})$ \cite{Terp:1981}. From this it follows that for any $\psi,\phi\in\W_0(\N)$, the spaces $L_p(\N,\psi)$ and $L_p(\N,\phi)$ are isometrically isomorphic \cite{Terp:1981,Kosaki:1984:ncLp}. This leads to a question whether it is possible to provide a construction of noncommutative $L_p$ spaces, which would be explicitly independent of the choice of $\psi\in\W_0(\N)$.
\subsection{Canonical noncommutative integration\label{canonical.int.section}}
The problem of construction of canonical integration theory over $W^*$-algebras $\N$ together with the associated canonical (weight-independent) construction of $L_p(\N)$ spaces was solved in two different but equivalent ways by Kosaki \cite{Kosaki:1980:PhD} and by Falcone and Takesaki \cite{Falcone:Takesaki:2001}.

The approach of Kosaki is based on the use of polar decomposition of elements of $\N_\star$ in terms of canonical relative modular operator \eqref{canonical.rel.mod.op}. For $\phi_1,\phi_2\in\N_\star$ with polar decompositions $\phi_1=\ab{\phi_1}(\,\cdot\,u_1)$ and $\phi_2=\ab{\phi_2}(\,\cdot\,u_2)$, $p\in[1,\infty[$, and $\lambda=\ee^{\ii r}\ab{\lambda}\in\CC$ with $r\in[0,2\pipi[$, consider the addition, multiplication and $*$ operations on $\N_\star$ given by
\begin{enumerate}
\item[1)] $\phi_1^{1/p}+\phi_2^{1/p}:=(\varphi(\,\cdot\,u))^{1/p}$, where $\varphi\in\N_\star^+$ and a partial isometry $u$ with $\supp(\varphi)=u^*u$ are determined by
\begin{equation}
        u\Delta^{1/p}_{\varphi,\ab{\phi_1}+\ab{\phi_2}}:=
        u_1\Delta^{1/p}_{\ab{\phi_1},\ab{\phi_1}+\ab{\phi_2}}+
        u_2\Delta^{1/p}_{\ab{\phi_2},\ab{\phi_1}+\ab{\phi_2}},
\end{equation}
\item[2)] $\lambda\cdot\phi_1^{1/p}:=(\ab{\lambda}^p\ab{\phi_1}(\,\cdot\,\ee^{\ii r}u))^{1/p}$,
\item[3)] $(\phi_1^{1/p})^*:=(\varphi(\,\cdot\,u))^{1/p}$, where $\varphi\in\N_\star^+$ and a partial isometry $u$ with $\supp(\varphi)=u^*u$ are determined by
\begin{equation}
        u\Delta^{1/p}_{\varphi,\ab{\phi_1}}:=(u_1\Delta^{1/p}_{\ab{\phi_1}})^*.
\end{equation}
\end{enumerate}
Like in \eqref{Kosaki.new.multiplication}-\eqref{Kosaki.new.addition}, $\rpktarget{PHIp}\phi^{1/p}$ is understood here as a \textit{symbol} referring to the element $\phi$ of $\N_\star$ subject to the above operations. The set $\N_\star^+$ equipped with the above structure becomes a vector space with involution $^*$, and will be denoted by $\MMM^p(\N)$. The map
\begin{equation}
        \n{\cdot}_p:\MMM^p(\N)\ni\phi^{1/p}\mapsto\n{\phi^{1/p}}_p:=(\ab{\phi}(\II))^{1/p}=\n{\phi}_{\N_\star}^{1/p}
\end{equation}
defines a norm on $\MMM^p(\N)$, with respect to which $\MMM^p(\N)$ is Cauchy complete. The Banach spaces $(\MMM^p(\N),\n{\cdot}_p)$ are denoted by $\rpktarget{LPNKOSAKI}L_p(\N)$. Kosaki shows that they satisfy noncommutative analogue of the Rogers--H\"{o}lder inequality, are uniform convex and uniform Fr\'{e}chet differentiable for $p\in\,]1,\infty[$, and $L_q(\N)$ is a Banach dual of $L_p(\N)$ for $\frac{1}{p}+\frac{1}{q}=1$ with $p\in[1,\infty[$. The space $L_2(\N)$ coincides with the Hilbert space of Kosaki's canonical representation. Moreover, given a $*$-isomorphism $\varsigma:\N_1\ra\N_2$ of $W^*$-algebras $\N_1$ and $\N_2$, the isometry $\varsigma_\star:{\N_2}_\star\ra{\N_1}_\star$, induced by 
\begin{equation}
        (\varsigma_\star\phi)(x)=\phi(\varsigma(x))\;\;\forall x\in\N_1,\;\forall\phi\in{\N_2}_\star
\end{equation}
gives rise to a surjective isometry $L_p(\N_1)\ra L_p(\N_2)$ \cite{Kosaki:1980:PhD}. This defines a functor $\ncLK_p:\WsIso\ra\ncL_p\Iso$, where $\ncL_p\Iso$ is a family of categories (indexed by $p\in[1,\infty]$) consisting of $L_p(\N)$ spaces with isometric isomorphisms.

The approach of Falcone and Takesaki relies on the properties of the standard core algebra $\core$ and Masuda's \cite{Masuda:1984} reformulation of Connes' noncommutative Radon--Nikod\'{y}m type theorem. The one-parameter automorphism group of $\fell(\N)$,
\begin{equation}
        \tilde{\sigma}_s(x\phi^{\ii t}):=\ee^{-\ii ts}x\phi^{\ii t}\;\;\;\;\forall x\phi^{\ii t}\in\N(t),
\label{tilde.sigma.def}
\end{equation}
corresponding to the unitary group $\tilde{u}(s)$ on $\widetilde{\H}$ given by
\begin{equation}
        (\tilde{u}(s)\xi)(t)=\ee^{-\ii st}\xi(t)\;\;\forall t,s\in\RR\;\forall\xi\in\widetilde{\H},
\end{equation}
extends uniquely to a group of automorphisms $\tilde{\sigma}_s:\core\ra\core$. The automorphism $\tilde{\sigma}_t$ provides a weight-independent replacement for a dual automorphism $\hat{\sigma}_t^\psi$ used in Haagerup's theory. The triple $(\tilde{\N},\RR,\tilde{\sigma})$ is a $W^*$-dynamical system. Analogously to \eqref{isomorphism.of.double.dual.psi} and \eqref{fixed.point.algebra.crossprod.psi}, there exist canonical isomorphisms
\begin{align}
        \core\rtimes_{\tilde{\sigma}}\RR&\iso\N\otimes\BBB(L_2(\RR,\dd\lambda)),\\
        \core_{\tilde{\sigma}}&\iso\N.
\end{align}
The action of $\tilde{\sigma}_s$ on $\core$ is integrable over $s\in\RR$, and 
\begin{equation}
        T_{\tilde{\sigma}}:\core^+\ni x\mapsto T_{\tilde{\sigma}}(x):=\int_\RR\dd s\,\tilde{\sigma}_s(x)\in\N^\ext,
\end{equation}
is an operator valued weight from $\core$ to $\core_{\tilde{\sigma}}\iso\N$. For any $\phi\in\W(\N)$, its dual weight over $\core$ is given by
\begin{equation}\rpktarget{DUALWEIGHT}
        \hat{\phi}:=\tilde{\phi}\circ T_{\tilde{\sigma}}\in\W(\core).
\end{equation}
Every $\phi\in\W_0(\N)$ can be considered as an analytic generator of the one parameter group of unitaries $\rpktarget{PHI.IT.DREI}\{\phi^{\ii t}\mid t\in\RR\}\subseteq\core$ acting on $\widetilde{\H}$ from the right, given by
\begin{equation}
        \phi=\exp\left(-\ii\frac{\dd}{\dd t}\left(\phi^{\ii t}\right)|_{t=0}\right).
\label{phi.in.FT.as.generator}
\end{equation}
This allows to equip $\core$ with a faithful normal semi-finite trace $\taucore_\phi:\core^+\ra[0,\infty]$, 
\begin{align}
        \taucore_\phi(x):&=\lim_{\epsilon\ra^+0}\hat{\phi}((\phi^{-1}(1+\epsilon\phi^{-1})^{-1})^{1/2}x((\phi^{-1}(1+\epsilon\phi^{-1})^{-1})^{1/2})\nonumber\\
        &=\lim_{\epsilon\ra^+0}\hat{\phi}(\phi^{-1/2}(1+\epsilon\phi^{-1})^{-1/2}x\phi^{-1/2}(1+\epsilon\phi^{-1})^{-1/2})\nonumber\\
        &=\lim_{\epsilon\ra^+0}\hat{\phi}((\phi+\epsilon)^{-1/2}x(\phi+\epsilon)^{-1/2}).
\label{canonical.trace}
\end{align}
This definition is independent of the choice of weight (e.g., $\taucore_\varphi=\taucore_\psi\;\forall\varphi,\psi\in\W_0(\N)$), which follows from the fact that 
\begin{equation}
\Connes{\taucore_\phi}{\taucore_\varphi}{t}=\Connes{\taucore_\phi}{\tilde{\varphi}}{t}
\Connes{\tilde{\varphi}}{\tilde{\psi}}{t}
\Connes{\tilde{\psi}}{\taucore_\varphi}{t}=\varphi^{-\ii t}\Connes{\varphi}{\psi}{t}\psi^{\ii t}=\varphi^{-\ii t}\varphi^{\ii t}\psi^{-\ii t}\psi^{\ii t}
=1
\end{equation}
for all $\phi,\varphi\in\W_0(\N)$ and for all $t\in\RR$. This allows to write $\taucore$ instead of $\taucore_\varphi$. Moreover, $\taucore$ has the scaling property
\begin{equation}
        \taucore\circ\tilde{\sigma}_s=\ee^{-s}\taucore\;\;\;\;\forall s\in\RR.
\label{scaling}
\end{equation}
This allows to call $\taucore$ a \df{canonical trace} of $\core$. It will play the role analogous to a natural trace $\tilde{\tau}_\psi$ on $\precore=\N\rtimes_{\sigma^\psi}\RR$. Nevertheless, the definition \eqref{canonical.trace} is not a straightforward generalisation of \eqref{nat.trace.Connes}, and is neither an application of \eqref{ext.normal.weights.bijection} nor of \eqref{Pedersen.Takesaki.perturbed.weight}. It is yet another type of `perturbed' construction of a weight, which is designed in this case for the purpose of direct elimination of the dependence of $\taucore_\phi$ on $\phi$.

Consider the category $\WstarCovRTr$ of quadruples $(\N,\RR,\alpha,\tau)$, where $\N$ is a semi-finite $W^*$-algebra, $\tau$ is a faithful normal semi-finite trace on $\N$, and $(\N,\RR,\alpha)$ is a $W^*$-dynamical system, with morphisms
\begin{equation}
        (\N_1,\RR,\alpha^1,\tau_1)\ra
        (\N_2,\RR,\alpha^2,\tau_2)
\end{equation}
given by such $*$-isomorphisms $\varsigma:\N_1\ra\N_2$ which satisfy
\begin{align}
        \varsigma\circ\alpha^1_t&=\alpha^2_t\circ\varsigma\;\;\forall t\in\RR,\label{iso.covariance}\\
        \tau_1&=\tau_2\circ\varsigma.\label{trace.covariance}
\end{align}
Falcone and Takesaki call the quadruple $(\core,\RR,\tilde{\sigma},\taucore)$ a \df{noncommutative flow of weights}, and prove that every $*$-isomorphism $\varsigma:\N_1\ra\N_2$ of von Neumann algebras extends to a $*$-isomorphism $\widetilde{\varsigma}:\core_1\ra\core_2$ satisfying \eqref{iso.covariance} and \eqref{trace.covariance}. This defines a functor 
\begin{equation}
        \FTflow:\VNIso\ra\WstarCovRTr.
\end{equation}
The restriction of $\tilde{\sigma}$ to the center $\zentr_\core$ is the Connes--Takesaki flows of weights $(\zentr_\core,\RR,\tilde{\sigma}|_{\zentr_\core})$ \cite{Falcone:Takesaki:2001}. Hence, the relationship between the Falcone--Takesaki noncommutative flow of weights and the Connes--Takesaki flow of weights can be summarised in terms of the commutative diagram
\begin{equation}
\xymatrix{
        \WsIso
        \ar[rr]^{\CanVN}
        &&
        \VNIso
        \ar[rr]^{\FTflow}
        &&
        \WstarCovRTr
        \ar[d]^{\zentr\circ\ForgTr}
        \\
        \WsfIIIIso
        \ar@{ >->}[u]
        \ar[rr]_{\CanVN}
        &&
        \VNfIIIIso
        \ar@{ >->}[u]
        \ar[rr]_{\CTflow}
        &&
        \WstarCovR,     
}
\label{ctft.cat.diag}
\end{equation}
where $\WsfIIIIso$ (respectively, $\VNfIIIIso$) is a category of type III factor $W^*$-algebras (respectively, von Neumann algebras) with $*$-isomorphisms, $\ForgTr$ denotes the forgetful functor that forgets about traces, while $\zentr:\WstarCovR\ra\WstarCovR$ is an endofunctor that assigns an object $(\zentr_\N,\RR,\alpha|_{\zentr_\N})$ to each $(\N,\RR,\alpha)$, and assigns a morphism $\varsigma^{12}_\zentr$ such that
\begin{equation}
        \varsigma^{12}_\zentr\circ\alpha^1_t|_{\zentr_{\N_1}}=
        \alpha^2_t|_{\zentr_{\N_2}}\circ\varsigma^{12}_\zentr
\end{equation}
to each $\varsigma:(\N_1,\RR,\alpha^1)\ra(\N_2,\RR,\alpha^2)$.

Given $\varphi\in\W(\N)$, 
\begin{equation}
        h^{\ii t}_\varphi:=\Connes{(\tilde{\varphi}\circ T_{\tilde{\sigma}})}{\taucore\,}{t}\;\;\forall t\in\RR
\label{haagerup.corresp}
\end{equation}
defines a map 
\begin{equation}
        \W(\N)\ni\varphi\mapsto h_\varphi\in\aff(\core)^+.
\end{equation}
From the Pedersen--Takesaki theorem it follows that $h_\varphi$ is a unique element of $\aff(\core)^+$ that satisfies
\begin{equation}
        \tilde{\varphi}\circ T_{\tilde{\sigma}}(\cdot)=\taucore_{h_\varphi}.
\label{nc.RN.on.core}
\end{equation}
Hence, $h_\varphi$ can be considered as (a reference-independent) `operator density' of $\varphi$. Define a \df{grade} $\rpktarget{GRAD}\grad(x)$ of a closed densely defined operator $x$ affiliated with $\core$ as such $\gamma\in\CC$ that
\begin{equation}
        \tilde{\sigma}_s(x)=\ee^{-\gamma s}x\;\;\;\;\forall s\in\RR.
\label{grade.def}
\end{equation}
If $\grad(x)=0$, then $x$ is bounded, but if $\re(\grad(x))\neq0$, then $x$ is unbounded. Equations \eqref{haagerup.corresp} and \eqref{nc.RN.on.core} define a bijection between the set $\W(\N)$ and the set $\NNN^1(\core)$ of all elements of $\aff(\core)^+$ that are of grade $1$. For any $x\in\aff(\core)^+$ with $p:=\re(\grad(x))>0$ there exists a unique $\varphi\in\W(\N)$ such that $h_\varphi$ is of grade $1$ and $h_\varphi=\ab{x}^{1/p}$. Moreover, $\varphi\in\N_\star^+$ if{}f $h_\varphi$ is $\taucore$-measurable \cite{Nelson:1974,Haagerup:1979:ncLp}. Let $\rpktarget{MMMPCORE}\MMM^p(\core,\taucore)$ denote the space of all $\taucore$-measurable operators of grade $1/p$ affiliated with $\core$ for $p\in\CC\setminus\{0\}$. The spaces $\MMM^p(\core,\taucore)$ embed into the topological $*$-algebra $\MMM(\core,\taucore)$ of all $\taucore$-measurable operators affiliated with $\core$. Given a polar decomposition of $\varphi=\ab{\varphi}(\,\cdot,u)$, $h_\varphi:=uh_{\ab{\varphi}}$ defines a unique extension of the map $\N_\star^+\ni\varphi\mapsto h_\varphi\in\MMM(\core,\taucore)$ to a natural bijection (linear isomorphism)
\begin{equation}
        \N_\star\ni\omega\mapsto h_\omega\in\MMM^1(\core,\taucore)
\label{SDHFT.iso}
\end{equation}
that preserves positivity and satisfies
\begin{equation}
        h_{\phi(x\,\cdot\,y)}=xh_\phi y\;\;\forall x,y\in\N.
\end{equation} 

Let
\begin{align}
        \mmm^+_{T_{\tilde{\sigma}}}&
        :=\{x^*y\in\core^+\mid
        \;\n{T_{\tilde{\sigma}}(x^*x)}<\infty,
        \;\n{T_{\tilde{\sigma}}(y^*y)}<\infty\}
        \subseteq\core^+,\\
        \hat{\mmm}^+_{T_{\tilde{\sigma}}}&
        :=\{y\in\mmm^+_{T_{\tilde{\sigma}}}\mid
        T_{\tilde{\sigma}}(y)=1\}.
\end{align}
If $\omega\in\N_\star$, then, for any $x\in\N$ and for any $y\in\hat{\mmm}^+_{T_{\tilde{\sigma}}}$, 
\begin{gather}
        y^{1/2}xh_\omega y^{1/2}\in\MMM(\core,\taucore),\\
        \omega(x)=\taucore(y^{1/2}xh_\omega y^{1/2}).
\end{gather}
This allows to introduce the integral of $x\in\MMM^1(\core,\taucore)$,
\begin{equation}
        \int x:=\taucore(y^{1/2}xy^{1/2}),
\end{equation}
whose value is independent of the choice of $y\in\hat{\mmm}^+_{T_{\tilde{\sigma}}}$. This allows to call it \df{canonical integral}. While $\taucore$ takes only the $+\infty$ value on nonzero elements of $\MMM^1(\core,\taucore)$, the canonical integral $\int$ takes finite values. This allows to extend \eqref{SDHFT.iso} to an isometric isomorphism of Banach spaces, with the norm on $\MMM^1(\core,\taucore)$ defined by $\n{x}_1:=\int\ab{x}$, and with $\n{h_\varphi}_1=\varphi(\II)$. The duality pairing between Banach spaces $\N$ and $\MMM^1(\core,\taucore)$ that identifies $\N_\star$ with $\MMM^1(\core,\taucore)$ is given by the bilinear form
\begin{equation}
        \N\times\MMM^1(\core,\taucore)\ni(y,x)\mapsto\duality{y,x}_{\core}:=\int yx\in\CC.
\label{int.duality.one}
\end{equation}
The noncommutative $\rpktarget{LPNFT}L_p(\N)$ spaces for $p\in\{z\in\CC\mid\re(z)\geq1\}$ are defined as the spaces $\MMM^p(\core,\taucore)$ equipped with, and Cauchy complete in, the norm
\begin{equation}
        \n{\cdot}_p:\MMM^p(\core,\taucore)\ni x\mapsto
        \n{x}_p:=\left(\int\ab{x}^{\re(p)}\right)^{1/{\re(p)}}\in\RR^+.
\label{norm.FT.Lp.space}
\end{equation}
By \eqref{SDHFT.iso} and \eqref{int.duality.one}, $L_1(\N)\iso\N_\star$, and it is natural to define $L_\infty(\N):=\MMM^\infty(\core,\taucore)\iso\N\iso\N(0)$, using the definition \eqref{grade.def} of grade with $\tilde{\sigma}_s(x)=x$ for $\grad(x)=0$. If $\phi\in\N_\star^+$ then $\n{\phi^{1/p}}_p=(\phi(\II))^{1/\re(p)}$. By definition, all $L_p(\N)$ for $p\in[1,\infty]$ are Banach spaces. The space $L_2(\N)$ is also a Hilbert space with respect to the inner product
\begin{equation}
L_2(\N)\times L_2(\N)\ni(x_1,x_2)\mapsto\s{x_1,x_2}_{L_2(\N)}:=\int x_2^*x_1\in\CC.
\label{FT.Hilbert.space}
\end{equation}
The duality \eqref{int.duality.one} extends to noncommutative $L_p(\N)$ space duality, given by the bilinear map
\begin{equation}
        L_p(\N)\times L_q(\N)\ni(x,y)\mapsto\duality{x,y}_{\core}:=\int xy\in\CC,
\label{dual.pairing.Lp.FT}
\end{equation}
with $1/p+1/q=1$, where $p\in\{\lambda\in\CC\mid\re(\lambda)>0\}$. For $p\in\CC$ such that $\re(p)<0$, one has $\MMM^p(\core,\taucore)=\{0\}$, so $L_p(\N)=\{0\}$. Moreover, 
\begin{equation}
        \rpktarget{LITNFT}L_{1/\ii t}(\N)=\N(t)\;\;\forall t\in\RR.
\end{equation}

Due to the properties of grade function, the elements of $\MMM(\core,\taucore)$ possess remarkable algebraic properties. The grade function satisfies:
\begin{align}
        \grad(x^*)&=(\grad(x))^*,\\
        \grad(\ab{x})&=\re(\grad(x))
        =\textstyle\frac{1}{2}(\grad(x)+\grad(x)^*),\\
        \grad(\overline{xy})&=\grad(x)+\grad(y),
\end{align}
where $\overline{xy}$ is the closure of $xy$. Moreover, from the bijection between the elements $\omega\in\N_\star^+$ and $h_\omega\in L_1(\N)$ it follows that
\begin{equation}
        \re(\grad(x))\geq0\limp\ab{x}^{1/\re(\grad(x))}\in\N^+_\star.
\end{equation}
If $\CC^+:=\{\lambda\in\CC\mid\re(\lambda)\geq0\}$, then the spaces $\bigcup_{\lambda\in\CC}L_{1/\lambda}(\N)$ and $\bigcup_{\lambda\in\CC^+}L_{1/\lambda}(\N)$ are $*$-algebras, which are identical due to 
\begin{equation}
        L_{1/\lambda}(\N)=\{0\}\;\;\;\forall\lambda\in\CC\setminus\CC^+.
\label{L.N.minus.empty}
\end{equation}
Moreover,
\begin{equation}
        \N(s)\N(t)\subseteq\N(s+t),\;\;\;\N(t)^*=\N(-t),\;\;\;\forall s,t\in\RR.
\end{equation}
If $\{x_i\}_{i=1}^n\subseteq\MMM(\core,\taucore)$, $\sum_{i=1}^n\grad(x_i)=:r\leq1$ and $\re(\grad(x_i))\geq0$ $\forall i\in\{1,\ldots,n\}$, then the noncommutative analogue of the Rogers--H\"{o}lder inequality holds \cite{Kosaki:1984:continuity},
\begin{equation}
\n{x_1\cdots x_n}_{1/r}\leq\n{x_1}_{1/\re(\grad(x_1))}\cdots\n{x_n}_{1/\re(\grad(x_n))}.
\end{equation}
The stronger condition $\sum_{i=1}^n\grad(x_i)=1$ implies that $x_1\cdots x_n\in L_1(\N)$, and in such case
\begin{equation}
                \int x_1\cdots x_n=\int x_nx_1\cdots x_{n-1}.
\label{permutation.core}
\end{equation}
For $x_i=y_i\phi^{z_i}$ with a fixed $\phi\in\N^+_{\star0}$, the equation \eqref{permutation.core} turns to the Araki multiple KMS condition for $\sigma^\phi$ and $\beta=1$ \cite{Araki:1968,Araki:1973:relative:hamiltonian,Araki:Masuda:1982,Masuda:1983}. More generally, the function
\begin{equation}
        \CC^n\ni(z_1,\ldots,z_n)\mapsto\phi_1^{z_1}y_1\cdots\phi_n^{z_n}y_n\phi_{n+1}^{1-z_1-\ldots-z_n}\in\N_\star
\end{equation}
is a bounded holomorphic function on the tube 
\begin{equation}
        \{(z_1,\ldots,z_n)\in\CC^n\mid\re(z_i)>0\;\;\forall i\in\{1,\ldots,n\},\;\sum_{i=1}^n\re(z_i)\leq1\},
\end{equation}
with respect to the norm topology of $\N_\star$ \cite{Yamagami:1992}. The algebraic relations in modular algebra of $\N$ can be used in order to rewrite Connes' cocycle as\rpktarget{CONNES.MODALG}
\begin{equation}
        \Connes{\omega}{\phi}{t}=\Delta_{\omega,\phi}^{\ii t}\Delta_{\phi}^{-\ii t}=\omega^{\ii t}\phi^{-\ii t},
\label{core.connes}
\end{equation}
which holds for all $\phi,\omega\in\W_0(\N)$, and for all $\phi,\omega\in\N_\star^+$ provided $\supp(\omega)\leq\supp(\phi)$, and to rewrite the Tomita--Takesaki modular automorphism as\rpktarget{TT.MODALG}
\begin{equation}
        \sigma_{t}^\phi(x)=\Delta^{\ii t}_\phi x\Delta^{-\ii t}_\phi=\phi^{\ii t}x\phi^{-\ii t},
\label{core.tt.auto}
\end{equation}
which holds for all $\phi\in\W_0(\N)$, and for all $\phi\in\W(\N)$, provided $x\in\N_{\supp(\phi)}$. These remarkable algebraic properties were observed by Woronowicz \cite{Woronowicz:1979} and were later developed by Connes \cite{Connes:1980,Connes:1982,Connes:1994} and Yamagami \cite{Yamagami:1992,Yamagami:1994}. Equation \eqref{core.tt.auto} is a representation independent generalisation of \eqref{sv1.eq}, and enables to define an inner product 
\begin{equation}
        \N(t)\times\N(t)\ni(x\phi^{\ii t},y\phi^{\ii t})\mapsto\s{x\phi^{\ii t},y\phi^{\ii t}}_{\N(t)}:=(y\phi^{\ii t})^*(x\phi^{\ii t})=\phi^{-\ii t}y^*x\phi^{\ii t}=\sigma_t^\phi(y^*x)\in\N.
\end{equation}
Recall that, by means of \eqref{phi.in.FT.as.generator}, $\phi$ can be considered as a generator of a group $\{\phi^{\ii t}\mid t\in\RR\}$ of unitaries in $\core$. As \eqref{connes.cocycle.as.generator} shows, Connes' spatial quotient $\rpktarget{CONN.SPAT.MODALG}\connes{\phi}{\psi}$ can be identified with an exponentiated generator of the one parameter group of unitaries
\begin{equation}
        \RR\ni t\mapsto\phi^{\ii t}(\,\cdot\,)\psi^{-\ii t}
\end{equation}
acting on a Hilbert space $\H(t)$ \cite{Yamagami:1994}. By taking the properties of grade and equations \eqref{core.connes} and \eqref{core.tt.auto} as elementary, one can consider a $*$-algebra generated algebraically by a given $W^*$-algebra $\N$ and the set of symbols $\rpktarget{PHI.IT.MODALG}\{\psi^{\ii t}\mid\psi\in\W(\N),t\in\RR\}$, equipped with the relations
\begin{equation}
\psi^{\ii t}\psi^{\ii s}=\psi^{\ii(t+s)},\;(\psi^{\ii t})^*=\psi^{-\ii t},\;\psi^{\ii0}=\supp(\psi),\;\psi^{\ii t}x\psi^{-\ii t}=\sigma^{\psi}_t(x),\;\psi^{\ii t}=\Connes{\psi}{\phi}{t}\phi^{\ii t},
\end{equation}
and define a \df{modular algebra} of $\N$ as a closure of this $*$-algebra with respect to the topology induced from its normal representations, see \cite{Yamagami:1992,Sherman:2001}. The modular algebra of $\N$ is unitarily isomorphic to a canonical core algebra of $\N$. The equation \eqref{L.N.minus.empty} means that for the negative powers of weights, $\phi^{-p}$ for $p>0$, there are no corresponding $L_{-p}(\N)$ spaces. However, as shown in \cite{Sherman:2001}, the right multiplication $\rpktarget{RRR.SHERMAN}\RRR(\phi^{-p})$ for $\phi\in\W_0(\N)$ is well defined and satisfies $\RRR(\phi^{-p})=(\RRR(\phi^p))^{-1}$ as well as
\begin{equation}
        \left(\connes{\psi}{\phi(J_\N\cdot J_\N)}\right)^{1/p}=\RRR(\phi^{-1/p})\LLL(\psi^{1/p}),
\label{sherman.minus.weight.spatial}
\end{equation}
where $\psi\in\W(\N)$ and $\rpktarget{LLL.SHERMAN}\LLL$ denotes left multiplication, which is well defined too. By means of \eqref{canonical.rel.mod.op}, this gives
\begin{equation}
        \int\psi^\gamma\phi^{1-\gamma}=\int\psi^\gamma\phi^{-\gamma}\phi=\int(\RRR(\phi^{-\gamma})\LLL(\psi^{\gamma})\II)\phi=\phi(\RRR(\phi^{-\gamma})\LLL(\psi^{\gamma})\II)=\s{\xi_\pi(\phi),\Delta^\gamma_{\psi,\phi}\xi_\pi(\phi)}_\H
\label{sherman.left.right}
\end{equation}
for any standard representation $(\H,\pi,J,\stdcone)$. In analogy with the equations \eqref{core.connes} and \eqref{core.tt.auto}, the equation \eqref{sherman.left.right} holds also when $\phi,\psi\in\N_\star^+$ and $\psi\ll\phi$, because in such case $\phi$ is faithful on $\N_{\supp(\phi)}$ and this algebra contains the support of $\phi$.

The spaces $L_p(\N)$ of Falcone and Takesaki are isometrically isomorphic to Kosaki's $L_p(\N)$ spaces. The space $L_2(\N)$ is also unitarily isomorphic to Kosaki's $L_2(\N)$ space.\footnote{While Kosaki's construction of $L_2(\N)$ is based on $\N_\star^+$, the construction of Falcone and Takesaki is based on $\W_0(\N)$. The latter allows to think of $L_2(\N)^+$ as a space that provides also a representation of weights. More precisely, for every normal weight $\phi$ on $\N$ the function $\N\ni x\mapsto(\phi(x^*x))^{1/2}\in[0,\infty]$ is subadditive and weakly-$\star$ lower semi-continuous. In this sense, normal weights on $\N$ can be considered as `infinite vectors' in $L_2(\N)^+$.} For any choice of a reference weight $\psi\in\W_0(\N)$, $L_p(\N)$ spaces are isometrically isomorphic to the noncommutative $L_p(\N,\psi)$ spaces in the sense of Haagerup--Terp, Connes--Hilsum, Kosaki--Terp, Araki--Masuda, Zolotar\"{e}v, Cecchini and Leinert, as well as to Izumi's complex extension Kosaki--Terp $L_p(\N,\psi)$ spaces to $p\in\CC$. This implies that the spaces $L_p(\N)$ are uniformly convex and uniformly Fr\'{e}chet differentiable for $p\in\,]1,\infty[$. If the von Neumann algebra $\N$ is semi-finite and some faithful normal semi-finite trace $\tau$ on $\N$ is chosen, then the Falcone--Takesaki $L_p(\N)$ spaces are isometrically isomorphic and order-isomorphic to noncommutative $L_p(\N,\tau)$ spaces.

In the Falcone--Takesaki theory the trace $\taucore$, the algebra $\core$, and the integral $\int$ are independent of the choice of the particular weight on $\N$, as opposed to the trace $\tau_\psi$ and the algebra $\precore=\N\rtimes_{\sigma^\psi}\RR$ appearing in the Haagerup--Terp theory, and the integral $\int\psi^\comm$ appearing in the Connes--Hilsum theory. The canonical (representation independent) character of the Falcone--Takesaki `noncommutative integral' $\int$ corresponds to the canonical (representation independent) character of Connes' cocycle as the noncommutative analogue of the Radon--Nikod\'{y}m quotient. Because the construction of $L_p(\N)$ spaces is completely determined by the noncommutative flow of weights, the assignment of $L_p(\N)$ spaces to von Neumann algebras $\N$ determines a family of functors $\VNIso\ra\ncL_p\Iso$ for all $p\in\CC$ with $\re(p)\in\,]1,\infty[$. Using functorial character of Kosaki's canonical representation, we can provide the right composition with the functor $\CanVN$, which extends this to a family of functors $\ncLFT_p:\WsIso\ra\ncL_p\Iso$.

Sherman \cite{Sherman:2005} proved the following generalisation of theorems by Banach \cite{Banach:1931}, Stone \cite{Stone:1937} and Kadison \cite{Kadison:1951}: every surjective isometry $T:L_p(\N_1)\ra L_p(\N_2)$ for $p\in\,]0,\infty[\setminus\{2\}$ determines a unique surjective Jordan $*$-isomorphism $\varsigma:\N_1\ra\N_2$ and a unique unitary $u\in\N_2$ such that
\begin{equation}
        T(\phi^{1/p})=u(\phi\circ\varsigma^{-1})^{1/p}\;\;\forall\phi\in{\N_1}^+_\star.
\label{Sherman.correspondence}
\end{equation} 
This theorem allows us to define a family of functors $\Sher^\sharp_p:\ncL_p\Iso\ra\WssJIso$, where $\WssJIso$ is a category of $W^*$-algebras with surjective Jordan $*$-isomorphisms. Sherman \cite{Sherman:2005} proved also that for any $p\in[1,\infty]\setminus\{2\}$, any $W^*$-algebras $\N_1$ and $\N_2$ are Jordan $*$-isomorphic if{}f $L_p(\N_1)$ and $L_p(\N_2)$ are isometrically isomorphic.\footnote{This theorem seems to be a good starting point for introducing a family of functors $\Sher_p^\flat:\WssJIso\ra\ncL_p\Iso$ determining an equivalence of a category $\ncL_p\Iso$  with a category $\WssJIso$,
\begin{equation}
\xymatrix{\ncL_p\Iso
                \ar@<0.5ex>[rr]^{\Sher^\sharp_p}
                \POS!L(.8)\ar@(ul,dl)_{\id_{\ncL_p\Iso}}
        &&
                \WssJIso
                \ar@<0.5ex>[ll]^{\Sher^\flat_p}
                \POS!R(.8)\ar@(dr,ur)_{\id_{\WssJIso}}
}
\label{sherman.equivalence}
\end{equation}
with natural isomorphisms $\Sher^\sharp_p\circ\Sher^\flat_p\Rightarrow\id_{\WssJIso}$ and $\Sher^\flat_p\circ\Sher^\sharp_p\Rightarrow\id_{\ncL_p\Iso}$. However, it remains unclear how to handle the choice of corresponding surjective isometries (e.g., how to globally fix the choice of $u$ in \eqref{Sherman.correspondence}).}

The functorial character of the above constructions can be summarised by the diagram
\begin{equation}
\xymatrix{
&&
\WssJIso
&&\\
&&&&\\
&&
\WsIso
\ar@{ >->}[uu]
        \ar[ddll]|{\ncLFT_{1/\gamma}\;\;\;}
        \ar[ddrr]|{\;\;\;\ncLFT_{1/(1-\gamma)}}
        \ar@/_2pc/[ddll]|{\ncLK_{1/\gamma}}
        \ar@/^2pc/[ddrr]|{\ncLK_{1/(1-\gamma)}} 
&&\\
&&&&\\
\ncL_{1/\gamma}\Iso
\ar@<0.5ex>[rrrr]^{(\cdot)^\banach}
\ar@/^4pc/[uuuurr]^{\Sher^\sharp_{1/\gamma}}
&&&&
\ncL_{1/(1-\gamma)}\Iso
\ar@<0.5ex>[llll]^{(\cdot)_\star}
\ar@/_4pc/[uuuull]_{\Sher^\sharp_{1/(1-\gamma)}}
}
\label{ncLp.cat.diag}
\end{equation}
where $\gamma\in\{z\in\CC\mid\re(z)\in\,]0,1]\}$, with the exception of: arrows $\ncLK_p$, which are considered for $p\in[1,\infty[$, and arrows 
$\Sher_p^\sharp$, which are considered for $p\in\,]1,\infty[\setminus\{2\}$. The functors $(\cdot)^\banach$ and $(\cdot)_\star$ are defined by the Banach space duality between $L_{1/\gamma}(\N)$ and $L_{1/(1-\gamma)}(\N)$. The contravariant functor $\rpktarget{BANACH.FUNCTOR}(\cdot)^\banach$ assigns to each Banach space $X$ its Banach dual space $X^\banach$ and to each isometric isomorphism $f:X\ra Y$ a dual isometric isomorphism $f^\banach:Y^\banach\ra X^\banach$ defined by
\begin{equation}
        \duality{f(x),\phi}_{X\times X^\banach}=\duality{x,f^\banach(\phi)}_{X\times X^\banach}\;\;\forall x\in X\;\forall\phi\in Y^\banach.
\label{banach.duality.functor}
\end{equation}
The contravariant functor $(\cdot)_\star\rpktarget{PREDUAL.FUNCTOR}$ assigns to each Banach space $X$ its predual Banach space $X_\star$ and to each isometric isomorphism $f:X\ra Y$ a predual isometric isomorphism $f_\star:Y_\star\ra X_\star$ defined by
\begin{equation}
        \duality{f_\star(x),\phi}_{X_\star\times X}=\duality{x,f(\phi)}_{X_\star\times X}\;\;\forall x\in X_\star\;\forall\phi\in Y.
\label{banach.preduality.functor}
\end{equation}

\subsection{Integration relative to a measure\label{comm.integr.section}}
Integration theory based on the notion of measure on countably additive bounded subsets of $\RR^n$, developed by Borel \cite{Borel:1898} and Lebesgue \cite{Lebesgue:1901,Lebesgue:1902,Lebesgue:1904,Lebesgue:1910}, was unified with the Stieltjes integral theory \cite{Stieltjes:1894,Stieltjes:1895,Riesz:1909:b} by Radon \cite{Radon:1913}. Together with the ideas of `general analysis' by Fr\'{e}chet \cite{Frechet:1906,Frechet:1907:a,Frechet:1910,Frechet:1925:Morale} and Moore \cite{Moore:1909,Moore:1915}, Radon's work became a point of departure of several different abstract integration theories on abstract spaces. The most important are the Daniell abstract integral theory on vector lattices \cite{Daniell:1918,Daniell:1919:a,Daniell:1920,Goldstine:1941,McShane:1944,Stone:1948,Stone:1949,McShane:1949,Riesz:SzokefalviNagy:1952,Shilov:Gurevich:1964,Weir:1974,Pfeffer:1977} (built upon some earlier ideas by Young \cite{Young:1911,Young:1913,Young:1914} and Riesz \cite{Riesz:1912,Riesz:1920}), Fr\'{e}chet's abstract measure theory on sets \cite{Frechet:1915:definition,Frechet:1915:sur:integrale,Frechet:1922,Frechet:1923:additives,Sierpinski:1927,Sierpinski:1928,Nikodym:1930,Saks:1933,Kolmogorov:1933,Maharam:1942:TAMS,Halmos:1950} (which includes an integration theory on topological spaces \cite{Haar:1933,vonNeumann:1935:Haarschen,vonNeumann:1936,Cartan:1940,Weil:1940,Aleksandrov:1940,Aleksandrov:1941,Aleksandrov:1943,Rokhlin:1949,Bourbaki:1952,Varadarajan:1961,Schwartz:1973}), and Carath\'{e}odory's abstract measure theory on boolean algebras \cite{Caratheodory:1938,Wecken:1939,Ridder:1941,Olmsted:1942,Maharam:1942:PNAS,Ridder:1946,Gomes:1946,Kappos:1948,Horn:Tarski:1948,Segal:1951,Caratheodory:1956,Kappos:1960,Fremlin:1974}. To a large extent, these theories are equivalent, and they all are `commutative' integration theories, in the sense that the elements subjected to integration form commutative algebras. For recent expositions, see \cite{Stroock:1990,Bichteler:1998,Fremlin:2000,Carrillo:2002,Bogachev:2007}. 

In Section \ref{integr.compar.section} we will establish the direct relationship between commutative and noncommutative integration theory in their canonical (that is, representation independent and functorial) formulations. For this purpose, in this section we will describe a formulation of Carath\'{e}odory's approach based on interplay between the properties of boolean algebras and Riesz and Banach lattices, and we will also briefly discuss Fr\'{e}chet's and Daniell's approaches. For the theory of boolean algebras we refer to \cite{Sikorski:1960,Koppelberg:1989,Givant:Halmos:2009}, while for the theory of Riesz and Banach lattices we refer to \cite{Kantorovich:Vulikh:Pinsker:1950,Nakano:1950,Vulikh:1961,Semadeni:1971,Luxemburg:Zaanen:1971,Fremlin:1974,Schaefer:1974,Lacey:1974,Lindenstrauss:Tzafriri:1977:1979,Zaanen:1983,MeyerNieberg:1991,Kusraev:Kutateladze:1999}. Our default reference for the contents of this section is \cite{Fremlin:2000}.

A \df{partially ordered set} (or a \df{poset}) \cite{Hausdorff:1914} is defined as a pair $(X,\leq)$, where $X$ is a set, and $\rpktarget{POSET}\leq$ is a relation on $X$ such that
\begin{equation}
        x\leq x,\;\;
        (x\leq y,\;y\leq x)\limp x=y,\;\;
        (x\leq y,\;y\leq z)\limp x\leq z\;\;
        \forall x,y,z\in X.
\end{equation}
If $(X,\leq)$ is a poset and $Y\subseteq X$, then $Y$ is called: \df{bounded above} if{}f $\exists x\in X$ $\forall y\in Y$ $y\leq x$; \df{bounded below} if{}f $\exists x\in X$ $\forall y\in Y$ $x\leq y$; \df{upwards directed} if{}f $Y$ is nonempty and every pair of elements of $Y$ is bounded above; \df{downwards directed} if{}f $Y$ is nonempty and every pair of elements is bounded below. A \df{supremum} (or the  \df{least upper bound}) of $Y\subseteq X$, denoted by $\rpktarget{SUP}\sup Y$, is defined as $x\in X$ such that 
\begin{equation}
        y\leq x,\;\;
        y\leq z\limp x\leq z\;\;
        \forall z\in X\;\forall y\in Y,
\end{equation}
while an \df{infimum} (or the \df{greatest lower bound}) of $Y\subseteq X$, denoted by $\rpktarget{INF}\inf Y$, is defined as $x\in X$ such that 
\begin{equation}
        x\leq y,\;\;
        z\leq y\limp z\leq x\;\;
        \forall z\in X\;\forall y\in Y.
\end{equation}
If $I$ is a set and $\{x_\iota\mid\iota\in I\}\subseteq X$, then $\sup\{x_\iota\mid\iota\in I\}=:\sup_{\iota\in I}\{x_\iota\}=:\sup_\iota\{x_\iota\}$ (and analogously for $\inf$). If $I=\NN$ and $i,n\in\NN$, then $\sup_i\{x_i\}=:\bigvee_ix_i$, $\inf_i\{x_i\}=:\bigwedge_ix_i$, $\sup\{x_1,\ldots,x_n\}=:x_1\lor\ldots\lor x_n$ and $\inf\{x_1,\ldots,x_n\}=:x_1\land\ldots\land x_n$. If $(X,\leq)$ is a poset, then $Y\subseteq X$ is called \df{order closed} if{}f $\sup Z_1\in Y$ for every nonempty upwards directed $Z_1\subseteq Y$ such that $\sup Z_1\in X$ and $\inf Z_2\in Y$ for every nonempty downwards directed $Z_2\subseteq Y$ such that $\inf Z_2\in X$. A poset $(X,\leq)$ is called: \df{Dedekind--MacNeille complete} \cite{Dedekind:1872,MacNeille:1937} if{}f every nonempty bounded above subset of $X$ has a supremum, \textit{or}, equivalently, if{}f every bounded below subset of $X$ has an infimum; \df{countably additive complete} if{}f every nonempty bounded above countable subset of $X$ has a supremum \textit{and} every nonempty bounded below countable subset of $X$ has an infimum; \df{lattice} \cite{Peirce:1880,Peirce:1885,Schroeder:1890,Ore:1935} if{}f every subset of $X$ consisting of two elements has a supremum and infimum. A lattice $X$ is called: \df{distributive} \cite{Schroeder:1890} if{}f
\begin{align}
        x\land(y\lor z)&=(x\land y)\lor(x\land z)\;\;\forall x,y,z\in X,\\
        x\lor(y\land z)&=(x\lor y)\land(x\lor z)\;\;\forall x,y,z\in X;
\end{align}
\df{boolean} \cite{Boole:1847,Boole:1854,Whitehead:1898} if{}f it is distributive, contains a \df{least element} $0\in X$ such that $0\leq x$ $\forall x\in X$ and a \df{greatest element} $1\in X$ such that $x\leq 1$ $\forall x\in X$, and
\begin{equation}
        \forall x\in X\;\exists y\in X\;\;\mbox{such that}\;\;x\land y=0,\;\;x\lor y=1,\;\;\mbox{and}\;\;y=:\lnot x;
\end{equation}
\df{Riesz} \cite{Daniell:1919:a,Riesz:1930,Riesz:1937,Riesz:1940} if{}f it is a vector space over $\RR$ such that
\begin{equation}
        x\leq y\limp x+z\leq y+z,\;\;x\geq0\limp\lambda x\geq 0\;\;\forall\lambda\geq0\;\forall x,y,z\in X;
\end{equation}
\df{Banach} \cite{Kantorovich:1935,Kantorovich:1937,Birkhoff:1940} if{}f it is a Riesz lattice equipped with a norm $\n{\cdot}:X\ra\RR^+$ such that $\ab{x}\leq\ab{y}\limp\n{x}\leq\n{y}$ and it is Cauchy complete with respect to this norm, where $\ab{x}:=x\lor(-x)$; an \df{f-algebra} \cite{Birkhoff:Pierce:1956} if{}f it is a Riesz lattice equipped with an associative multiplication $\cdot:X\times X\ra X$ such that $(X,+,\cdot\,)$ is an algebra over $\RR$, $x,y\geq0$ $\limp$ $x\cdot y\geq0$, and $(x\land y=0$, $z\geq0)$ $\limp$ $(x\cdot z)\land y=0$. For boolean lattice the two conditions required for countably additive completeness are equivalent. Every Dedekind--MacNeille complete lattice is countably additive complete. Every Riesz lattice is distributive \cite{Freudenthal:1936}. If $X$ is a vector space over $\CC$ and $Y$ is a real vector subspace of $X$ such that
\begin{equation}
        \forall x\in X\;\exists x_1,x_2\in Y\;\;x=x_1+\ii x_2,
\label{real.vector.sublattice}
\end{equation}
then $X$ is called a \df{complex Riesz lattice} if{}f $Y$ is a Riesz lattice and
\begin{equation}
        \ab{x}:=\sup\{\re(\ee^{\ii\lambda}x)
        =x_1\cos\lambda+x_2\sin\lambda\mid
        \lambda\in[0,2\pipi]\}\in Y\;\;\forall x\in X.
\end{equation}
Given a complex Riesz lattice $X$ one defines $\re(X):=\bigcup_{x\in X}\re(x)$, where $\re(x):=x_1$ in terms of decomposition provided by \eqref{real.vector.sublattice}. A complex Riesz lattice $X$ is called, respectively: \df{Dedekind--MacNeille complete}, \df{countably additive complete}, \df{distributive} if{}f $\re(X)$ satisfies the corresponding property. A complex Banach lattice is defined as a Banach lattice constructed over a complex Riesz lattice. If $Y$ is a real Banach lattice, then its complexification $X:=Y+\ii Y$ is a complex Banach lattice with respect to the norm $\n{x}_X:=\n{\ab{x}}_Y$ \cite{Ando:1969,Mittelmeyer:Wolff:1974}. Unless stated otherwise, all following statements about Riesz and Banach lattices apply to both real and complex case. If $X$ is a Riesz lattice and $x\in X$ then $x^+:=x\lor0$ and $x^-:=(-x)\lor0$ satisfy $x=x^+-x^-$ and $\ab{x}=x^++x^-$. A \df{Riesz dual} of a Riesz lattice $X$ is defined as a set $X^\riesz$ of all linear functions $X\ra\RR$ that map intervals $[x,y]:=\{z\in X\mid x\leq z\leq y\}$ for any $x,y\in X$ to bounded subsets of $\RR$. A Riesz lattice $X$ is called \df{archimedean} if{}f 
\begin{equation}
        \{nx\mid n\in\NN\}\;\mbox{is bounded above}\;\limp x\leq 0\;\;\forall x\in X.
\end{equation}
Every countably additive complete Riesz lattice is archimedean. Every Banach lattice is archi\-me\-de\-an. Every archimedean f-algebra is commutative \cite{Amemiya:1953,Birkhoff:Pierce:1956}. An element $e\in X^+:=\{x\in X\mid x\geq0\}$ of an archimedean Riesz lattice $X$ is called an \df{order unit} if{}f $\forall x\in X$ $\exists\lambda>0$ $\ab{x}\leq\lambda e$ \cite{Freudenthal:1936}. If $X$ is an archimedean Riesz lattice with an order unit $e$, then an \df{order unit norm} on $X$ is defined as a map $\n{\cdot}_e:X\ra\RR^+$ such that $\n{x}_e:=\min\{\lambda\in\RR\mid\ab{x}\leq\lambda e\}$. An \df{MI-space} \cite{Krein:Krein:1940,Kakutani:1941:M} is defined as a Banach lattice with an order unit norm. An \df{abstract $L_p$ space} \cite{Birkhoff:1938,Bohnenblust:1940,Kakutani:1941:M} is defined for $p\in[1,\infty[$ as a Banach lattice $X$ with norm such that
\begin{equation}
        \ab{x}\land\ab{y}=0\;
        \limp\;
        \n{x+y}^p=\n{x}^p+\n{y}^p\;\;\forall x,y\in X,
\end{equation}
and as a countably additive complete MI-space $X$ for $p=\infty$. An abstract $L_\infty$ space will be called \df{proper} if{}f it is Banach dual to some Banach space. Every abstract $L_p$ space for $p\in[1,\infty[$ is Dedekind--MacNeille complete. A commutative ring $\rpktarget{BOOLE}(\boole,+,\cdot\,)$ is called \df{boolean} if{}f $x^2=x$ $\forall x\in \boole$. Every boolean lattice defines a boolean ring with unit by $x+y:=(x\land\lnot y)\lor(\lnot x\lor y)$ and $x\cdot y:=x\land y$, and the converse is also true \cite{Stone:1936}. By this reason both are referred to as a \df{boolean algebra}. A simplest nontrivial example of a boolean algebra is $\two$, consisting of two elements $\{0,1\}$ such that $0\leq1$ and $0\neq1$. 

If $(X_1,\leq_1)$ and $(X_2,\leq_2)$ are partially ordered sets, then a function $f:X_1\ra X_2$ is called: \df{order preserving} if{}f $x\leq_1y\limp f(x)\leq_2 f(y)$ $\forall x,y\in X_1$; \df{order continuous} \cite{Nakano:1950} if{}f it is order preserving, $f(\sup Y)=\sup_{x\in Y}\{f(x)\}$ for every nonempty upwards directed $Y\subseteq X_1$ with $\sup Y\in X_1$, and $f(\inf Y)=\inf_{x\in Y}\{f(x)\}$ for every nonempty downwards directed $Y\subseteq X_1$ with $\inf Y\in X_1$; \df{sequentially order continuous} if{}f it is order preserving, $f(\sup_i\{x_i\})=\sup_i\{f(x_i)\}$ for every nondecreasing sequence $\{x_i\}\subseteq X_1$, and $f(\inf_i\{x_i\})=\inf_i\{f(x)\}$ for every nonincreasing sequence $\{x_i\}\subseteq X_1$. If $X_1$ and $X_2$ are lattices, then a \df{lattice homomorphism} is defined as a function $f:X_1\ra X_2$ such that $f(x\lor y)=f(x)\lor f(y)$ and $f(x\land y)=f(x)\land f(y)$. If $\boole_1$ and $\boole_2$ are boolean algebras, then a \df{boolean homomorphism} is defined as a ring homomorphism $f:\boole_1\ra\boole_2$ such that $f(1)=1$. If $X_1$ and $X_2$ are Riesz lattices, then a \df{Riesz homomorphism} is defined as a linear function $f:X_1\ra X_2$ such that any of equivalent conditions holds: $f(x^+)=(f(x))^+$; $f(\ab{x})=\ab{f(x)}$; $f(x\land y)=f(x)\land f(y)$; $f(x\lor y)=f(x)\lor f(y)$. If $X_1$ and $X_2$ are Banach lattices then a Riesz homomorphism $f:X_1\ra X_2$ is called: \df{unit preserving} if{}f $X_1$ has an order unit norm with an order unit $e_1$, $X_2$ has an order unit norm with an order unit $e_2$ and $f(e_1)=e_2$; \df{norm preserving} if{}f $\n{f(x)}_{X_2}=\n{x}_{X_1}$; \df{isometric} if{}f it is norm preserving and continuous with respect to norm topologies on $X_1$ and $X_2$. A \df{boolean isomorphism} is defined as a bijective boolean homomorphism, while a \df{Riesz isomorphism} is defined as a bijective Riesz homomorphism. Isometric Riesz isomorphisms of Banach lattices coincide with their isometric isomorphisms (surjective isometries). Every isometric Riesz isomorphism is order continuous. Every boolean homomorphism and every Riesz lattice homomorphism is a lattice homomorphism. Every bijective lattice homomorphism is order continuous. Every boolean homomorphism is order preserving. A multiplication in archimedean f-algebra is order continuous.

Let $\catname{B}$ denote a category of boolean algebras and boolean homomorphisms, and let $\Top$ denote a category of topological spaces and continuous functions. The \df{Stone spectrum} \cite{Stone:1936} of a boolean algebra $\boole$ is defined as a set $\rpktarget{STONESPEC}\sp_\mathrm{S}(\boole)$ of nonzero boolean homomorphisms from $\boole$ to $\two$,
\begin{equation}
        \sp_{\mathrm{S}}(\boole):=\Hom_{\catname{B}}(\boole,\two)\setminus\{0\},
\end{equation}
equipped with a topology of open sets given by 
\begin{equation}
\{\Y\subseteq\sp_\mathrm{S}(\boole)\mid\forall \xx\in \Y\;\exists x\in\boole\;\;\xx\in\hat{x}\subseteq\Y\},
\end{equation}
where $\hat{\cdot}:\boole\ra\Hom_\Top(\sp_\mathrm{S}(\boole),\two)$ is the \df{Stone representation map} defined by $\hat{x}:=\{\xx\in\sp_\mathrm{S}(\boole)\mid\xx(x)=1\}$. The set $\{\hat{x}\subseteq\sp_{\mathrm{S}}(\boole)\mid x\in\boole\}$ consists of all subsets of $\sp_\mathrm{S}(\boole)$ that are open and closed, and is boolean isomorphic to $\boole$. An order closed vector subspace $Y$ of a Riesz lattice $X$ is called a \df{band} if{}f $(x\in Y,\;\ab{y}\leq\ab{x})\limp y\in Y$. If $X$ is an archimedean Riesz lattice and $Z\subseteq X$, then 
\begin{equation}
        Z^\oc:=\{x\in X\mid\ab{x}\land\ab{y}=0\;\forall y\in Z\}
\end{equation}
is a band and $Z^\oc{}^\oc=Z$. A subset $Y$ of an archimedean Riesz lattice $X$ is called a \df{projection band} if{}f $Y+Y^\oc=X$. If $X$ is archimedean and Dedekind--MacNeille complete, then each band of $X$ is a  projection band. The set of all bands of an archimedean Riesz lattice $X$ forms a Dedekind--MacNeille complete boolean algebra $\boole$, with $Y\land Z:=Y\cap Z$, $Y\lor Z:=(Y+Z)^\oc{}^\oc$, $1:=X$, $0:=\{0\}$, $\lnot Y:=Y^\oc$, $(Y\leq Z):=(Y\subseteq Z)$, while the set of all projection bands of $X$ forms a boolean subalgebra of $\boole$. These two boolean algebras coincide if{}f $X$ is Dedekind--MacNeille complete.

A \df{measure} on a boolean algebra $\boole$ is defined as a function $\rpktarget{MEASURE}\mu:\boole\ra[0,\infty]$ such that $\mu(0)=0$. It is called: \df{countably additive} if{}f
\begin{equation}
        \mu(\bigvee_i x_i)=\sum_i\mu(x_i)\;\;\mbox{for}\;\;(i\neq j\limp x_i\land x_j=0);
\label{countably.additive.measure}
\end{equation}
\df{strictly positive} if{}f $x\neq0\limp\mu(x)>0$; \df{finite} if{}f $\cod(\mu)\subseteq\RR^+$; \df{semi-finite} if{}f
\begin{equation}
        \forall x\in\boole\;\exists y\in\boole\;\;\mu(x)=\infty\limp(y\leq x\;\mbox{and}\;0<\mu(y)<\infty).
\end{equation}
The space of all semi-finite countably additive measures on a boolean algebra $\boole$ will be denoted $\rpktarget{MEASURE.W}\W(\boole)$, while the subset of strictly positive elements of $\W(\boole)$ will be denoted $\rpktarget{MEASURE.W.ZERO}\W_0(\boole)$. A boolean algebra will be called: \df{ccb-algebra} if{}f it is \underline{c}ountably additive \underline{c}omplete; \df{Dcb-algebra} if{}f it is \underline{D}edekind--MacNeille \underline{c}omplete; \df{mcb-algebra} if{}f it allows a semi-finite strictly positive countably additive \underline{m}easure and is Dedekind--MacNeille \underline{c}omplete. There exist Dcb-algebras that do not admit any countably additive measure \cite{Szpilrajn:1934,Horn:Tarski:1948}. A pair $(\boole,\mu)$ of a ccb-algebra $\boole$ and a strictly positive countably additive measure $\mu$ on $\boole$ is called a \df{measure algebra}. A measure algebra $(\boole,\mu)$ is called: \df{semi-finite} if{}f $\mu$ is semi-finite; \df{localisable} (or \df{Maharam}) if{}f $\boole$ is an mcb-algebra and $\mu$ is semi-finite. An \df{evaluation} on a boolean algebra $\boole$ is defined as a function $\phi:\boole\ra\RR$ satisfying $\phi(0)=0$ and countably additive in the sense of \eqref{countably.additive.measure} with $\mu$ substituted by $\phi$. It is called: \df{positive} if{}f $\cod(\phi)\subseteq\RR^+$; \df{strictly positive} if{}f $x\neq0$ $\limp$ $\phi(x)>0$. The set of all evaluations on $\boole$ will be denoted $\eval(\boole)$, and its subsets of all positive (resp., strictly positive) elements will be denoted by $\eval(\boole)^+$ (resp. $\eval(\boole)^+_0$). Every positive evaluation is an element of $\W(\boole)$, hence the diagram
\begin{equation}
\xymatrix{
        \eval(\boole)^+_0
        \ar@{^{(}->}[r]
        \ar@{^{(}->}[d]
        &
        \W_0(\boole)
        \ar@{^{(}->}[d]
        \\
        \eval(\boole)^+
        \ar@{^{(}->}[r]
        &
        \W(\boole)
}
\end{equation}
is commutative. 

Let $\boole$ be an arbitrary boolean algebra, let $X$ be a vector space of all sums $\sum_{i=1}^n\lambda_ix_i$ with $\{\lambda_i\}\subseteq\RR$ and $\{x_i\}\subseteq\boole$, and let $Y$ be a vector subspace of $X$ spanned by the elements of $X$ of the form $(x_1\lor x_2)-x_1-x_2$ for $x_1,x_2\in\boole$ such that $x_1\land x_2=0$. The space $X/Y$ can be equipped with the norm
\begin{equation}
        \n{f}_\infty:=\min\{\lambda\geq0\mid\ab{f}\leq \lambda\chr(1)\}\;\;\forall f\in X/Y,
\end{equation}
where $\chr:\boole\ra X/Y$ is defined as a map from $x\in\boole$ to an image of $x\in X$ in $X/Y$. The space $\rpktarget{LPBINF}L_\infty(\boole)$ is defined as a Cauchy completion of $X/Y$ in $\n{\cdot}_\infty$.\footnote{Equivalently, one can define (real or complex) Banach lattice $L_\infty(\boole)$ as the space of all (real or complex) continuous functions on the Stone spectrum $\sp_\mathrm{S}(\boole)$, endowed with its multiplication, linear and order structures, and norm given by $\n{f}:=\sup_{\xx\in\sp_\mathrm{S}(\boole)}\{\ab{f(\xx)}\}$. However, for the purposes of Section \ref{integr.compar.section}, we want to avoid any dependence on Stone representation in the definition of $L_\infty(\boole)$.} The order unit of $L_\infty(\boole)$ is given by the constant function taking the value $1$ everywhere. The projection band algebra of $L_\infty(\boole)$ is boolean isomorphic to $\boole$. $L_\infty(\boole)$ is Dedekind--MacNeille complete if{}f $\boole$ is, and is countably additive complete if{}f $\boole$ is. If $f:\boole_1\ra\boole_2$ is a boolean homomorphism, then the formula
\begin{equation}
        L_\infty(f)(\chr(x))=\chr(f(x))\;\;\forall x\in\boole_1
\label{Linfty.morphism}
\end{equation}
determines a unique Riesz homomorphism $L_\infty(f):L_\infty(\boole_1)\ra L_\infty(\boole_2)$ which is unit preserving, and is surjective (resp.: injective; order continuous) if{}f $f$ is surjective (resp.: injective; order continuous). If $\boole$ is a ccb-algebra, then $\rpktarget{LPBZERO}L_0(\boole)$ is defined as a set of all functions $f:\RR\ra\boole$ such that 
\begin{equation}
        f(\lambda_1)=\sup_{\lambda_2>\lambda_1}f(\lambda_2)\;\forall\lambda_1\in\RR,\;\;
        \inf_{\lambda\in\RR}f(\lambda)=0,\;\;
        \sup_{\lambda\in\RR}f(\lambda)=1.
\end{equation}
The $L_0(\boole)$ space can be equipped with an f-algebra structure, provided by
\begin{equation}
        (x\cdot y)(\lambda_1)
        :=\sup\left\{
                x(\lambda_2)\land y\left(\frac{\lambda_1}{\lambda_2}\right)\mid
                \lambda_2\in\QQ,\;\lambda_2>0
        \right\}
        \;\;\forall x,y\geq0,
\end{equation}
and
\begin{equation}
        x\cdot y:=x^+\cdot x^+-x^+\cdot y^--x^-\cdot y^++x^-\cdot y^-\;\;\forall x,y\in L_0(\boole).
\end{equation}
For any measure algebra $(\boole,\mu)$, the map
\begin{equation}
        \n{\cdot}_1:L_0(\boole)\ni f\mapsto\int_0^\infty\dd\lambda\,\mu(\ab{f(\lambda)})\in[0,\infty],
\end{equation}
where $\dd\lambda$ is a Lebesgue measure on $\RR$, allows to define
\begin{equation}\rpktarget{LPBM.ONE}
        L_1(\boole,\mu):=\{f\in L_0(\boole)\mid\n{f}_1<\infty\}.
\end{equation}
Moreover, for $p\in\,]1,\infty[$,
\begin{equation}
        \ab{f(\lambda)}^p:=\left\{
        \begin{array}{ll}
                \ab{f(\lambda^{1/p})}&:\;\lambda\geq0\\
                1&:\;\lambda<0
        \end{array}
        \right.
\end{equation}
allows to define 
\begin{equation}\rpktarget{LPBM}
        L_p(\boole,\mu):=\{f\in L_0(\boole)\mid\ab{f}^p\in L_1(\boole,\mu)\}
\end{equation}
and 
\begin{equation}
        \n{\cdot}_p:L_p(\boole,\mu)\ni f\mapsto\n{\ab{f}^p}_1^{1/p}\in\RR^+.
\end{equation}
For $p\in[1,\infty[$ the maps $\n{\cdot}_p$ are norms on $L_p(\boole,\mu)$ under which $L_p(\boole,\mu)$ are Cauchy complete. The spaces $L_p(\boole,\mu)$ inherit an f-algebra structure from $L_0(\boole)$ and are Dedekind--MacNeille complete. If $\boole$ is a ccb-algebra and $\mu_1,\mu_2\in\W(\boole)$, then $L_p(\boole,\mu_1)$ and $L_p(\boole,\mu_2)$ are isometrically Riesz isomorphic. If $(\boole,\mu)$ is a localisable measure algebra, then the band algebra of $L_p(\boole,\mu)$ is boolean isomorphic to $\boole$. If $(\boole,\mu)$ is a measure algebra and $x\in L_1(\boole,\mu)$, then the function $\int\mu:L_1(\boole,\mu)\ra\RR$, defined by
\begin{equation}
        \int\mu x:=\n{x^+}_1-\n{x^-}_1=\int_0^\infty\dd\lambda\,\mu(x(\lambda))-\int_0^\infty\dd\lambda\,\mu(-x(\lambda)),
\end{equation}
is linear and order continuous, and satisfies
\begin{align}
\n{x}_1&=\int\mu\ab{x}\;\;\forall x\in L_1(\boole,\mu),\\
\n{x}_p&=\left(\int\mu\ab{x}^p\right)^{1/p}=\n{\ab{x}^p}_1^{1/p}\;\;\forall x\in L_p(\boole,\mu)\;\forall p\in[1,\infty[.
\end{align}
The space $\eval(\boole)$ is an abstract $L_1$ space, and if $(\boole,\mu)$ is a semi-finite measure algebra, then there exists a bijective Riesz isomorphism between $\eval(\boole)$ and $L_1(\boole,\mu)$. Hence, there exists a bijection between $L_1(\boole,\mu)^+$ and $\eval(\boole)^+$. For any measure algebra $(\boole,\mu)$ and $\gamma\in\,]0,1[$ there is a Riesz isomorphism $L_{1/\gamma}(\boole,\mu)^\riesz\iso L_{1/(1-\gamma)}(\boole,\mu)$ and a Banach space duality $L_{1/\gamma}(\boole,\mu)^\banach\iso L_{1/(1-\gamma)}(\boole,\mu)$ determined by the map
\begin{equation}
        L_{1/\gamma}(\boole,\mu)\times L_{1/(1-\gamma)}(\boole,\mu)\ni(x,y)\mapsto\int\mu xy\in\RR.
\end{equation}
The space $L_\infty(\boole)$ can be identified with the linear subspace of $L_0(\boole)$ generated by $\chr(1)$, and in such case $L_1(\boole,\mu)\times L_\infty(\boole)\ni(x,y)\mapsto x\cdot y\in L_1(\boole,\mu)$ is a bilinear maps, while
\begin{equation}
        L_1(\boole,\mu)\times L_\infty(\boole)\ni(x,y)\mapsto\int\mu xy\in\RR
\end{equation}
is a bilinear functional. If $(\boole,\mu)$ is semi-finite, then $L_1(\boole,\mu)$ is isometrically Riesz isomorphic to $L_\infty(\boole)^\riesz$. According to Segal's theorem \cite{Segal:1951}, the space $L_1(\boole,\mu)^\banach$ is isometrically Riesz isomorphic to $L_\infty(\boole)$ if{}f $(\boole,\mu)$ is localisable, and in such case all Banach preduals of $L_\infty(\boole)$ are isometrically (and Riesz) isomorphic. If $(\boole,\mu)$ is a measure algebra, then $\rpktarget{BOOLMEASIDEAL}\boole^\mu:=\{x\in\boole\mid\mu(x)<\infty\}$ is a boolean algebra and an ideal in $\boole$. If $(\boole,\mu)$ is semi-finite, then an embedding $\boole^\mu\subseteq\boole$ is an order continuous injective boolean homomorphism. For any measure algebras $(\boole_1,\mu_1)$ and $(\boole_2,\mu_2)$ a boolean homomorphism $f:\boole_1^{\mu_1}\ra\boole_2^{\mu_2}$ is called \df{measure preserving} if{}f $\mu_2(f(x))=\mu_1(x)$ $\forall x\in\boole_1^{\mu_1}$. Every measure preserving boolean homomorphism is injective. Every measure preserving boolean homomorphism $f:\boole_1^{\mu_1}\ra\boole_2^{\mu_2}$ is order continuous. If $p\in[1,\infty[$, then a measure preserving boolean homomorphism $f:\boole_1^{\mu_1}\ra\boole_2^{\mu_2}$ determines a unique injective, order continuous, isometric Riesz homomorphism $\widetilde{f}:L_p(\boole_1,\mu_1)\ra L_p(\boole_2,\mu_2)$ given by
\begin{equation}
        \widetilde{f}(\chr(x))=\chr(f(x))\;\;\forall x\in\boole_1^{\mu_1},
\label{Lp.homo}
\end{equation}
where the function $\chr:\boole\ra L_0(\boole)$, defined by
\begin{equation}\rpktarget{CHICHAR}
        (\chr(x))(\lambda):=
        \left\{
        \begin{array}{ll}
                1&:\;\lambda<0\\
                x&:\;\lambda\in[0,1[\\
                0&:\;1\leq\lambda,
        \end{array}
        \right.
\end{equation}
is additive, injective lattice homomorphism. The map $\widetilde{f}$ is surjective if{}f $f$ is surjective. Hence, every measure preserving boolean isomorphism determines a unique corresponding isometric Riesz isomorphism of associated $L_p$ spaces.

Every abstract $L_p$ space for $p\in[1,\infty[$ is Dedekind--MacNeille complete and countably additive complete. If $X$ is an abstract $L_{1/\gamma}$ space with $\gamma\in\,]0,1[$, then by And\={o}'s theorem \cite{Ando:1969}, its Banach dual $X^\banach$ is an abstract $L_{1/(1-\gamma)}$ space. If $X$ is an abstract $L_1$ space, then its Banach dual $X^\banach$ coincides with its Riesz dual $X^\riesz$ and is a proper abstract $L_\infty$ space. According to the Bohnenblust--Kakutani--Nakano theorem \cite{Bohnenblust:1940,Kakutani:1941:M,Kakutani:1941:L,Bohnenblust:Kakutani:1941,Ando:1969}:
\begin{enumerate}
\item[(i)] every abstract $L_p$ space $X$ for $p\in[1,\infty[$ is isometrically Riesz isomorphic to some $L_p(\boole,\mu)$ space, where $\boole$ is uniquely determined as an mcb-algebra of projection bands of $X$, while $\mu\in\W(\boole)$ is (nonuniquely) determined by $\boole$ and a norm of $X$, so that $(\boole,\mu)$ is a localisable measure algebra;
\item[(ii)] every abstract $L_\infty$ space $X$ determines a ccb-algebra $\boole$ of its projection bands, and $X$ is isometrically Riesz isomorphic to $L_\infty(\boole)$. Hence, every Dedekind--MacNeille complete abstract $L_\infty$ space $X$ is isometrically Riesz isomorphic to $L_\infty(\boole)$, where $\boole$ is a Dcb-algebra.
\end{enumerate}
By Segal's theorem \cite{Segal:1951}, this implies that
\begin{enumerate}
\item[(iii)] every proper abstract $L_\infty$ space $X$ is isometrically Riesz isomorphic to $L_\infty(\boole)$, where $\boole$ is an mcb-algebra.
\end{enumerate}

Consider the categories: $\locMeAlgIso$ of localisable measure algebras and measure preserving boolean isomorphisms; $\sfMeAlgIso$ of semi-finite measure algebras and measure preserving boolean isomorphisms; $\sfMeAlg$ of semi-finite measure algebras and measure preserving boolean homomorphisms; $\mcBIso$ of mcb-algebras and boolean isomorphisms; $\ccBIso$ of ccb-algebras and boolean isomorphisms; $\ccB$ of ccb-algebras and boolean homomorphisms; $\ccBi$ of ccb-algebras and injective boolean homomorphisms; $\mcBc$ of mcb-algebras and order continuous boolean homomorphisms; $\DcBc$ of Dcb-algebras and order continuous boolean homomorphisms; $\LpIso$ of abstract $L_p$ spaces for fixed $p\in[1,\infty[$ with unit preserving isometric Riesz isomorphisms; $\Lpinp$ of abstract $L_p$ spaces for fixed $p\in[1,\infty[$ with unit preserving norm preserving injective Riesz homomorphisms; $\LinfIso$ of proper abstract $L_\infty$ spaces with unit preserving isometric Riesz isomorphisms; $\Linfc$ of proper abstract $L_\infty$ spaces with order continuous unit preserving Riesz homomorphisms; $\aLinfIso$ of abstract $L_\infty$ spaces with unit preserving isometric Riesz isomorphisms; $\aLinfi$ of abstract $L_\infty$ spaces with injective unit preserving Riesz homomorphisms; $\aLinf$ of abstract $L_\infty$ spaces with unit preserving Riesz homomorphisms; $\DcaLinfc$ of Dedekind--MacNeille complete abstract $L_\infty$ spaces with order continuous unit preserving Riesz homomorphisms; $\bLinf$ of $L_\infty(\boole)$ spaces over boolean algebras $\boole$ with unit preserving Riesz homomorphisms. Let $\mathrm{Frg}:\sfMeAlg\ra\ccBi$ denote the forgetful functor that forgets about the measure, let $\mathrm{MeAlgL}_p:\sfMeAlg\ra\Lpinp$ denote the functor that associates $L_p(\boole,\mu)$ space to each $(\boole,\mu)\in\Ob(\sfMeAlg)$ and $\widetilde{f}$ given by \eqref{Lp.homo} to each $f\in\Mor(\sfMeAlg)$, let $\mathrm{L}_\infty:\catname{B}\ra\bLinf$ denote the functor that assigns $L_\infty(\boole)$ to each boolean algebra $\boole$ and maps the corresponding homomorphisms according to \eqref{Linfty.morphism}, and let $\PrjB$ be a functor from the category of archimedean Riesz lattices with Riesz homomorphisms to the category of boolean algebras with boolean homomorphisms that associates an algebra of projection bands to each Riesz lattice, and a boolean homomorphism of projection bands algebra that is determined by a Riesz homomorphism. Then the properties discussed above can be summarised in terms of the following commutative diagram:
\begin{equation}\xymatrix{%
\mcBc
\ar@{ >->}[r]
\ar@<-0.5ex>[d]_{\mathrm{L}_\infty}
&
\DcBc
\ar@{ >->}[r]
\ar@<-0.5ex>[d]_{\mathrm{L}_\infty}
&
\ccB
\ar@{ >->}[r]
\ar@<-0.5ex>[d]_{\mathrm{L}_\infty}
&
\catname{B}
\ar@<-0.5ex>[d]_{\mathrm{L}_\infty}
\\
\Linfc
\ar@{ >->}[r]
\ar@<-0.5ex>[u]_{\PrjB}
&
\DcaLinfc
\ar@{ >->}[r]
\ar@<-0.5ex>[u]_{\PrjB}
&
\aLinf
\ar@{ >->}[r]
\ar@<-0.5ex>[u]_{\PrjB}
&
\bLinf
\ar@<-0.5ex>[u]_{\PrjB}
\\
\LinfIso
\ar@{ >->}[r]
\ar@{ >->}[u]
\ar@<+0.5ex>[d]^{\PrjB}
\ar@<-1.5ex>@/_4.7pc/[ddd]_{(\cdot)^\banach}
&
\aLinfIso
\ar@{ >->}[r]
\ar@<+0.5ex>[d]^{\PrjB}
&
\aLinfi
\ar@{ >->}[u]
\ar@<+0.5ex>[d]^{\PrjB}
&
\\
\mcBIso
\ar@{ >->}[r]
\ar@<+0.5ex>@{ >->}@/^3.3pc/[uuu]
\ar@<+0.5ex>[u]^{\mathrm{L}_\infty}
&
\ccBIso
\ar@{ >->}[r]
\ar@<+0.5ex>[u]^{\mathrm{L}_\infty}
&
\ccBi
\ar@<+0.5ex>@{ >->}@/^3.3pc/[uuu]
\ar@<+0.5ex>[u]^{\mathrm{L}_\infty}
&
\\
\locMeAlgIso
\ar@{ >->}[r]
\ar[u]_{\mathrm{Frg}}
\ar[d]^(0.45){\!\mathrm{MeAlgL}_p}
&
\sfMeAlgIso
\ar@{ >->}[r]
\ar[u]_{\mathrm{Frg}}
\ar[dl]^{\;\mathrm{MeAlgL}_p}
&
\sfMeAlg
\ar[u]_{\mathrm{Frg}}
\ar[d]^{\mathrm{MeAlgL}_p}
&
\\
{\catname{L}_p\catname{Iso}}
\ar@{ >->}[rr]
\ar@/^4.7pc/[uuu]_{(\cdot)_\star}
\ar@/^3pc/[uu]|(0.76){\PrjB}
&
&
{\catname{L}_p\catname{inp}}
&
}
\label{commutative.integration.diagram}
\end{equation}
where $(\cdot)^\banach$ and $(\cdot)_\star$ denote the Banach space duality functors, defined by \eqref{banach.duality.functor} and \eqref{banach.preduality.functor}, respectively, and are considered only for $p=1$. The $(\mathrm{L}_\infty,\PrjB)$ in \eqref{commutative.integration.diagram} set up an equivalence of categories that are their domain and codomain, e.g.
\begin{align}
        \mathrm{L}_\infty\circ\PrjB\iso\id_{\bLinf}&,
        \;\;\;\PrjB\circ\mathrm{L}_\infty\iso\id_{\catname{B}};
        \label{bLinf.B.equiv}
        \\
        \mathrm{L}_\infty\circ\PrjB\iso\id_{\aLinf}&,
        \;\;\;\PrjB\circ\mathrm{L}_\infty\iso\id_{\ccB};
        \label{aLinf.ccB.equiv}\\
        \mathrm{L}_\infty\circ\PrjB\iso\id_{\Linfc}&,
        \;\;\;\PrjB\circ\mathrm{L}_\infty\iso\id_{\catname{mcBc}};
        \label{Linfc.mcBc.equiv}
\end{align}
and so on. The duality between $L_{1/\gamma}(\boole,\mu)$ and $L_{1/(1-\gamma)}(\boole,\mu)$ spaces for semi-finite measure algebra $(\boole,\mu)$ and $\gamma\in\,]0,1[$ extends to an And\={o} duality of abstract $L_p$ spaces, which takes the form of the commutative diagram of equivalence of categories,
\begin{equation}
\xymatrix{
        \catname{L}_{1/\gamma}\catname{Iso}
        \ar@<+0.5ex>[rr]^{(\cdot)^\banach}
        &&
        \catname{L}_{1/(1-\gamma)}\catname{Iso},
        \ar@<+0.5ex>[ll]^{(\cdot)_\star}
}
\;\mbox{ with }\;
(\cdot)^\banach\circ(\cdot)_\star\iso\id_{\catname{L}_{1/\gamma}\Iso}\iso(\cdot)_\star\circ(\cdot)^\banach.
\end{equation}

This concludes the description of an algebraic integration theory based on boolean algebras and Banach lattices. We finish this section with a brief discussion of the elementary notions of Fr\'{e}chet's and Daniell's approaches to integration.

We begin with Fr\'{e}chet's approach. For a given set $\X$, a \df{countably additive algebra} on $\X$ is defined as a family $\rpktarget{MHO}\mho(\X)$ of subsets of $\X$ such that
\begin{equation}
        \varnothing\in\mho(\X),\;\;
        \Y\in\mho(\X)\limp\X\setminus\Y\in\mho(\X),\;\;
        \bigcup_i\X_i\in\mho(\X)\;\;\mbox{for any sequence}\;\;\{\X_i\}\subseteq\mho(\X).
\end{equation}
A \df{countably additive ideal} \cite{Ulam:1930} of a countably additive algebra $\mho(\X)$ on $\X$ is defined as a family $\rpktarget{MHO.ZERO}\mho^0(\X)$ of subsets of $\mho(\X)$ such that
\begin{enumerate}
\item[1)] $\varnothing\in\mho^0(\X)$,
\item[2)] $(\X_1\in\mho^0(\X),\;\X_2\in\mho(\X),\;\X_2\subseteq\X_1)\limp\X_2\in\mho^0(\X)$,
\item[3)] $\bigcup_i\X_i\in\mho^0(\X)$ for any countable set $\{\X_i\}\subseteq\mho^0(\X)$.
\end{enumerate}
A \df{premeasurable space} is defined as a pair $(\X,\mho(\X))$, while a \df{measurable space} is defined as a triple $(\X,\mho(\X),\mho^0(\X))$, where $\mho(\X)$ is any countable additive algebra on $\X$, while $\mho^0(\X)$ any countably additive ideal of $\mho(\X)$. A \df{complete morphism} of premeasurable spaces, $(\X_1,\mho_1(\X_1),\mho^0_1(\X_1))\ra(\X_2,\mho_2(\X_2),\mho^0_2(\X_2))$, is defined as a map $f:\X_1\ra\X_2$ such that $f^{-1}(\Y)\in\mho_1(\X_1)$ $\forall\Y\in\mho_2(\X_2)$ and $f^{-1}(\Z)\in\mho^0_1(\X_1)$ $\forall\Z\in\mho^0(\X_2)$. A \df{measure} on a premeasurable space $(\X,\mho(\X))$ is defined as a function $\rpktarget{SET.MEASURE}\tmu:\mho(\X)\ra[0,\infty]$ such that $\tmu(\varnothing)=0$. A measure is called \df{countably additive} if{}f $\tmu(\bigcup_i\X_i)=\sum_i\tmu(\X_i)$ for any countable sequence $\{\X_i\}\subseteq\mho(\X)$ satisfying $i\neq j\limp\X_i\cap\X_j=\varnothing$. A set of all countably additive measures on $(\X,\mho(\X))$ will be denoted $\rpktarget{SET.MEASURE.SET}\Meas^+(\X,\mho(\X))$. Moreover, $\Meas(\X,\mho(\X)):=\{\tmu:=\tmu_1-\tmu_2\mid\tmu_1,\tmu_2\in\Meas^+(\X,\mho(\X))\}$. A \df{measure space} is defined as a triple $(\X,\mho(\X),\tmu)$, where $(\X,\mho(\X))$ is a premeasurable space, and $\tmu$ is countably additive measure on it. Given a measure space $(\X,\mho(\X),\tmu)$, a set $\Y\subseteq\X$ is called \df{$\tmu$-null} if{}f there exists $\Z\subseteq\mho(\X)$ such that $\Y\subseteq\Z$ and $\tmu(\Z)=0$. A family of all $\tmu$-null subsets of $\X$ is denoted by $\rpktarget{NULL}\nul(\X,\mho(\X),\tmu)$. A measure $\tmu:\mho(\X)\ra[0,\infty]$ is called \df{atomless} if{}f there exists no $\Y\in\mho(\X)$ satisfying ($\tmu(\Y)>0$ and for any $\Z\in\mho(\X)$ such that $\Z\subseteq\Y$ it holds that either $\Z$ or $\Y\setminus\Z$ is $\tmu$-null). One says that the property $Q(\xx)$ holds for \df{$\tmu$-almost every} $\xx\in\X$ if{}f $\{\xx\in\X\mid Q(\xx)\;\mbox{is false}\}\in\nul(\X,\mho(\X),\tmu)$. The set
\begin{equation}\rpktarget{MHO.TMU}
\mho^{\tmu}(\X):=\mho(\X)\cap\nul(\X,\mho(\X),\tmu)=\{\Y\in\mho(\X)\mid\tmu(\Y)=0\}
\end{equation}
is a countably additive ideal of $\mho(\X)$, hence, every measure space $(\X,\mho(\X),\tmu)$ determines a corresponding measurable space $(\X,\mho(\X),\mho^\tmu(\X))$. A measure space $(\X,\mho(\X),\tmu)$ is called: \df{semi-finite} if{}f 
\begin{equation}
        \forall\X_1\in\mho(\X)\;\exists\X_2\in\mho(\X)\;\;\tmu(\X_1)=\infty\limp(\X_2\subseteq\X_1\mbox{ and }0<\tmu(\X_2)<\infty);
\end{equation}
\df{localisable} (or \df{Maharam}) if{}f it is semi-finite and for all $Y\subseteq\mho(\X)$ there exists $\X_1\in\mho(\X)$ such that
\begin{enumerate}
\item[1)] $\Y\setminus\X_1\in\nul(\X,\mho(\X),\tmu)\;\forall\Y\in Y$,
\item[2)] $(\X_2\in\mho(\X),\;\Y\setminus\X_2\in\nul(\X,\mho(\X),\tmu)\;\forall\Y\in Y)\limp\X_1\setminus\X_2\in\nul(\X,\mho(\X),\tmu)$.
\end{enumerate}
A measurable space $(\X,\mho(\X),\mho^0(\X))$ will be called \df{localisable} (or \df{Maharam}) if{}f there exists a measure $\tmu$ on $(\X,\mho(\X))$ such that $\mho^0(\X)=\mho^\tmu(\X)$ and $(\X,\mho(\X),\mho^\tmu(\X))$ is localisable. If $(\X,\mho(\X))$ is a premeasurable space and $\Y\subseteq\X$, then $\mho^\X(\Y):=\{\Z\cap\Y\mid\Z\in\mho(\X)\}$ is a countably additive algebra on $\Y$. A function $f:\Y\ra\RR$ is called \df{$\mho(\X)$-measurable} if{}f $\{\xx\in\X\mid f(\xx)\leq\lambda\}\subseteq\mho^\X(\Y)$ $\forall\lambda\in\RR$. A function $f:\X\ra\RR$ is called \df{$\tmu$-simple} if{}f $f=\sum_{i=1}^n\lambda_i\chr_{\Y_i}$, where $n\in\NN$, $\{\lambda_i\}\subseteq\RR$, $\{\Y_i\}\subseteq\X$ are $\mho(\X)$-measurable sets with $\tmu(\Y_i)<\infty$, and $\chr_{\Y_i}$ are characteristic functions of $\Y_i$. A \df{$\tmu$-integral} of a $\tmu$-simple $f$ is defined as $\int\tmu f:=\sum_{i=1}^n\lambda_i\mu(\Y_i)$. A function $f:\X\ra\RR$ is called \df{$\tmu$-integrable} if{}f $f=f_a-f_b$, where  $f_o\in\{f_a,f_b\}$ satisfy
\begin{enumerate}
\item[1)] $\X\setminus\dom f_o$ is $\tmu$-null,
\item[2)] $f_o(\xx)\in\RR^+$ $\forall\xx\in\dom f_o$,
\item[3)] there exists a nondecreasing sequence $\{f_i\}$ of simple functions $f_i:\X\ra\RR^+$ such that $\sup_i\{\int\tmu f_i\}<\infty$ and $\lim_{i\ra\infty}f_i(\xx)=f_o(\xx)$ holds $\tmu$-almost everywhere.
\end{enumerate}
A \df{$\tmu$-integral} of $\tmu$-integrable $f$ is defined as $\int\tmu f:=\int\tmu f_a-\int\tmu f_b$. If $(\X,\mho(\X),\tmu)$ is a measure space, then the set of functions $f:\X\ra\RR$ such that
\begin{enumerate}
\item[i)] $\X\setminus\dom f$ is $\tmu$-null,
\item[ii)] $\exists\Y\subseteq\X$ such that $\X\setminus\Y$ is $\tmu$-null and $f|_\Y$ is $\mho^\X(\Y)$-measurable,
\end{enumerate}
is denoted by $\Lcal_0(\X,\mho(\X),\tmu)$. A space $\Lcal_\infty(\X,\mho(\X),\tmu)$ is defined as a set of $f\in\Lcal_0(\X,\mho(\X),\tmu)$ such that
\begin{equation}
        \exists\lambda\geq0\;\;
        \X\setminus\{\xx\in\dom f\mid
        \ab{f(\xx)}\leq\lambda\}
        \in\nul(\X,\mho(\X),\tmu).
\end{equation}
A space $\Lcal_p(\X,\mho(\X),\tmu)$, for $p\in\,]1,\infty[$, is defined as a set of all $f\in\Lcal_0(\X,\mho(\X),\tmu)$ such that $\ab{f}^p$ is $\tmu$-integrable. For $p\in[1,\infty]\cup\{0\}$ \cite{Riesz:1910,Radon:1913,Dunford:1938}
\begin{equation}\rpktarget{LPMEASURESET}
        L_p(\X,\mho(\X),\tmu):=\Lcal_p(\X,\mho(\X),\tmu)/=_\tmu,
\end{equation} 
where $=_\tmu$ is an equivalence relation on elements of $\Lcal_0(\X,\mho(\X),\tmu)$ such that $f_1=_\tmu f_2$ if{}f $f_1=f_2$ holds $\tmu$-almost everywhere. 

By Wecken's theorem \cite{Wecken:1939}, every measurable space $(\X,\mho(\X),\mho^0(\X))$ determines a ccb-algebra $\boole$ by $\boole:=\mho(\X)/\mho^0(\X)$, and, in particular, every measure space $(\X,\mho(\X),\tmu)$ determines a ccb-algebra\rpktarget{BOOLETMUALG}
\begin{equation}
        \boole_\tmu:=\mho(\X)/\mho^\tmu(\X)=\mho(\X)/\{\Y\in\mho(\X)\mid\tmu(\Y)=0\},
\label{wecken.eins}
\end{equation}
and a measure algebra $(\boole_\tmu,\mu)$, with 
\begin{equation}
        \mu([\Z]_{\boole_\tmu}):=\tmu(\Z)\;\;\forall\Z\in\mho(\X),
\label{wecken.zwei}
\end{equation}
where the map $\mho(\X)\ni\Z\mapsto[\Z]_{\boole_\tmu}\in\boole_\tmu$ is defined by  \eqref{wecken.eins}, and is sequentially order continuous. On the other hand, for every ccb-algebra $\boole$ the Loomis--Sikorski theorem \cite{Loomis:1947,Sikorski:1948} provides an explicit construction of a measurable space $(\sp_\mathrm{S}(\boole),\mho_{\mathrm{LS}}(\sp_\mathrm{S}(\boole)),\mho^0_{\mathrm{LS}}(\sp_\mathrm{S}(\boole)))$, such that $\mho_{\mathrm{LS}}(\sp_\mathrm{S}(\boole))/\mho^0_{\mathrm{LS}}(\sp_\mathrm{S}(\boole))$ is boolean isomorphic to $\boole$.\footnote{The elements of $\mho_{\mathrm{LS}}(\sp_\mathrm{S}(\boole))$ are clopen subsets of $\sp_{\mathrm{S}}(\boole)$, while $\mho^0_{\mathrm{LS}}(\sp_\mathrm{S}(\boole))$ consists of all meager subsets of $\sp_{\mathrm{S}}(\boole)$.} As a consequence, one can show that for every measure algebra $(\boole,\mu)$ there exists a measure preserving isomorphism to a measure algebra of some measure space. By Kelley--Namioka theorem \cite{Kelley:Namioka:1963}, the measure space is localisable if{}f the corresponding ccb-algebra is an mcb-algebra. Together with a contravariant equivalence between order continuous boolean homomorphisms and complete morphisms of measurable spaces (see e.g. \cite{Fremlin:2000}), this allows to establish the following categorical equivalence. Let $\locMeSp$ be a category consisting of localisable measurable spaces and complete morphisms, let $\mathrm{W}:\locMeSp^\op\ra\mcBc$ be a functor that assigns to each $(\X,\mho(\X),\mho^0(\X))$ its mcb-algebra $\mho(\X)/\mho^0(\X)$, and to each complete morphism $(\X_1,\mho_1(\X),\mho^0_1(\X_1))\ra(\X_2,\mho_2(\X_2),\mho^0_2(\X_2))$ the corresponding order continuous boolean homomorphism $\mho_2(\X_2)/\mho^0_2(\X_2)\ra\mho_1(\X_1)/\mho^0_1(\X_1)$. Let $\mathrm{LS}:\mcBc\ra\locMeSp^\op$ assign to each $\boole\in\Ob(\mcBc)$ its Loomis--Sikorski representation, and to each $f\in\Mor(\mcBc)$ the corresponding complete morphism in $\locMeSp^\op$. Then $\mathrm{W}$ and $\mathrm{LS}$ form an equivalence of categories,
\begin{equation}
        \mathrm{W}\circ\mathrm{LS}\iso\id_{\mcBc},\;\;\mathrm{LS}\circ\mathrm{W}\iso\id_{\locMeSp^\op}.
\label{mcBc.locMesp.duality}
\end{equation}
It seems that this equivalence is a special case of a more general equivalence between arbitrary measurable spaces and ccb-algebras with arbitrary boolean homomorphisms. However, it is unclear for us what are the morphisms between measurable spaces in this general case.

Every $L_p(\X,\mho(\X),\tmu)$ space for $p\in[1,\infty[$ is a Banach lattice that is isometrically Riesz isomorphic to $L_p(\boole_\tmu,\mu)$ with $(\boole_\tmu,\mu)$ determined by \eqref{wecken.eins} and \eqref{wecken.zwei}. On the other hand, an abstract $L_p$ space $X$ determines uniquely the mcb-algebra $\boole$ (as an algebra of its projection bands), but it determines neither the particular measurable space $(\X,\mho(\X),\mho^0(\X))$ nor the measure space $(\X,\mho(\X),\tmu)$. These structures can vary arbitrarily, as long as $\mho(\X)/\mho^0(\X)$, or, respectively, $\mho(\X)/\mho^\tmu(\X)$, are boolean isomorphic to  $\boole$. All mutually isometrically Riesz isomorphic $L_p(\X,\mho(\X),\tmu)$ spaces constructed over various measure spaces $(\X,\mho(\X),\tmu)$ can be identified with a single $L_p(\boole,\mu)$ space, with $\boole\iso\mho(\X)/\mho^\tmu(\X)$ and $\mu([\cdot]_\boole)=\tmu$. Finally, for any $(\X,\mho(\X),\tmu)$ one has an isometric Riesz isomorphism $L_\infty(\X,\mho(\X),\tmu)\iso L_\infty(\mho(\X)/\mho^\tmu(\X))\iso L_\infty(\boole)$. The restriction of validity of isometric isomorphism $L_1(\boole,\mu)^\banach\iso L_\infty(\boole)$ to localisable measure algebras $(\boole,\mu)$ is equivalent with restriction of validity of the Steinhaus--Nikod\'{y}m theorem \cite{Steinhaus:1919,Nikodym:1931} $L_1(\X,\mho(\X),\tmu)\iso L_\infty(\X,\mho(\X),\tmu)$ to localisable measure spaces, which was established by Segal \cite{Segal:1951}. Example showing that this property does not hold for arbitrary measure space was given in \cite{Saks:1933}.

For further purposes, we will recall few notions from topology and topological branch of Fr\'{e}chet's approach to measure theory. If $\X$ is a topological space, then $\Y\subseteq\X$ is called: \df{clopen} if{}f it is closed and open; \df{rare} (or \df{nowhere dense}) if{}f $\INT(\bar{\Y})=\varnothing$; \df{meager} if{}f it is a union of a sequence of rare subsets of $\X$; \df{cozero} if{}f $\X\setminus\Y$ is of the form $f^{-1}(\{0\})$ for some continuous function $f:\X\ra\RR$. A topological space $\X$ is called: \df{Hausdorff} \cite{Hausdorff:1914} if{}f for any $\xx_1,\xx_2\in\X$ such that $\xx_1\neq\xx_2$ there exist open sets $\X_1,\X_2\subseteq\X$ such that $\xx_1\in\X_1$, $\xx_2\in\X_2$, and $\X_1\cap\X_2=\varnothing$; \df{compact} if{}f every open cover of $\X$ has a finite subcover; \df{locally compact} if{}f every $\xx\in\X$ has a topological neighbourhood which is compact; \df{extremally disconnected} if{}f the closure of every open subset of $\X$ is open (and hence clopen); \df{stonean} if{}f $\X$ is extremally disconnected compact Hausdorff; \df{totally disconnected} (or \df{zero-dimensional}) if{}f every open subset of $\X$ is a union of its clopen subsets; \df{Stone} if{}f $\X$ is totally disconnected compact Hausdorff; \df{basically disconnected} if{}f the closure of every cozero subset of $\X$ is open; \df{Rickart} (or \df{quasi-stonean}) if{}f $\X$ is basically disconnected compact Hausdorff. Let $\mho_{\mathrm{Borel}}(\X)$ denote the smallest countably additive algebra on $\X$ containing all open sets of $\X$. If $\X$ is a Hausdorff topological space, a measure $\tmu$ on $\rpktarget{BOREL}\mho_{\mathrm{Borel}}(\X)$ is called: \df{compactly inner regular} if{}f
\begin{equation}
        \tmu(\Y)=\sup\{\tmu(\Z)\mid\Z\mbox{ is compact},\;\Z\subseteq\Y\}\;\;\forall\Y\in\mho(\X);
\end{equation}
\df{locally finite} if{}f $\forall\xx\in\X$ of $\xx$ there exists a neighbourhood $\Y\subseteq\X$ such that $\tmu(\Y)<\infty$; \df{Radon} if{}f it is compactly inner regular and locally finite; \df{normal} if{}f $\int\tmu\sup_\iota\{f_\iota\}=\sup_\iota\{\int\tmu f_\iota\}$ for each bounded monotone increasing net $\{f_\iota\}\subseteq\mathrm{C}(\X;\RR)$, or, equivalently, if{}f $\tmu(\Y)=0$ for every meager $\Y\subseteq\X$. Let $\rpktarget{NORMAL.RADON}\Meas^+_\star(\X,\mho_{\mathrm{Borel}}(\X))$ denote the set of all normal Radon measures on $\mho_{\mathrm{Borel}}(\X)$. A stonean topological space $\X$ is called \df{hyperstonean} if{}f $\bigcup\{\supp(\tmu)\mid\tmu\in\Meas^+_\star(\X,\mho_{\mathrm{Borel}}(\X))\}$ is a dense subset of $\X$, or, equivalently, if{}f for each nonempty open $\Y\subseteq\X$ there exists a Radon measure $\tmu$ on $\X$ such that $\tmu(\Z)=0$ for every rare $\Z\subseteq\X$ and $\tmu(\Y)>0$. Hence, $\nul(\X,\mho_{\mathrm{Borel}}(\X),\tmu)=\rare(\X)$, where $\rare(\X)$ is a set of all rare subsets of $\X$, so $\mho^{\tmu}(\X)=\mho_{\mathrm{Borel}}(\X)\cap\rare(\X)$. A continuous function $f:\X_1\ra\X_2$ between two compact Hausdorff topological spaces $\X_1$ and $\X_2$ is called: \df{quasi-open} if{}f $f^{-1}(\Y)$ is rare in $\X_1$ for all rare $\Y\subseteq\X_2$, or, equivalently, if{}f 
\begin{equation}
        \X\neq\varnothing\;\limp\;\INT(f(\X))\neq\varnothing\;\;\forall\X\subseteq\X_1;
\end{equation}
\df{open} if{}f $\overline{f^{-1}(\Y)}=f^{-1}(\bar{\Y})$ $\forall\Y\subseteq\X_2$, or, equivalently, if{}f
\begin{equation}
        \INT(f^{-1}(\Y))=f^{-1}(\INT(\Y))\;\;\forall\Y\subseteq\X_2.
\end{equation}
A continuous function between stonean spaces is open if{}f it is quasi-open. Every homeomorphism of topological spaces is open, and a bijective continuous function between two topological spaces is homeomorphism if{}f it is open.

According to Riesz representation theorem \cite{Riesz:1909:b}, if $\X$ is a locally compact Hausdorff topological space, then every $\RR$-valued positive linear functional $\phi$ on the set $\mathrm{C}_\mathrm{c}(\X;\RR)\rpktarget{CC.RIESZ.THM}$ of continuous functions on $\X$ with compact support determines a unique Radon measure $\tmu_\phi$ on $\mho_{\mathrm{Borel}}(\X)$ such that
\begin{equation}
        \phi(f)=\int_\X\tmu_\phi(\xx)f(\xx)\;\;\forall f\in\mathrm{C}_\mathrm{c}(\X;\RR).
\end{equation}
This implies that there exists an order preserving isometric isomorphism between the set $\rpktarget{RADON}\Rad(\X)^+$ of Radon measures on $\X$ and a subset of a Banach dual $\mathrm{C}(\X;\RR)^\banach$ of a space $\mathrm{C}(\X;\RR)\rpktarget{CONTI.RIESZ.THM}$ of all $\RR$-valued continuous functions on $\X$, which is determined by the map
\begin{equation}
        \Rad(\X)^+\ni\tmu\mapsto\int_\X\tmu(\xx)(\cdot)\in\mathrm{C}(\X;\RR)^\banach.
\end{equation}

Finally, we turn briefly to Daniell's approach. A \df{Stone lattice} \cite{Stone:1948,Stone:1949} is defined as a Riesz lattice $X$ with a (multiplicative) unit element $\rpktarget{DANIELUNIT}\II$ such that $f\in X^+$ $\limp$ $\inf\{f,\II\}\in X$. A \df{Daniell--Stone integral} \cite{Daniell:1918,Daniell:1920,Stone:1948,Stone:1949} is defined as a function $\omega:X\ra\RR^+$ that is linear and monotonically sequentially continuous (that is, $\omega(\inf\{f_i\})=\inf\{\omega(f_i)\}$ for every sequence $\{f_i\}\subseteq X$ such that $f_1\geq f_2\geq\ldots$ and $\inf\{f_i\}\in X$, or, equivalently, $f_1\leq f_2\leq\ldots$ and $\sup\{f_i\}\in X$). If $\X$ is a set, then a \df{Daniell lattice} over $\X$ is defined as such Stone lattice $X$ that is a subset of a set of functions $f:\X\ra\RR\cup\{+\infty\}$. Every Daniell--Stone integral $\omega$ on a Daniell lattice $X$ can be uniquely extended to a function $\hat{\omega}:\hat{X}\ra[0,\infty]$ where
\begin{align}
        \hat{X}&:=\{f:\X\ra\RR\cup\{+\infty\}\mid\exists\{f_n\}\in X\;\;\inf\{f_n\}=f\},\\
        \hat{\omega}(f)&:=\left\{
        \begin{array}{ll}
                \inf\{\lim\omega(h_i)\mid
                \{h_i\}\subseteq X,\;
                \inf\{h_i\}=h\geq f
                \}&:\,h\in X\\
                +\infty&:\,\mbox{otherwise}.
        \end{array}
        \right.
\end{align}
Such $\hat{\omega}$ is called a \df{Daniell--Stone extension} of $\omega$, and is an analogue of a semi-finite countably additive measure. (In particular, the Lebesgue integral can be constructed as a Daniell--Stone extension of the Riemann integral, while the Radon integral can be constructed as a Daniell--Stone extension of the Stieltjes integral.) It is convex, affine ($\lambda>0\limp\hat{\omega}(\lambda f)=\lambda\hat{\omega}(f)$), and subadditive ($\hat{\omega}(f_1+f_2)\leq\hat{\omega}(f_1)+\hat{\omega}(f_2)$). We will call the pair $(\hat{X},\hat{\omega})$ a \df{Daniell system}. For $p\in[1,\infty[$ and a Daniell system $(\hat{X},\hat{\omega})$, the function 
\begin{equation}
        \n{\cdot}_p:\hat{X}\ni x\mapsto\n{x}_p:=(\hat{\omega}(\ab{x})^p)^{1/p}\in[0,\infty]
\end{equation}
is a semi-norm on $X$. The space $\Lcal_p(X,\omega)$ is defined as a completion of the set $\{x\in X\mid\n{x}_p<\infty\}$ in the topology generated by $\n{\cdot}_p$. The function $\n{\cdot}_p$ becomes a norm on $L_p(X,\omega):=\Lcal_p(X,\omega)/\ker(\n{\cdot}_p)$, and $L_p(X,\omega)$ are Banach spaces with respect to $\n{\cdot}_p$.

The main structure that allows the transition between integral and measure are the characteristic functions of sets, which can be used to reconstruct either the measure on the underlying set or the integral on the space of functions. If $X$ is a Daniell lattice over $\X$, then a subset $\Y\subseteq\X$ is called \df{$X$-open} if{}f there exists a sequence of nonnegative functions $\{f_i\}\subseteq X$ that is monotone increasing $(f_1\leq f_2\leq\ldots)$ and $\II|_\Y=\sup\{f_i\}$. The family of $X$-open subsets of $\X$ is closed with respect to countable unions. Let $\mho^X(\X)$ denote the smallest countably additive algebra containing all $X$-open subsets of $\X$. The Daniell--Stone theorem states that for a given Daniell system $(\hat{X},\hat{\omega})$ there exists a countably additive measure $\tmu$ on $\mho^X(\X)$ that is uniquely determined by the conditions:
\begin{enumerate}
\item[1)] $X\subseteq L_1(\X,\mho^X(\X),\tmu)$,
\item[2)] $\omega(f)=\int\tmu f$ $\forall f\in X$,
\item[3)] $\tmu(\Z)=\inf\{\tmu(\Y)\mid\Z\subseteq\Y,\;\Y$\mbox{ is }$X$\mbox{-open in }$\X\}$ $\forall\Z\in\mho^X(\X)$.
\end{enumerate}
\subsection{Canonical commutative integration\label{canonical.comm.int.section}}
The role played by mcb-algebras in diagram \eqref{commutative.integration.diagram} suggests that it might be possible to deal with Banach lattice isomorphic $L_p(\boole,\mu)$ spaces for $p\in[1,\infty[$ without specifying any particular measure $\mu\in\W_0(\boole)$ associated with a given mcb-algebra $\boole$. In what follows, we will construct a family of canonical $L_p(\boole)$ spaces that are associated functorially to mcb-algebras $\boole$. This construction is new for $p\in[1,\infty[$, and it is aimed to provide a commutative counterpart to the Falcone--Takesaki construction of canonical noncommutative $L_p(\N)$ spaces. In principle, one could try to define the space $L_p(\boole)$ as an equivalence class of $L_p(\boole,\mu)$ spaces divided by the isometric Riesz isomorphisms generated by varying $\mu$ within $\W_0(\boole)$. However, this would remove too much structure, making category theoretic description inapplicable (or less applicable). Hence, instead of `isomorphism invariant' definition, we will provide `isomorphism covariant' construction, which follows the ideas of Neveu \cite{Neveu:1964} and Zhu \cite{Zhu:1998:lebesgue}. This will enable us to provide an explicit description of the relationship between the canonical integration theory and the commutative case of canonical noncommutative integration theory, \textit{without passing to representations} in terms of measure algebras, measure spaces, or Daniell systems. In order to keep the algebraic representation-independent formulation, we prove that every proper abstract $L_\infty$ space is a commutative $W^*$-algebra. Because converse is also true, this defines an equivalence between the category of mcb-algebras (with boolean isomorphisms), commutative $W^*$-algebras (with normal unital $*$-isomorphisms), and proper abstract $L_\infty$ spaces (with unit preserving isometric Riesz isomorphisms). Together with the full and faithful functor from the category of $L_p(\boole)$ spaces with isometric Riesz isomorphisms to the category of $L_p(\N)$ spaces with isometric isomorphisms, this shows that canonical commutative integration theory is precisely a commutative sector of canonical noncommutative integration theory. The relationships between various categories of commutative and noncommutative integration are summarised in the commutative diagram \eqref{canonical.integration.theory.big.stuff}

For any countably additive measures $\mu_1$ and $\mu_2$ on a ccb-algebra $\mu_1$ is called: \df{absolutely continuous} with respect to $\mu_2$, and denoted by $\rpktarget{MEASURE.LL}\mu_1\ll\mu_2$, if{}f
\begin{equation}
        \forall\epsilon_1>0\;\exists\epsilon_2>0\;\forall x\in\boole\;\;\;
        \mu_2(x)\leq\epsilon_2\;\limp\;\mu_1(x)\leq\epsilon_1,
\end{equation}
or, equivalently,
\begin{equation}
	\mu_2(x)=0\;\;\limp\;\;\mu_1(x)=0\;\;\forall x\in\boole;
\end{equation}
\df{compatible} \cite{Lewin:Lewin:1975} with respect to $\mu_2$ if{}f
\begin{equation}
	\forall x\in\boole\;\;0<\mu_1(x)<\infty\;\;\limp\;\;(\exists y\in\boole\;\;\mu_1(y)>0,\;\mu_2(y)<\infty).
\end{equation}
Note that $\mu_1,\mu_2\in\W_0(\boole)$ $\limp$ $\mu_1\ll\mu_2\ll\mu_1$. If $f\in L_1(\boole,\mu)$ and $x\in\boole$, then
\begin{equation}
        \int_x\mu f:=\int_0^\infty\dd\lambda\,\mu(x\land f(\lambda))=\int\mu f\chr(x).
\end{equation}
If $\boole$ is an mcb-algebra and $\mu_1,\mu_2\in\W(\boole)$, then the Segal--Lewin--Lewin theorem \cite{Segal:1951,Lewin:Lewin:1975} (see also \cite{Zaanen:1961,McShane:1962,Kelley:1966,Volcic:1970,Canderolo:Volcic:2002}) states that for $\mu_2\ll\mu_1$ and $\mu_2$ compatible with $\mu_1$
\begin{equation}
        \exists! f\in L_1(\boole,\mu_1)\;\forall x\in\boole\;\;\mu_2(x)=\int_x\mu_1 f.
\end{equation}
Such $f$ will be called a \df{Radon--Nikod\'{y}m quotient} and denoted by $\frac{\mu_2}{\mu_1}$. This theorem is a generalisation of the Lebesgue--Radon--Daniell--Nikod\'{y}m theorem \cite{Lebesgue:1910,Radon:1913,Daniell:1920:Stieltjes,Nikodym:1930} (which holds for countably additive finite measures). For $\mu_1,\mu_2\in\W_0(\boole)$ the compatibility condition of $\mu_1$ with respect to $\mu_2$ reduces to
\begin{equation}
	\forall x\in\boole^{\mu_1}\;\;\exists y\in\boole\;\;y\leq x,\;y\in\boole^{\mu_2}.
\end{equation}
If $\mu_1,\mu_2,\mu_3\in\W_0(\boole)$ are mutually compatible, then their Radon--Nikod\'{y}m quotients satisfy
\begin{align}
\left(\frac{\mu_i}{\mu_j}\right)^{-1}&=\frac{\mu_j}{\mu_i},\\
\frac{\mu_i}{\mu_j}\frac{\mu_j}{\mu_k}&=\frac{\mu_i}{\mu_k},
\end{align}
for all $i,j\in\{1,2,3\}$.

As a consequence, for $\gamma\in\,]0,1]$, $\mu_1,\mu_2\in\W_0(\boole)$ such that $\mu_1$ and $\mu_2$ are mutually compatible, $f_1\in L_{1/\gamma}(\boole,\mu_1)$, $f_2\in L_{1/\gamma}(\boole,\mu_2)$, the formula 
\begin{equation}
        (f_1,\mu_1)\sim_{1/\gamma}(f_2,\mu_2)\;\;:\iff\;\; f_1=f_2\left(\frac{\mu_2}{\mu_1}\right)^{1/\gamma}
\end{equation}
determines an equivalence relation on $L_0(\boole)\times\W_0(\boole)$, which defines the family of equivalence classes
\begin{equation}
        \{f\mu^\gamma:=(f,\mu)/\sim_{1/\gamma}\mid\mu\in\W_0(\boole),\;f\in L_{1/\gamma}(\boole,\mu)\}.
        \label{canonical.Lp}
\end{equation}
Let $\rpktarget{LPBN}L_{1/\gamma}(\boole)$ denote the set $\{f\mu^\gamma\mid\mu\in\W_0(\boole),\;f\in L_{1/\gamma}(\boole,\mu)\}$ equipped with the operations
\begin{align}
f_1\mu_1^\gamma+f_2\mu_2^\gamma
&:=\left(f_1\left(\frac{\mu_1}{\mu_4}\right)^\gamma+f_2\left(\frac{\mu_2}{\mu_4}\right)^\gamma\right)\mu_4^\gamma,\label{L.gamma.sum}\\
\lambda(f\mu^\gamma)
&:=(\lambda f)\mu^\gamma,\label{L.gamma.multiplication}\\
\n{f\mu^\gamma}_{1/\gamma}
&:=\left(\int\mu\ab{f}^{1/\gamma}\right)^\gamma,\label{L.gamma.norm}\\
f_1\mu_1^\gamma\land f_2\mu_2^\gamma
&:=\left(f_1\left(\frac{\mu_1}{\mu_4}\right)^\gamma\land f_2\left(\frac{\mu_2}{\mu_4}\right)^\gamma\right)\mu_4^\gamma,\label{L.gamma.land}\\
f_1\mu_1^\gamma\lor f_2\mu_2^\gamma
&:=\left(f_1\left(\frac{\mu_1}{\mu_4}\right)^\gamma\lor f_2\left(\frac{\mu_2}{\mu_4}\right)^\gamma\right)\mu_4^\gamma,\label{L.gamma.lor}
\end{align}
where $\mu_4$ is an arbitrary element of $\W_0(\boole)$ providing a representation of an equivalence class $f\mu^\gamma$ (hence, it is compatible with $\mu_1,\mu_2\in\W_0(\boole)$. 
\begin{proposition}\label{proposition.L.gamma.is.abstract}
$L_{1/\gamma}(\boole)$ is an abstract $L_{1/\gamma}$ space for $\gamma\in\,]0,1]$.
\end{proposition}
\begin{proof}
We need to check that $L_{1/\gamma}(\boole)$ satisfies the following properties: 1) it is a lattice; 2) it is a vector space over $\RR$; 3) $x\leq y$ $\limp$ $x+z\leq y+z$; 4) $x\geq0$ $\limp$ $\lambda x\geq0$ $\forall\lambda\geq0$; 5) $\ab{x}\leq\ab{y}$ $\limp$ $\n{x}\leq\n{y}$; 6) $\n{\cdot}_{1/\gamma}$ is a norm; 7) it is Cauchy complete in $\n{\cdot}_{1/\gamma}$; 8)  $\ab{x}\land\ab{y}=0$ $\limp$ $\n{x+y}^{1/\gamma}_{1/\gamma}=\n{x}^{1/\gamma}_{1/\gamma}+\n{y}^{1/\gamma}_{1/\gamma}$. We begin by noting that 2) follows directly from \eqref{L.gamma.sum}, \eqref{L.gamma.multiplication} and the vector space structure of $L_{1/\gamma}(\boole,\mu)$, 6) and 7) follow directly from 2), \eqref{L.gamma.norm} and the Banach space structure of $L_{1/\gamma}(\boole,\mu)$, while 1) follows directly from \eqref{L.gamma.land}, \eqref{L.gamma.lor} and the lattice structure of $L_{1/\gamma}(\boole,\mu)$. Hence, it remains to prove 3), 4), 5), and 8).
\begin{enumerate}
\item[3)] $f_1\mu_1^\gamma\leq f_2\mu_2^\gamma$ $\iff$ $f_1\left(\frac{\mu_1}{\mu_4}\right)^\gamma\mu_4^\gamma\leq f_2\left(\frac{\mu_2}{\mu_4}\right)^\gamma\mu_4^\gamma$,\\so $f_1\mu_1^\gamma+f_3\mu_3^\gamma=f_1\left(\frac{\mu_1}{\mu_4}\right)^\gamma\mu_4^\gamma+f_3\left(\frac{\mu_3}{\mu_4}\right)^\gamma\mu_4^\gamma\leq f_2\left(\frac{\mu_2}{\mu_4}\right)^\gamma\mu_4^\gamma+f_3\left(\frac{\mu_3}{\mu_4}\right)^\gamma\mu_4^\gamma=f_2\mu_2^\gamma+f_2\mu_3^\gamma$.
\item[4)] $f\mu^\gamma\geq0$ $\iff$ $f\geq0$ $\limp$ $\lambda f\geq0$ $\iff$ $(\lambda f)\mu^\gamma\geq0$ $\iff$ $\lambda(f\mu^\gamma)\geq0$.
\item[5)] Using the f-algebra structure of $L_{1/\gamma}(\boole,\mu_4)$, we obtain
\begin{equation}
        \ab{f\mu^\gamma}=(f\mu^\gamma)\lor(-f\mu^\gamma)=\left(f\left(\frac{\mu}{\mu_4}\right)^\gamma\lor-f\left(\frac{\mu}{\mu_4}\right)^\gamma\right)\mu_4^\gamma=\ab{f\left(\frac{\mu}{\mu_4}\right)^\gamma}\mu_4^\gamma.
\label{canonical.Lp.absolute}
\end{equation}
This allows us to write
\begin{align}
\ab{f_1\mu_1^\gamma}
&\leq\ab{f_2\mu_2^\gamma},\\
\ab{f_1\left(\frac{\mu_1}{\mu_4}\right)^\gamma}\mu_4^\gamma
&\leq\ab{f_2\left(\frac{\mu_2}{\mu_4}\right)^\gamma}\mu_4^\gamma,\\
\left(\ab{f_2\left(\frac{\mu_2}{\mu_4}\right)^\gamma}-\ab{f_1\left(\frac{\mu_1}{\mu_4}\right)^\gamma}\right)\mu_4^\gamma
&\geq0,\\
\ab{f_2\left(\frac{\mu_2}{\mu_4}\right)^\gamma}
&\geq\ab{f_1\left(\frac{\mu_1}{\mu_4}\right)^\gamma},\\
\n{f_2\left(\frac{\mu_2}{\mu_4}\right)^\gamma}
&\geq\n{f_1\left(\frac{\mu_1}{\mu_4}\right)^\gamma},\\
\left(\int\mu_4\ab{f_2\left(\frac{\mu_2}{\mu_4}\right)^\gamma}^{1/\gamma}\right)^\gamma
&\geq\left(\int\mu_4\ab{f_1\left(\frac{\mu_1}{\mu_4}\right)^\gamma}^{1/\gamma}\right)^\gamma,\\
\left(\int\mu_2\ab{f_2}^{1/\gamma}\right)^\gamma
&\geq\left(\int\mu_1\ab{f_1}^{1/\gamma}\right)^\gamma,\\
\n{f_2\mu_2^\gamma}_{1/\gamma}
&\geq\n{f_1\mu_1^\gamma}_{1/\gamma}.
\end{align} 
\item[8)] We have
\begin{align}
\n{f_1\mu_1^\gamma+f_2\mu_2^\gamma}^{1/\gamma}_{1/\gamma}&=\n{f_1+f_2}^{1/\gamma}_{1/\gamma},\\
\n{f_1\mu_1^\gamma}^{1/\gamma}_{1/\gamma}+\n{f_2\mu_2^\gamma}^{1/\gamma}_{1/\gamma}&=\n{f_1}^{1/\gamma}_{1/\gamma}+\n{f_2}^{1/\gamma}_{1/\gamma}.
\end{align}
In order to prove $\ab{f_1\mu_1^\gamma}\land\ab{f_2\mu_2^\gamma}=0$ $\iff$ $\ab{f_1}\land\ab{f_2}=0$, we need to use $(x\land y=0,\;z\geq0)$ $\limp$ $(x\cdot z)\land y=0$ in $L_{1/\gamma}(\boole,\mu_4)$, and the positivity of Radon--Nikod\'{y}m quotient, which gives us
\begin{align}
0&=
\ab{f_1\mu_1^\gamma}
        \land
        \ab{f_2\mu_2^\gamma}
        =
        \ab{f_1\left(
                \frac{\mu_1}{\mu_4}
        \right)^\gamma}
        \mu_4^\gamma
        \land
        \ab{f_2\left(
                \frac{\mu_2}{\mu_4}
        \right)^\gamma}
        \mu_4^\gamma\;\;\iff
        \\0&=
        \ab{f_1
                \left(
                        \frac{\mu_1}{\mu_4}
                \right)^\gamma
                \left(\left(
                        \frac{\mu_1}{\mu_4}
                \right)^{-1}\right)^\gamma}
        \mu_4^\gamma\land
        \ab{f_2
                \left(
                        \frac{\mu_2}{\mu_4}
                \right)^\gamma
                \left(\left(
                        \frac{\mu_2}{\mu_4}
                \right)^{-1}\right)^\gamma}
        \mu_4^\gamma=
        \left(
                \ab{f_1}\land\ab{f_2}
        \right)
        \mu_4^\gamma\;\;\iff
        \\0&=
                \ab{f_1}\land\ab{f_2}.
\end{align}
\end{enumerate}
Thus, an abstract $L_{1/\gamma}$ space structure of $L_{1/\gamma}(\boole)$ follows from an abstract $L_{1/\gamma}$ space structure of $L_{1/\gamma}(\boole,\mu_4)$ for $\mu_4\in\W_0(\boole)$.
\end{proof}
Hence, every mcb-algebra $\boole$ allows to construct a family of canonical commutative $L_p(\boole)$ spaces over $\boole$, with $p\in[1,\infty]$, which are abstract $L_p$ spaces and do not depend on the choice of measure on $\boole$. This assignment is functorial, with boolean homomorphisms $f:\boole_1\ra\boole_2$ mapped to the unit preserving Riesz homomorphisms $\widetilde{f}:L_p(\boole_1)\ra L_p(\boole_2)$, and with boolean isomorphisms mapped to the unit preserving Riesz isomorphisms, by means of \eqref{Lp.homo} with arbitrary choice of $\mu_1\in\W_0(\boole_1)$. This defines the family of functors $\mathrm{L}_{1/\gamma}:\mcBIso\ra\catname{L}_{1/\gamma}\catname{Iso}$, which provides a missing arrow in the diagram \eqref{commutative.integration.diagram} for $\gamma\in\,]0,1]$, and removes the need for use of categories of measure algebras in the foundations of the theory. Together with the functor $\PrjB$, the functor $\mathrm{L}_{1/\gamma}$ establishes an equivalence of the categories $\mcBIso$ and $\catname{L}_{1/\gamma}\catname{Iso}$:
\begin{equation}
\mathrm{L}_{1/\gamma}\circ\PrjB\iso\id_{\catname{L}_{1/\gamma}\catname{Iso}},\;\;\PrjB\circ\mathrm{L}_{1/\gamma}\iso\id_{\mcBIso}.
\label{L.gamma.mcBIso.duality}
\end{equation}
\begin{proposition}
The map $[\cdot]_\mu:L_{1/\gamma}(\boole)\ni x\mu^\gamma\mapsto x\in L_{1/\gamma}(\boole,\mu)$ is an isometric Riesz isomorphism.
\end{proposition}
\begin{proof}
Linearity follows from \eqref{L.gamma.sum} and \eqref{L.gamma.multiplication}, isometry follows from \eqref{L.gamma.norm}, while the property $[\ab{x}]_\mu=\ab{[x]_\mu}$ follows from \eqref{canonical.Lp.absolute}.
\end{proof}
Hence, for $\mu\in\W_0(\boole)$ the function $[\cdot]_\mu$ provides an isometrically Riesz isomorphic representation of $L_{1/\gamma}(\boole)$ space in terms of the $L_{1/\gamma}(\boole,\mu)$ space.

\begin{corollary}
For any mcb-algebra $\boole$ there exists a bijective Riesz homomorphism $L_1(\boole)\iso\eval(\boole)$, and the diagram
\begin{equation}
\xymatrix{
        L_1(\boole)^+_0
        \ar@{^{(}->}[r]
        \ar@{^{(}->}[d]
        &
        \W_0(\boole)
        \ar@{^{(}->}[d]
        \\
        L_1(\boole)^+
        \ar@{^{(}->}[r]
        &
        \W(\boole)
}
\label{commutative.state.weight.diag}
\end{equation}
commutes.
\end{corollary}
This is a strict analogue of \eqref{Wstar.states.weights.comm} for mcb-algebras. Moreover, if $\mu_1,\mu_2\in\W_0(\boole)$ and $\mu_1$ is compatible with respect to $\mu_2$, then $\int\mu_1 f=\int\mu_2\frac{\mu_1}{\mu_2}f$ $\forall f\in L_1(\boole,\mu_1)$. This allows us to define a \df{canonical integral},
\begin{equation}
        \int: L_1(\boole)\ni x\mapsto\int x:=\int\mu f\in\RR,
\end{equation}
where $[x]_\mu=f\in L_1(\boole,\mu)$, which is independent of the choice of an arbitrary $\mu\in\W_0(\boole)$. As a result, we obtain a bilinear functional
\begin{equation}
        L_{1/\gamma}(\boole)\times L_{1/(1-\gamma)}(\boole)\ni(x,y)\mapsto\int xy=\int yx\in\RR,
\end{equation}
which sets up a canonical duality between $L_{1/\gamma}(\boole)$ and $L_{1/(1-\gamma)}(\boole)$ spaces for $\gamma\in\,]0,1]$. This establishes direct analogy between the properties of the family $L_{1/\gamma}(\boole)$ spaces over mcb-algebras $\boole$ and the properties of the family of $L_{1/\gamma}(\N)$ spaces over $W^*$-algebras $\N$. In what follows, we will see that those two settings coincide in the case when $W^*$-algebra is commutative.
\subsection{Categories of integration theory\label{integr.compar.section}}
Consider the categories: $\cpH$ of \underline{c}om\underline{p}act \underline{H}ausdorff topological spaces and continuous functions; $\tdcpH$ of \underline{t}otally \underline{d}isconnected \underline{c}ompact \underline{H}ausdorff topological spaces and continuous functions; $\bdcpH$ of \underline{b}asically \underline{d}isconnected \underline{c}ompact \underline{H}ausdorff topological spaces and continuous functions; $\edcpHo$ of \underline{e}xtremally \underline{d}is\-con\-ne\-cted \underline{c}ompact \underline{H}ausdorff topological spaces and \underline{o}pen continuous functions; $\hypso$ of \underline{hyp}er\underline{s}tonean spaces and \underline{o}pen continuous functions; $\hypsh$ of \underline{hyp}er\underline{s}tonean spaces and \underline{h}omeomorphisms. Every Stone spectrum is a totally disconnected compact Hausdorff topological space. The functor $\sp_\mathrm{S}:\catname{B}\ra\tdcpH^\op$ assigns to each $\boole\in\Ob(\catname{B})$ its Stone spectrum $\sp_{\mathrm{S}}(\boole)$, and to each boolean homomorphism $f:\boole_1\ra\boole_2$ a continuous function $\sp_\mathrm{S}(f):\sp_\mathrm{S}(\boole_2)\ra\sp_\mathrm{S}(\boole_1)$. The functor $\mathrm{B}:\tdcpH^\op\ra\catname{B}$ assigns to each $\X\in\Ob(\tdcpH^\op)$ a boolean algebra of all clopen subsets of $\X$, and to each $f:\X_1\ra\X_2$ in $\tdcpH$ a boolean homomorphism $\mathrm{B}(f):\mathrm{B}(\X_2)\ra\mathrm{B}(\X_1)$. According to categorified version of Stone's theorem \cite{Stone:1936,Stone:1937}, these two functors define an equivalence of categories:
\begin{equation}
        \mathrm{B}\circ\sp_\mathrm{S}\iso\id_{\catname{B}},\;\;\;\sp_\mathrm{S}\circ\mathrm{B}\iso\id_{\tdcpH^{op}}.
\label{Stone.tdcpH.B.duality}
\end{equation}
Stone \cite{Stone:1937,Stone:1937:algebraic} (and Nakano \cite{Nakano:1941:System}, see also \cite{Gleason:1958}) showed that under restriction to Dcb-algebras and ccb-algebras this restricts, respectively, to equivalences:
\begin{align}
        \mathrm{B}\circ\sp_\mathrm{S}\iso\id_{\ccB},\;&\;\;\sp_\mathrm{S}\circ\mathrm{B}\iso\id_{\bdcpH^\op};\\
        \mathrm{B}\circ\sp_\mathrm{S}\iso\id_{\DcBc},\;&\;\;\sp_\mathrm{S}\circ\mathrm{B}\iso\id_{\edcpHo^\op}.
\end{align}
The original statements of these theorems were set theoretic (without discussing the morphisms and functors). For a detailed discussion of the functorial extension specified above, see \cite{Linton:1970,Johnstone:1982,Bezhanishvili:2010} and further references therein. Finally, for $\X\in\Ob(\hypso^\op)$, $\mathrm{B}(\X)$ is boolean isomorphic to its Loomis--Sikorski representation $\mho_{\mathrm{Borel}}(\X)/\mho_{\mathrm{Borel}}(\X)\cap\rare(\X)$, and from the fact that $\bigcup\{\supp(\tmu)\mid\tmu\in\Meas^+_\star(\X,\mho_{\mathrm{Borel}}(\X))\}$ is dense in $\X$ it follows that there exists $\mu\in\W_0(\mho_{\mathrm{Borel}}(\X)/\mho_{\mathrm{Borel}}(\X)\cap\rare(\X))$, hence $\mathrm{B}(\X)$ is an mcb-algebra (see e.g. \cite{Fremlin:2000} for details). This gives rise to equivalences:
\begin{align}
        \mathrm{B}\circ\sp_\mathrm{S}\iso\id_{\mcBc},\;&\;\;\sp_\mathrm{S}\circ\mathrm{B}\iso\id_{\hypso^\op};\label{mcBc.hypso.duality}\\
        \mathrm{B}\circ\sp_\mathrm{S}\iso\id_{\mcBIso},\;&\;\;\sp_\mathrm{S}\circ\mathrm{B}\iso\id_{\hypsh^\op}.\label{mcBIso.hypsh.duality}
\end{align}

Let $\MI$ denote a category of MI-spaces and order unit preserving Riesz homomorphisms. The functor $\mathrm{C}:\cpH^\op\ra\MI$ assigns to each $\X\in\Ob(\cpH^\op)$ a set $\mathrm{C}(\X;\RR)$ (or $\mathrm{C}(\X;\CC)$) of real (or complex) valued continuous functions on $\X$, equipped with a Banach lattice structure, with ordering defined by $f_1\leq f_2$ $\iff$ $f_1(\xx)\leq f_2(\xx)$ $\forall\xx\in\X$, with norm $\n{f}:=\sup_{\xx\in\X}\{\ab{f(\xx)}\}$, and order unit given by a function constantly equal to $1$. According to the Kre\u{\i}n--Kre\u{\i}n--Kakutani theorem \cite{Krein:Krein:1940,Krein:Krein:1943,Kakutani:1940,Kakutani:1941:M}, every real (resp., complex) MI-space is isometrically Riesz isomorphic to some $\mathrm{C}(\X;\RR)$ (resp., $\mathrm{C}(\X;\CC)$) space, with $\X$ unique up to a homeomorphism. This defines a functor $\KKK:\MI\ra\cpH^\op$, and an extension of this theorem to an equivalence of categories reads \cite{Banaschewski:1976}
\begin{equation}
        \mathrm{C}\circ\KKK\iso\id_{\MI},\;\;\;\KKK\circ\mathrm{C}\iso\id_{\cpH^\op}.
\end{equation}
Under restriction to the subcategories of $\MI$ given by $\bLinf$, $\aLinf$, $\DcaLinfc$, $\Linfc$, and $\LinfIso$ the equivalent subcategories of $\cpH^\op$ are, respectively, $\tdcpH^\op$, $\bdcpH^\op$, $\edcpHo^\op$, $\hypso^\op$, and $\hypsh^\op$. These results can be deduced directly from Stone duality type equivalences above and equivalences in \eqref{commutative.integration.diagram}, by means of composition of functors:
\begin{equation}
        \mathrm{L}_\infty\circ\mathrm{B}=\mathrm{C},\;\;\;\sp_{\mathrm{S}}\circ\PrjB=\KKK.
\label{LinfB.eq.C.spSPrjB.eq.KKK}
\end{equation}

Consider the categories: $\CsJ$ of $C^*$-algebras and Jordan $*$-homomorphisms; $\Cs$ of $C^*$-algebras and $*$-homomorphims; $\SZs$ of \c{S}tr\u{a}til\u{a}--Zsid\'{o} $C^*$-algebras and $*$-homomorphisms; $\Rs$ of Rickart $C^*$-algebras and $*$-homomorphisms; $\AWsc$ of $AW^*$-algebras and complete $*$-ho\-mo\-mor\-ph\-isms; $\cCsu$ of commutative $C^*$-algebras and unital $*$-homomorphisms; $\cSZsu$ of commutative \c{S}tr\u{a}til\u{a}--Zsid\'{o} $C^*$-algebras and unital $*$-homomorphisms; $\cRsu$ of commutative Rickart $C^*$-algebras and unital $*$-homomorphisms; $\cAWsuc$ of commutative $AW^*$-algebras and unital complete $*$-ho\-mo\-mor\-ph\-isms; $\cVNun$ of commutative von Neumann algebras and unital normal $*$-homomorphims; $\cWsun$ of commutative $W^*$-algebras and unital normal $*$-homomorphisms; $\cWsuIso$ of commutative $W^*$-algebras and unital $*$-isomorphisms. The functor $\mathrm{C}:\cpH^\op\ra\cCsu$ assigns to each $\X\in\Ob(\cpH^\op)$ a set $\rpktarget{CONTIFUN}\mathrm{C}(\X;\CC)$ of continuous complex valued functions on $\X$ equipped with multiplication defined pointwise, $^*$ defined by complex conjugation on codomain, norm defined by $\n{f}:=\sup_{\xx\in\X}\{\ab{f(\xx)}\}$, and unit given by a function constantly equal to $1$. The functor $\sp_{\mathrm{G}}:\cCsu\ra\cpH^\op$ sends $\C\in\Ob(\cCsu)$ to its \df{Gel'fand spectrum} $\rpktarget{GELFSPEC}\sp_{\mathrm{G}}(\C)$, defined as the set of all nonzero unital $*$-homomorphisms $f:\C\ra\CC$,
\begin{equation}
        \sp_\mathrm{G}(\C):=\Hom_{\cCsu}(\C,\CC)\setminus\{0\},
\end{equation}
equipped with the weak-$\star$ topology of $\C^\banach$. At the level of objects, this construction was introduced by Gel'fand \cite{Gelfand:1939,Gelfand:1941}. According to categorified version \cite{Negrepontis:1971} of Gel'fand--Na\u{\i}mark theorem \cite{Gelfand:Naimark:1943}, the functors $\sp_\mathrm{G}$ and $\mathrm{C}$ define an equivalence of categories
\begin{equation}
        \mathrm{C}\circ\sp_\mathrm{G}\iso\id_{\cCsu},\;\;\;\sp_\mathrm{G}\circ\mathrm{C}\iso\id_{\cpH^\op}.
\label{Stone.Gelfand.Naimark.duality}
\end{equation}
An analogous construction for $\mathrm{C}(\X;\RR)$ and $C^*$-algebras over $\RR$ was developed by Stone \cite{Stone:1940}, while its categorified version was discussed in detail in \cite{Johnstone:1982}. Thus, if both real and complex case are considered simultaneously, one can speak of the Stone--Gel'fand--Na\u{\i}mark equivalence of categories. \c{S}tr\u{a}til\u{a} and Zsid\'{o} \cite{Stratila:Zsido:1977:1979} proved that commutative \c{S}tr\u{a}til\u{a}--Zsid\'{o} $C^*$-algebras are characterised as those commutative $C^*$-algebras that have totally disconnected Gel'fand spectrum, which implies an equivalence
\begin{equation}
        \mathrm{C}\circ\sp_\mathrm{G}\iso\id_{\cSZsu},\;\;\;\sp_\mathrm{G}\circ\mathrm{C}\iso\id_{\tdcpH^\op}.
\label{Stratila.Zsido.duality}
\end{equation}
In order to consider further special cases of the Stone--Gel'fand--Na\u{\i}mark equivalence \eqref{Stone.Gelfand.Naimark.duality}, we need some facts from the theory of Rickart $C^*$-algebras and $AW^*$-algebras (see \cite{Kaplansky:1955,Bade:1971,Berberian:1972,Stratila:Zsido:1977:1979} for a detailed account). Rickart \cite{Rickart:1946} showed that a set of projections in a Rickart $C^*$-algebra $\C$ is a countably additive lattice, which turn to a ccb-algebra if{}f $\C$ is commutative. In such case $\C$ is a countably additive complex Banach lattice (a commutative $C^*$-algebra is always a Banach lattice with respect to its norm \cite{Gelfand:Naimark:1943}). By Stone theorem discussed above, ccb-algebras are dual to basically disconnected compact Hausdorff spaces. This implies the Stone--Gel'fand--Na\u{\i}mark type equivalence
\begin{equation}
        \mathrm{C}\circ\sp_\mathrm{G}\iso\id_{\cRsu},\;\;\;\sp_\mathrm{G}\circ\mathrm{C}\iso\id_{\bdcpH^\op},
\end{equation}
which was established (at the set theoretical level) already by Rickart \cite{Rickart:1946}. The $AW^*$-algebras form a subclass of Rickart $C^*$-algebras, and they are characterised by Dedekind--MacNeille completeness of their lattice of projections \cite{Kaplansky:1951,Kaplansky:1952}, which for commutative unital $AW^*$-algebras turn to a Dcb-algebras. Every commutative $AW^*$-algebra is a Dedekind--MacNeille complete Banach lattice, and a commutative $C^*$-algebra is an $AW^*$-algebra if{}f every bounded set of its self-adjoint elements has a supremum. By yet another theorem of Stone \cite{Stone:1949:CJM}, a compact Hausdorff topological space is extremally disconnected if{}f the Riesz lattice $\mathrm{C}(\X;\RR)$ (or $\mathrm{C}(\X;\CC)$) is Dedekind--MacNeille complete. This gives a duality between commutative unital $AW^*$-algebras and stonean topological spaces. In order to extend this duality to categorical formulation, we need to use Brown's theorem \cite{Pedersen:1986}: if $f:\X_2\ra\X_1$ is a continuous function between compact Hausdorff spaces, and $\mathrm{C}(f):\mathrm{C}(\X_1;\CC)\ra\mathrm{C}(\X_2;\CC)$ is a corresponding $*$-homomorphism of commutative $C^*$-algebras, then $\mathrm{C}(f)$ is complete if{}f $f$ is open. This allows us to establish an equivalence
\begin{equation}
        \mathrm{C}\circ\sp_\mathrm{G}\iso\id_{\cAWsuc},\;\;\;\sp_\mathrm{G}\circ\mathrm{C}\iso\id_{\edcpHo^\op}.
\label{Kaplansky.duality}
\end{equation}
According to Dixmier's theorem \cite{Dixmier:1951}, every hyperstonean topological space $\X$ determines a commutative $W^*$-algebra given by the space $\mathrm{C}(\X;\RR)$, which is isometrically isomorphic to a Banach dual of the Banach space $\Meas_\star(\X,\mho_{\mathrm{Borel}}(\X)):=\{\mu:=\tmu_1-\tmu_2\mid\tmu_1,\tmu_2\in\Meas^+_\star(\X,\mho_{\mathrm{Borel}}(\X))\}$. Grothendieck \cite{Grothendieck:1955} (see also \cite{Wada:1957}) proved the converse theorem: if $X$ is a Banach space, $\X$ is a compact Hausdorff space, $j:X\ra X^\banach{}^\banach$ is a canonical embedding, and $f:\mathrm{C}(\X;\RR)\ra X^\banach$ is an isometric isomorphism, then $\X$ is hyperstonean and $f^\banach\circ j: X\ra\Meas_\star(\X,\mho_{\mathrm{Borel}}(\X))$ is an isometric isomorphism. Both theorems hold also in the complex case. For commutative unital $W^*$-algebras a unital $*$-homomorphism is complete if{}f it is normal. Moreover, restriction to homeomorphisms is equivalent to restriction to $*$-isomorphisms. Hence, the Dixmier--Grothendieck theorem allows to restrict the Stone--Gel'fand--Na\u{\i}mark duality to a functorial equivalence of categories
\begin{align}
        \mathrm{C}\circ\sp_\mathrm{G}\iso\id_{\mcBc},\;&\;\;\sp_\mathrm{G}\circ\mathrm{C}\iso\id_{\hypso^\op};\label{Dixmier.Grothendieck.duality}\\
        \mathrm{C}\circ\sp_\mathrm{G}\iso\id_{\mcBIso},\;&\;\;\sp_\mathrm{G}\circ\mathrm{C}\iso\id_{\hypsh^\op}.
\end{align}
The Dixmier--Grothendieck theorem is in a quite interesting relationship with a result of Segal \cite{Segal:1951}, who established an explicit equivalence between localisable measure spaces and commutative unital von Neumann algebras: $(\X,\mho(\X),\tmu)$ is a localisable measure space if{}f $L_\infty(\X,\mho(\X),\tmu)$ is a commutative unital von Neumann algebra acting on a Hilbert space $L_2(\X,\mho(\X),\tmu)$, and, conversely, every commutative unital von Neumann algebra $\N$ on a Hilbert space $\H$ can be faithfully represented in this form, up to unitary equivalence $\H\iso L_2(\X,\mho(\X),\tmu)$. Together with earlier results on equivalence between normal unital $*$-homomorphisms, order continuous boolean homomorphisms, and complete morphisms of measurable spaces, this defines a pair of functors: $\Seg^\sharp:\locMeSp^\op\ra\cVNun$, $\Seg^\flat:\cVNun\ra\locMeSp^\op$ which set up an equivalence \cite{Pavlov:2010}
\begin{equation}
        \Seg^\sharp\circ\Seg^\flat\iso\id_{\cVNun},\;\;\;\Seg^\flat\circ\Seg^\sharp\iso\id_{\locMeSp^\op}.
\label{Segal.duality}
\end{equation}
We are however unaware of a direct proof of equivalence of morphisms of these two categories without using the above passage. If such proof had been provided, then a right composition of Segal duality \eqref{Segal.duality} with Sakai--Kosaki duality \eqref{Sakai.Kosaki.duality} together with a left composition with Wecken--Loomis--Sikorski duality \eqref{mcBc.locMesp.duality} and Stone duality \eqref{mcBc.hypso.duality}, would form an \textit{independent} (yet more complicated) statement of the duality between $\cWsun$ and $\hypso^\op$:
\begin{align}
        (\sp_\mathrm{S}\,\circ\,\mathrm{W}\,\circ\,\Seg^\flat\,\circ\, \CanVN)\,\circ\,(\FrgHlb\,\circ\,\Seg^\sharp\,\circ\,\mathrm{LS}\,\circ\,\mathrm{B})&\iso\id_{\cWsun},\nonumber\\ (\FrgHlb\,\circ\,\Seg^\sharp\,\circ\,\mathrm{LS}\,\circ\,\mathrm{B})\,\circ\,(\sp_\mathrm{S}\,\circ\,\mathrm{W}\,\circ\,\Seg^\flat\,\circ\, \CanVN)&\iso\id_{\hypso^\op}.
\label{long.segal.equivalence}
\end{align}

Independently of the above three `topological' families of dualities (Stone, Kre\u{\i}n--Kre\u{\i}n--Kakutani, Stone--Gel'fand--Na\u{\i}mark) it is possible to develop the fourth family of dualities, which is of algebraic character, relates a subclass of Banach lattices with a subclass of commutative Banach algebras and relies on Freudenthal's spectral theorem and some results in the theory of f-algebras.

\begin{proposition}\label{proposition.Linfty.cWstar}
The categories $\cWsun$ and $\Linfc$ are equivalent, the categories $\cWsuIso$ and $\LinfIso$ are equivalent.
\end{proposition}
\begin{proof}
Using Freudenthal's spectral theorem \cite{Freudenthal:1936}, Lyubovin \cite{Lyubovin:1954,Lyubovin:1956} and Vulikh \cite{Vulikh:1957:Lenin,Vulikh:1961} proved that each commutative von Neumann algebra is a Dedekind--MacNeille complete Banach lattice (for earlier proofs of this result, depedending on Gel'fand's representation theorem, see \cite{Kaplansky:1951,Fell:Kelley:1952}), while Luxemburg and Zaanen \cite{Luxemburg:Zaanen:1971} proved that each commutative von Neumann algebra is an MI-space. Both proofs hold for arbitrary $W^*$-algebra. Taking into account that each $W^*$-algebra has a unique Banach predual, we conclude that each commutative $W^*$-algebra is a complex \textit{proper} abstract $L_\infty$ space. Conversely, each real abstract $L_\infty$ space $X$ has a form $L_\infty(\boole)$ over a ccb-algebra $\boole$ of projection bands of $X$. Hence (see e.g. \cite{Fremlin:2000}) $X$ is a real commutative Banach algebra and an archimedean real f-algebra. As an f-algebra, it satisfies $\ab{y^2}=\ab{y}^2$, where $y^2:=y\cdot y$. As an archimedean f-algebra it satisfies \cite{Huijsmans:dePagter:1982}
\begin{equation}
        \forall x\in X\;\;\;x\geq0\;\iff\;\exists! y\in X\;\;\;x=y^2,
\label{f.algebra.square.root}
\end{equation}
while as a Banach lattice it satisfies $\ab{x}\leq\ab{y}$ $\limp$ $\n{x}_X\leq\n{y}_X$. Hence, $X$ satisfies $\n{y^2}_X=\n{y}_X^2$. Its Banach algebra complexification $X_\CC:=X+\ii X$, equipped with multiplication, involution, and norm:
\begin{align}
        (x_1+\ii x_2)\cdot(y_1+\ii y_2)&:=(x_1y_1-x_2y_2)+\ii(x_1y_1+x_2y_2),\\
        (x_1+\ii x_2)^*&:=x_1-\ii x_2,\\
        \n{x+\ii y}_{X_\CC}&:=\n{x^2+y^2}_X^{1/2},
\end{align}
is a commutative $C^*$-algebra (see e.g. \cite{vanNeerven:1997}). On the other hand, the Banach lattice complexification $\tilde{X}_\CC$ of $X$ is equipped with the norm \cite{Beukers:Huijsmans:dePagter:1983}
\begin{equation}
        \n{x+\ii y}_{\tilde{X}_\CC}:=\n{\ab{x+\ii y}}_X=\n{\sqrt{x^2+y^2}}_X.
\end{equation}
These two complexifications coincide, because $\n{x^{1/2}}_X=\n{x}^{1/2}_X$ $\forall x=y^2\geq0$. Thus, every abstract $L_\infty$ space is a commutative $C^*$-algebra with the multiplicative unit $\II$ given by the order unit. The Dedekind--MacNeille completeness of $X_\CC$ implies that its boolean algebra $\boole$ of projection bands is a Dcb-algebra, hence $X_\CC$ is a commutative $AW^*$-algebra \cite{Kaplansky:1955}. Finally, the existence of a unique predual turns $X$ into a commutative $W^*$-algebra. Every algebra homomorphism $f$ of f-algebras with multiplicative unit element is a Riesz homomorphism if{}f it satisfies $f(\ab{x})=\ab{f(x)}$ \cite{Huijsmans:dePagter:1984}. But this is equivalent to a condition that $f$ is a $*$-homomorphism, since $f(x^*x)=f(x)^*f(x)=\ab{f(x)}^2$, hence $f(\ab{x})=\ab{f(x)}$, which follow from \eqref{f.algebra.square.root}. From the equality of multiplicative unit $\II$ with an order unit, and coincidence of definitions of normality and order continuity, it follows that a function $f:X_1\ra X_2$ between two commutative unital $W^*$-algebras $X_1$ and $X_2$ is a normal (resp., unital) $*$-homomorphism if{}f it is order continuous (resp., unit preserving). Finally, the surjective isometries of commutative $W^*$-algebras coincide with their $*$-isomorphisms (and are normal), while the surjective isometries of Banach lattices coincide with their isometric Riesz isomorphisms (and are order continuous).
\end{proof}
This result follows the tradition of lattice theoretic analysis of a structure of commutative $W^*$-algebras without using topological representation theorems (see also \cite{SzokefalviNagy:1942,Riesz:SzokefalviNagy:1952,Vulikh:1957,Dodds:1969,Dodds:1974}). As a corollary, obtained by composition with the functors $\mathrm{L}_\infty$ and $\PrjB$, it provides a characterisation of mcb-algebras as those Dcb-algebras (and those ccb-algebras) which arise as lattices of projections in commutative $W^*$-algebras.\footnote{For a representation dependent characterisation which includes also an equivalence between normal $*$-homomorphisms and order continuous boolean homomorphisms, see e.g. \cite{Lurie:2011}.} Using some additional results, it is also possible to establish a similar lattice--algebra duality between MI-spaces and commutative $C^*$-algebras, as well as between abstract $L_\infty$ spaces and commutative unital Rickart $C^*$-algebras.

\begin{proposition}
The categories $\MI$ and $\cCsu$ are equivalent, the categories $\aLinf$ and $\cRsu$ are equivalent.
\end{proposition}
\begin{proof}
Every MI-space $X$ with an order unit $e$ can be equipped with a multiplication $\cdot:X\times X\ra X$ and an algebraic (multiplicative) unit $\II\geq0$ such that $X$ becomes an archimedean f-algebra, a Banach lattice, and $e=\II$ \cite{Martignon:1980}. The unique extension to the commutative unital $C^*$-algebra structure is then provided in the same way as in the proof of Proposition \ref{proposition.Linfty.cWstar}. Conversely, by Sherman's theorem \cite{Sherman:1951}, each $C^*$-algebra $\C$ is a Banach lattice with respect to an order \eqref{Cstar.order} if{}f it is commutative. The order unit $e$ of $\C$ is defined as an algebraic (multiplicative) unit $\II$. From Freudenthal's spectral theorem it follows that the spectrum of $\ab{x}$ is contained in $\RR^+$ for any $x\in\C$, and it is equal to $\n{\ab{x}}=\n{x}$. From $\n{x}=\sup\{\ab{x}\}$ it follows that $\n{x}=\min\{\lambda\in\RR\mid\ab{x}\leq\lambda\II\}$, hence a norm of $\C$ is an order unit norm, so $\C$ is an MI-space. The restriction to commutative Rickart $C^*$-algebras corresponds to restriction of lattices of projections to ccb-algebras, which in turn correspond to countably additive complete MI-spaces, which are the same as abstract $L_\infty$ spaces. The correspondence of unital $*$-homomorphisms and unit preserving Riesz homomorphisms was already established.
\end{proof}
\begin{definition} Functors generating above equivalences will be denoted by $\Freu^\sharp:\MI\ra\cCsu$ and $\Freu^\flat:\cCsu\ra\MI$.
\end{definition}
\begin{corollary}
The following diagrams commute in the weak sense (up to isomorphism):
\hspace*{-0.6cm}\begin{equation}\begin{tabular}{ccc}\xymatrix{
        \cCsu
        \ar@<0.5ex>[rr]^{\Freu^\sharp}
        \ar@<0.5ex>@/_1pc/[rd]^{\sp_\mathrm{G}}
        &&
        \MI
        \ar@<0.5ex>[ll]^{\Freu^\flat}
        \ar@<0.5ex>@/^1pc/[ld]^{\KKK}
        \\
        &
        \cpH^\op
        \ar@<0.5ex>@/_1pc/[ur]^{\mathrm{C}}
        \ar@<0.5ex>@/^1pc/[ul]^{\mathrm{C}}
        &
}
&
\xymatrix{
        \cRsu
        \ar@<0.5ex>[rr]^{\Freu^\sharp}
        \ar@<0.5ex>@/_1pc/[rd]^{\sp_\mathrm{G}}
        &&
        \aLinf
        \ar@<0.5ex>[ll]^{\Freu^\flat}
        \ar@<0.5ex>@/^1pc/[ld]^{\KKK}
        \\
        &
        \bdcpH^\op
        \ar@<0.5ex>@/_1pc/[ur]^{\mathrm{C}}
        \ar@<0.5ex>@/^1pc/[ul]^{\mathrm{C}}
        &
}
&
\xymatrix{
        \cWsun
        \ar@<0.5ex>[rr]^{\Freu^\sharp}
        \ar@<0.5ex>@/_1pc/[rd]^{\sp_\mathrm{G}}
        &&
        \Linfc
        \ar@<0.5ex>[ll]^{\Freu^\flat}
        \ar@<0.5ex>@/^1pc/[ld]^{\KKK}
        \\
        &
        \hypso^\op
        \ar@<0.5ex>@/_1pc/[ur]^{\mathrm{C}}
        \ar@<0.5ex>@/^1pc/[ul]^{\mathrm{C}}
        &
}
\end{tabular}
\label{Freu.comm}
\end{equation}
\end{corollary}
\begin{conjecture}
The functors $\Freu^\sharp$ and $\Freu^\flat$ provide an equivalence between categories $\bLinf$ and $\cSZsu$, as well as between $\DcaLinfc$ and $\cAWsuc$.
\end{conjecture}
From \eqref{Freu.comm}, \eqref{Stratila.Zsido.duality} and \eqref{Kaplansky.duality} it follows that this conjecture is true, but it remains an open problem how to prove it without passing to topological representation. 

The above results allow us to describe jointly the canonical structures of commutative and noncommutative integration theory.
\begin{proposition}\label{huge.diagram.int.theory}
The diagram

{\small
\begin{equation}\hspace*{-1cm}\xymatrix@=2em{%
&
\CsJ
&
\VNn
\ar@<+0.5ex>[r]^{\FrgHlb}
&
\Wsn
\ar@<+0.5ex>[l]|{\CanVN}
\ar@{ >->}[r]
&
\AWsc
\ar@{ >->}[rd]
&
\SZs
\ar@{ >->}[r]
&
\Cs
\ar@<-0.5ex>@{ >->}@/_2pc/[lllll]
\\
&
\WssJIso
\ar@{ >->}[u]
&
\cVNun
\ar@{ >->}[u]
\ar@<+0.5ex>[r]^{\FrgHlb}
\ar@<+0.5ex>[d]^{\Seg^\flat}
&
\cWsun
\ar@<+0.5ex>[l]^{\CanVN}
\ar@{ >->}[r]
\ar@{ >->}[u]
\ar@<+1.2ex>[ddd]^{\sp_\mathrm{G}}
\ar@<-0.7ex>@/^1.3pc/[ddddl]_(0.47){\Freu^\flat\!\!\!}
\ar@<+0.8ex>@/^2pc/[dddd]|(.6){\Prj}
&
\cAWsuc
\ar@{ >->}[ddr]
\ar@{ >->}[u]
\ar@<+0.5ex>[ddd]^{\sp_\mathrm{G}}
\ar@<-1ex>@/_2.4pc/[dddd]|{\Prj}
&
\Rs
\ar@{ >->}[u]
&
\cCsu
\ar@{ >->}[u]
\ar@<+0.5ex>[d]^(0.65){\Freu^\flat}
\ar@<+0.5ex>@/_0.5pc/[dl]^(0.65){\sp_\mathrm{G}}
\\
\catname{ncL}_{1/\gamma}\catname{Iso}
\ar@/^1.5pc/[ur]|(0.35){\Sher_{1/\gamma}^\sharp}
\ar@<+0.5ex>@/_2.4pc/[r]|{(\cdot)^\banach}
&
\WsIso
\ar@{ >->}[u]
\ar@{=}[l]|(.45){\,\gamma=0\,}
\ar@/^1.15pc/[l]|{\ncLFT_{1/\gamma}}
\ar@/_1pc/[l]|{\ncLK_{1/\gamma}}
\ar@<+0.5ex>@/^2.9pc/[l]|{(\cdot)_\star}
\ar@<-1ex>@{ >->}@/^1.5pc/[uurr]
&
\locMeSp^\op
\ar@<+0.5ex>[u]^{\Seg^\sharp}
\ar@<+0.5ex>@/_0.8pc/[rddd]^(0.2){\mathrm{W}}
&
&
&
\cpH^\op
\ar@<+0.5ex>[r]|{\,\mathrm{C}\,}
\ar@<+0.3ex>@/^0.5pc/[ur]^(0.3){\mathrm{C}}
&
\MI
\ar@<+0.5ex>[l]^{\KKK}
\ar@<+0.5ex>[u]^(0.4){\Freu^\sharp}
\\
&
\cWsuIso
\ar@{ >->}[u]
\ar@{ >->}@/_1.5pc/[uurr]
\ar@<+0.5ex>@/^0.6pc/[rd]^(0.65){\sp_\mathrm{G}}
\ar@<+0.5ex>[d]^{\Freu^\flat}
&
&
&
&
\cRsu
\ar@{ >->}[r]
\ar@<-1ex>@{ >->}@/^1.9pc/[uu]%
\ar@<+0.5ex>[d]^(0.56){\!\sp_\mathrm{G}}
\ar@/_1.94pc/[dd]|(0.3){\Prj}
\ar@<-0.5ex>@/^3.2pc/[ddd]_(0.5){\Freu^\flat\!\!}
&
\cSZsu
\ar@<+0.5ex>@{ >->}@/_3.2pc/[uu]
\ar@/^2.5pc/[dd]|{\Prj}
\ar@<+0.5ex>[d]^{\sp_\mathrm{G}}
\ar@{ >->}[luuu]
\\
&
\LinfIso
\ar@<+0.5ex>[u]^{\Freu^\sharp}
\ar@<+0.5ex>[ldd]^(0.55){(\cdot)_\star}
\ar@<+0.5ex>[dd]^(0.4){\PrjB}
\ar@<+0.5ex>[r]^{\KKK}
\ar@{ >->}@<+0.5ex>[rd]
\ar@{ >->}@<+0.5ex>@/_1.2pc/[rrdd]
&
\hypsh^\op
\ar@{ >->}[r]
\ar@<+0.5ex>[l]^{\mathrm{C}}
\ar@<+0.1ex>@/_0.5pc/[lu]^{\mathrm{C}}
\ar@<-0.5ex>[ldd]_(0.65){\mathrm{B}}
&
\hypso^\op
\ar@<-0.5ex>[ld]|{\mathrm{C}}
\ar@<-0.6ex>[uuu]^{\mathrm{C}}
\ar@<-0.9ex>[d]|{\mathrm{B}}
\ar@{ >->}[r]
&
\edcpHo^\op
\ar@{ >->}[r]
\ar@<+0.5ex>[uuu]^{\mathrm{C}}
\ar@<+0.5ex>[d]^{\mathrm{B}}
\ar@<-0.5ex>@/^2.5pc/[dd]_(0.3){\mathrm{C}\!\!}
&
\bdcpH^\op
\ar@{ >->}[r]
\ar@<+0.5ex>[d]^{\mathrm{B}}
\ar@<+0.5ex>@/_2.5pc/[dd]^(0.35){\!\!\mathrm{C}}
\ar@<+0.5ex>[u]^{\mathrm{C}}
&
\tdcpH^\op
\ar@{ >->}[luu]
\ar@<+0.5ex>[d]^{\mathrm{B}}
\ar@<+0.5ex>[u]^{\mathrm{C}}
\ar@<-0.5ex>@/^2pc/[dd]_(0.5){\mathrm{C}}
\\
&
&
\Linfc
\ar@<-0.1ex>@/_1.3pc/[ruuuu]_(0.42){\!\!\!\Freu^\sharp}
\ar@{ >->}@/_0.7pc/[rrd]
\ar@<+0.5ex>[r]|{\PrjB}
\ar@<-0.5ex>[ru]_{\KKK}
&
\mcBc
\ar@{ >->}[r]
\ar[u]_(0.6){\sp_\mathrm{S}}
\ar@<+0.5ex>[l]^{\mathrm{L}_\infty}
\ar@<+0.5ex>@/^0.8pc/[luuu]^(0.8){\mathrm{LS}}
&
\DcBc
\ar@{ >->}[r]
\ar@<+0.5ex>[u]^{\sp_\mathrm{S}\!\!}
\ar@<+0.5ex>[d]^(0.4){\mathrm{L}_\infty}
&
\ccB
\ar@{ >->}[r]
\ar@<+0.5ex>[u]^{\sp_\mathrm{S}\!\!}
\ar@<+0.5ex>[d]^(0.3){\!\mathrm{L}_\infty}
&
\catname{B}
\ar@<+0.5ex>[u]^{\sp_\mathrm{S}}
\ar@<+0.5ex>[d]^(0.35){\mathrm{L}_\infty}
\\
\catname{L}_{1/\gamma}\catname{Iso}
\ar@{ >->}[uuuu]
\ar@<+0.5ex>[ruu]^{(\cdot)^\banach}
\ar@<+0.5ex>[r]^{\PrjB}
&
\mcBIso
\ar@<+0.5ex>[l]^{\mathrm{L}_{1/\gamma}}
\ar@<+0.5ex>[uu]^{\mathrm{L}_\infty}
\ar@<-0.5ex>[uur]_(0.4){\sp_\mathrm{S}}
\ar@{ >->}@<+0.5ex>@/_0.9pc/[urr]
\ar@{ >->}[r]
&
\ccBIso
\ar@<-0.5ex>[r]_{\mathrm{L}_\infty}
&
\aLinfIso
\ar@<-0.5ex>[l]|{\PrjB}
\ar@{ >->}@/_2pc/[rr]
&
\DcaLinfc
\ar@{ >->}[r]
\ar@<+0.5ex>[u]^{\PrjB}
\ar@<-0.5ex>@/_2.5pc/[uu]_(0.78){\!\!\!\KKK}
&
\aLinf
\ar@{ >->}[r]
\ar@<+0.5ex>[u]^(0.6){\PrjB\!}
\ar@<+0.5ex>@/^2.5pc/[uu]^(0.23){\KKK\!\!\!}
\ar@<-0.5ex>@/_3.2pc/[uuu]_(0.17){\!\!\!\!\Freu^\sharp}
&
\bLinf
\ar@<+0.5ex>[u]^{\PrjB}
\ar@<-0.5ex>@/_2pc/[uu]_(0.66){\KKK}
\ar@{ >->}@<-1.5ex>@/_4pc/[uuuu]
}
\label{canonical.integration.theory.big.stuff}
\end{equation}} commutes in the following sense: the pairs of functors $(\mathrm{B},\sp_{\mathrm{S}})$, $(\mathrm{C},\sp_{\mathrm{G}})$, $(\mathrm{C},\KKK)$, $(\mathrm{L}_\infty,\PrjB)$, $(\mathrm{L}_{1/\gamma},\PrjB)$, $(\mathrm{LS},\mathrm{W})$, $(\FrgHlb,\CanVN)$,  $(\Seg^\sharp,\Seg^\flat)$, and $(\Freu^\sharp,\Freu^\flat)$ 
commute weakly (up to isomorphism) as equivalences of categories that are at their domains and codomains; the functors $(\cdot)^\banach$ and $(\cdot)_\star$ are considered only for $\gamma=1$; the functor denoted $\gamma=0$ is considered only for $\gamma=0$ and is an isomorphism of categories.
\end{proposition}
\begin{proof}
Follows directly from \eqref{Sakai.Kosaki.duality}, \eqref{ncLp.cat.diag}, \eqref{commutative.integration.diagram}, \eqref{mcBc.locMesp.duality}, \eqref{L.gamma.mcBIso.duality}, \eqref{Stone.tdcpH.B.duality}-\eqref{LinfB.eq.C.spSPrjB.eq.KKK}, \eqref{Stone.Gelfand.Naimark.duality}-\eqref{Segal.duality} and \eqref{Freu.comm}.
\end{proof}

The diagram \eqref{canonical.integration.theory.big.stuff} is valid for MI-spaces and $C^*$-algebras in both cases: real and complex. In the real case the functor $\mathrm{C}$ assigns $\mathrm{C}(\X;\RR)$ spaces, while in complex case it assigns $\mathrm{C}(\X;\CC)$ spaces. Note also that \eqref{canonical.integration.theory.big.stuff} contains no category of measure spaces. According to Fremlin, \cytat{another curiosity of traditional measure theory is the unimportance of any notion of homomorphism between measure spaces. This distinguishes it from all comparably abstract branches of mathematics. I believe that this deficiency occurs because the natural and important homomorphisms of the theory are between measure algebras, and not between measure spaces at all} \cite{Fremlin:1974}. From the category theoretic point of view, this means that measure spaces are just irrelevant for the \textit{foundations} of commutative (and noncommutative) integration theory. However, there is still some structural difference between the categories of commutative and noncommutative $L_p$ spaces: while both constructions assign $L_p$ spaces to underlying algebras canonically and functorially, the categories $\catname{L}_p\catname{Iso}$ possess also an internal characterisation in terms of abstract $L_p$ spaces. So far there is no internal characterisation of categories $\catname{ncL}_p\catname{Iso}$, with an exception of the case $p=2$, which was characterised by Connes \cite{Connes:1974}. Moreover, in \cite{Henson:Raynaud:Rizzo:2007} it was proved that the categories $\catname{ncL}_p\catname{Iso}$ cannot be axiomatised neither in the language of Banach spaces nor in that of operator spaces. This prompts for some further development in the noncommutative integration theory. We think that the key problem is to define a category of canonical core algebras equipped with canonical trace without using the Falcone--Takesaki noncommutative flow of weights, but using some suitable \textit{Banach bimodule} properties imposed on the category of topological $*$-algebras instead. However, we are unable to develop this speculative programme in more details here (one can compare this with some comments in \cite{Sherman:2001})\footnote{I have been recently (September 2012) informed by Dmitri\u{\i} Pavlov that he is working on his own approach to this problem, which includes a replacement of the Falcone--Takesaki core algebra by means of some alternative construction.}. 

Let us note that for the sake of brevity we have omitted the discussion of adjoint functors to embedding arrows in \eqref{canonical.integration.theory.big.stuff}, the discussion which of these arrows are full and/or faithful, and the discussion of additional equivalences for locally compact Hausdorff topological spaces and nonunital commutative $C^*$-algebras.

The roles played in the commutative integration theory by mcb-algebras $\boole$ and their representations in terms of measurable spaces $(\X,\mho(\X),\mho^0(\X))$ or measure spaces $(\X,\mho(\X),\tmu)$ are analogous to the roles played in the noncommutative integration theory by, respectively, $W^*$-algebras $\N$ and their standard representations $(\H,\pi(\N),J,\stdcone)$ or the GNS representations ($\H_\omega,\pi_\omega,\Omega_\omega)$. In particular, the Loomis--Sikorski representation $(\sp_\mathrm{S}(\boole),\mho_{\mathrm{LS}}(\sp_\mathrm{S}(\boole)),\mho^0_{\mathrm{LS}}(\sp_\mathrm{S}(\boole)))$ is an analogue of the Kosaki canonical representation $(L_2(\N),\pi_\N(\N),J_\N,L_2(\N)^+)$. Furthermore, if $\N$ is commutative, then each $\mu_\psi\in\W(\boole)$ determines a normal semi-finite trace on $\N$ by
\begin{equation}
	\psi(x)=\int\mu_\psi x\;\;\forall x\in\N^+.
\end{equation}
If $\mu_\phi\in\W_0(\boole)$ corresponds to $\phi\in\W_0(\N)$ and $f\in L_1(\boole,\mu_\phi)$ is its Radon--Nikod\'{y}m quotient with respect to $\mu_\psi\in\W(\boole)$, $f=\frac{\mu_\psi}{\mu_\phi}$, which means
\begin{equation}
        \psi(x)=\int\mu_\psi x=\int\mu_\phi\frac{\mu_\psi}{\mu_\phi} x=\phi\left(\frac{\mu_\psi}{\mu_\phi}x\right)\;\;\forall x\in L_\infty(\boole)^+,
\end{equation}
then the faithfulness of $\psi$ corresponds to strict positivity of $\mu_\psi$ and implies $\frac{\mu_\psi}{\mu_\phi}>0$. In such case, the map $\RR\ni t\mapsto\left(\frac{\mu_\psi}{\mu_\phi}\right)^{\ii t}\in\N$ satisfies
\begin{equation}
        \left(\frac{\mu_\psi}{\mu_\phi}\right)^{\ii t}=\Connes{\psi}{\phi}{t}\;\;\forall t\in\RR.
\end{equation}
Note that the boolean ideals $\boole^\mu\subseteq\boole$ for $\mu\in\W(\boole)$ play the role analogous to the ideals $\nnn_\psi\subseteq\N$ for $\psi\in\W(\N)$. In particular, the compatibility condition $x\in\boole^{\mu_1}$ $\limp$ $\exists y\leq x$ $y\in\boole^{\mu_2}$ plays a crucial role in the definition of the Radon--Nikod\'{y}m quotient $\frac{\mu_1}{\mu_2}$ of $\mu_1,\mu_2\in\W_0(\boole)$, which corresponds to the crucial role played by the condition $x\in\nnn_{\psi_1}$ $\limp$ $x\in\nnn_{\psi_2}$ in the extension of Connes' cocycle $\Connes{\psi_1}{\psi_2}{t}$ of $\psi_1,\psi_2\in\W_0(\N)$ to the (square root of) noncommutative analogue of the Radon--Nikod\'{y}m quotient, $\Connes{\psi}{\phi}{-\ii/2}$. The additional condition $\mu_1\ll\mu_2$ required for $\mu_1,\mu_2\in\W(\boole)$ corresponds to the additional condition $\psi_1\ll\psi_2$ required for $\psi_1,\psi_2\in\W(\N)$.
\subsection{Commutative and noncommutative Orlicz spaces}
In this section we will first review the main notions from the theory of commutative Orlicz spaces, and then we will move to discussion of the noncommutative version of this theory. For a detailed treatment of the theory of commutative Orlicz spaces, as well as the associated theory of modular spaces, see \cite{Nakano:1950,Nakano:1951:book,Nakano:1951,Zaanen:1953,Krasnoselskii:Rutickii:1958,Lindenstrauss:Tzafriri:1977:1979,Musielak:1983,Maligranda:1989,Rao:Ren:1991,Chen:1996,Rao:Ren:2002,Leonard:2007}. The Banach lattice analogues of commutative Orlicz spaces were studied in \cite{Nakano:1950,Kantorovich:Vulikh:Pinsker:1950,Orlicz:1964,Claas:Zaanen:1978,Wnuk:1980,Zaanen:1983,Wnuk:1984}.

If $X$ is a vector space over $\KK\in\{\RR,\CC\}$, and $\rpktarget{ORLICZFUN}\Orlicz:X\ra[0,\infty]$ is a convex function satisfying
\begin{enumerate}
\item[1)] $\Orlicz(0)=0$,
\item[2)] $\Orlicz(\lambda x)=0$ $\forall\lambda>0$ $\limp$ $x=0$,
\item[3)] $\ab{\lambda}=1$ $\limp$ $\Orlicz(\lambda x)=\Orlicz(x)$,
\end{enumerate}
then $\Orlicz$ is called a \df{pseudomodular function} \cite{Musielak:Orlicz:1959}. If 2) is replaced by
\begin{enumerate}
\item[2')] $\Orlicz(x)=0$ $\limp$ $x=0$,
\end{enumerate}
then $\Orlicz$ is called a \df{modular function} \cite{Nakano:1950,Nakano:1951:book,Nakano:1951}. If 2) is replaced by 
\begin{enumerate}
\item[2'')] $x\neq0$ $\limp$ $\lim_{\lambda\ra+\infty}\Orlicz(\lambda x)=+\infty$,
\end{enumerate}
then $\Orlicz$ is called a \df{Young function} \cite{Young:1912,Birnbaum:Orlicz:1930,Birnbaum:Orlicz:1931}. A Young function $f$ on $\RR$ is said to satisfy: \df{local $\triangle_2$ condition} if{}f \cite{Birnbaum:Orlicz:1930,Birnbaum:Orlicz:1931}
\begin{equation}
	\exists\lambda>0\;\;
	\exists x_0\geq0\;\;
	\forall x\geq x_0\;\;
	f(2x)\leq\lambda f(x);
\label{local.delta.two.condition}
\end{equation}
 \df{global $\triangle_2$ condition} if{}f $x_0$ in \eqref{local.delta.two.condition} is set to $0$. A convex function $f:\RR\ra\RR^+$ is called \df{N-function} \cite{Birnbaum:Orlicz:1931} if{}f $\lim_{x\ra^+0}\frac{f(x)}{x}=0$ and $\lim_{x\ra+\infty}\frac{f(x)}{x}=+\infty$. Every Young function $f$ allows do define a \df{Young--Birnbaum--Orlicz dual} \cite{Birnbaum:Orlicz:1931,Mandelbrojt:1939}
\begin{equation}
	f^\Young:\RR\ni y\mapsto f^\Young(y):=\sup_{x\geq0}\{x\ab{y}-f(x)\}\in[0,\infty].
\end{equation}
Every YBO dual is a nondecreasing Young function, and each pair $(f,f^\Young)$ satisfies the \df{Young inequality} \cite{Young:1912}
\begin{equation}
	xy \leq f(x)+f^\Young(y)\;\;\forall x,y\in\RR.
\end{equation}
If $f$ is also an N-function, then $f^\Young{}^\Young=f$. Every modular function $\Orlicz:X\ra[0,\infty]$ determines a \df{modular space} \cite{Nakano:1950,Nakano:1951:book}
\begin{equation}
	X_\Orlicz:=\{x\in X\mid\lim_{\lambda\ra^+0}\Orlicz(\lambda x)=0\}
\end{equation}
and the \df{Morse--Transue--Nakano--Luxemburg norm} on $X_\Orlicz$ \cite{Morse:Transue:1950,Nakano:1951:book,Luxemburg:1955,Weiss:1956},
\begin{equation}
	 \n{\cdot}_\Orlicz:X_\Orlicz\ni x\mapsto\n{x}_\Orlicz:=\inf\{\lambda>0\mid\Orlicz(\lambda^{-1}x)\leq 1\}\in\RR^+,
\end{equation}
which allows to define a Banach space
\begin{equation}
	L_\Orlicz(X):=\overline{X_\Orlicz}^{\n{\cdot}_\Orlicz}.
\end{equation}
An \df{Orlicz function} is defined as a function $f:[0,\infty[\,\ra[0,\infty]$ that is convex, continuous, nondecreasing, satisfying $f(0)=0$, $\lambda>0$ $\limp$ $f(\lambda)>0$, $\lim_{\lambda\ra+\infty}f(\lambda)=+\infty$. By definition, an Orlicz function is a restriction to $\RR^+$ of a modular Young function on $\RR$, equipped with the additional conditions of continuity and nondecreasing.\footnote{In all application discussed here, the condition of continuity of an Orlicz function $f$ can be relaxed to continuity at $[0,x_f[$ with left continuity at $x_f$, where $x_f:=\sup\{\lambda>0\mid f(\lambda)<\infty\}$.} Every Orlicz function $f$ defines a continuous modular function $\Orlicz_f$ on $L_0(\X,\mho(\X),\tmu;\RR)$ where $(\X,\mho(\X),\tmu)$ is a localisable measure space, by the formula
\begin{equation}
	\Orlicz_f:L_0(\X,\mho(\X),\tmu;\RR)\ni x\mapsto\Orlicz_f(x):=\int\tmu f(\ab{x})\in[0,\infty].
\label{Orlicz.modular}
\end{equation}
An \df{Orlicz space} \cite{Orlicz:1936} is defined as a Banach space
\begin{equation}
	L_{\Orlicz_f}(\X,\mho(\X),\tmu;\RR):=L_{\Orlicz_f}(L_0(\X,\mho(\X),\tmu;\RR)),
\rpktarget{LORLICZX}
\end{equation}
an \df{Orlicz class} \cite{Orlicz:1932} is defined by
\begin{equation}
	\tilde{L}_{\Orlicz_f}(\X,\mho(\X),\tmu;\RR):=\{x\in L_0(\X,\mho(\X),\tmu;\RR)\mid\int\tmu f(\ab{x})<\infty\},
\label{Orlicz.class}
\end{equation}
while a \df{Morse--Transue--Krasnosel'ski\u{\i}--Ruticki\u{\i} space} \cite{Morse:Transue:1950,Krasnoselskii:Rutickii:1952,Krasnoselskii:Rutickii:1954} is defined by\rpktarget{EORLICZX}
\begin{equation}
	E_{\Orlicz_f}(\X,\mho(\X),\tmu;\RR):=\{x\in L_0(\X,\mho(\X),\tmu;\RR)\mid\forall\lambda<0\;\;\int\tmu f(\lambda\ab{x})<\infty\}.
\label{MTKR.space}
\end{equation}
Every MTKR space is a Banach space with respect to the MTNL norm $\n{\cdot}_{\Orlicz_f}$. By application of the Lebesgue dominated convergence theorem, it follows that 
\begin{equation}
	L_{\Orlicz_f}(\X,\mho(\X),\tmu;\RR)=\{x\in L_0(\X,\mho(\X),\tmu;\RR)\mid\exists\lambda>0\;\;\int\tmu f(\lambda\ab{x})<\infty\}.
\end{equation}
Moreover, $L_{\Orlicz_f}(\X,\mho(\X),\tmu;\RR)=\Span_\RR B_{\Orlicz_f}(\X,\mho(\X),\tmu;\RR)$, where
\begin{equation}
	B_{\Orlicz_f}(\X,\mho(\X),\tmu;\RR):=\{x\in L_{\Orlicz_f}(\X,\mho(\X),\tmu;\RR)\mid\int\tmu f(\ab{x})\leq1\}
\label{unit.Orlicz.ball}
\end{equation}
is a unit closed ball in $L_{\Orlicz_f}(\X,\mho(\X),\tmu;\RR)$ and a convex subset of $\tilde{L}_{\Orlicz_f}(\X,\mho(\X),\tmu;\RR)$. The definitions of $L_{\Orlicz_f}(\X,\mho(\X),\tmu;\RR)$,  $\tilde{L}_{\Orlicz_f}(\X,\mho(\X),\tmu;\RR)$,  $E_{\Orlicz_f}(\X,\mho(\X),\tmu;\RR)$, and $B_{\Orlicz_f}(\X,\mho(\X),\tmu;\RR)$, as well as their above properties, can be extended by replacing $\RR$ with $[-\infty,+\infty]$ \cite{Zaanen:1953,Luxemburg:1955}, and by replacing Orlicz function $f$ by an arbitrary Young function $f:\RR\ra[0,\infty]$ \cite{Rao:Ren:1991}. In the latter case, \eqref{Orlicz.modular} does not define a modular function, but it is anyway a norm on $L_0(\X,\mho(\X),\tmu;\RR)$. Hence, the corresponding Orlicz space $L_f(\X,\mho(\X),\tmu;\RR)$ can be defined as a completion of $L_0(\X,\mho(\X),\tmu;\RR)$ in the MTNL norm $\n{\cdot}_{\Orlicz_f}$, and the same holds for $L_f(\X,\mho(\X),\tmu;[-\infty,+\infty])$. In what follows, we will use indices $(\cdot)_{\Orlicz}$ to refer to any of these two constructions: the one based on an arbitrary Young function $\Orlicz$, and the one based on a modular function determined by an Orlicz function $\Orlicz$. In order to keep the notation concise, in what follows we will also omit symbols $\X$, $\mho(\X)$, and ($\RR$ or $[-\infty,+\infty]$), whenever they will not make any difference.

As an example, the space $L_\infty(\tmu)$ can be determined as an Orlicz space $L_{\Orlicz_\infty}(\tmu)$, where
\begin{equation}
	\Orlicz_\infty(x):=\left\{
	\begin{array}{ll}
	0 &:\;x\in[0,1[\\
	+\infty &:\;x>1\\
	\Orlicz_\infty(-x) &:\;x<0
	\end{array}
	\right.
\end{equation}
is a Young function but not an Orlicz function, while the spaces $L_p(\tmu)$ can be defined as Orlicz spaces $L_\Orlicz(\tmu)$ with $\Orlicz(x)$ given by any of the Orlicz functions: $\frac{\ab{x}^p}{p}$, $\ab{x}^p$, $\frac{x^p}{p}$, or $x^p$. 

In general, $E_\Orlicz(\tmu)$ is the largest vector subspace of $L_0(\tmu)$ contained in $\tilde{L}_\Orlicz(\tmu)$, $\tilde{L}_\Orlicz(\tmu)$ is a convex subset of $L_\Orlicz(\tmu)$, while $L_\Orlicz(\tmu)$ is the smallest vector subspace of $L_0(\tmu)$ containing $\tilde{L}_\Orlicz(\tmu)$. Moreover, for any Young function $\Orlicz$, $L_\Orlicz(\tmu)\subseteq (L_{\Orlicz^\Young}(\tmu))^\banach$ and \cite{Orlicz:1932,Orlicz:1936,Morse:Transue:1950,Krasnoselskii:Rutickii:1958,Rao:1964,Rao:1968}
\begin{equation}
	(E_\Orlicz(\tmu))^\banach\iso L_{\Orlicz^\Young}(\tmu),
\end{equation}
so that
\begin{equation}
	\forall x\in L_{\Orlicz^\Young}(\tmu)\;\;\exists! y\in L_{\Orlicz^\Young}(\tmu)\;\;\forall z\in E_\Orlicz(\tmu)\;\;x(z)=\int\tmu zy.
\end{equation}
If $\Orlicz$ and $\Orlicz^\Young$ are Young N-functions, then \cite{Krasnoselskii:Rutickii:1958,Rao:Ren:1991} (c.f. also \cite{Billik:1957})
\begin{align}
	xy&\in L_1(\tmu)\;\;\forall(x,y)\in\tilde{L}_\Orlicz(\tmu)\times\tilde{L}_{\Orlicz^\Young}(\tmu),\\
		\n{xy}_{L_1(\tmu)}=\int\tmu\ab{xy}&\leq\n{x}_\Orlicz\n{y}_{\Orlicz^\Young}\;\;\forall(x,y)\in L_\Orlicz(\tmu)\times L_{\Orlicz^\Young}(\tmu).
\end{align}
For any Young function $\Orlicz$, $\ran(\DDD^+\Orlicz(E_\Orlicz(\tmu)))\subseteq L_{\Orlicz^\Young}(\tmu)$. The space $\tilde{L}_\Orlicz(\tmu)$ is a vector space if{}f $E_\Orlicz(\tmu)=\tilde{L}_\Orlicz(\tmu)$ as sets, and in such case also $\tilde{L}_\Orlicz(\tmu)=L_\Orlicz(\tmu)$ holds. If a Young function $\Orlicz$ satisfies (local $\triangle_2$ condition and $\tmu(\X)<\infty$) or (global $\triangle_2$ condition and $\tmu(\X)=\infty$) then:
\begin{enumerate}
\item[i)] $\tilde{L}_\Orlicz(\tmu)$ is a vector space,
\item[ii)] $E_\Orlicz(\tmu)\iso L_\Orlicz(\tmu)$,
\item[iii)] $E_\Orlicz(\tmu)=\tilde{L}_\Orlicz(\tmu)=L_\Orlicz(\tmu)$,
\item[iv)] $(L_\Orlicz(\tmu))^\banach\iso L_{\Orlicz^\Young}(\tmu)$.
\end{enumerate}
If $\n{x}_\Orlicz=1$ $\limp$ $\int\tmu\Orlicz(x)=1$ $\forall x\in L_\Orlicz(\tmu)$ (which holds for example when $\tmu(\X)<\infty$, $\tmu$ is atomless, and $\Orlicz$ satisfies global $\triangle_2$ condition), and if $\Orlicz$ is strictly convex on $\RR^+$, then $(L_\Orlicz(\tmu),\n{\cdot}_\Orlicz)$ is strictly convex \cite{Turett:1976}. If $\tmu$ is atomless, then $(L_\Orlicz(\tmu),\n{\cdot}_\Orlicz)$ is uniformly convex if{}f \cite{Luxemburg:1955,Kaminska:1982} ($\tmu(\X)<\infty$, $\Orlicz$ is strictly convex on $\RR^+$, satisfies local $\triangle_2$ condition with $x_0$, and is uniformly convex\footnote{%
A function $f:X\ra\,]-\infty,+\infty]$ on a Banach space $X$ is called \df{uniformly convex} \cite{Levitin:Polyak:1966} if{}f $f\neq+\infty$ and there exists an increasing function $g:\RR^+\ra\,]-\infty,+\infty]$ with $g(0)=0$, such that
\begin{equation}
	f(\lambda x+(1-\lambda)y)\leq\lambda f(x)+(1-\lambda)f(y)-\lambda(1-\lambda)g(\n{x-y})\;\;\forall x,y\in\{z\in X\mid f(z)\neq+\infty\}\;\;\forall t\in[0,1].
\end{equation}
} for $x\geq x_0$) or ($\tmu(\X)=\infty$, $\Orlicz$ is uniformly convex on $\RR^+$, and satisfies global $\triangle_2$ condition). If both $\Orlicz$ and $\Orlicz^\Young$ are continuous, $\Orlicz$ satisfies global $\triangle_2$ condition, and 
\begin{equation}
	\forall\epsilon>0\;\;\exists\lambda(\epsilon),x(\epsilon)\in\RR\;\;\forall x\geq x(\epsilon)\;\;\;\DDD^+f((1+\epsilon)x)\geq\lambda(\epsilon)\Orlicz(x),
\end{equation}
then $(L_\Orlicz(\tmu),\n{\cdot}_\Orlicz)$ is uniformly Fr\'{e}chet differentiable. Finally, $L_\Orlicz(\tmu)$ and $L_{\Orlicz^\Young}(\tmu)$ are reflexive if (both $\Orlicz$ and $\Orlicz^\Young$ satisfy global $\triangle_2$ condition, and $\tmu(\X)=\infty$) or (both $\Orlicz$ and $\Orlicz^\Young$ satisfy local $\triangle_2$ condition, and $\tmu(\X)<\infty$) \cite{Orlicz:1936,Rao:Ren:1991}.

The noncommutative Orlicz spaces associated with the algebra $\BH$ of bounded operators on a Hilbert space $\H$ were implicitly introduced by Schatten \cite{Schatten:1960} as ideals in $\BH$ generated by the so-called symmetric gauge functions, and were studied in more details by Gokhberg and Kre\u{\i}n \cite{Gokhberg:Krein:1965} (see also \cite{Grothendieck:1955}). First explicit study of those ideals which are direct noncommutative analogues of Orlicz spaces is due to Rao \cite{Rao:1971,Rao:Ren:1991}, where $\Orlicz$ is assumed to be a  continuous modular function, $\Orlicz(\ab{x})$ for $x\in\BH$ is understood in terms of the spectral representation, an analogue of the MNTL norm reads
\begin{equation}
	\BH\ni x\mapsto\n{x}_\Orlicz:=\inf\left\{\lambda>0\mid\tr\left(\Orlicz\left(\frac{\ab{x}}{\lambda}\right)\right)\leq1\right\},
\end{equation}
while the corresponding noncommutative Orlicz space is defined as
\begin{equation}
	\schatten_\Orlicz(\H):=\{x\in\BH\mid\n{x}_\Orlicz<\infty\}.
\rpktarget{SCHATTENORLICZ}
\end{equation}
The generalisation of Orlicz spaces to semi-finite $W^*$-algebras $\N$ equipped with a faithful normal semi-finite trace $\tau$ were proposed by Muratov \cite{Muratov:1978,Muratov:1979}, Dodds, Dodds, and de Pagter \cite{Dodds:Dodds:dePagter:1989}, and Kunze \cite{Kunze:1990}. Two latter constructions are based on the results of Fack and Kosaki \cite{Fack:Kosaki:1986}. Given any $y\in\MMM(\N,\tau)$, the \df{rearrangement function} is defined as \cite{Grothendieck:1955} (see also \cite{Yeadon:1975})
\begin{equation}
	\rearr{y}{\tau}:[0,\infty[\,\ni t\mapsto\rearr{y}{\tau}(t):=\inf\{s\geq0\mid\tau(\pvm^{\ab{x}}(]s,+\infty[)\leq t\}\in[0,\infty].
\rpktarget{REARRANGEMENT}
\end{equation}
If $x\in\MMM(\N,\tau)^+$ and $f:[0,\infty[\ra[0,\infty[$ is a continuous nondecreasing function, then \cite{Fack:Kosaki:1986}
\begin{align}
	\tau(f(x))&=\int_0^\infty\dd t f(\rearr{x}{\tau}(t)),\label{FK.property.one}\\
	\rearr{f(x)}{\tau}(t)&=f(\rearr{x}{\tau}(t))\;\;\forall t\in\RR^+.\label{FK.property.two}
\end{align}
Using this result, Kunze \cite{Kunze:1990} defined a noncommutative Orlicz space associated with a pair $(\N,\tau)$ and an arbitrary Orlicz function $\Orlicz$ as\rpktarget{LORLICZNT}
\begin{equation}
	L_\Orlicz(\N,\tau):=\Span_\CC\{x\in\MMM(\N,\tau)\mid\tau(\Orlicz(\ab{x}))\leq1\},
\label{Kunze.nc.Orlicz}
\end{equation}
equipped with a quantum version of a MNTL norm,
\begin{equation}
	\n{\cdot}_\Orlicz:\MMM(\N,\tau)\ni x\mapsto\inf\{\lambda>0\mid\tau(\Orlicz(\lambda^{-1}\ab{x}))\leq1\},
\end{equation}
under which, as he proves, \eqref{Kunze.nc.Orlicz} is a Banach space. From linearity, it follows that
\begin{equation}
	L_\Orlicz(\N,\tau)=\{x\in\MMM(\N,\tau)\mid\exists\lambda>0\;\;\tau(\Orlicz(\lambda\ab{x}))<\infty\}.
\end{equation}
On the other hand,  Dodds, Dodds, and de Pagter \cite{Dodds:Dodds:dePagter:1989} defined (implicitly) a noncommutative Orlicz space associated with $(\N,\tau)$ and an Orlicz function $\Orlicz$ as\rpktarget{LORLICZNT.ZWEI}
\begin{equation}
	L_\Orlicz(\N,\tau):=\{x\in\MMM(\N,\tau)\mid\rearr{x}{\tau}\in L_\Orlicz(\RR^+,\mho_{\mathrm{Borel}}(\RR^+),\dd\lambda)\}.
\label{DDdP.nc.Orlicz}
\end{equation}
By means of \eqref{FK.property.one} and \eqref{FK.property.two}, these two definitions are equivalent\footnote{See a discussion in \cite{Labuschagne:Majewski:2008,Labuschagne:2013} of the case when continuity of $\Orlicz$ is relaxed to continuity on $[0,x_\Orlicz[$ and left continuity at $x_\Orlicz$ with $x_\Orlicz\neq+\infty$.}. Kunze \cite{Kunze:1990} showed that, for $\Orlicz$ satisfying global $\triangle_2$ condition,
\begin{align}
	L_\Orlicz(\N,\tau)&=\{x\in\MMM(\N,\tau)\mid\tau(\Orlicz(\ab{x}))<\infty\},\\
	L_\Orlicz(\N,\tau)&=E_\Orlicz(\N,\tau),\\
	(L_\Orlicz(\N,\tau))^\banach&\iso L_{\Orlicz^\Young}(\N,\tau),
\end{align}
where $E_\Orlicz(\N,\tau)$ is defined for any Orlicz function $\Orlicz$ as
\begin{equation}
	E_\Orlicz(\N,\tau):=\overline{\N\cap L_\Orlicz(\N,\tau)}^{\n{\cdot}_\Orlicz}.
\rpktarget{EORLICZNT}
\end{equation}
In \cite{Ayupov:Chilin:Abdullaev:2012} it is shown that if $\tau_1$ and $\tau_2$ are faithful normal semi-finite traces on a semi-finite $W^*$-algebra $\N$, and $\Orlicz$ is an Orlicz function satisfying global $\triangle_2$ condition, then $L_\Orlicz(\N,\tau_1)$ and $L_\Orlicz(\N,\tau_2)$ are isometrically isomorphic. Further analysis of the structure of $L_\Orlicz(\N,\tau)$ spaces in the context of modular function was provided by Sadeghi \cite{Sadeghi:2012}, who showed that the map $\MMM(\N,\tau)\ni x\mapsto \tau(\Orlicz(\ab{x}))\in[0,\infty]$ is a modular function for any Orlicz function $\Orlicz$. He also notes that the results of \cite{Chilin:Krygin:Sukochev:1992} and \cite{Dodds:Dodds:dePagter:1993} allow to infer, respectively, the uniform convexity and reflexivity of the spaces $(L_\Orlicz(\N,\tau),\n{\cdot}_\Orlicz)$ from the corresponding properties of the commutative Orlicz spaces $(L_\Orlicz(\RR^+,\mho_{\mathrm{Borel}}(\RR^+),\dd\lambda),\n{\cdot}_\Orlicz)$. This leads to conclusion that $(L_\Orlicz(\N,\tau),\n{\cdot}_\Orlicz)$ is: uniformly convex if $\Orlicz$ is uniformly convex and satisfies global $\triangle_2$ condition; reflexive if $\Orlicz$ and $\Orlicz^\Young$ satisfy global $\triangle_2$ condition.\footnote{The statement of the sufficient condition for reflexivity in Collorary 4.3 of \cite{Sadeghi:2012} is missing the requirement of the global $\triangle_2$ condition for $\Orlicz^\Young$.}

Al-Rashed and Zegarli\'{n}ski \cite{AlRashed:Zegarlinski:2007,AlRashed:Zegarlinski:2011} proposed a construction of a family of noncommutative Orlicz spaces associated with a faithful normal state on a countably finite $W^*$-algebra. Ayupov, Chilin and Abdullaev \cite{Ayupov:Chilin:Abdullaev:2012} proposed the construction of a family of noncommutative Orlicz spaces $L_\Orlicz(\N,\psi)$ for a semi-finite $W^*$-algebra $\N$, a faithful normal locally finite weight $\psi$, and an Orlicz function $\Orlicz$ satisfying global $\triangle_2$ condition. This construction extends Trunov's theory of $L_p(\N,\tau)$ spaces \cite{Trunov:1979,Trunov:1981,Zolotarev:1988}. Labuschagne \cite{Labuschagne:2013} provided a construction of the family of noncommutative Orlicz spaces $L_\Orlicz(\N,\psi)$ associated with an arbitrary $W^*$-algebra and a faithful normal semi-finite weight $\psi$. This construction uses Haagerup's approach to noncommutative integration and is quite complicated, losing the direct structural analogy between commutative and noncommutative Orlicz spaces.\footnote{In addition, various constructions of noncommutative Orlicz spaces associated with a Young function $\Orlicz(x)=\cosh(x)-1$ were given in \cite{Streater:2004:Orlicz,Jencova:2005,Streater:2008,Labuschagne:Majewski:2008,Streater:2009:book,Streater:2010:Banach,Jencova:2010,Majewski:Labuschagne:2014}.} We propose here an alternative construction, based on the Falcone--Takesaki approach to noncommutative integration. 

\begin{definition}
For an arbitrary $W^*$-algebra $\N$ and arbitrary Orlicz function $\Orlicz$, we define a \df{noncommutative Orlicz space} as a vector space\rpktarget{LORLICZN}
\begin{equation}
	L_\Orlicz(\N):=\{x\in\MMM(\core,\taucore)\mid\exists\lambda>0\;\;\taucore(\Orlicz(\lambda\ab{x}))<\infty\},\label{nc.Orlicz.space}
\end{equation}
equipped with the norm
\begin{equation}
	\n{\cdot}_\Orlicz:\MMM(\core,\taucore)\ni x\mapsto\inf\{\lambda>0\mid\taucore(\Orlicz(\lambda^{-1}\ab{x}))\leq1\}.\label{nc.Orlicz.norm}
\end{equation}
In addition, we define
\begin{equation}
	E_\Orlicz(\N):=\overline{\N\cap L_\Orlicz(\N)}^{\n{\cdot}_\Orlicz}.
\rpktarget{EORLICZN}
\end{equation}
\end{definition}

Because $\core$ is a semi-finite von Neumann algebra, while $\taucore$ is a faithful normal semi-finite trace on $\core$, all above results on the Banach space structure of $L_\Orlicz(\N,\tau)$ immediately apply to $L_\Orlicz(\N)$. In particular, if $\Orlicz$ satisfies global $\triangle_2$ condition, then $L_\Orlicz(\N)=E_\Orlicz(\N)$ and $L_\Orlicz(\N)^\banach\iso L_{\Orlicz^\Young}(\N)$. Moreover: if $\Orlicz$ is also uniformly convex, then $L_\Orlicz(\N)$ is uniformly convex and $L_{\Orlicz^\Young}(\N)$ is uniformly Fr\'{e}chet differentiable; if $\Orlicz^\Young$ also satisfies global $\triangle_2$ condition, then $L_\Orlicz(\N)$ is reflexive. A particular example of the space $E_\Orlicz(\N)$ is considered in \cite{Jencova:2005,Jencova:2010}.

For any choice of a normal semi-finite weight $\psi$ on $\N$, $\core$ can be represented as $\N\rtimes_{\sigma^\psi}\RR$ by means of \eqref{core.iso.map}. In such case our construction provides an alternative to Labuschagne's. They do not coincide, because the representation of a canonical trace $\taucore$ on $\N\rtimes_{\sigma^\psi}\RR$, given by $\taucore(\coreiso_\psi\cdot\coreiso_\psi^*)$, differs from Haagerup's trace $\tilde{\tau}_\psi$. 

\begin{proposition}
Every $*$-isomorphism $\varsigma:\N_1\ra\N_2$ of $W^*$-algebras gives rise to a corresponding isometric isomorphism $L_\Orlicz(\N_1)\ra L_\Orlicz(\N_2)$.
\end{proposition}
\begin{proof}
By the Falcone--Takesaki construction, and its composition \eqref{ctft.cat.diag} with Kosaki's construction, $\varsigma:\N_1\ra\N_2$ induces a $*$-isomorphism $\widetilde{\varsigma}:\core_1\ra\core_2$ of semi-finite von Neumann algebras and a mapping $(\core_1,\RR,\tilde{\sigma}^1,\taucore_1)\ra(\core_2,\RR,\tilde{\sigma}^2,\taucore_2)$ satisfying
\begin{align}
	\widetilde{\varsigma}\circ\tilde{\sigma}^1_t&=\tilde{\sigma}^2_t\circ\widetilde{\varsigma}\;\;\forall t\in\RR,\label{widetilde.covariance.insanity.one}\\
	\taucore_1&=\taucore_2\circ\widetilde{\varsigma}.\label{widetilde.covariance.insanity.two}
\end{align}
By Collorary 38 in \cite{Terp:1981}, every $*$-isomorphism of semi-finite von Neumann algebras satisfying \eqref{widetilde.covariance.insanity.one} and \eqref{widetilde.covariance.insanity.two} extends to a topological $*$-isomorphism of corresponding spaces of $\tau$-measurable operators affiliated with these algebras, and this extension preserves the property \eqref{widetilde.covariance.insanity.two}. The $*$-isomorphism $\widetilde{\varsigma}$ extends to $\bar{\varsigma}:\MMM(\core_1,\taucore_1)\ra\MMM(\core_2,\taucore_2)$ by $\bar{\varsigma}(\cdot)=u(\cdot)u^*$, where $u$ is a unitary operator implementing $\widetilde{\varsigma}(\cdot)=u(\cdot)u^*$. It remains to show that $\bar{\varsigma}$ is an isometric isomorphism. Using \eqref{FK.property.one}, we can rewrite \eqref{nc.Orlicz.norm} as
\begin{equation}
	\n{x}_\Orlicz=\inf\{\lambda>0\mid\int_0^\infty\dd t\Orlicz(\lambda^{-1}\rearr{\ab{x}}{\taucore}(t))<\infty\},
\end{equation}
where
\begin{equation}
	\rearr{\ab{x}}{\taucore}(t)=\inf\{s\geq0\mid\taucore(\pvm^{\ab{x}}(]s,+\infty[))\leq t\}.
\end{equation}
For $\MMM(\core_2,\taucore_2)=\MMM(\widetilde{\varsigma}(\core_1),\taucore_1\circ\widetilde{\varsigma}^{-1})$ and $x\in\MMM(\core_1,\taucore_1)$ we have
\begin{equation}
\taucore_1\circ\widetilde{\varsigma}^{\,-1}\left(\pvm^{\ab{\widetilde{\varsigma}(x)}}(s,+\infty[)\right)=\taucore_1\circ\widetilde{\varsigma}^{\,-1}\circ\widetilde{\varsigma}(\pvm^{\ab{x}}(]s,+\infty[))=\taucore_1(\pvm^{\ab{x}}(]s,+\infty[)).
\end{equation}
Hence, $\bar{\varsigma}:L_\Orlicz(\N_1)\ra L_\Orlicz(\N_2)$ is an isometric isomorphism.
\end{proof}

Denoting the category of noncommutative Orlicz spaces $L_\Orlicz(\N)$ with isometric isomorphisms by $\catname{ncL}_\Orlicz\catname{Iso}$, we conclude that our construction determines a functor
\begin{equation}
	\mathrm{ncL}_\Orlicz:\WsIso\ra\catname{ncL}_\Orlicz\catname{Iso}.
\end{equation}
Following the results of Sherman \cite{Sherman:2005}, we end this section with an interesting problem: for which Orlicz functions $\Orlicz$ there exists a functor from $\catname{ncL}_\Orlicz\catname{Iso}$ to the category of $W^*$-algebras with surjective Jordan $*$-isomorphisms?
